%% file: fichier_principal.tex
\documentclass[twoside,leqno,10pt,a4paper]{report}

\usepackage[a4paper]{meta-donnees}
\usepackage[latin1]{inputenc}
\usepackage[T1]{fontenc}
\usepackage[francais,english]{babel}
\usepackage{amsmath}
\usepackage{verbatim}
\usepackage{amsfonts}
\usepackage[all]{xy}
\usepackage{graphicx}
\usepackage{amssymb}
\usepackage{fancyhdr}
\usepackage{MnSymbol}
\usepackage{amsthm}
\usepackage{appendix}
\usepackage{vmargin}
\usepackage[nottoc, notlof, notlot]{tocbibind}

\usepackage{hyperref}

\newcommand{\leftexp}[2]{{\vphantom{#2}}^{#1}{#2}}
\newcommand{\incl}[1][r]{\ar@<-0.2pc>@{^(-}[#1] \ar@<+0.2pc>@{-}[#1]}

\input{commandes}

\theoremstyle{plain}
\newtheorem{theoreme}{Théorème}[section]
\newtheorem{lemme}[theoreme]{Lemme}
\newtheorem{proposition}[theoreme]{Proposition}
\newtheorem{corollaire}[theoreme]{Corollaire}
\newtheorem*{theoreme*}{Théorème A}
\newtheorem*{theoreme**}{Théorème B}
\newtheorem*{theoreme***}{Théorème C}
\newtheorem*{question*}{Question}
\newtheorem*{question1}{Question 1}
\newtheorem*{question2}{Question 2}
\newtheorem*{question3}{Question 3}
\newtheorem*{question4}{Question 4}
\newtheorem*{question5}{Question 5}
\newtheorem*{question6}{Question 6}
\newtheorem*{question7}{Question 7}
\newtheorem*{conjecture*}{Conjecture}
\newtheorem*{conjecture1}{Conjecture 1}
\newtheorem*{conjecture2}{Conjecture 2}
\newtheorem*{proposition*}{Proposition D}
\newtheorem*{proposition**}{Proposition E}
\newtheorem*{proposition***}{Proposition}

\theoremstyle{definition}
\newtheorem{definition}[theoreme]{Définition}

\newtheorem{conjecture}[theoreme]{Conjecture}

\newtheorem{notation}[theoreme]{Notation}
\theoremstyle{remark}
\newtheorem{remarque}[theoreme]{Remarque}

\begin{document}
\renewcommand{\proofname}{\textbf{Preuve:}}

\selectlanguage{francais}

\input{template_Couv_Univ_Grenoble}

\pagestyle{empty}
\newpage
\mbox{}
\newpage
\input{remerciements}

\pagestyle{fancy}
\newpage




\pagestyle{fancy}
\fancyhf{}
\lhead[\thepage]{}
\chead[Schémas de Hilbert invariants et théorie classique des invariants]{Ronan Terpereau}
\rhead[]{\thepage}

\tableofcontents

\input{intro}


\input{rappels_schema_Hilbert}

\input{differentes_situations}

\input{reduction}

\input{cas_complet_SLn}

\input{cas_rg_1}


\input{position_probleme_GLn}

\input{GL1}

\input{nilcone_GLn}

\input{GL2}

\input{GL3}

\input{symplectique_GLn_1}

\input{symplectique_GLn_2}


\input{position_probleme_On} 

\input{O2}

\input{O3}

\input{position_probleme_SOn}

\input{SO3}

\input{position_probleme_Sp2n}

\input{nilcone_Sp2n}

\input{Sp4}

\input{symplectique_On_1}

\input{symplectique_On_2}

\input{symplectique_Spn_1}

\input{symplectique_Spn_2}

\input{perspectives}

\appendix

\input{appendiceA1}

\input{appendiceA2}

\input{appendiceB}

\newpage

\bibliographystyle{alpha}
\bibliography{biblio}

\pagestyle{empty}
\newpage
\mbox{}
\newpage
\mbox{}
\newpage
\mbox{}

\begin{abstract}

Pour toute variété affine $W$ munie d'une opération d'un groupe réductif $G$, le schéma de Hilbert invariant est un espace de modules qui classifie les sous-schémas fermés de $W$, stables par l'opération de $G$, et dont l'algèbre affine est somme directe de $G$-modules simples avec des multiplicités finies préalablement fixées.

Dans cette thèse, on étudie d'abord le schéma de Hilbert invariant, noté $\HH$,  qui paramètre les sous-schémas fermés $GL(V)$-stables $Z$ de $W=V^{\oplus n_1} \oplus V^{* \oplus n_2}$ tels que $k[Z]$ est isomorphe à la représentation régulière de $GL(V)$ comme $GL(V)$-module. Si $\dim(V) \leq 2$, on montre que $\HH$ est une variété lisse, et donc que le morphisme de Hilbert-Chow $\gamma:\ \HH \rightarrow W/\!/G$ est une résolution des singularités du quotient $W/\!/G$. En revanche, si $\dim(V)=3$, on montre que $\HH$ est singulier. 
Lorsque $\dim(V)\leq 2$, on décrit $\HH$ par des équations et aussi comme l'espace total d'un fibré vectoriel homogène au dessus d'un produit de deux grassmanniennes. 

On se place ensuite dans le cadre symplectique en prenant $n_1=n_2$ et en remplaçant $W$ par la fibre en $0$ de l'application moment $\mu:\ W \rightarrow \End(V)$. On considère alors le schéma de Hilbert invariant $\HHm$ qui paramètre les sous-schémas contenus dans $\mu^{-1}(0)$. On montre que $\HHm$ est toujours réductible, mais que sa composante principale $\HHmp$ est lisse lorsque $dim(V)\leq 2$. Dans ce cas, le morphisme de Hilbert-Chow est une résolution (parfois symplectique) des singularités du quotient $\mu^{-1}(0)/\!/G$. Lorsque $dim(V) \leq 2$, on décrit $\HHmp$ comme l'espace total d'un fibré vectoriel homogène au-dessus d'une variété de drapeaux. 

Enfin, on obtient des résultats similaires lorsque l'on remplace $GL(V)$ par un autre groupe classique ($SL(V)$, $SO(V)$, $O(V)$, $Sp(V)$) que l'on fait opérer d'abord dans $W=V^{\oplus n}$, puis dans la fibre en zéro de l'application moment.\\

\vspace{2mm}

\begin{center} \textbf{Asbtract}  \end{center}

Let $W$ be an affine variety equipped with an action of a reductive group $G$. The invariant Hilbert scheme is a moduli space that classifies the $G$-stable closed subschemes of $W$ such that the affine algebra is the direct sum of simple $G$-modules with previously fixed finite multiplicities.

In this thesis, we first study the invariant Hilbert scheme, denoted $\HH$. It parametrizes the $GL(V)$-stable closed subschemes $Z$ of $W=V^{\oplus n_1} \oplus V^{* \oplus n_2}$ such that $k[Z]$ is isomorphic to the regular representation of $GL(V)$ as $GL(V)$-module. If $\dim(V) \leq 2$, we show that $\HH$ is a smooth variety, so that the Hilbert-Chow morphism $\gamma:\ \HH \rightarrow W/\!/G$ is a resolution of singularities of the quotient $W/\!/G$. However, if $\dim(V)=3$, we show that $\HH$ is singular. When $\dim(V) \leq 2$, we describe $\HH$ by equations and also as the total space of a homogeneous vector bundle over the product of two Grassmannians.

Then we consider the symplectic setting by letting $n_1=n_2$ and replacing $W$ by the zero fiber of the moment map $\mu:\ W \rightarrow \End(V)$. We study the invariant Hilbert scheme $\HHm$ that parametrizes the subschemes included in $\mu^{-1}(0)$. We show that $\HHm$ is always reducible, but that its main component $\HHmp$ is smooth if $\dim(V) \leq 2$. In this case, the Hilbert-Chow morphism is a resolution of singularities (sometimes a symplectic one) of the quotient $\mu^{-1}(0)/\!/G$. When $\dim(V)=3$, we describe $\HHmp$ as the total space of a homogeneous vector bundle over a flag variety.

Finally, we get similar results when we replace $GL(V)$ by some other classical group ($SL(V)$, $SO(V)$, $O(V)$, $Sp(V)$)  acting first on $W=V^{\oplus n}$, then on the zero fiber of the moment map.\\

\vspace{2mm}

\begin{center} \textbf{Mots-clés}  \end{center}

\noindent Schéma de Hilbert invariant, résolution des singularités, théorie des invariants, variété déterminantielle, orbite nilpotente, singularité symplectique, fibré vectoriel homogène.

\vspace{1.5mm}

\begin{center} \textbf{Classification mathématique}  \end{center}

\begin{center} 13A50 14E15 14L24 14L30 14M12 14M17 20G05 20G20 53D20 \end{center}

\end{abstract}



\end{document}

%% file: commandes.tex
\newcommand{\ZZ}{\mathcal{Z}}
\newcommand{\PPP}{\mathcal{P}} 
\newcommand{\XX}{\mathcal{X}}
\newcommand{\RRR}{\mathcal{R}} 
\newcommand{\NNN}{\mathcal{N}}
\newcommand{\FF}{\mathcal{F}}

\newcommand{\II}{\mathcal{I}} 
 
\newcommand{\OO}{\mathcal{O}}
\newcommand{\HH}{\mathcal{H}}

\newcommand{\HHm}{\mathcal{H}_{\mathrm{s}}}
\newcommand{\HHmx}{\mathcal{H}_{\mathrm{s}}^I}
\newcommand{\HHmy}{\mathcal{H}_{\mathrm{s}}^{I\!I}}
\newcommand{\HHmxp}{\mathcal{H}_{\mathrm{s}}^{I,\mathrm{prin}}}
\newcommand{\HHmyp}{\mathcal{H}_{\mathrm{s}}^{I\!I,\mathrm{prin}}}
\newcommand{\HHmgz}{\text{$G^0$-}\mathcal{H}_{\mathrm{s}}}
\newcommand{\HHmg}{\text{$G$-}\mathcal{H}_{\mathrm{s}}}
\newcommand{\HHmgzp}{\text{$G^0$-}\mathcal{H}_{\mathrm{s}}^{\mathrm{prin}}}
\newcommand{\HHmgp}{\text{$G$-}\mathcal{H}_{\mathrm{s}}^{\mathrm{prin}}}

\newcommand{\HHr}{\mathcal{H}^{\mathrm{red}}}

\newcommand{\HHp}{\mathcal{H}^{\mathrm{prin}}}
\newcommand{\HHmp}{\mathcal{H}_{\mathrm{s}}^{\mathrm{prin}}}

\newcommand{\BB}{\mathcal{B}}
\newcommand{\MM}{\mathcal{M}}

\newcommand{\NN}{\mathbb{N}}
\newcommand{\Gm}{\mathbb{G}_m}
\newcommand{\Aff}{{\mathbb{A}}_{k}^{1}}

\newcommand{\PP}{\mathbb{P}}
\newcommand{\ZZZ}{\mathbb{Z}}

\renewcommand{\gg}{\mathfrak{g}} 
 
\newcommand{\hh}{\mathfrak{h}} 
\newcommand{\pp}{\mathfrak{p}} 
\newcommand{\GG}{\mathcal{G}} 

\newcommand{\Irr}{\mathrm{Irr}}

\newcommand{\Sing}{\mathrm{Sing}}
\newcommand{\Pic}{\mathrm{Pic}}
\newcommand{\Gr}{\mathrm{Gr}}
\newcommand{\IG}{\mathrm{IGr}}
\newcommand{\OG}{\mathrm{OGr}}
\newcommand{\Hom}{\mathrm{Hom}}

\newcommand{\Univ}{\mathrm{Univ}}
\newcommand{\tr}{\mathrm{tr}}
\newcommand{\Ker}{\mathrm{Ker}}
\renewcommand{\Im}{\mathrm{Im}}
\newcommand{\Ind}{\mathrm{Ind}}
\newcommand{\Aut}{\mathrm{Aut}}
\newcommand{\Stab}{\mathrm{Stab}} 
\newcommand{\End}{\mathrm{End}}
\newcommand{\Mor}{\mathrm{Mor}}
\newcommand{\Hilb}{\mathrm{Hilb}}
\newcommand{\rg}{\mathrm{rg}}
\newcommand{\Spec}{\mathrm{Spec}}
\newcommand{\multi}{\mathrm{multi}}

\renewcommand{\labelitemi}{$\bullet$}

%% file: template_Couv_Univ_Grenoble.tex




\Sethpageshift{23mm}   
\Setvpageshift{11mm}   

\Specialite{Mathématiques}
\Arrete{7 août 2006}
\Auteur{Ronan TERPEREAU}
\Directeur{Michel BRION}
\Laboratoire{l'Institut Fourier}
\EcoleDoctorale{l'Ecole Doctorale MSTII}         
\Titre{Schémas de Hilbert invariants et théorie classique des invariants}
\Depot{5 novembre 2012}




\Jury{

\UGTDirecteur{M. Michel BRION}{Directeur de recherche CNRS, Université Grenoble I}       
\UGTRapporteur{M. Hanspeter KRAFT}{Professeur, Université de Bâle}     
\UGTExaminateur{M. Manfred LEHN}{Professeur, Université de Mayence}     
\UGTExaminateur{M. Laurent MANIVEL}{Directeur de recherche CNRS, Université Grenoble I}     
\UGTRapporteur{M. Christoph SORGER}{Professeur, Université de Nantes}      
\UGTExaminateur{M. Mikhail ZAIDENBERG}{Professeur, Université Grenoble I}    

}

\MakeUGthesePDG    

%% file: remerciements.tex
\mbox{}
\vspace{20mm}
\begin{center}  \Huge \textbf{Remerciements} \end{center}
\vspace{10mm}

Tout d'abord, je tiens à exprimer toute ma gratitude à Michel Brion pour m'avoir proposé de travailler sur un sujet aussi riche et intéressant. Michel, pendant ces trois années tu as toujours été présent pour m'aider et me motiver lorsque cela était nécessaire. Tu m'as donné beaucoup de ton temps, et tu m'as souvent fait profiter de ton intuition et de tes profondes connaissances des mathématiques. Pour tout cela, je te remercie mille fois. Merci aussi pour tous les conseils que tu as pu me donner, et pour toutes les personnes intéressées par mon sujet que tu m'as permis de rencontrer. \\

Un grand merci à Hanspeter Kraft et Christoph Sorger pour avoir accepté de rapporter ma thèse. Merci également à Manfred Lehn, Laurent Manivel et Mikhail Zaidenberg pour avoir accepté de figurer dans mon jury. Je suis très honoré de l'intérêt que vous portez à mes travaux.   \\

Je souhaite exprimer ma reconnaissance aux personnes avec qui j'ai pu discuter de sujets connexes à celui de ma thèse: Jonas Budmiger, Tanja Becker,  Stavros Papadakis et Bart Van Steirteghem. Merci aussi à ceux qui m'ont invité à parler de mes travaux: Stéphanie Cupit-Foutou, Ivan Pan, Guido Pezzini, Alvaro Rittatore et les organisateurs de GAeL XIX.\\

Cette thèse a été préparée au sein de l'Institut Fourier. Je tiens à remercier le personnel administratif de l'Institut Fourier qui m'a offert un très bon cadre de travail.
Je remercie aussi tous les thésards et post-docs de l'Institut pour l'ambiance sympathique (et studieuse!) qui y règne. Un grand merci tout particulier à Claudia, Clélia, Christian, Guenaëlle, Junyan, Kévin, Mateusz, Mickaël et Roland. \\

Un merci du fond du coeur à mes amis pour tous les moments passés en votre compagnie pendant ces vingt-six dernières années. Merci pour tous les voyages, les week-ends et les soirées passées en votre compagnie. \\

Enfin, je tiens à remercier très chaleureusement ma famille pour son soutien constant et pour m'avoir toujours encouragé à faire ce qui me plaisait dans la vie: mes parents Evelyne et Elian, ma soeur Cindy et ma grand-mère Josiane. Maman, papa, sans vous je ne serais pas arrivé là où j'en suis et je sais que j'ai beaucoup de chance d'avoir des parents aussi formidables que vous l'êtes. \\

Pour finir, un immense merci à Marie, ma tendre moitié, qui a toujours été à mes côtés pendant cette thèse. Merci pour ta gentillesse et ta bonne humeur, et surtout, merci de toujours croire en moi. \\

%% file: intro.tex
\chapter*{Introduction} 
\addcontentsline{toc}{chapter}{Introduction} 

\section*{Contexte}

La motivation de ce mémoire provient d'une construction classique de résolution des singularités de certaines variétés qui sont des quotients par des groupes finis. Plus précisément, soient $G$ un groupe fini, $W$ une représentation linéaire de dimension finie de $G$ et $\nu:\ W \rightarrow W/G$ le morphisme de passage au quotient. En général, le morphisme $\nu$ n'est pas plat et le quotient $W/G$ est singulier. On peut construire une platification universelle de $\nu$ en considérant le $G$-schéma de Hilbert $\Hilb^G(W)$. Celui-ci paramètre les sous-schémas fermés $G$-stables de $W$ dont l'anneau des coordonnées est isomorphe à la représentation régulière de $G$ comme $G$-module. On a le diagramme commutatif suivant:
$$\xymatrix{
    \XX \ar[r]^(0.3){\pi}  \ar[d]_{p}   & \Hilb^G(W)  \ar[d]^{\gamma}  \\
      W  \ar[r]_(0.4){\nu}              &  W/G 
    }$$
où $\pi$ est la famille universelle (plate par définition), $p$ est la projection naturelle et $\gamma$ est le morphisme \textit{de Hilbert-Chow} qui à un sous-schéma $Z$ de $W$ associe le point $Z/G$ de $W/G$. Le morphisme de Hilbert-Chow est propre, et c'est un isomorphisme au dessus de l'ouvert de platitude $U$ de $\nu$ (celui-ci consiste en les $G$-orbites dans $W$ dont le groupe d'isotropie est trivial). On définit la composante principale de $\Hilb^G(W)$ par 
$$\Hilb^G(W)^{\mathrm{prin}}:=\overline{\gamma^{-1}(U)}.$$ 
La restriction $\gamma:\ \Hilb^G(W)^{\mathrm{prin}} \rightarrow W/G$ est birationnelle et propre. Il est donc naturel de se demander si le morphisme $\gamma$, éventuellement restreint à $\Hilb^G(W)^{\mathrm{prin}}$, est une résolution des singularités de $W/G$. Autrement dit, on souhaite connaître les couples $(G,W)$ tels que $\Hilb^G(W)$ soit une variété lisse. Et lorsque $\Hilb^G(W)$ est singulier, on veut savoir si sa composante principale est lisse. La réponse n'est pas connue en général, mais on a tout de même (entre autres) les résultats suivants:
\begin{itemize}
\item Si $\dim(W) \leq 2$, alors $\Hilb^G(W)$ est une variété lisse. En particulier, si $W={\mathbb{A}}_{k}^{2}$ et si $G \subset SL(W)$, alors Ito et Nakamura ont montré que $\gamma$ est l'unique résolution crépante de $W/G$ (voir \cite{IN1,IN2}). Signalons que cette construction joue un rôle clé dans la correspondance de McKay.
\item Si $\dim(W)=3$ et $G \subset SL(W)$, Bridgeland, King et Reid ont montré par des méthodes homologiques que $\Hilb^G(W)$ est encore une variété lisse, et que $\gamma$ est une résolution crépante de $W/G$ (voir \cite{BKR}).
\item Si $\dim(W)=4$, alors $\Hilb^G(W)$ peut être pathologique. Par exemple, si $G \subset SL_2(k)$ est le groupe tétrahédral binaire et si $W$ est la somme directe de deux copies de la représentation standard, alors Lehn et Sorger ont montré que $\Hilb^G(W)^{\mathrm{prin}}$ est lisse, mais que $\Hilb^G(W)$ est réductible (voir \cite{LS}).\\      
\end{itemize}

Par la suite, on travaille sur un corps $k$ algébriquement clos de caractéristique $0$, et on considère la situation plus générale où $G$ est un groupe réductif, $W$ est une représentation linéaire de dimension finie de $G$ et $\nu:\ W \rightarrow W/\!/G$ est le morphisme de passage au quotient. Ici $W/\!/G$ désigne le quotient catégorique, c'est-à-dire le spectre de l'algèbre des invariants $k[W]^G$; celle-ci est intègre et de type fini (car $G$ est réductif), donc $W/\!/G$ est une variété irréductible. En général, comme c'était déjà le cas lorsque $G$ est fini, le quotient $W/\!/G$ est singulier et le morphisme $\nu$ n'est pas plat. On dispose toujours d'une platification universelle de $\nu$; celle-ci est donnée par le schéma de Hilbert invariant construit par Alexeev et Brion (\cite{AB, Br}). On en rappelle brièvement la définition (voir la section \ref{generalitéesHilbert} pour plus de détails). On note $\Irr(G)$ l'ensemble des classes d'isomorphisme des $G$-modules irréductibles et $h$ une fonction de $\Irr(G)$ dans $\NN$. Une telle fonction $h$ est appelée fonction de Hilbert. Le schéma de Hilbert invariant $\Hilb_h^G(W)$ paramètre les sous-schémas fermés $G$-stables $Z$ de $W$, tels que 
$$k[Z] \cong \bigoplus_{M \in \Irr(G)} M^{\oplus h(M)}$$ 
comme $G$-module. On note $$\HH:=\Hilb_h^G(W)$$ 
pour alléger les notations. Si $h(V_0)=1$, où $V_0$ est la représentation triviale de $G$, alors on a le diagramme suivant:    
$$\xymatrix{
    \XX \ar[r]^{\pi}  \ar[d]_{p}   & \HH  \ar[d]^{\gamma}  \\
      W  \ar[r]_(0.4){\nu}              &  W/\!/G 
    }$$
où $\pi$ est la famille universelle, $p$ est la projection naturelle et $\gamma$ est un morphisme propre qui fait commuter le diagramme. Si de plus on fixe $h=h_W$, la fonction de Hilbert de la fibre générique de $\nu$, alors le morphisme \textit{de Hilbert-Chow} $\gamma:\ \HH \rightarrow W/\!/G$ est un isomorphisme au dessus de l'ouvert de platitude $U$ du morphisme $\nu$. En particulier, $\nu$ est plat si et seulement si $\HH \cong W/\!/G$. Par analogie avec le cas des groupes finis, on définit la composante principale de $\HH$ par 
$$\HHp:=\overline{\gamma^{-1}(U)}.$$ 
La restriction $\gamma:\ \HHp \rightarrow W/\!/G$ est birationnelle et propre, d'où la 

\begin{question*}
Dans quels cas le morphisme $\gamma$, éventuellement restreint à la composante principale, est-il une résolution des singularités de $W/\!/G$?
\end{question*}   

Le schéma de Hilbert invariant a été très étudié ces dernières années (\cite[...]{budmiger,cupit,jansou,jansouRessayre,PaBart}), mais surtout en relation avec des problèmes de classification. Lorsque le groupe $G$ est infini, la question précédente est totalement ouverte!

\section*{Résultats}

Dans ce mémoire, on apporte des éléments de réponse à la question ci-dessus lorsque $G$ est un groupe classique et que $W$ est une représentation classique de $G$. Nous décrivons explicitement $\HH$ lorsque:
\begin{itemize} \renewcommand{\labelitemi}{$\bullet$}
\item $G=SL_n(k)$ et $W:=V^{\oplus n'}$ est la somme directe de $n'$ copies de la représentation standard (théorème \ref{SLLn}),
\item $G=GL_n(k)$, $n \leq 2$, et $W:=V^{\oplus n_1} \oplus V^{* \oplus n_2}$ est la somme directe de $n_1$ copies de la représentation standard et de $n_2$ copies de la représentation duale (théorèmes \ref{Hcasn11111} et \ref{casn2}),
\item $G=O_n(k)$ ou $Sp_{2n}(k)$, $n \leq 2$, et $W$ est la somme directe de $n'$ copies de la représentation standard (théorèmes \ref{casn2O2} et \ref{casn2Sp4}).
\end{itemize}
Dans chacun de ces cas, on donne d'une part des équations pour $\HH$, et on réalise d'autre part $\HH$ comme l'espace total d'un fibré homogène sur une grassmannienne (ou sur un produit de deux grassmanniennes lorsque $G=GL_n(k)$). Notre description implique le

\begin{theoreme*} 
Le schéma de Hilbert invariant $\HH$ est une variété lisse dans les cas suivants:
\begin{itemize} \renewcommand{\labelitemi}{$\bullet$}
\item $G=SL_n(k)$ et $n$ quelconque,
\item $G=GL_n(k)$ et $n \leq 2$,
\item $G=O_n(k)$ et $n \leq 2$,
\item $G=Sp_{2n}(k)$ et $n \leq 2$.
\end{itemize}
\end{theoreme*}   
Il convient d'ajouter à cette liste les cas "triviaux" où le morphisme $\nu$ est plat et le quotient $W/\!/G$ est lisse (voir les corollaires \ref{cas_facile}, \ref{cas_facileOn}, \ref{cas_facileSOn} et \ref{cas_facileSpn}). 
Cependant, nous montrons qu'en général $\gamma:\ \HH \rightarrow W/\!/G$ n'est pas une résolution. 

\begin{theoreme**}  \emph{[théorèmes \ref{casn3}, \ref{casn3O3} et \ref{casn3SO3}]}\\
Le schéma de Hilbert invariant $\HH$ est singulier dans les cas suivants:
\begin{itemize} \renewcommand{\labelitemi}{$\bullet$}
\item $G=GL_3(k)$ et $n_1,n_2 \geq 3$,
\item $G=O_3(k)$ et $n' \geq 3$,
\item $G=SO_3(k)$ et $n'=3$.
\end{itemize}
\end{theoreme**} 

Dans la section \ref{lesdiffsituations}, on définit $G'$ un sous-groupe algébrique réductif de $\Aut^G(W)$. L'existence d'une opération de $G'$ dans $W$, $W/\!/G$ et $\HH$, telle que tous les morphismes que l'on manipule soient $G'$-équivariants, joue un rôle essentiel dans ce mémoire. Par exemple, pour décrire l'ouvert de platitude du morphisme de passage au quotient $\nu:\ W \rightarrow W/\!/G$ pour les différents groupes classiques, il est (presque) suffisant de connaître la dimension de la fibre de $\nu$ en un point de chaque orbite de $G'$. De même, pour déterminer si $\HH$ est lisse ou singulier, il suffit (par un argument de semi-continuité) de déterminer l'espace tangent de $\HH$ en un point de chaque orbite fermée.

Tous les quotients étudiés dans ce mémoire sont munis d'une stratification naturelle 
\begin{equation*} 
(*) \ \, \ \, \ X_0 \subset X_1 \subset \cdots \subset X_N=W/\!/G,
\end{equation*} 
où chaque $X_i$ est l'adhérence d'une orbite de $G'$. Lorsque $\HH$ est lisse (c'est-à-dire dans les cas donnés par le théorème A), on montre que le morphisme de Hilbert-Chow $\gamma$ s'identifie à la composition des éclatements successifs de certaines strates de $W/\!/G$. Donnons un exemple pour rendre les choses plus claires. Soient $G=GL_2(k)$ et $n_1,\,n_2>2$, alors $W/\!/G$ est la variété déterminantielle $X_2:=\{ M \in \MM_{n_2,n_1}(k)\ |\ \rg(M) \leq 2\}$. Pour résoudre les singularités de $X_2$, on éclate $X_0=\{0\}$, puis on éclate la transformée stricte de $X_1:=\{ M \in \MM_{n_2,n_1}(k)\ |\ \rg(M) \leq 1\}$. La variété obtenue est lisse et nous montrons qu'elle est isomorphe à $\HH$. De plus, le morphisme de Hilbert-Chow $\gamma$ s'identifie à la composition de ces deux éclatements.

On s'est également intéressé aux schémas de Hilbert invariants liés aux réductions symplectiques. Plus précisément, soient $G$ un groupe classique et 
$$W:=V^{\oplus d} \oplus V^{* \oplus d}$$ 
la somme directe de $d$ copies de la représentation standard et de $d$ copies de la représentation duale. La variété $W$ est naturellement munie d'une forme symplectique qui est préservée par $G$. On note $\mu:\ W \rightarrow \gg^*$, où $\gg$ est l'algèbre de Lie de $G$, \textit{l'application moment} pour l'opération de $G$ dans $W$, et on définit la réduction symplectique de $W$ par 
$$W/\!/\!/G:=\mu^{-1}(0)/\!/G.$$ 
Dans tous nos exemples le quotient $W/\!/G$ n'est pas symplectique (sauf lorsqu'il est trivial). En revanche, les composantes irréductibles de $W/\!/\!/G$ (munies de leur structure réduite) sont toujours symplectiques. Lorsque $G$ est un groupe fini, de "bonnes" propriétés géométriques pour $W/\!/\!/G$ donnent de "bonnes" propriétés géométriques pour $V^{\oplus d}/\!/G$, et on espère que cela reste vrai lorsque $G$ est un groupe réductif arbitraire. On a par exemple la     

\begin{conjecture*} [Kaledin, Lehn, Sorger]
Si chaque composante irréductible de $W/\!/\!/G$, munie de sa structure réduite, admet une résolution symplectique, alors le quotient $V^{\oplus d}/\!/G$ est lisse.
\end{conjecture*}

Tous les exemples traités dans ce mémoire vérifient cette conjecture. Par ailleurs, lorsque $W/\!/\!/G$ est réduit et irréductible, on a encore le morphisme de Hilbert-Chow 
$$\gamma:\ \Hilb_{h_s}^G(\mu^{-1}(0)) \rightarrow W/\!/\!/G,$$ 
où $h_s$ est la fonction de Hilbert de la fibre générique du morphisme $\mu^{-1}(0) \rightarrow \mu^{-1}(0)/\!/G$. On note 
$$\HHm:=\Hilb_{h_s}^G(\mu^{-1}(0))$$ 
pour alléger les notations (l'indice \textit{s} est pour \textit{symplectique}). Comme précédement, on se demande dans quels cas le morphisme de Hilbert-Chow $\gamma$, éventuellement restreint à la composante principale, est une résolution des singularités de $W/\!/\!/G$. \\
Lorsque $G=\Gm$, $O_2(k)$ ou $Sp_2(k)$, nous décrivons explicitement $\HHm$ par des équations. Mentionnons que le cas $G=Sp_2(k)$ et $d=3$ a été traité par Becker dans \cite{Tanja2}. Puis, pour certaines valeurs de $n$ et de $d$, nous décrivons la composante principale $\HHmp$ comme l'espace total d'un fibré homogène sur une variété de drapeaux. Notre description implique le

\begin{theoreme***} \emph{[corollaires \ref{symplisssse}, \ref{symplisssseOn} et \ref{resuss}]}\\
La composante principale $\HHmp$ est lisse dans les cas suivants:
\begin{itemize} \renewcommand{\labelitemi}{$\bullet$}
\item $G=GL_n(k)$ et 
    \begin{itemize}
    \item $d$ est pair et $d \leq n+1$,
    \item $n=1$ et $d$ est quelconque,
    \item $n=2$ et $d \neq 3$, 
    \end{itemize}
\item $G=O_n(k)$ et  
    \begin{itemize}
    \item $d \leq \frac{n+1}{2}$,
    \item $n \leq 2$ et $d$ est quelconque, 
    \end{itemize}
\item $G=Sp_{2n}(k)$ et  
    \begin{itemize}
    \item $n=1$ et $d \geq 3$, 
    \item $n=2$ et $d \geq 5$. 
    \end{itemize}
\end{itemize}
\end{theoreme***}

\noindent A cette liste s'ajoutent les cas "triviaux" où $\HHm \cong W/\!/\!/G$. Nous avons également obtenu la

\begin{proposition*}
\emph{[proposition \ref{noIrred}]}\\ 
Si $G=GL_n(k)$ et $d \geq 2n$, le schéma $\HHm$ est toujours réductible.
\end{proposition*}

Toutes les réductions symplectiques $W/\!/\!/G$ qui apparaissent dans ce mémoire sont des adhérences d'orbites nilpotentes (ou bien des revêtements doubles d'adhérences d'orbites nilpotentes lorsque $G=SO_n(k)$). En particulier, ce sont des variétés symplectiques et elles admettent une stratification naturelle analogue à (*).

Pour démontrer les théorèmes A, B et C, on montre dans chaque cas l'existence d'un \textit{principe de réduction} (section \ref{princreduction}) qui permet de réaliser $\HH$ comme l'espace total d'un fibré $G'$-homogène sur une base projective dont la fibre est isomorphe à un schéma de Hilbert invariant plus simple que $\HH$. Par exemple, pour déterminer $\HH$ lorsque $G=SL_n(k)$ et $n' \geq n$, il suffit de le déterminer pour $n'=n$. On n'a malheureusement pas pu obtenir de tel \textit{principe de réduction} pour $SO_n(k)$. L'ingrédient clé qui fait fonctionner \textit{le principe de réduction} est la

\begin{proposition**} \emph{[proposition \ref{morphismegrass}]}\\ 
Pour tout $M \in \Irr(G)$, il existe un $G'$-sous-module $F_M \subset \Hom^G(M,k[W])$ de dimension finie qui engendre $\Hom^G(M,k[W])$ comme $k[W]^G,G'$-module, et il existe un morphisme de schémas $G'$-équivariant
$${\delta}_{M}:\ \HH \rightarrow \Gr(h_W(M),F_{M}^{*}).$$  
\end{proposition**} 

Une fois que l'on a effectué cette étape de réduction, on fixe $B'$ un sous-groupe de Borel de $G'$ et on cherche les points fixes de $B'$ dans $\HH$ (qui sont en nombre fini dans tous mes exemples). Si $\HH$ admet un unique point fixe $z_0$, alors $\HH$ est connexe et $z_0$ appartient à la composante principale $\HHp$. Sinon, on utilise la théorie des bases de Gröbner pour déterminer les points fixes de $B'$ qui sont dans $\HHp$. On calcule alors la dimension de l'espace tangent à $\HH$ en chacun de ces points fixes, et on la compare avec la dimension de $\HHp$. Si toutes ces dimensions coïncident, alors nécessairement $\HHp$ est une composante connexe lisse de $\HH$. Sinon, plusieurs pathologies sont possibles. Le schéma $\HH$ peut être réductible, non réduit ou encore ses composantes irréductibles munies de leur structure réduite peuvent être singulières. Déterminer dans laquelle de ces situations l'on se trouve est un problème difficile. Je n'ai généralement pas été en mesure de le résoudre.

Il doit être souligné que la proposition E joue un rôle de premier plan dans tout ce mémoire, et qu'elle sera sans doute très utile à l'étude de futurs exemples de schémas de Hilbert invariants. On peut montrer qu'il existe un sous-ensemble fini $\mathcal{E}$ de $\Irr(G)$ tel que le morphisme
$$\gamma \times \prod_{M \in \mathcal{E}} \delta_M:\ \HH \longrightarrow W/\!/G \times \prod_{M \in \mathcal{E}} \Gr(h_W(M),F_M^*)$$
soit une immersion fermée (c'est une conséquence de la construction de $\HH$ comme sous-schéma fermé du schéma de Hilbert multigradué). Ce constat suggère de choisir une représentation irréductible "simple" $M_1$ de $G$ et de regarder si $\gamma \times \delta_{M_1}$ est une immersion fermée de $\HH$. Si tel est le cas, alors il faut ensuite identifier l'image. Sinon, on on choisit une autre représentation irréductible "simple" $M_2$ et on regarde si $\gamma \times \delta_{M_1} \times \delta_{M_2}$ est une immersion fermée de $\HH$. Cette procédure s'arrête en un nombre fini d'étapes et permet de construire une immersion fermée de $\HH$ qui soit explicite et aussi simple que possible. 

Une partie "Questions ouvertes" figure à la fin de ce mémoire. Elle permet de faire le point sur les résultats que l'on a obtenus et sur les problèmes qu'il reste à résoudre pour espérer une meilleure compréhension des schémas de Hilbert invariants. Ce mémoire s'achève par les annexes A et B. Dans l'annexe A, on montre d'une part que les résolutions de $W/\!/G$ données par le morphisme de Hilbert-Chow lorsque $\HH$ est lisse (voir le théorème A) ne sont jamais crépantes; on étudie d'autre part les liens entre la restriction du morphisme de Hilbert-Chow $\gamma:\ \HHmp \rightarrow W/\!/\!/G$ et les résolutions symplectiques de la réduction symplectique $W/\!/\!/G$. L'annexe B est consacrée à l'étude de certaines propriétés géométriques des fibres de la famille universelle $\pi:\ \XX \rightarrow \HH$. \\

\newpage

%% file: rappels_schema_Hilbert.tex
\chapter{Le schéma de Hilbert invariant} \label{rappels_Hilbert}

\section{Généralités}  \label{generalitéesHilbert}

L'article d'exposition \cite{Br} fournit une introduction détaillée aux schémas de Hilbert invariants. Dans cette section on donne quelques définitions et propriétés fondamentales de ces schémas. 

On suppose que le lecteur a des notions de base de la théorie des schémas (\cite{Ha}), la théorie des groupes algébriques linéaires (\cite{Bo}) et la théorie des représentations des groupes algébriques réductifs (\cite{FH}). On travaille sur un corps $k$ algébriquement clos de caractéristique $0$ et tous les schémas considérés par la suite sont supposés séparés et de type fini sur $k$. Pour nous, une variété sera toujours un schéma réduit et irréductible. Soient $G$ un groupe algébrique réductif et $\Irr(G)$ l'ensemble des classes d'isomorphisme des $G$-modules irréductibles. On note $V_0$ la représentation triviale de $G$. Si $N$ est un $G$-module rationnel, alors on a une décomposition de $N$ en $G$-modules irréductibles:
\begin{equation} \label{decompoGmod}
N \cong \bigoplus_{M \in \Irr(G)} N_{(M)} \otimes M
\end{equation}
où $N_{(M)}:=\Hom^G(M,N)$ est l'espace vectoriel des morphismes $G$-équivariants de $M$ dans $N$. Le $G$-module $N_{(M)}\otimes M$ s'appelle la composante isotypique de $N$ associée à $M$ et $\dim(N_{(M)})$ est la multiplicité de $M$ dans $N$. On note $\NN=\{0,1,\ldots\}$ l'ensemble des entiers naturels. Si pour tout $M \in \Irr(G)$ on a $\dim(N_{(M)})< \infty$, alors on définit
$$\begin{array}{ccccc}
h & : & \Irr(G) & \rightarrow & \NN \\
& & M & \mapsto & \dim(N_{(M)}) \\
\end{array}$$
la fonction de Hilbert du $G$-module $N$. Plus généralement, on appelle fonction de Hilbert une fonction $h:\ \Irr(G) \rightarrow \NN$.

Soient $S$ un schéma, $\ZZ$ un $G$-schéma et $\pi: \ZZ \rightarrow S$ un morphisme affine, de type fini et $G$-invariant. D'après \cite[§2.3]{Br}, le faisceau $\FF:=\pi_{*} \OO_{\ZZ}$ admet la décomposition suivante comme $\OO_S,G$-module: 
\begin{equation} \label{eq1}
\FF \cong \bigoplus_{M \in \Irr(G)} \FF_{(M)} \otimes M
\end{equation}
où l'opération de $G$ dans $\FF$ est induite par l'opération de $G$ dans chaque $M$, et chaque $\FF_{(M)}:=\Hom^{G}(M,\FF)$ est un $\FF^G$-module cohérent. On dit que la famille $\pi$ est à multiplicités finies si $\FF^G$ est un $\OO_S$-module cohérent. Si de plus $\pi$ est plat, alors chaque $\OO_S$-module $\FF_{(M)}$ est localement libre de rang fini (et ce rang est constant sur une base connexe). Toutes les familles considérées par la suite seront supposées plates et à multiplicités finies.

\begin{definition}
Soient $h:\ \Irr(G) \rightarrow \NN$ une fonction de Hilbert et $W$ un $G$-schéma affine. On définit le foncteur de Hilbert ${Hilb}_{h}^{G}(W)$: ${Sch}^{op} \rightarrow Ens$ par: 
$$S \mapsto \left\{ 
\begin{array}{c} 
\xymatrix{
\ZZ \ar@{.>}_{\pi}[dr]  \ar@{^{(}->}[r] &   S \times W \ar@{->}[d]^{p_1}\\
                     &  S } \end{array}
\middle| 
\begin{array}{l} \ZZ  \text{ sous-schéma fermé $G$-stable},  \\
{\pi} \text{ morphisme plat},\\
{\pi}_{*} {\OO}_{\ZZ} \cong \bigoplus_{M \in \Irr{(G)}} {\FF}_{M} {\otimes} M,\\
 {\FF}_{M} \ \text{loc. libre de rang $h(M)$ sur } {\OO}_S.  \end{array} 
\right\} .$$
\end{definition} 

On fixe $W$ un $G$-schéma affine. On appelle famille plate de $G$-sous-schémas fermés de $W$ au dessus de $S$ un élément $(\pi:\ \ZZ \rightarrow S) \in {Hilb}_{h}^{G}(W)(S)$, avec $h$ est une fonction de Hilbert arbitraire. D'après \cite[Theorem 2.11]{Br}, le foncteur ${Hilb}_{h}^{G}(W)$ est représenté par un schéma quasi-projectif ${\Hilb}_{h}^{{G}}(W)$: le schéma de Hilbert invariant associé au $G$-schéma affine $W$ et à la fonction de Hilbert $h$. On note ${\mathrm{Univ}}_{h}^{G}(W) \subset W \times {\Hilb}_{h}^{{G}}(W)$, muni de la seconde projection, la famille universelle correspondante. Par définition, la famille universelle vérifie la propriété suivante: pour tout $(\pi:\ \ZZ \rightarrow S) \in {Hilb}_{h}^{G}(W)(S)$, il existe un unique morphisme $g:\ S \rightarrow {\Hilb}_{h}^{{G}}(W)$ tel que 
$$\ZZ \cong {\mathrm{Univ}}_{h}^{G}(W) {\times}_{ {\Hilb}_{h}^{{G}}(W)} S .$$  

\begin{remarque}
La construction du schéma de Hilbert invariant par Alexeev et Brion repose sur une réduction (via la théorie des représentations) au schéma de Hilbert multigradué construit par Haiman et Sturmfels dans \cite{HS}. 
\end{remarque}

L'algèbre des invariants ${k[W]}^G \subset k[W]$ est de type fini, donc c'est l'algèbre d'un schéma affine $W/\!/G$, muni d'un morphisme de passage au quotient $\nu:\ W \rightarrow W/\!/G$. On dit que $W/\!/G$ est le quotient catégorique de $W$ par $G$. Les propriétés du quotient $W/\!/G$ sont données dans \cite[§3]{SB}, en particulier:
\begin{itemize}
\item si $W$ est irréductible (resp. réduit), alors $W/\!/G$ est irréductible (resp. réduit),
\item si $W$ est normal, alors $W/\!/G$ est normal (\cite[§3.2, Théorème 2]{SB}),
\item si $W$ est un $G$-module rationnel, alors $W/\!/G$ est de Cohen-Macaulay (\cite[§3.4, Théorème 4]{SB}),
\item si $W$ est un $G$-module rationnel de dimension finie et si tout caractère multiplicatif de $G$ est trivial, alors $W/\!/G$ est de Gorenstein (\cite[§4.4, Théorème 4]{SB}).
\end{itemize}
Le lemme \ref{diagcom} et la proposition  \ref{chow} qui suivent sont démontrés dans \cite[Proposition 3.15]{Br}, nous en redonnons les démonstrations avec plus de détails.

\begin{lemme} \label{diagcom}
Soient $G$ un groupe algébrique réductif, $W$ un $G$-schéma affine, $h$ une fonction de Hilbert telle que $h(V_0)=1$, $S$ un schéma arbitraire et $(\pi:\ \ZZ \rightarrow S) \in {Hilb}_{h}^{G}(W)(S)$. Alors il existe un morphisme $f:\ S \rightarrow W/\!/G$ tel que le diagramme suivant commute:  
\begin{equation}  \label{kré}
\xymatrix{
      \ZZ  \ar[r]^{{p}_2} \ar[d]_{\pi} & W \ar@{->>}[d]^{\nu} \\
       S \ar[r]_(0.4)f & W/\!/G}
\end{equation}
\end{lemme}

\begin{proof}
On note $p_2:\ \ZZ \rightarrow W$ le morphisme obtenu en composant l'inclusion $\ZZ \hookrightarrow S \times W$ avec la seconde projection $S \times W \rightarrow W$. 
Montrons que ${\FF}^{G} \cong {\OO}_{S}$ via le morphisme ${\pi}^{\#}: {\OO}_{S} \rightarrow \FF$: 
$${\FF}^{G} \cong  \bigoplus_{M \in \Irr(G)} {\FF}_{(M)} {\otimes} M^G= {\FF}_{V_0} {\otimes} {V_0}^G= {\FF}_{V_0}  .$$ 
Puis, $\pi$ induit le morphisme ${\OO}_{S} \stackrel{\phi}{\rightarrow} {\FF}_{V_0} \cong {\FF}^G$ entre ${\OO}_{S}$-modules localement libres de rang $1$. Or, pour chaque point fermé $s \in S(k)$, la fibre de ${\OO}_{S}$ en $s$ vaut 
$${\OO}_{S}(s):= {\OO}_{S,s}/{\mathcal{M}}_s=k(s)=k$$ 
puisque $k$ n'admet pas d'extension de dimension finie, et la fibre de ${\FF}_{V_0}$ en $s$ vaut: $${\FF}_{V_0}(s)\cong {({\pi}_{*} {\OO}_{\ZZ})}^{G}(s)= {({\pi}_{*} {\OO}_{\ZZ})}^{G} {\otimes} k(s)=  {\left( { \left({\pi}_{*} {\OO}_{\ZZ}\right)} {\otimes} k\left( s \right) \right)}^{G}={k[{\ZZ}_s]}^{G} \cong k  .$$ 
Et sur chaque fibre, le morphisme ${\phi}(s):\ k \rightarrow k$ est un isomorphisme. Donc, d'après le lemme de Nakayama, pour chaque $s \in S(k)$, le morphisme ${\phi}_s:\ {\OO}_{S,s} \rightarrow {\FF}_{{V_0},s}$ est un isomorphisme, et donc $\phi$ est un isomorphisme. Autrement dit, le morphisme $\pi:\ \ZZ \rightarrow S$ induit un isomorphisme $\pi/\!/G:\ \ZZ/\!/G \rightarrow S$. On note $f:\ S \rightarrow W/\!/G$ l'application obtenue en composant $(\pi/\!/G)^{-1}$ avec l'inclusion $\ZZ/\!/G \subset W/\!/G$ et on vérifie alors que $f$ fait commuter le diagramme (\ref{kré}).    
\end{proof}

D'après le lemme \ref{diagcom} et par définition du produit fibré, on a une immersion fermée $\ZZ \hookrightarrow S {\times}_{W/\!/G} W$. On a donc le diagramme commutatif suivant:
\begin{equation} \label{diag1}
        \xymatrix{
    \ZZ \ar@{^{(}->}[r]^(0.3){i} \ar@{.>}[rd]_{\pi}  & S {\times}_{W/\!/G} W \ar[d]^{p_1} \ar[r]^(0.6){p_2} &W \ar@{->>}[d]^(0.4){\nu} \\
                        & S \ar[r]_(0.4){f} & W/\!/G 
    }
\end{equation}

Soit $h$ une fonction de Hilbert telle que $h(V_0)=1$. Si l'on prend $S={\Hilb}_{h}^{G}(W)$ dans le lemme \ref{diagcom}, alors on obtient l'existence d'un morphisme $\gamma:\ {\Hilb}_{h}^{G}(W) \rightarrow W/\!/G$ tel que le diagramme (\ref{diag1}) commute. Ce morphisme $\gamma$ est appelé morphisme de Hilbert-Chow et va jouer un rôle très important par la suite. Ce morphisme est propre (\cite[Proposition 3.12]{Br}) et donc projectif.

On suppose dorénavant que le schéma affine $W$ est réduit (mais non nécessairement irréductible) et que $W/\!/G$ est irréductible. D'après \cite[Theorem 14.4]{Ei}, le morphisme $\nu:\ W \rightarrow W/\!/G$ est plat sur un ouvert non vide de $ W/\!/G$. On note ${(W/\!/G)}_{*}$ le plus grand ouvert de platitude.

\begin{definition} 
La décomposition (\ref{eq1}) implique que toutes les fibres au dessus de ${(W/\!/G)}_{*}$ ont la même fonction de Hilbert. On note $h_W$ cette fonction de Hilbert; celle-ci est égale à la fonction de Hilbert de la fibre générique de $\nu$. Enfin, on note $\HH:={\Hilb}_{h_W}^{{G}} (W)$ et $\XX:={\mathrm{Univ}}_{h_W}^{{G}} (W)$.
\end{definition}

\begin{remarque} \label{fibre_groupe}
Si l'une des fibres de $\nu$ au dessus de ${(W/\!/G)}_{*}$ est isomorphe au groupe $G$, alors $h_W(M)=\dim (M)$ puisque 
$$k[G] \cong \bigoplus_{M \in \Irr(G)} M^{\oplus \dim(M)}$$ comme $G$-module à gauche.
\end{remarque}

\noindent La proposition qui suit fournit un rôle privilégié à $h_W$:

\begin{proposition} \label{chow}
Soient $G$ un groupe algébrique réductif, $W$ un $G$-schéma affine, $\nu: W \twoheadrightarrow W/\!/G$ le morphisme de passage au quotient, ${(W/\!/G)}_{*}$ l'ouvert de platitude de $\nu$ et $h=h_W$ la fonction de Hilbert de la fibre générique de $\nu$. Alors $\gamma$ induit un isomorphisme 
$$ {\gamma}^{-1}({(W/\!/G)}_{*})  \stackrel{\cong}{\longrightarrow}  {(W/\!/G)}_{*} \ .$$
\end{proposition}

\begin{proof}
Quitte à remplacer $W$ par un ouvert affine $G$-stable contenu dans ${\nu}^{-1}({(W/\!/G)}_{*})$, on peut supposer que $\nu$ est plat sur $W$. On souhaite montrer que, dans ce cas, $\gamma:\ \HH \rightarrow W/\!/G$ est un isomorphisme. On fixe $S$ un schéma arbitraire, $(\pi:\ \ZZ \rightarrow S) \in {Hilb}_{h_W}^{G}(W)(S)$ et on reprend les notations du diagramme (\ref{diag1}).

Par hypothèse, $\nu$ est plat sur $W$, donc $p_1:\ S \times_{W/\!/G} W \rightarrow S$ est une famille plate dont chaque fibre $F_s$ a pour fonction de Hilbert $h_W$. De même, par définition, $\pi:\ \ZZ \rightarrow S$ est aussi une famille plate dont toutes les fibres ${\ZZ}_s$ ont pour fonction de Hilbert $h_W$.
Ensuite, on a le morphisme surjectif de ${\OO}_{S}$-modules ${{p}_1}_{*} {\OO}_{S {\times}_{W/\!/G} W} \twoheadrightarrow {{\pi}}_{*} {\OO}_{\ZZ}$. Mais les morphismes ${\pi}$ et $p_1$ sont tous deux $G$-invariants, donc il s'agit en fait d'un morphisme de ${\OO}_{S},G$-modules. On en déduit, pour chaque $M \in \Irr(G)$, un morphisme surjectif de ${\OO}_{S},G$-modules: $({{p}_1}_{*} {\OO}_{S {\times}_{W/\!/G} W})_{(M)} \twoheadrightarrow ({{\pi}}_{*} {\OO}_{\ZZ})_{(M)}$. En passant aux fibres, on obtient pour chaque $s \in S(k)$ le morphisme surjectif: 
$$ {k[F_s]}_{(M)} \twoheadrightarrow {k[{\ZZ}_s]}_{(M)} .$$ 
Or, par définition de $\ZZ$ et par hypothèse sur $\nu$, les espaces vectoriels ${k[{\ZZ}_s]}_{(M)}$ et ${k[F_s]}_{(M)}$ ont la même dimension $h_W(M)$ et donc ${k[F_s]}_{(M)} \cong {k[{\ZZ}_s]}_{(M)}$. D'après le lemme de Nakayama $({{p}_1}_{*} {\OO}_{S {\times}_{W/\!/G} W})_{(M)} \cong ({{\pi}}_{*} {\OO}_{\ZZ})_{(M)}$, d'où un isomorphisme $\ZZ \cong {S {\times}_{W/\!/G} W}$. Il s'ensuit que $W/\!/G$ représente le foncteur de Hilbert, $\nu:\ W \rightarrow W/\!/G$ est la famille universelle et $\gamma:\ \HH  \rightarrow W/\!/G$ est un isomorphisme. 
\end{proof}

\begin{definition} \label{Horb}
On définit la composante principale de $\HH$ par 
$$\HHp:=\overline{{\gamma}^{-1}({(W/\!/G)}_{*})}.$$ 
La variété $\HHp$ est une composante irréductible de $\HHr$, le schéma de Hilbert invariant $\HH$ muni de sa structure réduite. La restriction du morphisme de Hilbert-Chow $\gamma:\ \HHp \rightarrow W/\!/G$ est un isomorphisme au dessus de ${(W/\!/G)}_{*}$ et donc un morphisme birationnel et projectif.
\end{definition}


\noindent On termine cette section avec la proposition suivante (\cite[Proposition 3.10]{Br}):

\begin{proposition} \label{groupaction}
Soit $G'$ un groupe algébrique tel que $G' \subset {\Aut}^G(W)$, alors $G'$ opère dans le schéma de Hilbert invariant $\HH$ et dans $\XX \subset W \times \HH$  de telle sorte que tous les morphismes qui apparaissent dans le diagramme (\ref{diag1}) soient $G'$-équivariants.
\end{proposition}

\begin{remarque} 
Sous les hypothèses de la proposition \ref{groupaction}, le morphisme ${\pi}^{\#}: {\OO}_{\HH} \rightarrow \FF:={\pi}_{*}{\OO}_{\XX}$ est un morphisme de $G'$-modules, et les ${\FF}_{(M)}$ qui apparaissent dans la décomposition (\ref{eq1}) sont des ${\OO}_{S},G'$-modules. 
\end{remarque}

\section{Points fixes pour l'opération d'un groupe de Borel et espaces tangents} \label{ConnexitéetTangence}

On fixe un sous-groupe algébrique $G' \subset {\Aut}^G(W)$ et un sous-groupe de Borel $B' \subset G'$. Dans cette section, on démontre une série de lemmes qui seront utiles par la suite pour montrer des résultats de connexité et déterminer les espaces tangents en certains points de $\HH$. 

\begin{lemme} \label{fixespoints}
On suppose que $W/\!/G$ admet une unique orbite fermée pour l'opération de $G'$ et que cette orbite est un point $x$. Alors chaque fermé $G'$-stable de $\HH$ contient au moins un point fixe pour l'opération de $B'$. Si de plus $\HH$ admet un unique point fixe de $B'$, alors $\HH$ est connexe. 
\end{lemme}

\begin{proof}
Le morphisme $\gamma$ est projectif et $G'$-équivariant donc la fibre ensembliste $\gamma^{-1}(x)$ est une $G'$-variété projective. Soit $C$ un fermé $G'$-stable de ${\HH}$, alors $\gamma (C)$ est un fermé $G'$-stable de $W/\!/G$. Donc $x \in \gamma(C)$, autrement dit $C \cap \gamma^{-1}(x)$ est non-vide. Donc $C \cap \gamma^{-1}(x)$ contient au moins un point fixe pour l'opération de $B'$, d'après le théorème de point fixe de Borel (\cite[Theorem 10.4]{Bo}). Enfin, chaque composante connexe de $\HH$ admet au moins un point fixe de $B'$, d'où la dernière assertion du lemme.  
\end{proof}  


\begin{lemme} \label{Hlisse4}
On suppose, comme dans le lemme \ref{fixespoints}, que $W/\!/G$ admet une unique orbite fermée $x$ pour l'opération de $G'$. Alors on a l'équivalence:
$$ \HH=\HHp \text{ est une variété lisse}  \Leftrightarrow  \left\{
    \begin{array}{l}
        \forall Z \in \HH^{B'},\ \dim(T_{Z} \HH)=\dim(\HHp), \text{ et }\\
        \HH \text{ est connexe.}
    \end{array}
\right.$$
\end{lemme}

\begin{proof}
Le sens $\Rightarrow$ est clair. Montrons l'autre implication. 
On note $d:=\dim (\HHp)$. L'ensemble $E:=\{Z \in \HHr \ \mid \ \dim(T_Z\HH)>d\}$ est un fermé $G'$-stable de $\HHr$. On suppose que $E$ est non-vide. D'après le lemme \ref{fixespoints}, le fermé $E$ contient un point fixe de $B'$, noté $Z_0$; alors $\dim(T_{Z_0} \HH)>d$ ce qui contredit notre hypothèse de départ. Il s'ensuit que $E$ est vide et donc $\HH$ est une variété lisse. On a supposé de plus que $\HH$ est connexe, donc $\HH$ est irréductible d'où $\HH=\HHp$. 
\end{proof}

Soit $Z \in \HH(k)$ et $I \subset k[W]$ l'idéal de $Z$. On note $R:=k[W]/I$ l'algèbre des fonctions régulières de $Z$. On rappelle le résultat important suivant (\cite[Proposition 3.5]{Br}):
\begin{proposition} \label{isoTangent}
On a un isomorphisme canonique
$$T_Z \HH \cong \Hom_{R}^{G}(I/I^2,R) .$$
\end{proposition}

On suppose maintenant que $R \cong k[G]$ comme $G$-module, c'est-à-dire pour chaque $M \in \Irr(G)$, on a $h_W(M)=\dim(M)$. Soit $N$ un $G$-sous-module de $k[W]$ contenu dans $I$ tel que le morphisme naturel de $R,G$-modules $\delta:\ R \otimes N \rightarrow I/I^2$ soit surjectif; et soit $\RRR$ un $G$-sous-module de $R \otimes N$ tel que l'on ait la suite exacte de $R,G$-modules
\begin{equation} \label{complexeCornFlakes}
\begin{array}{cccccc}
R \otimes \RRR  &  \stackrel{\rho}{\longrightarrow} & R \otimes N & \stackrel{\delta}{\longrightarrow} & I/I^2 & \rightarrow 0 \\
  &&f \otimes 1& \mapsto & \overline{f} & 
\end{array}
\end{equation}
où l'on note $\overline{f}$ l'image de $f \in I$ dans $I/I^2$.  

\begin{remarque}
Le module $N$ est appelé module des générateurs de $I/I^2$. Le module $\RRR$ est appelé module des relations entre les générateurs de $I/I^2$.
\end{remarque} 

\begin{lemme} \label{dimRmod}
Si $\delta$ est un isomorphisme, alors 
$$\dim(\Hom_{R}^{G}(I/I^2,R)= \dim (N) .$$ 
\end{lemme}

\begin{proof}
On a l'isomorphisme canonique d'espaces vectoriels: 
\begin{equation*} 
\begin{array}{ccc}
{\Hom}_R^{G}(R \otimes N,R) & \cong & {\Hom}^{G}(N,R) \\
 f & \mapsto & (n  \mapsto f(1 \otimes n))\\
 (r \otimes n  \mapsto r \, g(n))  & \leftmapsto & g
\end{array}
\end{equation*} 
et donc 
\begin{align*}
\dim({\Hom}_R^{G}(R \otimes N,R))&=\dim({\Hom}^{G}(N,R))\\
       &=\dim({\Hom}^{G}(N , k[G] )) \\
       &=\dim({(k[G] \otimes N^*)}^{G}) \\
       &=\dim(\Mor^G(G,N^*))\\
       &=\dim(N)
\end{align*} 
où l'on note $\Mor^G(G,N^*)$ l'espace vectoriel des morphismes de schémas $G$-équivariants de $G$ dans $N^*$.
\end{proof}


On applique le foncteur contravariant et exact à gauche ${\Hom}_R(\,.\,,R)$ à la suite exacte (\ref{complexeCornFlakes}) puis on prend les $G$-invariants. On obtient alors la suite exacte d'espaces vectoriels de dimension finie:
\begin{equation} \label{complexeCornFlakes2}
  \xymatrix{
    0 \ar[r] &{\Hom}_R^{G}(I/I^2,R) \ar[r]^{\delta^*} & {\Hom}_R^{G}(R {\otimes} N,R) \ar[r]^{\rho^*} \ar[d]^{\cong} & {\Hom}_R^{G}(R \otimes \RRR,R) \ar[d]^{\cong} \\
        &   &  \Hom^G(N,R) & \Hom^G(\RRR,R)
  }
\end{equation}

On a donc $T_Z \HH \cong \Im(\delta^*) = \Ker(\rho^*)$. De plus, si l'idéal $I$ est $B'$-stable, alors on peut choisir pour $N$ et $\RRR$ des $B' \times G$-modules tels que tous les morphismes de la suite exacte (\ref{complexeCornFlakes}) soient des morphismes de $R,B' \times G$-modules et que tous les morphismes de la suite exacte (\ref{complexeCornFlakes2}) soient des morphismes de $B'$-modules.

\begin{lemme} \label{InegTang}
Avec les notations précédentes, on a
$$\dim({\Hom}_R^G(I/I^2,R)) = \dim(N)- \rg(\rho^*) .$$
\end{lemme}

\begin{proof}
\begin{align*}
\dim({\Hom}_R^{G}(I/I^2,R)) &= \dim(\Ker(\rho^*)) \\
                            &= \dim(\Hom_R^{G}(R \otimes N,R))- \rg(\rho^*)\\  
                            &= \dim(N)- \rg(\rho^*) \text{ d'après le lemme \ref{dimRmod}.} 
\end{align*} 
\end{proof}

\section{Construction de morphismes équivariants vers des grassmanniennes}



Soit $G'$ un groupe algébrique tel que $G' \subset {\Aut}^G(W)$ comme précédemment. Si $E$ un espace vectoriel et $m$ un entier, on note $\Gr(m,E)$ la grassmannienne des sous-espaces vectoriels de dimension $m$ dans $E$. Le but de cette section est de démontrer la

\begin{proposition} \label{morphismegrass}
Pour tout $M \in \Irr(G)$, il existe un $G'$-sous-module $F_M \subset {k[W]}_{(M)}$ de dimension finie qui engendre ${k[W]}_{(M)}$ comme $k[W]^G,G'$-module, et il existe un morphisme de schémas $G'$-équivariant
$${\delta}_{M}: \HH \rightarrow \Gr(h_W(M),F_{M}^{*}).$$  
\end{proposition}

\begin{proof}
Avec les notations du diagramme (\ref{diag1}), l'inclusion $\XX \hookrightarrow W {\times}_{W/\!/G} \HH$ est $G' \times G$-équivariante et donc induit un morphisme surjectif de ${\OO}_{\HH},G' \times G$-modules ${p_2}_{*} {\OO}_{\HH {\times}_{W/\!/G} W} \rightarrow \FF:={\pi}_{*} {\OO}_{\XX}$. Mais ${p_2}_{*} {\OO}_{\HH {\times}_{W/\!/G} W}={\OO}_{\HH} {\otimes}_{k[W/\!/G]} k[W]$, où l'on rappelle que $k[W/\!/G]={k[W]}^{G}$ par définition de $W/\!/G$. On peut alors considérer la décomposition en ${\OO}_{\HH},G' \times G$-modules 
\begin{equation}
{\OO}_{\HH} {\otimes}_{k[W/\!/G]} k[W]\cong \bigoplus_{M \in \Irr(G)} {\OO}_{\HH} {\otimes}_{k[W/\!/G]} {k[W]}_{(M)} {\otimes} M
\end{equation}
où l'opération de $G'$ dans ${\OO}_{\HH} {\otimes}_{k[W/\!/G]} k[W]$ est induite par l'opération de $G'$ dans $W$ (et $G'$ opère trivialement dans $M$).
On en déduit, pour chaque $M \in \ \Irr(G)$, un  morphisme surjectif de ${\OO}_{\HH},G'$-modules 
\begin{equation} \label{m1}
{\OO}_{\HH} {\otimes}_{k[W/\!/G]} {k[W]}_{(M)} \twoheadrightarrow \FF_{(M)} . 
\end{equation}
Il s'ensuit que l'espace vectoriel ${k[W]}_{(M)}$ engendre ${\FF}_{(M)}={\Hom}^{G}(M,\FF)$ comme ${\OO}_{\HH},G'$-module. Malheureusement, ${k[W]}_{(M)}$ est en général un espace vectoriel de dimension infinie. Cependant,  ${k[W]}_{(M)}$ est un ${k[W]}^{G}$-module de type fini, donc il existe un $G'$-module $F_M$ de dimension finie qui engendre ${k[W]}_{(M)}$ comme ${k[W]}^{G}$-module: 
\begin{equation} \label{m2}
{k[W]}^{G} {\otimes} F_M \twoheadrightarrow {k[W]}_{(M)} .
\end{equation} 
Ensuite, on déduit de (\ref{m1}) et (\ref{m2}) le morphisme surjectif de ${\OO}_{\HH},G'$-modules: 
\begin{equation}  \label{pacfib}
{\OO}_{\HH} {\otimes} F_M \twoheadrightarrow {\FF}_{(M)}
\end{equation}
où l'on rappelle que ${\FF}_{(M)}$ est un ${{\OO}_{\HH}}$-module localement libre de rang $h_W(M)$. Or, d'après \cite[Exercice 6.18]{EH}, un tel morphisme est équivalent à la donnée d'un morphisme de schémas  
$$\delta:\ \HH \rightarrow \Gr(\dim (F_M)-h_W(M),F_{M}).$$ 
On vérifie que ce morphisme $\delta$ est $G'$-équivariant. Enfin, on a l'isomorphisme $G'$-équivariant
$$\Gr(\dim(F_M)-h_W(M),F_M) \cong \Gr(h_W(M),F_M^*)$$ 
ce qui termine la preuve de la proposition.
\end{proof}


\begin{remarque} 
La proposition \ref{morphismegrass} reste vraie plus généralement si l'on considère une fonction de Hilbert $h$ telle que $h(V_0)=1$. 
\end{remarque}

On termine cette section par une description ensembliste du morphisme ${\delta}_M$. On rappelle que, pour tout $G$-module $M$, on a les isomorphismes canoniques suivants: 
\begin{equation} \label{isocanonique}
\begin{array}{rccccc}
k{[W]}_{(M)}:=&{\Hom}^{G}(M,k[W]) &  \cong  & {(M^{*} \otimes k[W])}^{G} & \cong & \Mor^G(W,M^*) . \\
&(m \mapsto \phi(m) \, f) & &  \phi \otimes f  &  & (w \mapsto f(w) \, \phi)
\end{array}
\end{equation}
Via ces isomorphismes, les éléments du $G'$-module $F_M \subset k[W]_{(M)}$ s'identifient à des morphismes $G$-équivariants de $W$ dans $M^{*}$ et l'application $\delta_M$ est donnée ensemblistement par:
\begin{equation}
\delta_M: \HH(k) \rightarrow \Gr(\dim (F_M)-h_W(M),F_M),\ Z \mapsto \Ker(f_Z) 
\end{equation}
où 
$$\begin{array}{ccccc} \label{descirption_explicite}
f_Z & : & F_M & \twoheadrightarrow & \FF_{M,Z}  \\
& & q & \mapsto & q_{|Z} 
\end{array}$$
est l'application linéaire surjective obtenue en passant aux fibres dans (\ref{pacfib}).

%% file: differentes_situations.tex
\section{Les différentes situations étudiées}

\subsection{Les différentes situations étudiées}  \label{lesdiffsituations}

Soient $E$ et $F$ des espaces vectoriels de dimension finie sur $k$, on note $\Hom(E,F)$ l'espace vectoriel des applications $k$-linéaires de $E$ dans $F$. On a toujours un isomorphisme canonique
\begin{equation} \label{iso_canonique}
 \Hom(E,F) \cong E^* \otimes F .
\end{equation} 
Nous allons nous intéresser aux cinq situations qui suivent.\\

\begin{itemize}

\item \textit{Situation $1$}: soient $V$ et $V'$ des espaces vectoriels de dimensions $n$ et $n'$ respectivement. On note $W:={\Hom}(V',V)$, $G:=SL(V)$ et $G':=GL(V')$. Le groupe $G' \times G$ opère dans $W$ de la façon suivante:
\begin{equation}  \label{actionSLLn}
\forall w \in W,\ \forall (g',g) \in G' \times G,\ (g',g).w:=g \circ w \circ g'^{-1} .
\end{equation}

\item \textit{Situation $2$}: soient $V$, $V_1$ et $V_2$ des espaces vectoriels de dimensions $n$, $n_1$ et $n_2 \in {\NN}^{*}$ respectivement. On note 
$W:={\Hom}(V_1,V) \times {\Hom}(V,V_2)$, $G:=GL(V)$ et $G':=GL(V_1) \times GL(V_2)$. Le groupe $G' \times G$ opère dans $W$ de la façon suivante:
\begin{equation}  \label{aktionGLn}
\forall (u_1,u_2) \in W,\ \forall (g_1,g_2) \in G',\ \forall g \in G,\ (g_1,g_2,g).(u_1,u_2):=(g \circ u_1 \circ g_1^{-1},g_2 \circ u_2 \circ g^{-1}) .
\end{equation} 

\item \textit{Situation $3$}: comme la situation $1$ sauf que l'on considère $G:=O(V)$, le groupe des automorphismes de $V$ qui préservent la forme quadratique $q$ sur $V$ définie par: 
\begin{equation} \label{defQ}
\forall x:=(x_1,\ldots,x_n) \in V,\ q(x_1,\ldots,x_n):= \sum_{i=1}^{n} x_i^2 .
\end{equation} 

\item \textit{Situation $4$}: comme la situation $1$ sauf que l'on considère $G:=SO(V)=O(V) \cap SL(V)$.\\

\item \textit{Situation $5$}: comme la situation $1$ sauf que l'on suppose que $n$ est pair et on considère $G:=Sp(V)$, le groupe des automorphismes de $V$ qui préservent la forme symplectique $\Omega$ sur $V$ définie par:
\begin{equation}  \label{defFormSymp}
\forall \ x:=(x_1,\ldots,x_n),\ y:=(y_1,\ldots,y_n) \in V,\ \Omega(x_1,\ldots,x_n,y_1,\ldots,y_n):= \sum_{i=1}^{n/2} x_{2i-1} y_{2i}-y_{2i-1} x_{2i} .
\end{equation}
\end{itemize}

La situation $1$, qui est de loin la plus simple, sera traitée dans la section \ref{casSln}. Les autres situations sont plus compliquées et seront traitées dans les chapitres \ref{chp2} et \ref{chp3}. Dans chaque situation, $V$ est la représentation standard de $G$ et on note $V^*$ sa duale. Remarquons que, dans les situations $3$ à $5$, le groupe $G$ préserve une forme bilinéaire non-dégénérée, et donc $V \cong V^*$ comme $G$-module. On remarque également que les opérations de $G'$ et $G$ sur $W$ commutent dans les cinq situations, donc d'après la proposition \ref{groupaction}, le groupe $G'$ opère dans $W/\!/G$, dans $\HH$ et dans $\XX$.

Pour tous $p,q \in \NN^*$, on note $\MM_{p,q}(k)$ l'espace vectoriel des matrices de taille $p \times q$ à coefficients dans $k$. Nous serons amené par la suite à faire des calculs explicites sur des idéaux de $k[W]$. On fixe donc une bonne fois pour toute des bases 
\begin{itemize} \renewcommand{\labelitemi}{$\bullet$}
\item $\BB:=\{b_1,\ldots,b_n\}$ de $V$,
\item $\BB':=\{c_1,\ldots,c_{n'}\}$ de $V'$,
\item $\BB_1:=\{e_1,\ldots,e_{n_1}\}$ de $V_1$,
\item $\BB_2:=\{f_1,\ldots,f_{n_2}\}$ de $V_2$,
\end{itemize}
et on note $\BB^*$, $\BB'^*$, $\BB_1^*$, $\BB_2^*$ les bases duales associées. Via le choix de ces bases, on a des isomorphismes
\begin{itemize} \renewcommand{\labelitemi}{$\bullet$}
\item $\Hom(V',V) \cong \MM_{n,n'}(k)$, 
\item $\Hom(V_1,V) \times \Hom(V,V_2) \cong \MM_{n,n_1}(k) \times \MM_{n_2,n}(k)$,
\item $\Hom(V_1,V_2) \cong \MM_{n_2,n_1}(k)$. 
\end{itemize}

Dans la situation 2, on note
\begin{align*}
&B_1:=\Stab_{G'}(\left\langle e_1\right\rangle, \left\langle e_1, e_2\right\rangle, \ldots , \left\langle e_1, \ldots,  e_{n_1-1}\right\rangle) ,\\
&B_2:=\Stab_{G'}(\left\langle f_1 \right\rangle, \left\langle f_1, f_2 \right\rangle, \ldots , \left\langle f_1, \ldots,  f_{n_2-1} \right\rangle).
\end{align*} 
Autrement dit, $B_1$ (resp. $B_2$) est le sous-groupe de Borel de $GL_{n_1}(k)$ (resp. $GL_{n_2}(k)$) formé des matrices triangulaires supérieures. Pour $i=1,2$, on note $U_i$ le radical unipotent de $B_i$ et $T_i$ le tore maximal des matrices diagonales de $B_i$. Alors $B':=B_1 \times B_2$ est un sous-groupe de Borel de $G'$, son radical unipotent est $U':=U_1 \times U_2$ et $T':=T_1 \times T_2$ est un tore maximal de $B'$. Enfin, on note $B$ le sous-groupe de Borel formé des matrices triangulaires inférieures de $G\cong GL_n(k)$, $U$ le radical unipotent de $B$ et $T$ le tore maximal des matrices diagonales de $B$.

Dans les situations 1, 3, 4 et 5, on note 
$$ B':=\Stab_{G'}(\left\langle c_1\right\rangle, \left\langle c_1, c_2\right\rangle, \ldots , \left\langle c_1, \ldots,  c_{n'-1}\right\rangle).$$
Autrement dit, $B'$ est le sous-groupe de Borel de $G'$ formé des matrices triangulaires supérieures. On note $U'$ le radical unipotent de $B'$ et $T'$ le tore maximal des matrices diagonales de $B'$. Enfin, on fixe $B$ un sous-groupe de Borel de $G$, on note $U$ le radical unipotent de $B$ et $T$ un tore maximal dans $B$.

\subsection{Rappels concernant la théorie des représentations des groupes classiques}  \label{rappelsthedesreps}

Notre référence pour la théorie des représentations des groupes classiques est \cite{FH}. Dans cette section, on fixe des notations et on rappelle quelques faits qui nous seront utiles par la suite. Soient 
\begin{itemize} \renewcommand{\labelitemi}{$\bullet$}
\item $E$ un espace vectoriel de dimension finie $d$, 
\item $G$ un sous-groupe fermé, connexe et réductif de $GL(E)$,
\item $B$ un sous-groupe de Borel de $G$,
\item $T$ un tore maximal de $B$.
\end{itemize}
Le groupe des caractères $\Lambda:=\XX(T)$ ne dépend pas (à isomorphisme près) des choix de $B$ et $T$ et est appelé le réseau des poids de $G$. On note $\Phi=\Phi(G,T)$ les racines de l'algèbre de Lie $\gg$ de $G$, c'est-à-dire les poids de $T$ dans $\gg$ pour l'action adjointe. Le choix de $B$ définit le sous-ensemble $\Phi_+ \subset \Phi$ des racines positives. On note $\Pi$ le sous-monoïde de $\Lambda$ engendré par les éléments de $\Phi_+$ et on définit un ordre partiel sur $\Lambda$ de la façon suivante:
$$\forall \lambda, \mu \in \Lambda,\  \mu \leq \lambda \Leftrightarrow \lambda-\mu \in \Pi .$$
Si $M$ est un $G$-module irréductible, alors $M$ admet un unique plus haut poids $\lambda$ pour l'ordre partiel $\leq$ et ce poids détermine entièrement $M$. On note alors $M=E_{\lambda}$. On appelle poids dominants, noté $\Lambda_+$, l'ensemble des poids qui apparaissent comme plus hauts poids de $G$-modules irréductibles. On a donc une correspondance bijective
$$\lambda \in \Lambda_+ \leftrightarrow E_{\lambda} \in \Irr(G).$$
La formule des dimensions de Weyl permet d'exprimer explicitement la dimension de $E_{\lambda}$ en fonction de $\lambda$ (\cite[Corollary 24.6]{FH}).

\begin{remarque}
En pratique, on utilisera la notation $S^{\lambda}(E)$ (resp. $\Gamma_{\lambda}(E)$) pour désigner $E_{\lambda}$ lorsque $G=GL(E)$ (resp. $G=SO(E)$ ou $G=Sp(E)$). \end{remarque}

\subsubsection{Cas du groupe linéaire}   \label{threpGLn}

Soit $G:=GL(E)$ le groupe linéaire, alors $\Lambda$ est un $\ZZZ$-module libre de rang $d$ et on note $\{ \epsilon_1, \ldots, \epsilon_{d}\}$ une base de $\Lambda$. Pour un choix approprié de $B$, on a 
$$\Lambda_+=\{ r_1 \epsilon_1 + \ldots + r_{d} \epsilon_{d} \in \Lambda \ |\ r_1 \geq \ldots \geq r_d \}.$$
Si $r_d \geq 0$ (resp. $r_1 \leq 0$), alors la représentation $S^{\lambda}(E)$ est polynomiale (resp. duale d'une représentation polynomiale). En particulier, lorsque $r_d \geq 0$, alors $E \mapsto S^{\lambda}(E)$ est un foncteur appelé foncteur de Schur. Si $r_1.r_d<0$, alors il existe $1 \leq t \leq d$ tel que $r_t \geq 0 \geq  r_{t+1}$ et un unique $d$-uplet $(k_1, \ldots,k_d) \in {\NN}^d$ tel que 
\begin{equation}
\lambda=k_1 {\epsilon}_1+k_2 {\epsilon}_2+ \ldots+k_t {\epsilon}_t-k_{t+1} {\epsilon}_{t+1}- \ldots-k_d {\epsilon}_d .
\end{equation}

\subsubsection{Cas du groupe spécial orthogonal}  \label{threpspeortho}

Soit $G:=SO(E)$ le groupe spécial orthogonal, alors $\Lambda$ est un $\ZZZ$-module libre de rang $d'$, où $d':=E(\frac{d}{2})$ est la partie entière inférieure de $\frac{d}{2}$. On note $\{ \epsilon_1, \ldots, \epsilon_{d'}\}$ une base de $\Lambda$. Pour un choix approprié de $B$, on a 
$$\Lambda_+=\left\{ r_1 \epsilon_1 + \ldots + r_{d'} \epsilon_{d'} \in \Lambda \   \middle|   \begin{array}{ll}
        r_1 \geq \ldots \geq r_{d'-1} \geq \pm r_{d'} & \text{ si } d=2d'\\
        r_1 \geq \ldots \geq  r_{d'} \geq 0 & \text{ si } d=2d'+1
    \end{array} \right\}.$$

\subsubsection{Cas du groupe orthogonal}  \label{threportho}

Soit $G:=O(E)$ le groupe orthogonal, alors on a la suite exacte naturelle de groupes:
\begin{equation} \label{rqsuiteexacteOn}
\xymatrix{
    0 \ar[r]  & SO(E) \ar[r] & O(E) \ar^{\det}[r] & \ZZZ_2 \ar[r]  & 0
}
\end{equation}
où $\ZZZ_2$ désigne le groupe d'ordre $2$. 
Cette suite est scindée, donc 
\begin{equation} \label{OnSonScinde}
O(E) \cong SO(E) \ltimes \ZZZ_2
\end{equation} 
où l'on identifie $\ZZZ_2$ à $\left\{ \begin{bmatrix} \pm 1 & 0  \\
0 & Id \end{bmatrix} \right\}$.
Et ce produit est direct lorsque $d$ est impair. On distingue alors deux cas. 
\begin{enumerate}
\item Si $d$ est impair, alors les représentations irréductibles de $G$ sont paramétrées par les couples $(\lambda, s) \in \Lambda_+ \times \{ \pm 1 \}$, où $\Lambda_+$ est l'ensemble des poids dominants de $SO(E)$. On note $M_0$ (resp. $\epsilon$) la représentation triviale (resp. signe) de $\ZZZ_2$. La représentation irréductible de $O(E) \cong SO(E) \times \ZZZ_2$ associée à $(\lambda,s)$ est:
\begin{itemize} \renewcommand{\labelitemi}{$\bullet$}
\item $\Gamma_{\lambda}(E)^{+}:=\Gamma_{\lambda}(E) \otimes M_0$ si $s=1$,
\item $\Gamma_{\lambda}(E)^{-}:=\Gamma_{\lambda}(E) \otimes \epsilon$ si $s=-1$.
\end{itemize}
\item Si $d$ est pair, alors une représentation irréductible de $O(E)$ est soit une représentation irréductible de $SO(E)$ stabilisée par $\ZZZ_2$ (par exemple la représentation standard si $d \geq 4$), soit la somme directe de deux représentations irréductibles de $SO(E)$ échangées par $\ZZZ_2$. Plus précisément, si $\lambda=r_1 \epsilon_1 + \ldots + r_{d'} \epsilon_{d'} \in \Lambda_+$, alors on peut considérer la représentation de $O(E)$ induite: 
\begin{itemize} 
\item si $r_{d'} \neq 0$, alors
$$\Ind_{SO(E)}^{O(E)}(\Gamma_{\lambda}(E)) \cong \Gamma_{\lambda}(E) \oplus \Gamma_{\lambda'}(E)$$
comme $SO(E)$-module, avec $\lambda'=r_1 \epsilon_1 + \ldots + r_{d'-1} \epsilon_{d'-1}-r_{d'} \epsilon_{d'}$. Le $O(E)$-module $\Ind_{SO(E)}^{O(E)}(\Gamma_{\lambda}(E))$ est irréductible et $\ZZZ_2$ opère en échangeant $\Gamma_{\lambda}(E)$ et $\Gamma_{\lambda'}(E)$.   
\item si $r_{d'} = 0$, alors
$$\Ind_{SO(E)}^{O(E)}(\Gamma_{\lambda}(E)) \cong \Gamma_{\lambda}(E) \oplus \Gamma_{\lambda}(E)$$
comme $SO(E)$-module. Le $O(E)$-module $\Ind_{SO(E)}^{O(E)}(\Gamma_{\lambda}(E))$ se décompose en deux modules irréductibles $\Gamma_{\lambda}(E)^{+}$ et $\Gamma_{\lambda}(E)^{-}$ sur lesquels $\ZZZ_2$ opère trivialement et par le signe respectivement. 
\end{itemize}
Et toutes les représentations irréductibles de $O(E)$ sont de cette forme.
\end{enumerate}

\subsubsection{Cas du groupe symplectique}  \label{threpsymppp}

On suppose que $d=2d'$ pour un certain $d' \geq 1$ et soit $G:=Sp(E)$ le groupe symplectique. Alors $\Lambda$ est un $\ZZZ$-module libre de rang $d'$ et on note $\{ \epsilon_1, \ldots, \epsilon_{d'}\}$ une base de $\Lambda$. Pour un choix approprié de $B$, on a 
$$\Lambda_+=\{ r_1 \epsilon_1 + \ldots + r_{d'} \epsilon_{d'} \in \Lambda \ |\ r_1 \geq \ldots  \geq r_{d'} \geq 0\}.$$

%% file: reduction.tex
\section{Le principe de réduction et le cas du groupe spécial linéaire}

\subsection{Le principe de réduction} \label{princreduction}

On utilise la proposition \ref{morphismegrass} pour construire, dans les situations 1,2,3 et 5, un morphisme $G'$-équivariant $\rho$ de $\HH$ vers un espace homogène $G'/P$ où $P$ est un sous-groupe parabolique de $G'$. Puis, on montre que la fibre schématique de $\rho$ en $eP$ s'identifie à un schéma de Hilbert invariant $\HH'$ plus "simple" que $\HH$ (proposition \ref{reduction1}). On obtient ensuite un résultat analogue pour la famille universelle $\pi:\ \XX \rightarrow \HH$ (proposition \ref{reduction2}).

\subsubsection{Construction d'un morphisme \texorpdfstring{$\HH \rightarrow G'/P$}{} dans la situation 2}  \label{redGrassGLn}
Dans le lemme qui suit, on utilise la théorie classique des invariants pour déterminer un $G'$-sous-module de dimension finie de $\Hom^G(V^*,k[W])$ (resp. $\Hom^G(V,k[W])$) qui engendre ce $k[W]^G$-module. La référence que nous utiliserons systématiquement pour les résultats de la théorie classique des invariant est \cite{Pro}.

\begin{lemme} \label{exi1}
\begin{enumerate}
\item Le $k[W]^G$-module ${k[W]}_{(V)}$ est engendré par ${\Hom}^{G}(V,W^{*})$.
\item Le $k[W]^G$-module ${k[W]}_{(V^{*})}$ est engendré par ${\Hom}^{G}(V^*,W^{*})$.
\end{enumerate}
\end{lemme}

\begin{proof}
Les preuves étant analogues pour la représentation standard et sa duale, on se contente de traiter uniquement le cas de la représentation standard, c'est-à-dire de montrer que le morphisme naturel 
\begin{equation} \label{mmorpp}
{k[W]}^G \otimes {\Hom}^{G}(V,W^{*}) \rightarrow  {k[W]}_{(V)}
\end{equation}
de ${k[W]}^{G},G'$-modules est surjectif. On identifie $k[W]$, l'algèbre des fonctions régulières sur $W$, avec $S(W^*)$, l'algèbre symétrique de $W$. Alors  
\begin{align*}
  {k[W]}_{(V)} &\cong {(S(W^{*}) \otimes V^*)}^{G}\\
               &\cong \bigoplus_{p \geq 0, q \geq 0} {(S^p({{V}^{*}}^{n_1}) \otimes S^q({V}^{n_2}) \otimes V^*)}^{G}\\
               &\cong \bigoplus_{\substack{p_1,\ldots,p_{n_1} \geq 0,\\ q_1, \ldots,q_{n_2} \geq 0}}  {(S^{p_1}({V}^{*}) \otimes \cdots \otimes S^{p_{n_1}}({V}^{*})  \otimes S^{q_1}({V}) \otimes \cdots \otimes S^{q_{n_2}}(V) \otimes V^*)}^{G}.
\end{align*} 
On fixe $(p_1,\ldots p_{n_1},q_1,\ldots,q_{n_2}) \in {\NN}^{n_1+n_2}$. Alors 
\begin{align*}
&{(S^{p_1}({V}^{*}) \otimes \cdots \otimes S^{p_{n_1}}({V}^{*})  \otimes S^{q_1}({V}) \otimes  \cdots  \otimes S^{q_{n_2}}({V}) \otimes V^*)}^{G} \\
&\cong {\left(k[ V \oplus  \cdots \oplus V \oplus {V}^{*} \oplus  \cdots  \oplus {V}^{*} \oplus V]_{(p_1, \ldots ,p_{n_1},q_1, \ldots ,q_{n_2},1)}\right)}^{G} 
\end{align*}
est l'espace des invariants multihomogènes de multidegré $(p_1, \ldots ,p_{n_1},q_1, \ldots ,q_{n_2},1)$.
On applique l'opérateur de polarisation $\PPP$ défini dans \cite[§3.21]{Pro}: 
\begin{align*}
&{\left(k[ V \oplus  \cdots \oplus V \oplus {V}^{*} \oplus  \cdots  \oplus {V}^{*} \oplus V]_{(p_1, \ldots ,p_{n_1},q_1, \ldots ,q_{n_2},1)}\right)}^{G} \\
&\stackrel{\PPP}{\rightarrow} {\left(k[ {V}^{p_1} \oplus  \cdots \oplus {V}^{p_{n_1}} \oplus {{V}^{*}}^{q_1} \oplus  \cdots  \oplus {{V}^{*}}^{q_{n_2}} \oplus V]_{\multi}\right)}^{G}.
\end{align*} 
D'après le premier théorème fondamental pour $GL(V)$ (voir \cite[§9.1.4]{Pro}), on a nécessairement $p+1=q$, avec $p:=p_1+ \ldots +p_{n_1}$ et $q:=q_1+ \ldots +q_{n_2}$ et 
$${(k[ {V}^{p_1} \oplus  \cdots \oplus {V}^{p_{n_1}} \oplus {{V}^{*}}^{q_1} \oplus  \cdots  \oplus {{V}^{*}}^{q_{n_2}} \oplus V]_{\multi})}^{G}$$
est engendré comme espace vectoriel par 
$$\left \{  f_{\sigma}:=\prod_{i=1}^{q} (i\ | \ \sigma(i)) ,\ \sigma \in {\Sigma}_{q} \right \}$$ 
où l'on note $\Sigma_q$ le groupe des permutations de $1,\ldots,q$ et pour chaque couple $(i,j)$, $i=1, \ldots ,q$, $j=1, \ldots ,q$, la forme bilinéaire $(i\ | \ j)$ sur $V^p \oplus V^{*q}$ est définie par 
\begin{equation}  \label{formebili}
\forall v_1,\ldots,v_q\in V,\ \forall {\phi}_1,\ldots,{\phi}_q \in {V}^{*},\ (i\ | \ j):(v_1,\ldots,v_q,{\phi}_1,\ldots,{\phi}_q) \mapsto {\phi}_j(v_i).
\end{equation}
Ensuite, d'après \cite[§3.2.2, Theorem]{Pro}, le $k[W]^G$-module 
$${(S^{p_1}({V}^{*}) \otimes  \cdots  \otimes S^{p_{n_1}}({V}^{*})  \otimes S^{q_1}({V}) \otimes  \cdots  \otimes S^{q_{n_2}}({V}) \otimes V)}^{G}$$ 
est engendré comme espace vectoriel par 
$$\left \{  \RRR f_{\sigma},\ \sigma \in \Sigma_{q} \right \}  $$ 
où l'on note $\RRR$ l'opérateur de restitution défini dans \cite[§3.2.2]{Pro}. \\
On fixe $\sigma \in \Sigma_{q}$, alors 
\begin{align*}
f_{\sigma}=\prod_{i=1}^{q} (i\ | \ \sigma(i)) &= \left( \prod_{i , \sigma(i) \neq q} (i\ | \ \sigma(i)) \right) \times ({\sigma}^{-1}(q)\ | \ q) \\
                                             &= f'_{\sigma} \times ({\sigma}^{-1}(q)\ | \ q).
\end{align*} 
Donc, pour tous $v_1, \ldots ,v_{n_1}, v \in V$ et pour tous ${\phi}_1, \ldots , {\phi}_{n_2} \in V^{*}$, on a
$$\RRR f_{\sigma}(v_1, \ldots ,v_{n_1}, {\phi}_1, \ldots , {\phi}_{n_2},v)= (\RRR f'_{\sigma}(v_1, \ldots ,v_{n_1},{\phi}_1, \ldots ,{\phi}_{n_2})) \times {\phi}_{i_0}(v)$$
pour un certain $1 \leq i_0 \leq n_2$. Or $\RRR f'_{\sigma}(v_1, \ldots ,v_{n_1},{\phi}_1, \ldots ,{\phi}_{n_2})  \in {k[W]}^{G}$ et $(v \mapsto (\phi_{i_0} \mapsto {\phi}_{i_0}(v)))  \in  \Hom^G(V,W^*)$ d'où le résultat. 
\end{proof} 

On note $F_1:={\Hom}^{G}(V,W^{*})$ et $F_2:={\Hom}^{G}(V^{*},W^{*})$. On a 
$$W^{*} \cong (V_1 \otimes V^{*}) \oplus (V \otimes V_{2}^{*})$$ 
comme $G' \times G$-module et donc $F_1 \cong V_2^*$ et $F_2 \cong V_1$ comme $G'$-modules.\\
Ecrivons explicitement ces deux isomorphismes dans les bases que l'on a fixées, cela nous sera utile lors de la démonstration du lemme \ref{fibrehil}. 
On a:
\begin{equation}
\hspace{1cm}  V_{2}^{*} \cong F_1={\Hom}^{G}(V,W^{*}),\ f_{i}^{*} \rightarrow \left( v \rightarrow \left( 
\begin{bmatrix} 0 
\end{bmatrix}, 
\begin{bmatrix} 0 & \cdots &0 &v &0 &\cdots &0
\end{bmatrix} \right) \right)
\end{equation}
où $v \in V$ et $\left( 
\begin{bmatrix} 0 
\end{bmatrix}, 
\begin{bmatrix} 0 & \cdots &0 &v &0 &\cdots &0
\end{bmatrix} \right)  \in W^{*} \cong \Hom(V,V_1) \times \Hom(V_2,V)$. Le vecteur colonne $v$ occupe la $i$-ème colonne et les autres colonnes sont toutes nulles.

\noindent De manière analogue, on a:
\begin{equation}
V_{1} \cong F_2={\Hom}^{G}(V^{*},W^{*}),\ e_j \rightarrow \left(\phi \rightarrow \left( 
\begin{bmatrix}
0\\
 \vdots \\
0\\
\phi\\
0\\
 \vdots \\
0
\end{bmatrix}, 
\begin{bmatrix} 0 
\end{bmatrix} \right) \right)
\end{equation} 
où $\phi \in V^{*}$ et $\left( 
\begin{bmatrix}
0\\
 \vdots \\
0\\
\phi\\
0\\
 \vdots \\
0
\end{bmatrix}, 
\begin{bmatrix} 0 
\end{bmatrix} \right) \in W^*$. Le vecteur ligne $\phi$ occupe la $j$-ème ligne et les autres lignes sont toutes nulles.

\noindent Puis, via l'isomorphisme (\ref{isocanonique}), le $G'$-module $F_1$ s'identifie au sous-espace vectoriel de ${\Mor}^{G}(W,V^{*})$ engendré par les $n_2$ projections linéaires $p_1, \ldots ,p_{n_2}$ de $W \cong (V_{1}^{*} \otimes V) \oplus (V^{*} \otimes V_2)$ sur $V^{*}$ et le $G'$-module $F_2$ s'identifie au sous-espace vectoriel de ${\Mor}^{G}(W,V)$ engendré par les $n_1$ projections linéaires $q_1, \ldots ,q_{n_1}$ de $W$ sur $V$. Avec ces notations, le lemme \ref{exi1} admet la reformulation suivante: tout morphisme $G$-équivariant de $W$ dans $V^*$ (resp. de $W$ dans $V$) peut s'écrire comme combinaison linéaire de la forme $\sum_i f_i p_i$ (resp. $\sum_i f_i q_i$), où $f_i \in k[W]^G$. Par la suite, on fera parfois l'abus d'écrire
\begin{align*}
&F_1=\left\langle p_i,i=1, \ldots ,n_2 \right \rangle  \subset \Mor^G(W,V^{*}),  \\ 
&F_2=\left\langle  q_j,j=1, \ldots ,n_1 \right \rangle  \subset \Mor^G(W,V). 
\end{align*}

On note $m_1:=h_W(V^*)$, $m_2:=h_W(V)$ et on fait l'hypothèse suivante: $1 \leq m_i \leq n_i$, pour $i=1,2$. Cette hypothèse sera toujours vérifiée dans les exemples que nous allons traiter. Le but est simplement d'éviter de construire des morphismes triviaux. La proposition \ref{morphismegrass} nous donne des morphismes $G'$-équivariants:   
$${\rho}_1:\ \HH \rightarrow \Gr(m_1,V_{1}^{*}),\ Z \mapsto \Ker(q \rightarrow q_{|Z})\ \text{ où }\ q \in F_2,$$
et
$${\rho}_2:\ \HH \rightarrow \Gr(m_2,V_2),\ Z \mapsto \Ker(q \rightarrow q_{|Z})\ \text{ où }\ q \in F_1.$$ 
Les grassmanniennes $\Gr(m_1,V_1^*)$ et $\Gr(m_2,V_2)$ sont des espaces homogènes pour les opérations naturelles de $GL(V_1)$ et $GL(V_2)$ respectivement, et donc $\Gr(m_1,V_1^*) \times \Gr(m_2,V_2)$ est un espace homogène pour l'opération de $G'=GL(V_1) \times GL(V_2)$. On note $E_1$ le sous-espace vectoriel de $V_{1}^{*}$ engendré par les $m_1$ premiers vecteurs de la base ${\BB}_{1}^{*}$ et on note $E_{2}$ le sous-espace vectoriel de $V_2$ engendré par les $m_2$ premiers vecteurs de la base $\BB_2$. On a un isomorphisme $G'$-équivariant  
\begin{equation} \label{isograss}
\Gr(m_1,V_1^*) \times \Gr(m_2,V_2) \cong G'/ P
\end{equation}
où $P$ est le stabilisateur de $(E_1,E_2) \in \Gr(m_1,V_1^*) \times \Gr(m_2,V_2)$ dans $G'$, et donc un sous-groupe parabolique. On compose le morphisme $\rho_1 \times \rho_2$ avec l'isomorphisme (\ref{isograss}), on obtient un morphisme $G'$-équivariant:
$$\rho:\ \HH \rightarrow G'/P .$$

\subsubsection{Construction d'un morphisme \texorpdfstring{$\HH \rightarrow G'/P$}{} dans les situations 1, 3 et 5} \label{redGrass2}
  
\noindent On commence par énoncer le

\begin{lemme} \label{exi1SLn}
Le $k[W]^G$-module $k[W]_{(V^*)}$ est engendré par $\Hom^G(V^*,W^*)$.
\end{lemme}

\begin{proof}
La preuve est analogue à celle du lemme \ref{exi1}. Les ingrédients clés sont d'une part les opérateurs de polarisation et de restitution (\cite[§3.2]{Pro}) et d'autre part le premier théorème fondamental pour les groupes classiques (\cite[§11.2.1]{Pro}).
\end{proof}

\begin{remarque}  \label{rkOninutile}
Le lemme \ref{exi1SLn} est faux pour $G=SO(V)$ si $n' \geq \dim(V) \geq 2$. En effet, on vérifie que le $k[W]^G$-module $k[W]_{(V^*)}$ est engendré par $\Hom^G(V^*,W^* \oplus k[W]_{n-1})$. On écarte par la suite le cas $n' < \dim(V)$ car alors nous verrons que $W/\!/SO(V) \cong W/\!/O(V)$ et on obtient des résultats identiques dans les situations 3 et 4. 
\end{remarque}

On a $ \Hom^G(V^*,W^*) \cong V'$ comme $G'$-module. On note $m:=h_W(V^*)$ et on suppose que $1 \leq m \leq n'$. Cette hypothèse sera toujours vérifiée dans les exemples que nous traiterons. La proposition \ref{morphismegrass} nous donne un morphisme $G'$-équivariant: 
$$\HH \rightarrow \Gr(m,V'^*).$$
On note $E$ le sous-espace vectoriel de $V'^*$ engendré par les $m$ premiers vecteurs de la base $\BB'^*$ et $P$ le stabilisateur de $E$ dans $G'$.
Alors $\Gr(m,V'^*) \cong G'/P$, d'où un morphisme $G'$-équivariant:
$$\rho:\ \HH \rightarrow G'/P .$$

\subsubsection{Réduction}  \label{zectionred}

Soient $P$ un sous-groupe parabolique de $G'$ et $F$ un $P$-schéma. On considère $G'$ comme une $P$-variété pour l'opération par multiplication à droite de $P$:
$$\forall g' \in G',\ \forall p \in P,\ p.g':=g'  p^{-1}.$$ 
On rappelle que si $P$ est un sous-groupe parabolique de $G'$, alors le morphisme de passage au quotient $G' \rightarrow G'/P$ est localement trivial pour la topologie de Zariski. D'après \cite[§I.5.16]{Jan}, le quotient du $P$-schéma $G' \times F$ par l'opération du groupe $P$ est naturellement muni d'une structure de $G'$-schéma. Il s'ensuit que, si $X$ est un $G'$-schéma muni d'un morphisme $G'$-équivariant vers $G'/P$, alors on a un isomorphisme $G'$-équivariant
\begin{equation} \label{iissoo}
X \cong G' \times^P F
\end{equation}
où $F$ est la fibre schématique en $eP$. Et de nombreuses propriétés géométriques et topologiques de $X$ peuvent se lire sur $F$: lissité, connexité, irréductibilité,...

On se place maintenant dans l'une des situations 1,2,3 ou 5. Le morphisme $\rho:\ \HH \rightarrow G'/P$ est $G'$-équivariant et donc, d'après ce qui précède, on a un isomorphisme $G'$-équivariant
$$\HH \cong G' \times^P F$$
où $F$ est la fibre schématique de $\rho$ en $eP$. Pour déterminer $\HH$ comme $G'$-schéma, on est donc ramené à déterminer $F$ comme $P$-schéma.

\begin{notation}
Si $E$ est un sous-espace vectoriel de $V'$, on note $E^{\perp}$ l'orthogonal de $E$ dans $V'^*$. On note:
\begin{itemize} \renewcommand{\labelitemi}{$\bullet$}
\item $W':= \left\{
    \begin{array}{ll}
        \Hom(V_1/E_1^{\perp},V) \times \Hom(V,E_2) & \text{ dans la situation 2,}\\
        \Hom(V'/E^{\perp},V) & \text{ dans les situations 1, 3 et 5}
    \end{array}
\right.$ 
\item $\HH':={\Hilb}_{h_W}^{{G}} (W'),$ 
\item $\XX':={\Univ}_{h_W}^{{G}} (W')$ et $\pi':\ \XX' \rightarrow \HH'$ la famille universelle.
\end{itemize}
\end{notation}

\begin{lemme} \label{fibrehil}
La fibre $F$ du morphisme ${\rho}$ est isomorphe au schéma de Hilbert invariant $\HH'$ et l'opération de $P$ dans $F$ coïncide, via cet isomorphisme, avec l'opération de $P$ dans $\HH'$ induite par l'opération de $P$ dans $W'$.
\end{lemme}

\begin{proof}
On donne ici la démonstration dans la situation $2$, la démonstration dans les autres situations est analogue.\\ 
Par définition de $F$, pour tout schéma $S$, on a: 
 $$ \Mor(S,F) = \left\{ 
\begin{array}{c} 
\xymatrix{
\ZZ \ar@{.>}_{\pi}[dr]  \ar@{^{(}->}[r] &   S \times W \ar@{->}[d]^{p_1}\\
                     &  S } \end{array}
\middle|
\begin{array}{l} {\ZZ}  \text{ sous-schéma fermé $G$-invariant},  \\
{\pi} \text{ morphisme plat},\\
{\pi}_{*} {\OO}_{{\ZZ}} = \bigoplus_{M \in \Irr{G}} {\FF}_{M} {\otimes} M,\\
 {\FF}_{M} \ \text{loc. libre de rang $h_W(M)$ sur } {\OO}_S, \\
\forall s \in S(k),\ \rho({{\ZZ}}_s)=E_1 \times E_2. \end{array} 
\right\} . $$
Or 
\begin{align*}
\rho({{\ZZ}}_s)=E_1 \times E_2 &\Leftrightarrow  {{p}_{n+1}}_{|{{\ZZ}}_s}= \cdots ={{p}_{n_1}}_{|{{\ZZ}}_s}={{q}_{n+1}}_{|{{\ZZ}}_s}= \cdots ={{q}_{n_2}}_{|{{\ZZ}}_s}=0\\
                          &\Leftrightarrow  {{\ZZ}}_s \subset \Hom(V'/E_1^{\perp},V) \times \Hom(V,E_2)=W' .
\end{align*} 
Donc \begin{align*} \Mor(S,F) &= \left\{ 
\begin{array}{c} 
\xymatrix{
\ZZ \ar@{.>}_{\pi}[dr]  \ar@{^{(}->}[r] &   S \times W' \ar@{->}[d]^{p_1}\\
                     &  S } \end{array}
\middle|
\begin{array}{l} {\ZZ}  \text{ sous-schéma fermé $G$-invariant},  \\
{\pi} \text{ morphisme plat},\\
{\pi}_{*} {\OO}_{{\ZZ}} = \bigoplus_{M \in \Irr{G}} {\FF}_{M} {\otimes} M,\\
 {\FF}_{M} \ \text{loc. libre de rang $h_{W}(M)$ sur } {\OO}_S.  \end{array} 
\right\} \\
&={Hilb}_{h_{W}}^{{G}} (W')(S) \\
&\cong \Mor(S,{\Hilb}_{h_{W}}^{{G}} (W'))  
\end{align*} où le dernier isomorphisme est une conséquence directe de la définition du schéma de Hilbert invariant comme foncteur représentable. Il s'ensuit que $F \cong \HH'$ comme $P$-schéma.  
\end{proof}

D'où la 

\begin{proposition} \label{reduction1}
On a un isomorphisme $G'$-équivariant 
$$\begin{array}{ccccc}
\psi_1 & : & G' {\times}^{P} \HH' & \cong & \HH  \\
& & (g',z)P & \mapsto & g'.z 
\end{array}$$
\end{proposition}


\begin{remarque}
La fonction de Hilbert $h_W$ qui apparaît dans la définition de $\HH'$ et de $\XX'$ n'est pas égale à la fonction de Hilbert $h_{W'}$ de la fibre générique du morphisme de passage au quotient $W' \rightarrow W'/\!/G$ en général et on n'a donc pas de "vraie" réduction dans ce cas. Cependant, dans la plupart des exemples que nous traiterons, on aura bien $h_{W'}=h_W$. 
\end{remarque}

Ensuite, soit $\pi:\ \XX \rightarrow \HH$ la famille universelle. On note $\delta$ la composée
$$\XX \stackrel{\pi}{\longrightarrow} \HH \stackrel{\rho}{\longrightarrow} G'/P.$$
Le morphisme $\delta$ est $G'$-équivariant et munit $\XX$ d'une structure de fibré $G'$-homogène. On note $F'$ la fibre schématique de $\delta$ en $eP$, alors d'après (\ref{iissoo}) on a un isomorphisme $G'$-équivariant 
$$\XX \cong G' {\times}^{P} F'.$$ 
Donc, pour déterminer $\XX$ comme $G'$-schéma, on est ramené à déterminer $F'$ comme $P$-schéma.

\begin{lemme} \label{fibrehil2}
On a un isomorphisme $F' \cong \XX'$ et le morphisme $\pi_{|F'}:\ F' \rightarrow \HH'$ s'identifie à la famille universelle $\pi':\ \XX' \rightarrow \HH'$. De plus, l'opération de $P$ dans $F'$ coïncide, via cet isomorphisme, avec l'opération de $P$ dans $\XX'$ induite par l'opération de $P$ dans $W'$.
\end{lemme} 

\begin{proof}
On identifie $\HH'$ à un sous-schéma fermé de $\HH$ grâce au lemme \ref{fibrehil}. Par définition du schéma de Hilbert invariant, on a le diagramme cartésien
  $$
  \xymatrix{
    \XX'  \ar@{^{(}->}[r] \ar[d]_{\pi'}  & \XX \ar[d]^{\pi} \\
    \HH' \ar@{^{(}->}[r] & \HH
  }$$ 
qui induit le diagramme cartésien
  $$
    \xymatrix{
   G' \times^P \XX' \ar@{^{(}->}[r] \ar[d]_{Id \times \pi'}  & G' \times^P \XX \ar[d]^{Id \times \pi} \ar[r]^{\cong}   & G'/P \times \XX \ar[d]^{Id \times \pi} \\
    G' \times^P \HH' \ar@{^{(}->}[r] & G' \times^P  \HH \ar[r]_(0.4){\cong} & G'/P \times \HH
  } $$
On en déduit que le diagramme suivant est cartésien:
  $$
    \xymatrix{
   G' \times^P \XX' \ar[r]^(0.6){\psi_2} \ar[d]_{Id \times \pi'}  &  \XX \ar[d]^{ \pi} \\
    G' \times^P \HH' \ar[r]_(0.6){\psi_1} &   \HH
  } $$
où l'on note $\psi_2$ le morphisme obtenu en composant le morphisme $G' \times^P \XX' \rightarrow G'/P \times \XX$ avec la seconde projection $G'/P \times \XX \rightarrow \XX$. D'après la proposition \ref{reduction1}, le morphisme $\psi_1$ est un isomorphisme, donc $\psi_2$ également. On a donc
$$ \xymatrix{ 
    G' \times^P \XX' \ar[rr]^{\cong} \ar@{->>}[rd] && G' \times^P F' \ar@{->>}[ld] \\ & G'/P }$$
et donc $\XX' \cong F'$ comme $P$-schéma. 
\end{proof}

D'où la 

\begin{proposition} \label{reduction2}
On a un isomorphisme $G'$-équivariant 
$$\begin{array}{ccccc}
\psi_2 & : & G' {\times}^{P} \XX' & \cong & \XX \\
& & (g',x)P & \mapsto & g'.x 
\end{array}$$
\end{proposition}


\begin{corollaire} \label{reductionHdiag}
On a le diagramme commutatif de $G'$-schémas suivant:
\begin{equation} \label{diag_com}
 \xymatrix{
     G' {\times}^{P} \XX' \ar[d]_{\psi_2} \ar[rr]^{G' \times^P \pi'} && G' {\times}^{P} \HH' \ar[d]_{\psi_1} \ar[rr]^{G' \times^P \gamma'} && G' {\times}^{P} W'/\!/G \ar[d]^{\phi} \\
     \XX \ar@{.>}[rd]_{\delta} \ar[rr]_{\pi} && \HH \ar[rr]_{\gamma} \ar@{->>}[ld]^{\rho}   && W/\!/G \\
       &G'/P
}
  \end{equation}
où 
\begin{itemize} \renewcommand{\labelitemi}{$\bullet$}
\item $G' \times^P \pi'$ et $G' \times^P \gamma'$ sont les morphismes induits par $\pi'$ et $\gamma'$ respectivement,
\item $\psi_1$ et $\psi_2$ sont les isomorphismes des propositions \ref{reduction1} et \ref{reduction2} respectivement, 
\item $\phi$ est le morphisme induit par l'inclusion de $W'/\!/G$ dans $W/\!/G$ comme sous-variété $P$-stable,
\item $\gamma':\ \HH' \rightarrow W'/\!/G$ est le morphisme de Hilbert-Chow.
\end{itemize}
\end{corollaire}

\begin{proof}
D'après ce qui précède, la seule chose à vérifier est que le carré de droite est commutatif. On considère le diagramme
\begin{equation}  \label{diagCOM}
  \xymatrix{
    \HH'  \ar[r] \ar[d]_{\gamma'}  & \HH \ar[d]^{\gamma} \\
    W'/\!/G \ar[r]_{\phi} & W/\!/G
  }  
\end{equation}
où la flèche du haut est l'immersion fermée donnée par le lemme \ref{fibrehil}. Par définition du morphisme de Hilbert-Chow (voir la section \ref{generalitéesHilbert}), le diagramme (\ref{diagCOM}) est commutatif. Puis le diagramme 
\begin{equation} 
 \xymatrix{
     G' {\times}^{P} \HH' \ar[d]_{G' \times^P \gamma'} \ar[r] & G' {\times}^{P} \HH \ar[d]^{G' \times^P \gamma} \ar[r]^{\cong} & G'/P \times \HH \ar[d]^{Id \times \gamma} \\
     G' {\times}^{P} W'/\!/G  \ar[r] & G' {\times}^{P} W/\!/G  \ar[r]_(0.4){\cong} & G'/P \times W/\!/G  }
  \end{equation}   
est aussi commutatif et le résultat s'ensuit.
\end{proof} 

Par la suite, lorsque nous mentionnerons le principe de réduction, nous ferons implicitement référence aux propositions \ref{reduction1} et \ref{reduction2}. Dans la section \ref{casSln}, nous illustrons le principe de réduction en traitant le cas de la situation $1$. Nous verrons que dans cette situation, le principe de réduction permet de "trivialiser" l'étude du schéma de Hilbert invariant.

%% file: cas_complet_SLn.tex
\subsection{Cas du groupe spécial linéaire}  \label{casSln}

On se place dans la situation $1$: on a $G:=SL(V)$, $G':=GL(V')$ et $W:=\Hom(V',V)$. L'opération de $G' \times G$ dans $W$ est donnée par (\ref{actionSLLn}). On note $C(\Gr(n,V'^*)) \subset \Lambda^n (V'^*)$ le cône affine au dessus de la grassmannienne $\Gr(n,V'^*)$ identifiée à une sous-variété fermée de $\PP(\Lambda^n (V'^*))$ via le plongement de Plücker (\cite[§4.1, Example 1]{Sh}). Dans toute cette section, les points de $\Gr(n,V'^*)$ sont ainsi vus comme des points de $\PP(\Lambda^n (V'^*))$. On note $Bl_0(C(\Gr(n,V'^*)))$ la variété obtenue en éclatant $C(\Gr(n,V'^*))$ en $0$. C'est aussi l'espace total du fibré en droite $\OO_{\Gr(n,V'^*)}(-1)$ sur $\Gr(n,V'^*)$, c'est-à-dire 
$$Bl_0(C(\Gr(n,V'^*)))=\{(x,L) \in C(\Gr(n,V'^*)) \times \Gr(n,V'^*)\ |\ x \in L\}.$$ 
Nous allons démontrer le 

\begin{theoreme} \label{SLLn}
Si $n'=n>1$, alors $\HH \cong \Aff$ et $\gamma$ est un isomorphisme.\\
Si $n'>n>1$, alors $\HH \cong Bl_0(C(\Gr(n,V'^*)))$ et $\gamma$ est l'éclatement de $C(\Gr(n,V'^*))$ en $0$.\\ 
En particulier $\HH$ est toujours une variété lisse et donc $\gamma$ est une résolution de $W/\!/G$ lorsque ce quotient est singulier. 
\end{theoreme}

\begin{remarque}
Si $n'<n$, alors nous verrons que $W/\!/G=\{0\}$, donc le morphisme de passage au quotient $\nu:\ W \rightarrow W/\!/G$ est plat et donc $\HH \cong \{0\}$ d'après le corollaire \ref{cas_faciles_SLn}. De même, si $n=1$, alors nous verrons que $\nu$ est l'identité, donc $\nu$ est plat et $\HH \cong V'^*$. 
\end{remarque}

Le cas $n'=n>1$ est le corollaire \ref{cas_faciles_SLn}. Le cas $n'>n>1$ est la proposition \ref{iso_final2SLn}. La famille universelle $\pi:\ \XX \rightarrow \HH$ lorsque $1<n<n'$ est étudiée à la fin de cette section. 

\subsubsection{Etude du morphisme de passage au quotient}

\begin{notation}  \label{lignescolonnes}
On rappelle que l'on identifie $W \cong \MM_{n,n'}(k)$ via les bases $\BB$ et $\BB'$ fixées dans la section \ref{lesdiffsituations}. On notera parfois $w \in W$ sous forme de vecteurs colonnes ou de vecteurs lignes de la façon suivante:
\begin{itemize} \renewcommand{\labelitemi}{$\bullet$}
\item $w=\begin{bmatrix} C_1 & \cdots & C_{n'} \end{bmatrix}$ où les $C_i$ sont les vecteurs colonnes de la matrice $w$ et s'identifient naturellement à des éléments de $V$,
\item $w=\begin{bmatrix} L_1 \\ \vdots \\ L_n \end{bmatrix}$ où les $L_j$ sont les vecteurs lignes de la matrice $w$ et s'identifient naturellement à des éléments de $V'^*$.
\end{itemize}
\end{notation}

D'après le premier théorème fondamental pour $SL(V)$ (voir \cite[§11.1.2]{Pro}) l'algèbre des invariants $k[W]^G$ est engendrée par les $[i_1, \ldots, i_n]$, où pour tout $w \in W$ et pour tout $n$-uplet $1 \leq i_1 < \ldots < i_n \leq n'$, on définit $[i_1, \ldots, i_n]$ par:
\begin{equation} \label{SLinvariants} 
[i_1, \ldots, i_n](w):=\det \left( \begin{bmatrix} C_{i_1} & \cdots & C_{i_n} \end{bmatrix} \right) .
\end{equation}
On a le morphisme naturel $G' \times G$-équivariant
$$\Hom(V',V) \rightarrow  \Hom(\Lambda^n(V'),\Lambda^n(V)),\ w \mapsto \Lambda^n(w),$$
et $\Lambda^n(V) \cong V_0$, d'où un isomorphisme $\Hom(\Lambda^n(V'),\Lambda^n(V)) \cong \Lambda^n(V'^*)$. Le morphisme de passage au quotient $\nu:\ W \rightarrow W/\!/G$ est obtenu en composant ces deux morphismes:
$$\begin{array}{lrcl}
 \nu:  & \Hom(V',V)  & \rightarrow  &  \Lambda^n (V'^*) \\
        & w  & \mapsto     &  L_1 \wedge \ldots \wedge L_n. 
\end{array}$$ 
On distingue trois cas de figure:
\begin{itemize}
\item si $n'<n$, alors $W/\!/G=\Lambda^n (V'^*)=\{0\}$ et le morphisme $\nu$ est trivial, 
\item si $n'=n$, alors $W/\!/G=\Lambda^n (V'^*) \cong \Aff$ et $\nu(w)=\det(w)$,
\item si $n' >n$, alors $W/\!/G=C(\Gr(n,V'^*))$.   
\end{itemize}
En particulier, lorsque $n'>n$, on a $\nu(W)=\Lambda^n (V'^*)$ si et seulement si $n=1$ ou $n=n'-1$.

\begin{lemme} \label{qu_lisseSLn}
La variété quotient $W/\!/G$ est lisse sauf lorsque $1<n<n'-1$, auquel cas $W/\!/G$ admet une unique singularité en $0$.
\end{lemme}

\begin{proof} 
Si $n=1$ ou $n'-1 \leq n$, alors $W/\!/G$ est un espace affine, donc une variété lisse. Si $1<n<n'-1$, alors $W/\!/G = C(\Gr(n,V'^*))$ est un cône affine dans $\Lambda^n (V'^*)$ mais n'est pas un espace affine. Et $\Gr(n,V'^*)$ est une variété lisse, donc le cône $C(\Gr(n,V'^*))$ est singulier uniquement en $0$. 
\end{proof}  

La variété $W/\!/G$ est normale (\cite[§3.2, Théorème 2]{SB}) et de Gorenstein (\cite[§4.4, Théorème 4]{SB}) car le groupe des caractères de $G$ est trivial. Lorsque $n' \geq n$, la variété $W/\!/G$ est la réunion de deux orbites pour l'opération de $G'$: l'orbite fermée $\{0\}$ et l'orbite ouverte $U:=W/\!/G-\{0\}$.

\begin{lemme} \label{fibgeneSLn}
Lorsque $n' \geq n$, la fibre de $\nu$ en un point de $U$ est isomorphe à $G$.
\end{lemme}

\begin{proof}
On identifie $V \cong k^n$ et $V' \cong k^{n'}$ via les bases $\BB$ et $\BB'$ fixées précédement. On définit $A \in W \cong \MM_{n,n'}(k)$ par 
$$A_{i,j}=\left\{
    \begin{array}{ll}
        1 &\text{ si } i=j \leq n, \\
        0 &\text{ sinon,} 
    \end{array}
\right.$$
et on pose $B:= \nu(A)=c_1^{*} \wedge \ldots \wedge c_n^*$. Alors $B \in U$ et nous allons montrer que $\nu^{-1}(B)\cong SL_n(k)$. \\  
On a 
$${\nu}^{-1}(B)=\{ w \in W \ |\  L_1 \wedge \ldots \wedge L_n= c_1^{*} \wedge \ldots \wedge c_n^* \}  . $$ 
On décompose chaque $L_i \in V'^*$ dans la base $\BB'^*$: 
$$\forall 1 \leq i \leq n,\ L_i= \sum_{j=1}^{n'} a_{i,j} c_j^*  . $$ 
Alors 
$$L_1 \wedge \ldots \wedge L_n= \sum_{1\leq j_1<\ldots<j_{n} \leq n'} \det \left( \begin{bmatrix} C_{j_1} & \cdots & C_{j_n} \end{bmatrix} \right)\ c_{j_1}^* \wedge \ldots \wedge c_{j_n}^*$$ 
et donc nécessairement 
$$\det\left( \begin{bmatrix} C_{j_1} & \cdots & C_{j_n} \end{bmatrix} \right)=\left\{
    \begin{array}{ll}
        0 & \text{ lorsque } (j_1, \ldots ,j_{n}) \neq (1, \ldots ,n), \\
        1 &\text{ sinon.} 
    \end{array}
\right.$$
Il s'ensuit que les $C_j$ sont nuls, pour $j=n+1, \ldots ,n'$, et que la matrice carrée de taille $n$ définie par $\begin{bmatrix} C_1 & \cdots & C_n \end{bmatrix}$ appartient à $SL_n(k)$. Réciproquement, une matrice $w \in W$ de la forme $\begin{bmatrix} C_1 & \cdots & C_n & 0 & \cdots & 0 \end{bmatrix}$ vérifie $\nu(w)=B$. On en déduit que la fibre de $\nu$ en $B$ est isomorphe à $G$.
\end{proof}

En fait le lemme \ref{fibgeneSLn} est une conséquence d'un résultat de Luna (\cite[§2.1, Theorem 6]{SB}) qui dit que, si $n' \geq n$, alors $\nu$ est un $G$-fibré principal au dessus de l'ouvert $U$.

\begin{proposition} \label{ouvert_platitudeSLn}
\begin{itemize}
\item Si $n'\leq n$ ou $n=1$, alors $\nu$ est plat sur $W/\!/G$ tout entier.
\item Si $n'> n>1$, alors $U$ est l'ouvert de platitude de $\nu$ dans $W/\!/G$.
\end{itemize}
\end{proposition}

\begin{proof}
Si $n'<n$, alors le morphisme $\nu$ est trivial donc plat sur $W/\!/G=\{0\}$.\\
Si $n'=n$, alors $W/\!/G \cong \Aff$ et $\nu$ est le déterminant. Dans ce cas, d'après \cite[Exercice 10.9]{Ha}, le morphisme $\nu$ est plat sur $W/\!/G$ tout entier.\\
On suppose enfin $n' > n$. On sait que $\nu$ est plat sur un ouvert non-vide de $W/\!/G$, donc nécessairement $\nu$ est plat sur $U$ par $G'$-homogénéité. Ensuite, on vérifie que $\nu^{-1}(0)=\{w \in W\ |\ \rg(w) \leq n-1 \}$ est de dimension $(n'+1)(n-1)$. D'après le lemme \ref{fibgeneSLn}, la dimension de la fibre de $\nu$ en un point de $U$ vaut $n^2-1$. Les fibres d'un morphisme plat ont nécessairement toutes la même dimension, donc si $n>1$, l'ouvert $U$ est l'ouvert de platitude de $\nu$. En revanche, si $n=1$, alors $W/\!/G$ est lisse et toutes les fibres de $\nu$ ont la même dimension donc d'après \cite[Exercice 10.9]{Ha}, le morphisme $\nu$ est plat sur $W/\!/G$. 
\end{proof}


\begin{corollaire} \label{fctHilbSLn}
Si $n' \geq n$, la fonction de Hilbert de la fibre générique de $\nu$ est donnée par:
$$\forall M \in \Irr(G),\ h_W(M)=\dim(M)  .$$
\end{corollaire}

\noindent On déduit des propositions \ref{chow} et \ref{ouvert_platitudeSLn} le

\begin{corollaire} \label{cas_faciles_SLn}
Le morphisme de Hilbert-Chow $\gamma$ est un isomorphisme dans les cas suivants:
\begin{itemize} \renewcommand{\labelitemi}{$\bullet$}
\item $n'<n$ et alors $\HH$ est un point réduit correspondant au schéma $W$ tout entier,
\item $n'=n$, alors $\HH \cong \Aff$ et $\det:\ W \rightarrow \Aff$ est la famille universelle,
\item $n'>n=1$, alors $\HH \cong V'^*$ et $Id:\ V'^* \rightarrow V'^*$ est la famille universelle.
\end{itemize}
\end{corollaire}

\subsubsection{Détermination de \texorpdfstring{$\HH$}{H} lorsque \texorpdfstring{$1<n<n'$}{1<n<n'}}

Soient 
\begin{itemize} \renewcommand{\labelitemi}{$\bullet$}
\item $E:=\left\langle c_1^*,  \ldots , c_n^* \right\rangle \in \Gr(n,V'^*)$ et $P:=\Stab_{G'}(E)$,
\item $L_0$ l'image de $E$ par le plongement de Plücker $\Gr(n,V'^*)\hookrightarrow \PP(\Lambda^n V'^*)$,
\item $W':=\Hom(V'/E^{\perp},V)$ et $\nu':\ W' \rightarrow W'/\!/G$ le morphisme de passage au quotient,
\item $\HH':={\Hilb}_{h_{W'}}^{{G}} (W') \cong \Aff$ d'après le corollaire \ref{cas_faciles_SLn}, 
\item $\XX':={\Univ}_{h_{W'}}^{{G}} (W')$ et $\pi':\ \XX' \rightarrow \HH'$ la famille universelle.
\end{itemize}
D'après la proposition \ref{reduction1}, on a un isomorphisme $G'$-équivariant 
$$\HH \cong  G' \times^P \HH' $$
où $P$ opère dans $\HH'$ par multiplication par l'inverse du déterminant. En particulier $\HH$ est une variété lisse. On a vu que $W/\!/G=\Gr(n,V'^*)$ et donc $\PP(W/\!/G)=\Gr(n,V'^*) \cong G'/P$. On rappelle que l'on a défini un morphisme $G'$-équivariant $\rho:\ \HH \rightarrow \Gr(n,V'^*)$ dans la section \ref{princreduction}.  

\begin{lemme} \label{morphisme_dans_YSLn}
Le morphisme $\gamma \times \rho$ envoie $\HH$ dans $Bl_0(W/\!/G)$.
\end{lemme}

\begin{proof}
D'après la proposition \ref{ouvert_platitudeSLn}, le morphisme $\nu$ est plat sur $U$. Donc, d'après la proposition \ref{chow}, le morphisme $\gamma$ se restreint en un isomorphisme de $\gamma^{-1}(U)$ sur $U$. On note $Z_0$ l'unique point de $\HH$ tel que $\gamma(Z_0)=c_1^* \wedge  \ldots  \wedge c_n^*$. On vérifie que 
$$Q:=\Stab_{G'}(c_1^* \wedge  \ldots  \wedge c_n^*)=\left\{
\begin{bmatrix} A & 0 \\
                 B     & C 
\end{bmatrix},\ A \in SL_n(k),\ B \in \MM_{n'-n,n}(k) \text{ et } C \in GL_{n'-n}(k) \right\}  .$$ 
Le morphisme $\gamma$ est $G'$-équivariant donc $Z_0$ est stable par $Q$. Ensuite, $\rho$ est aussi $G'$-équivariant, donc $\rho(Z_0)$ est une droite $Q$-stable de $W/\!/G$. Mais $W/\!/G$ admet une unique droite $Q$-stable qui est la droite $\left\langle  c_1^* \wedge  \ldots  \wedge c_n^* \right\rangle$ et donc $\gamma(Z_0) \in \rho (Z_0)$. Il s'ensuit que $(\gamma \times \rho) (Z_0) \in {Bl}_0(W/\!/G)$. Puis, comme $\gamma \times \rho:\ \HH \rightarrow W/\!/G \times \PP(W/\!/G)$ est $G'$-équivariant et que ${Bl}_0(W/\!/G)$ est $G'$-stable, pour chaque $Z \in {\gamma}^{-1}(U)$, on a $(\gamma \times \rho)(Z) \in {Bl}_{0}(W/\!/G)$. Enfin ${Bl}_{0}(W/\!/G)$ est un fermé de $W/\!/G \times \PP (W/\!/G)$, donc ${(\gamma \times \rho)}^{-1}({Bl}_0(W/\!/G))$ est un fermé de $\HH$ contenant ${\gamma}^{-1}(U)$, d'où le morphisme de variétés $\gamma \times \rho:\ \HH \rightarrow {Bl}_0(W/\!/G)$. 
\end{proof}


La seconde projection $p_2:\ {Bl}_0(W/\!/G) \rightarrow G'/P$ est un morphisme $G'$-équivariant et munit donc ${Bl}_0(W/\!/G)$ d'une structure de fibré en droites $G'$-homogène au dessus de $G'/P$. On note $D_0$ la fibre schématique de $p_2$ en $eP$ et on vérifie que l'opération de $P$ dans $D_0 \cong \Aff$ coïncide avec l'opération de $P$ dans $\HH'$. D'après (\ref{iissoo}), on a un isomorphisme $G'$-équivariant 
$${Bl}_0(W/\!/G) \cong G' \times^{P} D_0.$$

\begin{proposition} \label{iso_final2SLn}
Le morphisme $\gamma \times \rho:\ \HH \rightarrow {Bl}_0(W/\!/G)$ est un isomorphisme.
\end{proposition}

\begin{proof}
On a le diagramme commutatif suivant: 
\begin{equation} \label{KommSLn}
 \xymatrix{ 
   \HH \ar[rr]^{\gamma \times \rho}   && {Bl}_0(W/\!/G)   \\
  {{G'} {\times}^{P} \HH'}  \ar@{.>}[rr]^{\theta}  \ar@{->>}[rd] \ar[u]^{\cong} && {{G'} {\times}^{P} D_0} \ar@{->>}[ld] \ar[u]_{\cong} \\ & G'/{P} }  
  \end{equation} 
où l'on note $\theta$ le morphisme de variétés induit par $\gamma \times \rho$ tel que le carré commute. 
Par construction, le morphisme $\theta$ est $G'$-équivariant. Soit $\theta_e:\ \HH' \rightarrow D_0$ le morphisme $P$-équivariant obtenu en restreignant $\theta$ à la fibre en $eP \in G'/P$. D'après le corollaire \ref{reductionHdiag}, le morphisme $\theta_e$ s'identifie au morphisme de Hilbert-Chow $\gamma':\ \HH' \rightarrow W'/\!/G$; ce dernier est un isomorphisme d'après le corollaire \ref{cas_faciles_SLn}. Il s'ensuit que $\theta$, et donc $\gamma \times \rho$, est un isomorphisme. 
\end{proof}

\subsubsection{La famille universelle lorsque \texorpdfstring{$1<n<n'$}{1<n<n'}} \label{n12SLn}

On rappelle que l'on note $\pi:\ \XX \rightarrow \HH$ la famille universelle. Soit $T$ (resp. $\underline{V'}$, $\underline{V}$) le fibré tautologique (resp. les fibrés triviaux de fibre $V'$,$V$) au dessus de la grassmanienne $\Gr(n,V'^*)$. Avec les notations précédentes, on a
\begin{align*}
\XX &\cong G' \times^P \XX' \text{ d'après la proposition \ref{reduction2},}\\
    & \cong G' \times^P W' \text{ d'après le corollaire \ref{cas_faciles_SLn}.}
\end{align*}
Autrement dit, le morphisme $\rho \circ \pi:\ \XX \rightarrow G'/P$ permet d'identifier $\XX$ à l'espace total du fibré vectoriel $$\Hom(\underline{V'}/T^{\perp},\underline{V}) \cong T \otimes \underline{V}.$$ 
Notre but est d'obtenir une description alternative de $\XX$ comme sous-variété de $W \times \PP(W/\!/G)$. On rappelle que $\PP(W/\!/G)=\Gr(n,V'^*)$ et donc, un point de $\PP(W/\!/G)$ peut être considéré comme une droite de $W/\!/G$ ou bien comme un sous-espace de dimension $n$ de $V'^*$ suivant le contexte.

\begin{proposition}  \label{familleUnivSLn}
On a un isomorphisme $G'$-équivariant 
$$\XX \cong  \{(w,L) \in W \times \PP(W/\!/G) \ |\ L^{\perp} \subset \Ker(w) \}.$$  
Et la famille universelle s'identifie à $(w,L) \in \XX \mapsto (\nu(w),L) \in Bl_0(W/\!/G)$.
\end{proposition}

\begin{proof}
On note $\ZZ:=\{(w,L) \in W \times \PP(W/\!/G) \ |\ L^{\perp} \subset \Ker(w) \}$. La seconde projection est un morphisme $G'$-équivariant vers l'espace homogène $\PP(W/\!/G) \cong G'/P$, donc d'après (\ref{iissoo}), on a un isomorphisme $G'$-équivariant $\ZZ \cong G' \times^P \Hom(V'/L_0^{\perp},V)$. Autrement dit, $\ZZ$ s'identifie à l'espace total du fibré vectoriel $\Hom(\underline{V'}/T^{\perp},\underline{V})$ au dessus de $\Gr(n,V'^*)$, et donc $\ZZ \cong \XX$ comme $G'$-variété.\\
Ensuite, d'après le corollaire \ref{reductionHdiag}, la famille universelle s'identifie à la composition des morphismes suivants: 
$$ \xymatrix{ \ZZ \ar@/_2pc/[rrrrrr] \ar[rr]^(0.4){\cong} && G' \times^P W' \ar[rr]^{G' \times^P \nu'} && G' \times^P W'/\!/G \ar[rr]^{\cong} && Bl_0(W/\!/G)  }$$\\ \\
Et on vérifie que le morphisme obtenu est bien $(w,L) \in \ZZ \mapsto (\nu(w),L) \in Bl_0(W/\!/G)$.
\end{proof}

\begin{remarque}
Soit $Z \subset W$ un sous-schéma fermé $G$-stable tel que $k[Z] \cong k[G]$ comme $G$-module. Alors $Z$ s'identifie à un point $(x,L)$ de $\HH$ et est défini comme sous-schéma de $W$ par
$$Z=\left \{  w \in W \ \middle|\  L^{\perp} \subset \Ker(w) \text{ et } \nu(w)=x \right\} .$$  
En particulier:\\
$\bullet$ Si $x \neq 0$, alors $Z$ est isomorphe à $SL_n(k)$.\\
$\bullet$ Si $x=0$, alors $Z$ est isomorphe à la variété déterminantielle $\{ M \in \MM_n(k) \ |\ \det(M)=0\}$. 
\end{remarque}

%% file: cas_rg_1.tex
\subsection{Détermination de la composante principale dans un cas simple}

\noindent Dans la section \ref{casSln}, la détermination de $\HH$ s'est faite en trois étapes:
\begin{enumerate}
\item étudier le morphisme de passage au quotient $W \rightarrow W/\!/G$,
\item montrer que $\HH=\HHp$ avec le principe de réduction,
\item montrer que $\gamma \times \rho$ est un isomorphisme entre $\HHp$ et $Bl_0(W/\!/G)$.
\end{enumerate} 
Nous allons voir que l'étape 3 est en fait un cas particulier d'un énoncé plus général (proposition \ref{enoncepgene}).

Soient, comme dans la section \ref{generalitéesHilbert}, $G$ un groupe algébrique réductif, $W$ un $G$-schéma affine réduit et $G' \subset \Aut^G(W)$ un sous-groupe algébrique. On suppose que $G'$ contient un sous-groupe central isomorphe à $\Gm$ tel que
\begin{itemize}
\item l'opération associée de $\Gm$ dans $W/\!/G$ a un point fixe $0$ qui est l'unique orbite fermée,
\item le quotient $X:=(W/\!/G \backslash \{0\})/\Gm$ est une unique orbite de $G'$. 
\end{itemize}
En particulier $W/\!/G$ est la réunion de deux $G'$-orbites: $\{0\}$ et $U:=W/\!/G \backslash \{0\}$. 
On suppose de plus l'existence d'un morphisme $G'$-équivariant $\rho:\ \HH \rightarrow X$. Cette hypothèse est, a priori, très forte, mais elle sera vérifiée dans plusieurs exemples que nous traiterons par la suite. Le morphisme de passage au quotient $\nu:\ W \rightarrow W/\!/G$ est plat sur un ouvert non-vide, donc sur $U$ par $G'$-homogénéité. Si $\nu$ est plat partout, alors $\HH \cong W/\!/G$ et $\gamma$ est un isomorphisme. Sinon, on a la

\begin{proposition} \label{enoncepgene}
Sous les hypothèses précédentes, le morphisme $\gamma \times \rho$ envoie isomorphiquement $\HHp$ dans $Bl_0(W/\!/G):=\{(f,L) \in W/\!/G \times X\ |\ f \in L\}$.
\end{proposition}

\begin{proof}
On a un morphisme $G'$-équivariant $\gamma \times \rho:\ \HHp \rightarrow W/\!/G \times X$. Soit $y_0 \in U$ et $H:=\Stab_{G'}(y_0)$, alors d'après la proposition \ref{chow}, il existe un unique $z_0 \in \HH$ tel que $\gamma(z_0)=y_0$. Le morphisme $\gamma$ est $G'$-équivariant, donc $H \subset \Stab_{G'}(z_0)$. De même, le morphisme $\rho$ est $G'$-équivariant, donc $H \subset P_0:=\Stab_{G'}(\rho(z_0))$. On note $H'$ le sous-groupe algébrique de $G'$ engendré par $H$ et par $\Gm$. On a $H' \subset P_0$ et 
$$\dim(H')=1+\dim(H)=1+\dim(G')-\dim(W/\!/G)=\dim(G')-\dim(\PP(W/\!/G))=\dim(P_0).$$ 
On en déduit que $H'=P_0$ et donc $H$ stabilise un unique point de $\PP(W/\!/G)$ qui est le point correspondant à la droite $\left\langle y_0 \right\rangle$ de $W/\!/G$. 
Il s'ensuit que, pour tout $z \in \gamma^{-1}(U)$, on a $\gamma \times \rho (z) \in Bl_0(W/\!/G)$. Puis, $(\gamma \times \rho)^{-1}(Bl_0(W/\!/G))$ est un fermé de $\HHp$ qui contient l'ouvert $\gamma^{-1}(U)$, donc c'est $\HHp$ tout entier. On a donc montré que $\gamma \times \rho$ envoie $\HHp$ dans $Bl_0(W/\!/G)$, il reste à voir que c'est un isomorphisme.\\
On considère le diagramme commutatif suivant:
\begin{equation}
 \xymatrix{ \HHp \ar[rr]^{\gamma \times \rho} \ar[rd]_{\rho} && Bl_0(W/\!/G) \ar[ld]^{p_2} \\ & X }
 \end{equation}     
où $p_2$ est la seconde projection. Tous les morphismes qui apparaissent dans ce diagramme sont $G'$-équivariants. 
On fixe $x \in X$ et soit $P:=\Stab_{G'}(x)$, alors $X \cong G'/P$. On note $F_1$ et $F_2$ les fibres schématiques en $eP$ des morphismes $\rho$ et $p_2$ respectivement. D'après (\ref{iissoo}), on a des isomorphismes $G'$-équivariant $\HHp \cong G' \times^P F_1$ et $Bl_0(W/\!/G) \cong G' \times^P F_2$ et la restriction de $\gamma \times \rho$ à $F_1$ donne un morphisme $P$-équivariant $f:\ F_1 \rightarrow F_2$ tel que $\gamma \times \rho$ s'identifie à $G' \times^P f$, le morphisme induit par $f$. Le morphisme $\gamma$ est birationnel et propre (voir la section \ref{generalitéesHilbert}), donc $\gamma \times \rho$ est birationnel et propre, et donc $f$ également. On a bien sûr $F_2 \cong \Aff$, et donc $f$ un isomorphisme. Le résultat s'ensuit.   
\end{proof}

On fait maintenant l'hypothèse supplémentaire que $\HH=\HHp$. On a le carré cartésien
$$ \xymatrix{
    Bl_0(W/\!/G) \times_{W/\!/G} W \ar[r]^(0.7){q_2} \ar[d]_{q_1}  & W \ar[d]^{\nu} \\
    Bl_0(W/\!/G) \ar[r]_{p_1} & W/\!/G
  }$$
où $q_1$, $q_2$ sont les projections naturelles. La variété $\ZZ:=\overline{q_1^{-1}(p_1^{-1}(U))}$ est une composante irréductible de $Bl_0(W/\!/G) \times_{W/\!/G} W$ et on a le

\begin{corollaire} \label{kororg1}
Le morphisme $q_1:\ \ZZ \rightarrow Bl_0(W/\!/G)$ s'identifie à la famille universelle.  
\end{corollaire}

\begin{proof}
On considère le diagramme commutatif suivant:
\begin{equation*}
 \xymatrix{ \ZZ \ar[rr]^{q_1} \ar@{.>}[rd]_{\delta} && Bl_0(W/\!/G) \ar[ld]^{p_2} \\ & X }
 \end{equation*} 
où $\delta$ est défini comme la composition $p_2 \circ q_1$. On fixe $x_0 \in X$ et on note $F_1'$ et $F_2'$ les fibres schématiques en $x_0$ des morphismes $\delta$ et $p_2$ respectivement. On a $F_2' \cong \Aff$, donc d'après \cite[Proposition 9.7]{Ha}, le morphisme ${q_1}_{|F_1'}$ est plat et donc $q_1$ est plat. Par la propriété universelle du schéma de Hilbert invariant, il existe un morphisme $\psi:\ Bl_0(W/\!/G) \rightarrow \HH$ tel que le diagramme suivant soit cartésien:
$$\xymatrix{
    \ZZ \ar[r] \ar[d]_{q_1}  & \XX \ar[d]^{\pi} \\
    Bl_0(W/\!/G) \ar[r]_(0.6){\psi} & \HH
  }$$
D'après la proposition \ref{chow} et par définition de la famille $q_1:\ \ZZ \rightarrow Bl_0(W/\!/G)$, le morphisme $(\gamma \times \rho) \circ \psi$ est l'identité sur l'orbite ouverte de $Bl_0(W/\!/G)$, et donc sur $Bl_0(W/\!/G)$ tout entier. Or, d'après la proposition \ref{enoncepgene}, $\gamma \times \rho$ est un isomorphisme et donc $\psi$ est un isomorphisme. Le résultat s'ensuit. 
\end{proof}

On finit cette section par quelques remarques sur les résultats obtenus.
\begin{itemize}
\item Le morphisme $\psi$ qui apparaît dans la preuve du corollaire \ref{kororg1} est l'inverse du morphisme $\gamma \times \rho$ de la proposition \ref{enoncepgene}.
\item La variété $Bl_0(W/\!/G)$ est la réunion de deux $G'$-orbites, et donc la famille universelle $q_1:\ \ZZ \rightarrow Bl_0(W/\!/G)$ admet deux types de fibres: les fibres générales, qui sont les fibres en n'importe quel point $(x,L)$, avec $x \neq 0$, et les fibres spéciales, qui sont les fibres en n'importe quel point $(0,L)$. La fibre générale de $q_1$ est bien sûr isomorphe à la fibre générale du morphisme de passage au quotient $W \rightarrow W/\!/G$.   
\end{itemize}

%% file: position_probleme_GLn.tex
\chapter{Cas du groupe linéaire}  \label{chp2}

\section{Cas de \texorpdfstring{$GL(V)$}{GL(V)} opérant dans \texorpdfstring{$V^{\oplus n_1} \oplus V^{* \oplus n_2}$}{n1V+n2V*}} \label{GLngénéral}

On se place dans la situation $2$: on a $G:=GL(V)$, $G':=GL(V')$, $W:={\Hom}(V_1,V) \times {\Hom}(V,V_2)$ et l'opération de $G' \times G$ dans $W$ est donnée par (\ref{aktionGLn}).

\subsection{Etude du morphisme de passage au quotient} \label{description_quotient}

Les résultats essentiels de cette section sont les propositions \ref{descriptiongeofib} et \ref{ouvertplatitude} qui décrivent les fibres et l'ouvert de platitude de $\nu$. La plupart des résultats de cette section se trouvent dans \cite[§II.4.1]{Kra}. Cependant, d'une part le livre \cite{Kra} est écrit dans la langue de Goethe, ce qui peut être une difficulté pour certains lecteurs, et d'autre part, nos formulations, notations et méthodes sont quelque peu différentes. Donc, dans un souci de cohérence et de complétude, nous pensons que cette section a pleinement sa place dans ce mémoire. 

D'après le premier théorème fondamental pour $GL(V)$ (voir \cite[§9.1.4]{Pro}), la $k$-algèbre des invariants ${k[W]}^G$ est engendrée par les invariants $(i\ |\ j)$, où pour chaque couple $(i,j)$, $i=1,\ldots,n_1$, $j=1,\ldots,n_2$, la forme bilinéaire $(i\ |\ j)$ sur $W \cong (V_{1}^{*} {\otimes} V) \oplus (V^{*} {\otimes} V_2) \cong V^{\oplus n_1} \oplus {V^{* \oplus n_2}}$ est définie par (\ref{formebili}). 

\noindent Il s'ensuit que le morphisme de passage au quotient $\nu:\ W \rightarrow W/\!/G$ est donné par:

$$\begin{array}{lrcl}
 \nu:  &{\Hom}(V_1,V) \times {\Hom}(V,V_2)  & \rightarrow  & {\Hom}(V_1,V_2)  \\
        & (u_1,u_2)  & \mapsto      &  u_2 \circ u_1 .
\end{array}$$
Et donc 
$$W/\!/G=\{ f \in  {\Hom}(V_1,V_2) \ \mid  \ \rg(f) \leq n \}=:{\Hom}(V_1,V_2)^{\leq n} $$
est une variété déterminantielle.\\
Si $n_1,n_2 >n$, alors la variété $W/\!/G$ est de dimension $nn_1+nn_2-n^2$, de Cohen-Macaulay (\cite[Theorem 3.1]{ACGH}), normale et son lieu singulier est ${\Hom}(V_1,V_2)^{\leq n-1}$ (\cite[§ II.2.2]{ACGH}). Nous verrons dans la section \ref{appendice1} que $W/\!/G$ est de Gorenstein si et seulement si $n_1=n_2$.\\
Sinon, $W/\!/G=\Hom(V_1,V_2)$ est un espace affine.\\ 
On note 
$$N:=\min(n_1,n_2,n).$$ 
L'opération de $G'$ dans $W$ induit une opération dans $W/\!/G$ telle que $\nu$ soit $G'$-équivariant. Cette opération est donnée par
$$\forall (g_1,g_2) \in G',\  \forall f \in  {\Hom}(V_1,V_2),\ (g_1,g_2).f:=g_2 \circ f \circ g_{1}^{-1} $$
La variété quotient $W/\!/G$ se décompose en $N+1$ orbites pour cette opération: 
\begin{equation}
U_i:=\{ f \in  {\Hom}(V_1,V_2) \ \mid  \ \rg(f) = i \}
\end{equation}
pour $i=0, \ldots,N$, les adhérences de ces orbites sont imbriquées de la façon suivante: 
\begin{equation}
\{0\}=\overline{U_0} \subset \overline{U_1} \subset \cdots\subset \overline{U_{N}}=W/\!/G .
\end{equation}
En effet, pour chaque $i=0,\ldots,N$, on a $\overline{U_i}={\Hom}(V_1,V_2)^{\leq i}$. En particulier, l'orbite $U_{N}$ est un ouvert dense de $W/\!/G$.

\begin{definition} \label{defnilk}
On appelle nilcône de $\nu$, noté $\NNN(W,G)$, la fibre schématique en $0$ du morphisme $\nu$. 
\end{definition}

Certaines propriétés géométriques de $\NNN(W,G)$ sont étudiées dans \cite{KS}. En particulier, le schéma $\NNN(W,G)$ est toujours réduit mais il est irréductible si et seulement si $n_1+n_2 \leq n$ (\cite[Theorem 9.1]{KS}). Nous allons déterminer les composantes irréductibles de $\NNN(W,G)$ et leurs dimensions. Soit $m \in \{0, \ldots,n\}$, on définit 
$$X_m:=\left\{(u_1,u_2)\in W\ \middle| 
    \begin{array}{ll}
       \Im(u_1) \subset \Ker(u_2),\\
       \rg(u_1) \leq \min(n_1,m),\\
       \dim(\Ker(u_2)) \geq \max(n-n_2,m). 
    \end{array}
 \right\}$$ 
et soit 
$$\xymatrix{ &  Z_m \ar@{->>}[ld]_{p_1} \ar@{->>}[rd]^{p_2} \\   X_m && \Gr(m,V) }$$ 
où 
$$Z_m:=\{(u_1,u_2,L) \in \Hom(V_1,V) \times \Hom(V,V_2)\times \Gr(m,V)\ \mid  \ \Im(u_1) \subset L \subset \Ker(u_2)\}$$ 
et les $p_i$ sont les projections naturelles. On fixe $L_0 \in \Gr(m,V)$. La seconde projection munit $Z_m$ d'une structure de fibré vectoriel homogène au dessus de $\Gr(m,V)$, de fibre en $L_0$ isomorphe à $F_m:=\Hom(V_1,L_0) \times \Hom(V/L_0,V_2)$. Autrement dit, on a $Z_m=\Hom(\underline{V_1},T) \times \Hom(\underline{V}/T,\underline{V_2})$ où $T$ est le fibré tautologique de $\Gr(m,V)$ et $\underline{V},\underline{V_1},\underline{V_2}$ désignent les fibrés triviaux de fibres $V,V_1$ et $V_2$ respectivement. Donc $Z_m$ est une variété lisse de dimension:
 \begin{align*}
\dim(Z_m)&=\dim(\Gr(m,V))+\dim(\Hom(V_1,L_0) \times \Hom(V/L_0,V_2))\\
         &=m(n-m)+n_1 m+(n-m)n_2.
  \end{align*}

\begin{proposition} \label{compirredfibzero}
Chaque $X_m$ est un fermé irréductible de $W$ et les composantes irréductibles de $\NNN(W,G)$ sont:
$$\left\{
    \begin{array}{ll}
        X_i,\ i={\max(0,n-n_2)}, {\max(0,n-n_2)}+1, \ldots ,{\min(n,n_1)} &\text{ si } n<n_1+n_2,\\
        X_{n_1} &\text{ si } n\geq n_1+n_2.
    \end{array}
\right.$$
De plus, lorsque $m \leq n_1$ ou $m \geq n-n_2$, l'application $p_1:\ Z_m \rightarrow X_m$ est birationnelle. 
\end{proposition}

\begin{proof}
Déjà, par définition, les $X_m$ sont des fermés de $W$. Le morphisme $p_1$ est surjectif, et $Z_m$ est irréductible, donc $X_m$ est irréductible. 
Ensuite
$$\NNN(W,G)=\{(u_1,u_2)\in \Hom(V_1,V) \times \Hom(V,V_2)\ \mid  \ \Im(u_1) \subset \Ker(u_2) \}= \bigcup_{i=0}^{n} X_i . $$
Si  $n_1 \leq n-n_2$, alors  $$\left\{
    \begin{array}{ll}
        X_0 \subset  \cdots \subset X_{n_1},   \\
        X_{n_1} = \cdots=X_{n-n_2},\\
        X_{n-n_2} \supset  \cdots \supset X_n,
    \end{array}
\right.$$  
donc $X=X_{n_1}$.\\
Si  $n_1 > n-n_2$, alors  $$\left\{
    \begin{array}{ll}
        X_0 \subset  \cdots \subset X_{\max(0,n-n_2)},   \\
        X_{\min(n,n_1)} \supset  \cdots \supset X_n,
    \end{array}
\right.$$  
et on vérifie comme conséquence de la définition des $X_m$ qu'il n'y a pas d'autre relation d'inclusion.
Ensuite, soient
$$Z'_m:=\{(u_1,u_2,L) \in Z_m\ \mid  \ \rg(u_1)=\min(m,n_1) \text{ et } \dim(\Ker(u_2))=\max(m,n-n_2)\}$$ et 
$$X'_m:=\{(u_1,u_2)\in X_m\ \mid  \ \rg(u_1)=\min(m,n_1) \text{ et } \dim(\Ker(u_2))=\max(m,n-n_2)\} .$$ 
Alors $Z'_m$ (resp. $X'_m$) est un ouvert dense de $Z_m$ (resp. de $X_m$) et on a $Z'_m=p_{1}^{-1}(X'_m)$.
Si $m \leq n_1$ ou $m \geq n-n_2$, alors $p_1:\ Z'_m \rightarrow X'_m$ est un isomorphisme, donc $p_1$ est birationnelle.
\end{proof}

\begin{corollaire} \label{fibzero}
La dimension de $\NNN(W,G)$ est: 
\begin{itemize} \renewcommand{\labelitemi}{$\bullet$}
\item $n n_2$ lorsque $n \leq n_2-n_1$,
\item $n n_1$ lorsque $n \leq n_1-n_2$,
\item $\frac{1}{4}n(n+2n_1+2n_2)+\frac{1}{4}{(n_1-n_2)}^2$ lorsque $|n_1-n_2|<n<n_1+n_2$ et $n+n_1-n_2$ est pair, 
\item $\frac{1}{4}n(n+2n_1+2n_2)+\frac{1}{4}{(n_1-n_2)}^2-\frac{1}{4}$ lorsque $|n_1-n_2|<n<n_1+n_2$ et $n+n_1-n_2$ est impair, 
\item $n n_1+n n_2 -n_1 n_2$ lorsque $n \geq n_1+n_2$.
\end{itemize}
\end{corollaire}

\begin{proof}
D'après la proposition \ref{compirredfibzero}, il nous suffit de calculer la dimension des $X_m$ pour certains $m$ particuliers. 
On note $P(m):=m(n-m)+n_1 m+n_2 (n-m)$ la dimension de $Z_m$. Lorsque $m \leq n_1$ ou $m \geq n-n_2$, on a $\dim({X_m})=\dim(Z_m)=P(m)$.


Si $n \geq n_1+n_2$, alors 
$$\dim(\NNN(W,G))=\dim(X_{n_1})=n n_1+ n n_2- n_1 n_2 . $$

Si $n < n_1+n_2$, alors \begin{align*} 
\dim(\NNN(W,G))&=\dim \left(\bigcup_{i=\max(0,n-n_2)}^{\min(n,n_1)} X_i \right)\\
                 &=\max_{i={\max(0,n-n_2)}, \ldots,{\min(n,n_1)}} \dim(X_i) \\
                 &=\max_{i={\max(0,n-n_2)}, \ldots,{\min(n,n_1)}} P(i). 
                 \end{align*}
On est donc ramené à étudier les variations du polynôme $P$:\\
$$P'(m)=0 \Leftrightarrow m=\frac{1}{2}(n+n_1-n_2)=:m'.$$
Si $m' \leq \max(0,n-n_2)$, alors $m' \leq 0$ et donc  $\dim(\NNN(W,G))=P(0)= n n_2$.\\
Si $m' \geq \min(n,n_1)$, alors $m' \geq n$ et donc $\dim(\NNN(W,G))=P(n)=n n_1$.\\
Si $\max(0,n-n_2)<m'<\min(n,n_1)$, on a deux possibilités: ou bien $m' \in \NN$ et alors $\dim(\NNN(W,G))=P(m')$, ou bien  $m' \notin \NN$ et alors $\dim(\NNN(W,G))=P(m'+\frac{1}{2})$.\\
Et ces différents cas fournissent le résultat annoncé. 
\end{proof}

Nous allons maintenant nous intéresser à la description géométrique des fibres de $\nu$ au dessus de chaque orbite $U_i$. On rappelle que par homogénéité toutes les fibres au dessus d'une orbite donnée sont isomorphes, il suffit donc de décrire la fibre de $\nu$ en un point de chaque orbite.

\begin{notation} \label{defj}
Soit $ \ 0 \leq r \leq N$, on note 
$$J_r =\begin{bmatrix}
{I}_r \ \ &0_{r,n_1-r} \\
0_{n_2-r,r}  \ &0_{n_2-r,n_1-r} 
\end{bmatrix}$$
où $I_r$ est la matrice identité de taille $r$. La matrice $J_r$ s'identifie à un élément de $U_r$ via l'isomorphisme $\Hom(V_1,V_2) \cong \MM_{n_2,n_1}(k)$.        
\end{notation}
On fixe $r \in \{0, \ldots,N\}$ et on définit 
$$w_r:=\left( \begin{bmatrix}
I_r &0 \\
0   &0 \end{bmatrix},\begin{bmatrix}
I_r  &0 \\
0  &0 \end{bmatrix} \right) \in W$$ 
et $G_r$ le stabilisateur de ${w_r}$ dans $G$. On note $E_r$ (resp. $E_r^*$) la représentation standard (resp. duale) de $G_r$ et $V_0$ la représentation triviale de $G_r$.

\begin{lemme} \label{Gw}
$$G_r =\left\{ \begin{bmatrix}
I_r   &0 \\
0     &M \end{bmatrix},\ M\in {GL}_{n-r}(k) \right\} \cong {GL}_{n-r}(k)$$ et $V=r V_0 \oplus E_r$ comme $G_r$-module.    
\end{lemme}

\begin{proof}
On écrit $g \in GL(V)$ sous forme d'une matrice par blocs: $g=\begin{bmatrix}
A  &B \\
C  &D \end{bmatrix}$, alors $$g.\begin{bmatrix}
I_r   &0 \\
0     &0 \end{bmatrix}=\begin{bmatrix}
I_r   &0 \\
0     &0 \end{bmatrix} \Rightarrow A=I_{r} \text{ et } C=0$$ puis  
$$\begin{bmatrix}
I_r   &0 \\
0     &0 \end{bmatrix}.g^{-1}=\begin{bmatrix}
I_r   &0 \\
0     &0\end{bmatrix} \Rightarrow  B=0$$ d'où $g= \begin{bmatrix}
I_r  &0 \\
0  &M \end{bmatrix}$ avec $M \in {GL}_{n-r}(k)$. Réciproquement, si $g$ est de cette forme, alors $g \in G_r$. Ensuite, le sous-espace vectoriel de $V$ engendré par les $n-r$ derniers vecteurs de la base ${\BB}$ est la représentation standard de $G_r$ et $G_r$ opère trivialement dans le sous-espace vectoriel de $V$ engendré par les $r$ premiers vecteurs de la base ${\BB}$. 
\end{proof}

\begin{lemme} \label{orbfermee}
L'orbite $G.{w_r} \subset W$ est fermée dans $W$, et c'est l'unique orbite fermée contenue dans ${\nu}^{-1}(J_r)$.
\end{lemme}

\begin{proof}
On a $\nu ({w_r})=J_r$ et d'après le lemme \ref{Gw}, $G_r \cong {GL}_{n-r}(k)$ est un sous-groupe réductif de $G$. On applique \cite[§I.6.2.5, Theorem 10]{SB} qui nous fournit l'équivalence: 
$$ G.{w_r} \text{ est fermée dans } W \Leftrightarrow C_G(G_r).{w_r} \text{ est fermée dans } W.$$ 
Puis $C_G(G_r)=\left\{\begin{bmatrix}
M  &0 \\
0  &\lambda I_{n-r} \end{bmatrix},\ M \in {GL}_{r}(k),\ \lambda \in \Gm \right\}$, où $\Gm$ est le groupe multiplicatif. Donc $$C_G(G_r).{w_r}=\left\{ \left( \begin{bmatrix}
M   &0 \\
0   &0 \end{bmatrix},\begin{bmatrix}
M^{-1} &0 \\
0     &0 \end{bmatrix} \right)\ ,\ M \in {GL}_{r}(k)\right\}$$ est un fermé de $W$, et donc $G.{w_r}$ est une orbite fermée de ${\nu}^{-1}(J_r)$.\\     
Enfin, d'après \cite[§II.3.1, Théorème 1]{SB}, la fibre ${\nu}^{-1}(J_r)$ contient une unique orbite fermée, d'où le résultat.
\end{proof}

\begin{definition} \label{slice}
Soit $x \in W$ tel que l'orbite $G.x$ est fermée dans $W$ et $G_x$ le stabilisateur de $x$ dans $G$. Alors, d'après \cite[§6.2.1]{SB}, le $G_x$-module $T_x(G.x)$ admet un supplémentaire $G_x$-stable $M_x$ dans $W$. La représentation $M_x$ de $G_x$ ainsi construite est appelée représentation slice de $G$ en $x$. 
\end{definition}

\begin{lemme} \label{slice_explicite}
On a un isomorphisme de $G_r$-modules: 
$$ M_{w_r} \cong (n_1-r) E_r \oplus (n_2-r) E_r^{*} \oplus r(n_1+n_2-r) V_0  . $$
\end{lemme}

\begin{proof}
Par définition de la représentation slice $M_{w_r}$ de $G_r$, on a $M_{w_r} \cong W/T_{w_r}(G.{w_r})$ comme $G_r$-module. On déduit du lemme \ref{Gw} la décomposition de $W$ comme $G_r$-module: 
\begin{align*}
 W&\cong(V_{1}^{*} \otimes V) \oplus (V^{*} \otimes V_2) \\
           &\cong (V_{1}^{*} \otimes (r V_0 \oplus E_r)) \oplus ((r V_0 \oplus E_r^{*}) \otimes V_2) \\
           &\cong (V_{1}^{*} \otimes E_r) \oplus (E_r^{*} \otimes V_2) \oplus ((V_{1}^{*} \otimes r V_0) \oplus (r V_0  \otimes V_2))\\
           &\cong n_1 E_r \oplus n_2 E_r^{*} \oplus r(n_1+n_2) V_0.    
\end{align*}
Puis $T_{w_r}(G.{w_r}) \cong \gg/\gg_r$ où l'on note $\gg$ l'algèbre de Lie de $G$ et $\gg_r$ l'algèbre de Lie de $G_r$. Or 
\begin{align*}
\gg &\cong V^{*} \otimes V\\
       &\cong (r V_0 \oplus E_r^{*}) \otimes (r V_0  \oplus E_r)\\
       &\cong (r V_0  \otimes r V_0 ) \oplus ( r V_0  \otimes  E_r) \oplus (E_r^{*} \otimes r V_0 )  \oplus(E_r^{*} \otimes E_r)  
\end{align*} 
et 
$$\gg_r \cong E_r^{*} \otimes E_r$$ 
donc \begin{align*} T_{w_r}(G.{w_r}) &\cong (r V_0  \otimes r V_0 ) \oplus ( r V_0  \otimes  E_r) \oplus (E_r^{*} \otimes r V_0 )\\
& \cong r E_r \oplus r E_r^{*} \oplus r^2 V_0 \end{align*} 
d'où 
\begin{align*} 
M_{w_r} &\cong W/T_{w_r}(G.{w_r})\\
        &\cong (n_1-r) E_r \oplus (n_2-r) E_r^{*} \oplus r(n_1+n_2-r) V_0 
\end{align*} 
comme $G_r$-module.
\end{proof}

On note ${\nu}_M:\ M_{w_r} \rightarrow M_{w_r}/\!/G_r$ le morphisme de passage au quotient et  
$$\NNN(N_{w_r},G_r):={{\nu}_M}^{-1}({\nu}_M(0))$$
le nilcône de ${\nu}_M$. Le groupe $G_r$ opère naturellement dans $G$ par multiplication à droite, ainsi que dans $\NNN(M_{w_r},G_r)$ par définition de ${\nu}_M$. On peut donc considérer le quotient 
$$G {\times}^{G_r} \NNN(M_{w_r},G_r)$$ 
qui est naturellement muni d'une structure de $G$-schéma d'après \cite[§I.5.14]{Jan}.\\
Ensuite, on définit ${(W/\!/G)}^{(G_r)} \subset W/\!/G$ l'ensemble des $G$-orbites fermées de $W$ telles que $G_r$ est conjugué au stabilisateur d'un point de ces orbites. En particulier, $G.{w_r} \in {(W/\!/G)}^{(G_r)}$ d'après le lemme \ref{orbfermee}. On note $W^{(G_r)}:={\nu}^{-1}({(W/\!/G)}^{(G_r)}) \subset W$. Alors, d'après \cite[§6.2.3, Theorem 8]{SB}, les ensembles ${(W/\!/G)}^{(G_r)}$ et $W^{(G_r)}$ sont des sous-variétés lisses de $W/\!/G$ et $W$ respectivement. Il s'ensuit que ${\nu '}:=\nu_{| W^{(G_r)}}:\ W^{(G_r)} \rightarrow {(W/\!/G)}^{(G_r)}$ est un morphisme de variétés, et en fait, toujours d'après \cite[§6.2.3, Theorem 8]{SB}, c'est même une fibration de fibre isomorphe (comme schéma) à 
$$F_{w_r}:=G \times^{G_r} \NNN(M_{w_r},G_r) . $$ 
En particulier, on a 
$${\nu}^{-1}(J_r)= {\nu '}^{-1}(J_r) \cong G {\times}^{G_r} \NNN(M_{w_r},G_r) . $$
Soient $F_1$, $F_2$ et $F_3$ des espaces vectoriels de dimensions $n_1-r$, $n_2-r$ et $r(n_1+n_2-r)$ respectivement dans lesquels $G_r$ opère trivialement. D'après le lemme \ref{slice_explicite} on a un isomorphisme de $G_r$-modules 
$$ M_{w_r} \cong \Hom(F_1, E_r) \times \Hom(E_r, F_2) \times F_3  . $$   
Ensuite, le morphisme de passage au quotient ${\nu}_M$ est donné par:
 $$\begin{array}{lrcl}
 {\nu}_{M}:\  &\Hom(F_1, E_r) \times \Hom(E_r, F_2) \times F_3  & \rightarrow  & \Hom(F_1, F_2) \times F_3  . \\
        & (u_1',u_2',x)  & \mapsto      &  (u_2' \circ u_1',x) 
\end{array}$$
Donc $\NNN(M_{w_r},G_r):= {\nu}_{M}^{-1}({\nu}_M(0))={\nu}_{M}^{-1}(0) \cong {{\nu}_{M}'}^{-1}(0)$ avec
 $$\begin{array}{lrcl}
 {\nu'_{M}}:\  &\Hom(F_1, E_r) \times \Hom(E_r, F_2) & \rightarrow  & \Hom(F_1, F_2)  .  \\
        & (u_1',u_2')  & \mapsto      &  u_2' \circ u_1' 
\end{array}$$ 
La proposition qui suit résume notre étude de la fibre du morphisme $\nu$ en $J_r$, pour $r=0, \ldots,N$.

\begin{proposition} \label{descriptiongeofib}
Avec les notations précédentes, on a un isomorphisme $G$-équivariant
$${\nu}^{-1}(J_r) \cong G {\times}^{G_r} {\nu'_M}^{-1}(0)  . $$ 
En particulier, si l'on note $H:=G_{N}$, on a 
$${\nu}^{-1}(J_{N}) \cong \left\{
    \begin{array}{ll}
        G &\text{ si } N=n, \\
        G/H &\text{ si } N=n_1=n_2<n, \\
        G {\times}^{H} \Hom(E_{N},F_2) &\text{ si } N=n_1<\min(n,n_2), \\
        G {\times}^{H} \Hom(F_1,E_{N}) &\text{ si } N=n_2<\min(n,n_1). 
    \end{array}
\right.$$ 
où $H$ opère de la façon suivante dans $\Hom(F_1,E_{N}) \times \Hom(E_{N},F_2)$:
$$\forall h \in H,\ \forall (u_1',u_2') \in  \Hom(F_1,E_{N}) \times \Hom(E_{N},F_2),\ h.(u_1',u_2'):=(h \circ  u_1',u_2' \circ h^{-1}).$$    
\end{proposition}

Comme pour le cas $G=SL_n(k)$ étudié dans la section \ref{casSln}, si $N=n$, alors \cite[§2.1, Theorem 6]{SB} implique que $\nu$ est un $G$-fibré principal au dessus de l'ouvert $U_{N}$.

\begin{corollaire} \label{dimfibre}
Soit $r \in \{0, \ldots,N\}$, alors la dimension de la fibre du morphisme $\nu$ en $J_r$ vaut: 
\begin{itemize} \renewcommand{\labelitemi}{$\bullet$}
\item $n \, n_2+n \, r-n_2  r$ lorsque $n-r \leq n_2-n_1$,
\item $n \, n_1+n \, r-n_1 r$  lorsque $n-r \leq n_1-n_2$,
\item $\frac{1}{2}(n-r)(n_1+n_2)+\frac{1}{4}{(r+n)}^2+\frac{1}{4}{(n_1-n_2)}^2$ lorsque $|n_1-n_2|<n-r<n_1+n_2-2r$ et $n+n_1-n_2-r$ est pair, 
\item $\frac{1}{2}(n-r)(n_1+n_2)+\frac{1}{4}{(r+n)}^2+\frac{1}{4}{(n_1-n_2)}^2-\frac{1}{4}$ lorsque $|n_1-n_2|<n-r<n_1+n_2-2r$ et $n+n_1-n_2-r$ est impair, 
\item $n n_1+n n_2 -n_1 n_2$ lorsque $n \geq n_1+n_2-r$.
\end{itemize}
\end{corollaire}

\begin{proof}
D'après la proposition \ref{descriptiongeofib}:  
\begin{align*}
\dim({\nu}^{-1}(J_r))  &=\dim(G {\times}^{G_r} {\nu'_M}^{-1}(0))\\
        &=\dim(G)+\dim({\nu'_M}^{-1}(0))-\dim(G_r) \\
        &=n^2-(n-r)^2+\dim({\nu'_M}^{-1}(0))\\
        &=2nr-r^2+\dim({\nu'_M}^{-1}(0)).
\end{align*}
Enfin, le corollaire \ref{fibzero} donne $\dim({\nu'_M}^{-1}(0))$ en fonction de $n$, $n_1$, $n_2$ et $r$. Le résultat annoncé s'ensuit. \end{proof}


Pour chaque triplet $(n,n_1,n_2)$, le corollaire \ref{dimfibre} permet d'une part de calculer la dimension de la fibre générique de $\nu$, d'autre part de déterminer l'ouvert de platitude de $\nu$. 

\begin{proposition} \label{ouvertplatitude}
La dimension de la fibre générique et l'ouvert de platitude de $\nu$ sont donnés par le tableau suivant:
\vspace*{0.5mm}
\begin{center}
\begin{tabular}{|c|c|c|}
  \hline
  configuration & dim. de la fibre générique & ouvert de platitude \\
  \hline
  $n>\max(n_1,n_2)$ &     $n n_1+n n_2-n_1 n_2$ & $U_{N} \cup  \cdots \cup U_{\max(n_1+n_2-n-1, 0)}$\\
  $n=\max(n_1,n_2)$ &     $n^2$ & $U_{N} \cup U_{N-1} $\\  
  $\min(n_1,n_2) \leq n<\max(n_1,n_2)$     & $n n_1+n n_2-n_1 n_2$ & $U_{N}$\\
  $n < \min(n_1,n_2)$  & $n^2$ & $U_{N}$\\ 
  \hline    
\end{tabular}\\
\end{center}
\vspace*{0.5mm}
\end{proposition}

\begin{proof}
On sait que $\nu$ est plat sur $U_{N}$. On distingue alors deux cas de figure.\\
$\bullet$ Soit $W/\!/G$ est une variété lisse (c'est-à-dire $n \geq n_1$ ou $n \geq n_2$), auquel cas d'après \cite[Exercice 10.9]{Ha}, le morphisme $\nu$ est plat sur l'orbite $U_r$, $0 \leq r \leq N-1$ si et seulement si la dimension de la fibre au dessus de $U_r$ est égale à la dimension de la fibre au dessus de $U_{N}$. Il s'agit donc, pour chaque triplet $(n,n_1,n_2)$, de déterminer les valeurs de $r$ telles que la dimension de la fibre ${\nu}^{-1}(J_r)$ coïncide avec la dimension de la fibre générique.\\
$\bullet$ Soit $W/\!/G$ n'est pas une variété lisse (c'est-à-dire $n_1,n_2>n$), alors le critère précédent ne s'applique plus. Cependant, les fibres d'un morphisme plat ont nécessairement toutes la même dimension, or d'après le corollaire \ref{dimfibre}, la dimension de la fibre de $\nu$ en $J_{r} \in U_{r},\ r=0, \ldots,N-1$, est strictement plus grande que la dimension de la fibre générique, donc $U_{N}$ est l'ouvert de platitude de $\nu$. 
\end{proof}

\begin{corollaire} \label{ouverttoutplat}
Le morphisme $\nu$ est plat sur $W/\!/G$ tout entier si et seulement si $n \geq n_1+n_2-1$ et dans ce cas $W/\!/G=\Hom(V_1,V_2)$. 
\end{corollaire}

\noindent Le corollaire qui suit est une conséquence de la proposition \ref{chow} et du corollaire \ref{ouverttoutplat}:

\begin{corollaire} \label{cas_facile}
Si $n \geq n_1+n_2-1$, alors $\HH \cong \Hom(V_1,V_2)$ et $\gamma$ est un isomorphisme.
\end{corollaire}

On s'intéresse, dans la proposition qui suit, à la fonction de Hilbert de la fibre générique de $\nu$. Autrement dit on va déterminer, pour chaque $M \in \Irr(G)$, la multiplicité $h_W(M)$ du $G$-module $M$ dans le $G$-module $k[{\nu}^{-1}(J_{N})]$. On note comme précédement $H:=G_{N} \cong GL_{n-N}(k)$ le stabilisateur de $w_{N}$ dans $G$.

\begin{proposition} \label{fcthilb}
La fonction de Hilbert de la fibre générique de $\nu$ est donnée par:\\
$\forall M \in \Irr(G),\ h_W(M)= \left\{
    \begin{array}{ll}
        \dim(M) &\text{ si } N=n, \\
        \dim(M^{H}) &\text{ si } N=n_1=n_2<n, \\
        \dim({(M {\otimes} k[\Hom(E_{N},F_2)])}^{{H}}) &\text{ si } N=n_1<\min(n,n_2), \\
        \dim({(M {\otimes} k[\Hom(F_1,E_{N})])}^{{H}}) &\text{ si } N=n_2<\min(n,n_1). 
    \end{array}
\right.$
\end{proposition}

\begin{proof}
On utilise la description de la fibre de $\nu$ en un point de $U_N$ fournie par la proposition \ref{descriptiongeofib}.\\
$\bullet$ Si $N=n$, alors $$k[{\nu}^{-1}(J_{N})] \cong k[G] \cong \bigoplus_{M \in \Irr(G)} M^{*} \otimes M$$ 
comme $G \times G$-module, d'où $h_W(M)=\dim(M)$.\\
$\bullet$ Si $N=n_1=n_2<n$, alors 
$$k[{\nu}^{-1}(J_{N})] \cong k[G/H] \cong {k[G]}^{H} \cong \bigoplus_{M \in \Irr(G)} M^{*H} \otimes M$$ 
comme $G$-module à gauche et $\dim(M^{*H})=\dim(M^H)$, puisque $H$ est réductif.\\
$\bullet$ Si $N=n_1<\min(n_2,n)$, alors 
\begin{align*} 
k[{\nu}^{-1}(J_{N})] &\cong k[G {\times}^{H} \Hom(E_{N},F_2)]\\
       &\cong{k[G {\times} \Hom(E_{N},F_2)]}^{{H}}\\
       &\cong{\left(k[G] {\otimes} k[\Hom(E_{N},F_2)]\right)}^{H}\\
       &\cong \bigoplus_{M \in \Irr(G)} M^{*} \otimes { \left(M {\otimes} k[\Hom(E_{N},F_2)] \right)}^{{H}}
       \end{align*} 
d'où $h_W(M)=\dim \left({ \left(M {\otimes} k[\Hom(E_{N},F_2)]\right)}^H \right)$.\\
De même, si $N=n_2<\min(n_1,n)$, alors $h_W(M)=\dim \left({\left(M {\otimes} k[\Hom(F_1),E_{N}]\right)}^{{H}} \right)$.  
\end{proof}

\begin{remarque}
Lorsque $N=n_1=n_2<n$ et $M=r_1 \epsilon_1+ \cdots + r_n \epsilon_n$, on peut calculer $\dim(M^{H})$ explicitement en fonction des $r_i$ en utilisant les  
bases dites de "Gelfand-Tsetlin" (voir par exemple \cite{GT}).
\end{remarque}

%% file: GL1.tex
\subsection{Etude du cas \texorpdfstring{$\min(\dim(V),n_1,n_2)=1$}{N=1}} \label{section_n_1}

Dans cette section, on traite le cas $\min(n,n_1,n_2)=1$ qui est très similaire à la situation $1$ traitée dans la section \ref{casSln}. On a $W/\!/G=\Hom(V_1,V_2)^{\leq 1}=U_1 \cup \{0\}$ et $\rho:\ \HH \rightarrow \Gr(h_W(V^*),V_1^*) \times \Gr(h_W(V),V_2)$ est le morphisme construit dans la section \ref{princreduction}. Si $n \geq n_1+n_2-1$, alors $\HH \cong \Hom(V_1,V_2)$ d'après le corollaire \ref{ouverttoutplat}, on peut donc exclure ce cas par la suite. Il reste les cas suivants (où la fonction de Hilbert $h_W$ est calculée grâce à la proposition \ref{fcthilb}):
\begin{itemize}
\item si $n=1$, alors $h_W(V^*)=h_W(V)=1$,
\item si $n_2>n_1=1$, alors $h_W(V^*)=n_2$ et $h_W(V)=1$, 
\item si $n_1>n_2=1$, alors $h_W(V)=n_1$ et $h_W(V^*)=1$.
\end{itemize}   

Si $n=1$, alors on a le morphisme $\rho:\ \HH \rightarrow \PP(V_1^*) \times \PP(V_2)$. Puis, le plongement de Segre donne
\begin{equation}  \label{segre}
\PP(V_1^*) \times \PP(V_2) \cong \PP ({\Hom(V_1,V_2)}^{\leq 1})=\PP(W/\!/G) 
\end{equation}
et on peut donc considérer le morphisme $\rho:\ \HH \rightarrow \PP(W/\!/G)$. 

Si $\min(n_1,n_2)=1<n$, alors comme $n_1$ et $n_2$ jouent des rôles symétriques, on peut supposer que $n_2>n_1=1$. On a le morphisme $G'$-équivariant $\rho:\ \HH \rightarrow \PP(V_2)$, où $GL(V_1)$ opère trivialement dans $\PP(V_2)$. Or, $\PP(W/\!/G) \cong \PP(V_1^* \otimes V_2) \cong \PP(V_2)$ comme $G'$-variété et on peut donc considérer, comme dans le cas $n=1$, le morphisme $\rho:\ \HH \rightarrow \PP(W/\!/G)$.   

Dans tous les cas, on obtient un morphisme $G'$-équivariant $\HH \rightarrow \PP(W/\!/G)$, où la variété projective $\PP(W/\!/G)$ est $G'$-homogène, et donc on peut appliquer la proposition \ref{enoncepgene} pour déterminer $\HHp$. On note 
$${Bl}_0(W/\!/G):= \left \{(f,L) \in W/\!/G \times \PP (W/\!/G) \ \mid  \ f \in L \right \} = \OO_{\PP(W/\!/G)}(-1)$$ 
l'éclatement de $W/\!/G$ en $0$. On a la 

\begin{theoreme} \label{Hcasn11111}
Si $\min(n_1,n_2,n)=1$ et $n<n_1+n_2-1$, alors on a un isomorphisme $G'$-équivariant
$$\HHp \cong {Bl}_0(W/\!/G)$$ 
et le morphisme de Hilbert-Chow s'identifie à l'éclatement ${Bl}_0(W/\!/G) \rightarrow W/\!/G$. \\ 
De plus, si $n=1$, alors $\HH=\HHp$ et la famille universelle est le morphisme $(u_1,u_2,L)\in \XX \mapsto (u_2 \circ u_1, L) \in \HH$, où la variété $\XX$ est définie par
$$\XX \cong \left \{  (u_1,u_2,L) \in W \times \PP(W/\!/G) \ \middle| \ \begin{array}{l}
\Ker(L) \subset \Ker(u_1),\\
\Im(u_2) \subset \Im(L).
\end{array}  \right\}.$$
\end{theoreme}

\begin{proof}
La première partie de la proposition découle de la proposition \ref{enoncepgene}.\\
Ensuite, si $n=1$, alors d'après la proposition \ref{reduction1} et le corollaire \ref{cas_facile}, on a un isomorphisme de $G'$-schémas $\HH \cong G' \times^P \Aff$, où $P$ est un sous-groupe parabolique de $G'$, et donc $\HH=\HHp$.  Enfin, pour montrer la partie de la proposition qui concerne la famille universelle $\XX \rightarrow \HH$, on procède comme pour la proposition \ref{familleUnivSLn}.
\end{proof}

\begin{remarque}
Si $n=1$, on considère le diagramme suivant
$$\xymatrix{ & \PP(V_1^*) \times \PP(V_2) \ar@{->>}[ld]_{p_1} \ar@{->>}[rd]^{p_2} \\  \PP(V_1^*) && \PP(V_2) }$$ 
où $p_1$ et $p_2$ sont les projections naturelles. On note $T_1$ (resp. $T_2$) le tiré en arrière du fibré tautologique de $\PP(V_1^*)$ (resp. de $\PP(V_2)$) et $\underline{V}$ le fibré trivial de fibre $V$ au dessus de $\PP(V_1^*) \times \PP(V_2)$.   
D'après la proposition \ref{reduction2} et le corollaire \ref{cas_facile}, la variété $\XX$ s'identifie à l'espace total du fibré vectoriel
$$(T_1 \otimes\underline{V}) \oplus  (\underline{V^*} \otimes T_2)$$ 
au dessus de $\PP(V_1^*) \times \PP(V_2)$.
\end{remarque}

\begin{remarque}
Soient $n=1$ et $Z \subset W$ un sous-schéma fermé $\Gm$-stable tel que $k[Z] \cong k[\Gm]$ comme $\Gm$-module. Alors $Z$ s'identifie à un point $(f,L)$ de $\HH$ et est défini comme sous-schéma de $W$ par
$$Z=\left \{  (u_1,u_2) \in W \ \middle|\ \begin{array}{l}
u_2 \circ u_1=f,\\
\Ker(L) \subset \Ker(u_1),\\
\Im(u_2) \subset \Im(L).
\end{array}  \right\} .$$  
En particulier:\\
$\bullet$ Si $f \neq 0$, alors $Z$ est isomorphe à $\Gm$.\\
$\bullet$ Si $f=0$, alors $Z$ est la réunion de deux droites qui se coupent en l'origine. 
\end{remarque}

%% file: nilcone_GLn.tex
\subsection{Description de l'algèbre du nilcône}  \label{sectionJ}

Dans les sections \ref{etudeGL2} et \ref{etudeGL3}, nous allons étudier les cas $n=2$ et $n=3$ respectivement. Comme pour le cas cas $n=1$ étudié dans la section \ref{section_n_1}, nous allons utiliser le principe de réduction (voir la section \ref{princreduction}) pour nous ramener du cas $n_1, n_2 \geq n$ au cas $n_1=n_2=n$. Le résultat essentiel de cette section est la proposition \ref{decompoiso} qui donne la description de l'algèbre du nilcône $\NNN(W,G)$ comme $G' \times G$-module lorsque $n_1=n_2=n$. Cette description nous sera très utile par la suite.

\begin{notation} 
On note $J$ l'idéal engendré par les $G$-invariants homogènes de degré positif de $k[W]$. 
\end{notation}
L'idéal $J$ est $G' \times G$-stable par définition. On note
\begin{equation} \label{defslV}
sl(V):=\{ f \in \End(V)\ \mid  \ \tr(f)=0 \}
\end{equation} 
le $G$-module irréductible de plus haut poids $\epsilon_1-\epsilon_n$, où $\tr(f)$ désigne la trace de l'endomorphisme $f$. On a $V^* \otimes V \cong \End(V) \cong sl(V) \oplus V_0$ et
\begin{align}
{k[W]}_2&=S^2((V_1 \otimes V^{*}) \oplus (V \otimes V_{2}^{*})) \label{kW2} \\
       &=S^2(V_1 \otimes V^{*}) \oplus (V_1 \otimes V^{*} \otimes V \otimes V_{2}^{*}) \oplus S^2(V \otimes V_{2}^{*}) \notag \\
       &= (S^2(V_1) \otimes S^2(V^{*})) \oplus  ({\Lambda}^2 (V_1) \otimes {\Lambda}^2 (V^{*})) \oplus (V_1 \otimes V^{*} \otimes V \otimes V_{2}^{*}) \notag \\
       &\  \oplus (S^2(V) \otimes S^2(V_{2}^{*})) \oplus  ({\Lambda}^2 (V) \otimes {\Lambda}^2 (V_{2}^{*})) \notag  \\
       &= (S^2(V_1) \otimes S^2(V^{*})) \oplus (S^2(V) \otimes S^2(V_{2}^{*}))  \oplus  ({\Lambda}^2 (V_1) \otimes {\Lambda}^2 (V^{*}))  \notag  \\
       &\ \oplus  ({\Lambda}^2 (V) \otimes {\Lambda}^2 (V_{2}^{*})) \oplus (V_1 \otimes V_{2}^{*} \otimes sl(V)) \oplus (V_1 \otimes V_{2}^{*} \otimes V_0) \notag
\end{align}
comme $G' \times G$-module. Donc $J \cap {k[W]}_2 = V_1 \otimes V_{2}^{*} \otimes V_0 \cong \Hom(V_2,V_1)$ comme $G' \times G$-module et ce module engendre l'idéal $J$.

On rappelle que l'on a défini les sous-groupes $B', T', U'$ de $G'$ (resp. $B, T, U$ de $G$) dans la section \ref{lesdiffsituations} et que l'on identifie $\Hom(V_1,V)$ à $\MM_{n,n_1}(k)$ (resp. $\Hom(V,V_2)$ à $\MM_{n_2,n}(k)$) via les bases $\BB$, $\BB_1$ et $\BB_2$. 
La proposition qui suit est prouvée dans \cite[§9]{KS}:

\begin{proposition} \label{KSheadings}
Soient $x_i$, $1 \leq i \leq n$, le $i$-ème mineur principal de $\Hom(V_1,V)$ et $y_j$, $1 \leq j \leq n$, le $j$-ième mineur antiprincipal (c'est-à-dire extrait en partant du coin inférieur droit) de $\Hom(V,V_2)$. Alors la $T' \times T$-algèbre ${(k[W]/J)}^{U' \times U}$ est engendrée par les $x_i$ et les $y_j$ et les relations entre ces générateurs sont engendrées par les monômes $\{x_i y_j \mid  i+j>n\}$. Autrement dit, on a une suite exacte 
$$0 \rightarrow J' \rightarrow k[x_1, \ldots,x_n,y_1, \ldots,y_n] \rightarrow {(k[W]/J)}^{U' \times U} \rightarrow 0 $$
où $J'$ est l'idéal monomial engendré par $\{x_i y_j\ \mid \ i+j>n\}$. 
\end{proposition}

On reprend les notations de la section \ref{rappelsthedesreps}. On a la

\begin{proposition} \label{decompoiso}
Soit $\lambda \in \Lambda+$, alors la composante isotypique associée au $G$-module $S^{\lambda}(V)$ dans $k[W]/J$ est: 
\begin{itemize} \renewcommand{\labelitemi}{$\bullet$}
\item $S^{\lambda}(V_2^*) \otimes S^{\lambda}(V)$ si $r_n \geq 0$,
\item $S^{\lambda}(V_1^*) \otimes S^{\lambda}(V)$ si $r_1 \leq 0$,
\item $S^{k_n {\epsilon}_{1}+k_{n-1} {\epsilon}_{2}+ \ldots +k_{t+1} {\epsilon}_{n-t}}(V_1) \otimes S^{k_1 \epsilon_1+ \ldots+ k_t \epsilon_t }(V_2^*) \otimes           S^{\lambda}(V)$ sinon. 
\end{itemize}
De plus, la représentation $S^{\lambda}(V)$ apparaît dans $k[W]_p/(J \cap k[W]_p)$ si et seulement si $p=\sum_i k_i$. 
\end{proposition}

\begin{proof}
D'après la proposition \ref{KSheadings}, on a un isomorphisme de $T \times T'$-algèbres:
$${(k[W]/J)}^{U \times U'} \cong k[x_1, \ldots,x_n,y_1, \ldots,y_n]/J'.$$
D'après \cite[§11.4.2]{Pro}, on a les isomorphismes de $T' \times T$-modules suivants:
$$\left\{
    \begin{array}{l}
       k[x_1, \ldots,x_n] \cong {k[V_{1}^{*} \otimes V]}^{U\times U'}, \\
       k[y_1, \ldots,y_n] \cong {k[V^{*} \otimes V_2]}^{U \times U'}. 
      \end{array}
\right. $$
Donc la formule de Cauchy-Littlewood (\cite[$§8.3$, Corollary $3$]{Fu}) nous permet de conclure lorsque $r_n \geq 0$ ou $r_1 \leq 0$. Il reste à étudier le cas $r_1.r_n <0$. 
On vérifie que le poids du monôme 
$$  x_{n-t}^{k_{t+1}} x_{n-t-1}^{k_{t+2}-k_{t+1}} x_{n-t-2}^{k_{t+3}-k_{t+2}}  \cdots x_{1}^{k_{n}-k_{n-1}} y_{t}^{k_{t}} y_{t-1}^{k_{t-1}-k_{t}} y_{t-2}^{k_{t-2}-k_{t-1}}  \cdots y_{1}^{k_{1}-k_{2}}$$ 
est 
$$(\lambda,\ k_n {\epsilon}_{1}+k_{n-1} {\epsilon}_{2}+ \ldots+k_{t+1} {\epsilon}_{1+n-(t+1)},\ -k_t {\epsilon}_{1+n-t}- \ldots-k_1 {\epsilon}_{n})$$
et que $\lambda$ détermine uniquement ce monôme. Il s'ensuit que la composante isotypique du $G$-module $S^{\lambda}(V)$ dans $k[W]/J$ est  
$$ S^{k_n {\epsilon}_{1}+k_{n-1} {\epsilon}_{2}+ \ldots +k_{t+1} {\epsilon}_{n-t}}(V_1) \otimes S^{k_1 \epsilon_1+ \ldots+ k_t \epsilon_t }(V_2^*) \otimes           S^{\lambda}(V) .$$ 
En particulier, la représentation $S^{\lambda}(V)$ apparaît dans $k[W]_p/(J \cap k[W]_p)$ si et seulement si
\begin{align*}
p&=(k_n-k_{n-1})+2(k_{n-1}-k_{n-2})+\ldots+(n-t)k_{t+1}+(k_1-k_2)+2(k_2-k_3)+\ldots+t k_t\\
 &=k_1+k_2+ \ldots+k_n.
\end{align*} 
\end{proof}

\begin{remarque} 
Si $M$ est une représentation polynomiale ou la duale d'une représentation polynomiale, alors il découle de la proposition \ref{decompoiso} que la multiplicité de $M$ dans $k[W]/J$ est égale à sa dimension.
\end{remarque}

\begin{definition}
Si $I$ est un idéal homogène de $k[W]$, on appelle fonction de Hilbert classique de $I$ la fonction définie par:
$$ \forall p \in \NN,\ f_I(p):=\dim \left( k[W]_p/(k[W]_p \cap I)\right) .$$  
\end{definition}

\begin{corollaire} \label{fcthilbclassn3}
Si $I$ un idéal $G$-stable et homogène de $k[W]$  contenant $J$ et qui a pour fonction de Hilbert $h$, alors la fonction de Hilbert classique de $I$ est donnée par:
$$\forall p \in \NN,\ f_I(p)= \sum_{\substack{(r_1,\ldots,r_n) \in \ZZZ^n \\ r_1 \geq \ldots \geq r_n \\ |r_1|+\ldots+|r_n|=p}}  h(S^{r_1 \epsilon_1+ \ldots+r_n \epsilon_n}(V) ) \dim(S^{r_1 \epsilon_1+ \ldots+r_n \epsilon_n}(V) ) .$$
\end{corollaire}

\begin{proof}
Soit $\lambda=r_1 {\epsilon}_1+\ldots+r_n {\epsilon}_n \in \Lambda+$, alors d'après la proposition \ref{decompoiso} le $G$-module $S^{\lambda}(V)$ apparaît dans $k[W]_p/(k[W]_p \cap I)$ uniquement si $p=|r_1|+\ldots+|r_n|$. 
La multiplicité de chaque $G$-module étant fixée par la fonction de Hilbert $h$ de $I$, on en déduit la dimension de chaque $k[W]_p/(k[W]_p \cap I)$ par la formule annoncée.
\end{proof}

%% file: GL2.tex
\subsection{Etude du cas \texorpdfstring{$\dim(V)=2$}{n=2}}  \label{etudeGL2}

Dans toute cette section, on fixe $n=2$. On a $G \cong GL_2(k)$, $W/\!/G=\Hom(V_1,V_2)^{\leq 2}$ et $\rho:\ \HH \rightarrow \Gr(2,V_1^*) \times \Gr(2,V_2)$ est le morphisme construit dans la section \ref{princreduction}.
On note 
$$Y_0:=\left\{(f,L)\in W/\!/G \times \PP (W/\!/G)\ |\ f \in L \right\}=\OO_{\PP(W/\!/G)}(-1)$$ 
l'éclatement de $W/\!/G$ en l'origine et $Y_1$ l'éclatement de $Y_0$ le long de la transformée stricte de $\Hom(V_1,V_2)^{\leq 1}$. Nous verrons à la fin de cette section que $Y_1$ est isomorphe à l'éclatement de la section nulle du fibré $\Hom(\underline{V_1}/T_1^{\perp},T_2)$ au dessus de $\Gr(2,V_1^*) \times \Gr(2,V_2)$ où $T_1$ (resp. $T_2$) est le fibré tautologique de $\Gr(2,V_1^*)$ (resp. de $\Gr(2,V_2)$) et $\underline{V_1}$ désigne le fibré trivial de fibre $V_1$ au dessus de $\Gr(2,V_1^*) \times \Gr(2,V_2)$.   
Le but de cette section est de démontrer le

\begin{theoreme} \label{casn2} 
Si $n_1,n_2 \geq 2$, alors $\HH \cong Y_1$ est une variété lisse et $\gamma$ est une résolution des singularités de $W/\!/G$.
\end{theoreme} 

\noindent Lorsque $n_1=n_2=2$, alors on a $Y_1=Y_0$ et $\gamma$ est l'éclatement en $0$ de $W/\!/G$. Pour démontrer le théorème \ref{casn2}, on commence par établir le cas particulier $n_1=n_2=2$ (proposition \ref{gammaiso}). Le cas général $n_1,n_2 \geq 2$ est ensuite traité à l'aide du principe de réduction (voir la section \ref{princreduction}).

\begin{remarque}
Si $n_1+n_2 \leq 3$, alors $\HH \cong W/\!/G$ est une variété lisse et $\gamma$ est un isomorphisme d'après le corollaire \ref{cas_facile}.
\end{remarque}

\begin{remarque} \label{casparticuliers}
Les cas particuliers ($n_1=1$, $n_2>2$) et ($n_2=1$, $n_1 >2$) ont été traités dans la section \ref{section_n_1}.
\end{remarque}


\noindent On suppose dans un premier temps que $n_1=n_2=2$.

\subsubsection{Points fixes de \texorpdfstring{$\HH$}{H} pour l'opération de \texorpdfstring{$B'$}{B'}}  \label{sectionptfixe}

On souhaite montrer que $\HH$ est une variété lisse. D'après le lemme \ref{Hlisse4}, il suffit de montrer que les points fixes de $\HH$ pour l'opération de $B'$ sont tous dans $\HHp$ et de vérifier que la dimension de l'espace tangent à $\HH$ en chacun de ces points fixes est égale à la dimension de $\HHp$. On commence donc par déterminer les points fixes de $B'$ dans $\HH$.

\begin{notation}
On note $D:=\left\langle  e_1 \otimes f_2^*\right\rangle$ l'unique droite $B'$-stable de $V_1 \otimes V_{2}^{*}$. \\
On note $I$ l'idéal de $k[W]$ engendré par $(V_1 \otimes V_{2}^{*} \otimes V_0) \oplus (D \otimes sl(V)) \subset {k[W]}_2$. 
\end{notation}

\begin{remarque}
L'idéal $I$ est homogène, $B' \times G$-stable et contient l'idéal $J$ engendré par les $G$-invariants homogènes de degré positif de $k[W]$.
\end{remarque}

\begin{theoreme} \label{pointfixeborel}
L'idéal $I$ est l'unique point fixe de $\HH$ pour l'opération de $B'$.
\end{theoreme}

\begin{proof}
Soit $Z \in \HH(k)$, vu comme sous-schéma fermé de $W$, et $I_Z$ l'idéal de $Z$ dans $k[W] \cong S(W^{*})=\bigoplus_{p \geq 0} S^{p}(W^{*})$ où l'on identifie comme précédemment $k[W]$, l'algèbre des fonctions régulières sur $W$, avec $S(W^{*})$, l'algèbre symétrique de $W^{*}$. Alors
$$Z \in {\HH}^{B'}  \Leftrightarrow \forall b \in B',\ \forall f \in I_Z,\ b.f \in I_Z . $$
En particulier, $I_Z$ est nécessairement un idéal homogène puisqu'il est stable pour l'opération du groupe des matrices scalaires inversibles qui est un sous-groupe de $B'$.\\
Donc, les points fixes de $\HH$ pour l'opération de $B'$ correspondent exactement aux idéaux homogènes $I_Z$ de $k[W]$ tels que:\\
i) $I_Z$ est $B' \times G$-stable, \\
ii) $k[W]/ I_Z =\bigoplus_{M \in \Irr(G)} M^{\oplus \dim(M)}$ comme $G$-module.\\
Comme $k[W]$ est une algèbre graduée et que $I_Z$ est un idéal homogène, $k[W]/I_Z$ est encore une algèbre graduée: 
$$k[W]/I_Z= \bigoplus_{p \geq 0} {k[W]}_p/(I_Z \cap {k[W]}_p) .$$ 
Nous allons donc pouvoir étudier $I_Z$ degré par degré.\\
$\bullet$ Composante de degré $0$:\\
On a $I_Z \cap {k[W]}_0=0$, sinon $I_Z$ contient les constantes de $k[W]$ et donc $I_Z=k[W]$.\\
$\bullet$ Composante de degré $1$:\\
On a $I_Z \cap {k[W]}_1 \neq {k[W]}_1$, sinon $k[W]/I_Z=V_0$ comme $G$-module. \\
$\bullet$ Composante de degré $2$: on utilise la décomposition (\ref{kW2}).\\
Pour avoir la décomposition souhaitée de $k[W]/I_Z$ comme $G$-module, on a nécessairement ${k[W]}_2 \cap I_Z \supseteq sl(V) \oplus 4 V_0$. En effet, le $G$-module $k[W]/I_Z$ contient déjà une copie de la représentation triviale (les constantes), il ne peut donc pas en contenir d'autre. Ensuite, ${k[W]}_2$ contient $4$ copies de $sl(V)$ qui est un $G$-module de dimension $3$, donc ${k[W]}_2 \cap I_Z$ contient au moins une copie de $sl(V)$. Comme $k[W]_2 \cap I_Z$ est $B'$-stable, il contient $D \otimes sl(V)$ car $D$ est l'unique droite $B'$-stable de $V_1 \otimes V_{2}^{*}$. Il s'ensuit que $I_Z$ contient $(V_1 \otimes V_{2}^{*} \otimes V_0) \oplus (D \otimes sl(V))$ et donc $I_Z \supset I$.
Le lemme qui suit implique que cette inclusion est en fait une égalité et achève ainsi la démonstration du théorème \ref{pointfixeborel}:

\begin{lemme} \label{l3generators}
L'idéal $I$ a pour fonction de Hilbert $h_W$.
\end{lemme}

\textbf{Preuve du lemme:}
On rappelle que par définition de $h_W$, pour chaque $M \in \Irr(G)$, $h_W(M)=\dim(M)$ et donc il faut montrer que 
$$k[W]/I = \bigoplus_{M \in \Irr(G)} M^{\oplus \dim(M)}$$  
comme $G$-module. On a l'inclusion d'idéaux $J \subset I$, d'où la suite exacte de $B' \times G$-modules:
$$ 0 \rightarrow I/J \rightarrow k[W]/J \rightarrow k[W]/I  \rightarrow 0 $$ et donc 
$$k[W]/I \cong \frac{k[W]/J}{I/J}$$ 
comme $B' \times G$-module. 
La proposition \ref{decompoiso} nous fournit la décomposition isotypique du $G$-module $k[W]/J$:\\
-Si $M$ est un $G$-module polynomial (ou le dual d'un $G$-module polynomial), alors la multiplicité de $M$ dans $k[W]/J$ est $\dim(M)$. \\
-Sinon, $M = S^{k_1{\epsilon}_1-k_2 {\epsilon}_2}(V)$ pour un unique couple $k_1,k_2 >0$ et la multiplicité de $S^{k_1{\epsilon}_1-k_2 {\epsilon}_2}(V)$ dans $k[W]/J$ est égale à $\dim(S^{k_2}(V_1) \otimes S^{k_1}(V_{2}^{*}))={(k_1+1)(k_2+1)}$.\\
On note $w \in W$ sous la forme $w=\left( \begin{bmatrix}
x_{11}  &x_{12} \\
x_{21}  & x_{22} 
\end{bmatrix},
\begin{bmatrix}
 y_{22} & y_{12} \\
 y_{21} & y_{11}
\end{bmatrix} \right)$ et on identifie $k[W]$ à $k[x_{ij},y_{ij},\ 1 \leq i,j \leq 2]$. D'après la proposition \ref{KSheadings}, on a 
$${(k[W]/J)}^U \cong k[x_{11},x_{12},x_{11} x_{22}-x_{21} x_{12},y_{11},y_{12},y_{11} y_{22}-y_{21} y_{12}]/K$$
où
\begin{align*} 
K=(&x_{11} (y_{11} y_{22}-y_{21} y_{12}), x_{12} (y_{11} y_{22}-y_{21} y_{12}),\\
   &y_{11} (x_{11} x_{22}-x_{21} x_{12}), y_{12} (x_{11} x_{22}-x_{21} x_{12}),\\
   &(x_{11} x_{22}-x_{21} x_{12})(y_{11} y_{22}-y_{21} y_{12})).
\end{align*}
Et ${(I/J)}^U$ est l'idéal de ${(k[W]/J)}^U$ engendré par $x_{11}y_{11}$. Donc     
$$I/J = \bigoplus_{k_1,k_2 >0} {\left(D.\left(S^{k_2-1}(V_1) \otimes S^{k_1-1}(V_{2}^{*})\right)\right)} \otimes S^{k_1{\epsilon}_1-k_2 {\epsilon}_2}(V)$$
comme $B' \times G$-module. Il s'ensuit que
\begin{align*} 
k[W]/I = &V_0 \oplus {k[V_{1}^{*} \otimes V]}_{+} \oplus {k[V^{*} \otimes V_2]}_{+} \\
           &\oplus  \left(\bigoplus_{k_1,k_2 >0 } \frac{S^{k_2}(V_1) \otimes S^{k_1}(V_{2}^{*})}{(D.(S^{k_2-1}(V_1) \otimes S^{k_1-1}(V_{2}^{*})))} \otimes S^{k_1{\epsilon}_1-k_2 {\epsilon}_2}(V)\right)
\end{align*} 
comme $B' \times G$-module, où ${k[V_{1}^{*} \otimes V]}_{+}$ (resp. ${k[V^{*} \otimes V_2]}_{+}$) désigne l'idéal maximal homogène de $k[V_{1}^{*} \otimes V]$ (resp. de $k[V^{*} \otimes V_2]$). \\
On remarque que les multiplicités des représentations polynomiales (et de leurs duales) dans $k[W]/I$ sont les mêmes que dans $k[W]/J$. En revanche, la multiplicité de 
$ S^{k_1{\epsilon}_1-k_2 {\epsilon}_2}(V)$ dans $k[W]/I$ est: 
\begin{align*} 
\dim \left(\frac{S^{k_2}(V_1) \otimes S^{k_1}(V_{2}^{*})}{(D.(S^{k_2-1}(V_1) \otimes S^{k_1-1}(V_{2}^{*})))}\right)&=(k_1+1)(k_2+1)-k_1 k_2\\
&=k_1+k_2+1\\
&=\dim(S^{k_1{\epsilon}_1-k_2 {\epsilon}_2}(V)) .
\end{align*} 
\end{proof}

\begin{remarque}
On a $\Stab_{G'}(I)=B'$, donc l'unique orbite fermée de $\HH$ est isomorphe à $G'/B' \cong \PP^1 \times \PP^1$.
\end{remarque}

\noindent Le corollaire qui suit découle du lemme \ref{fixespoints} et du théorème \ref{pointfixeborel}.

\begin{corollaire} \label{Hconnexe}
Le schéma $\HH$ est connexe.
\end{corollaire}

\subsubsection{Espace tangent de \texorpdfstring{$\HH$}{H} en \texorpdfstring{$I$}{I}} \label{dimTang}

\noindent On note $Z_0:=\Spec(k[W]/I)$. Nous allons démontrer la  

\begin{proposition} \label{dimTangent}
$\dim(T_{Z_0} \HH)=4$.
\end{proposition}

On identifie $k[W]$ à $k[x_{ij},y_{ij},\ 1 \leq i,j \leq 2]$ comme dans la preuve du lemme \ref{l3generators} et on explicite des bases de certains $B' \times G$-modules qui apparaissent dans $k[W]_2$: 

 $\left.
    \begin{array}{l}
      f_1:=y_{22} x_{11}+y_{12} x_{21}\\
      f_2:=y_{22} x_{12}+y_{12} x_{22}\\
      f_3:=y_{21} x_{11}+y_{11} x_{21}\\
      f_4:=y_{21} x_{12}+y_{11} x_{22}
      \end{array}
\right \}  \text{ est une base de } V_1 \otimes V_2^* \otimes V_0, $\\

$\left.
    \begin{array}{l}
         h_1:=x_{11} y_{11}\\
         h_2:=x_{11} y_{21}\\  
         h_3:=x_{21} y_{11}\\ 
         h_4:=x_{21} y_{21}
      \end{array}
\right \} \text{ est une base de }  D \otimes V^* \otimes V \cong D \otimes (sl(V) \oplus V_0) .$\\
On reprend les notations de la section \ref{ConnexitéetTangence}. Soient $R:=k[W]/I$ et 
$$N:=\left\langle  f_1,f_2,f_3,f_4,h_1,h_2,h_3,h_4 \right\rangle \subset k[W]$$ 
qui est un $B' \times G$-module qui engendre l'idéal $I$. D'après \cite[Macaulay2]{Mac2}, les relations entre les générateurs ci-dessus du $R$-module $I/I^2$ sont données dans la table \ref{table1}.

\begin{table}[ht]
\begin{center}
\center \includegraphics[scale=0.8]{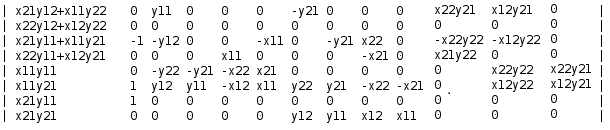}
\end{center}
\caption{\label{table1} Relations entre les générateurs de $I/I^2$}
\end{table}

\noindent La colonne de gauche nous donne les générateurs considérés et chaque autre colonne fournit une relation entre les générateurs des lignes correspondantes.
En particulier, on a les relations suivantes données par les colonnes $6$, $5$ et $4$ respectivement:
$$ \left \{
    \begin{array}{l}
   r_1:=-f_1 \otimes y_{21}+h_2 \otimes y_{22}+ h_4 \otimes y_{12},\\
   r_2:=-f_3 \otimes x_{11}+h_1\otimes x_{21}+ h_2 \otimes x_{11}, \\
   r_3:=f_4 \otimes x_{11}-h_1 \otimes x_{22}- h_2\otimes x_{12}. 
      \end{array}
      \right.$$
On a $\dim(N)=7$ et donc d'après le lemme \ref{InegTang} on a $\dim(T_{Z_0} \HH)=7- \rg(\rho^*)$. D'après le lemme \ref{fixespoints}, la variété $\HHp$ contient au moins un point fixe pour l'opération de $B'$, donc $Z_0 \in \HHp$ et donc $\dim(T_{Z_0} \HH) \geq \dim(\HHp)=4$. Donc, pour montrer la proposition \ref{dimTangent}, il suffit de montrer le

\begin{lemme} 
$\rg(\rho^*) \geq 3.$
\end{lemme}

\begin{proof}
Pour $i=1, \ldots ,4$, on définit ${\psi}_i \in {\Hom}_R^G(R \otimes N,R)$ par 
$$\left\{
    \begin{array}{ll}
        {\psi}_i( h_j \otimes 1)=0 &\text{ pour $j=1,\ldots,4$}, \\
        {\psi}_i(f_j \otimes 1)={\delta}_i^j &\text{ pour $j=1,\ldots,4$},
    \end{array}
\right.
$$
où ${\delta}_i^j$ est le symbole de Kronecker qui vaut $1$ lorsque $i=j$, et qui vaut $0$ sinon. Les ${\psi}_i$ forment une famille libre de $\Hom_R^G(R {\otimes} N,R)$, nous allons voir que $\rho^*({\psi}_1)$, $\rho^*({\psi}_3)$ et $\rho^*({\psi}_4)$ forment une famille libre de $\Hom_R^G(R \otimes \RRR,R)$ ce qui démontrera le lemme. Soient ${\lambda}_1$, ${\lambda}_3$, ${\lambda}_4 \in k$ tels que 
\begin{equation} \label{relLinaire1}
\sum_{i=1,3,4} {\lambda}_i \, \rho^*({\psi}_i)=0.
\end{equation}
On évalue (\ref{relLinaire1}) en $r_1 \otimes 1$, on obtient:
\begin{align*}
0&={\sum}_{i=1,3,4} {\lambda}_i \, \rho^*({\psi}_i)(r_1 \otimes 1)\\
 &= {\sum}_{i=1,3,4} {\lambda}_i \, {\psi}_i(r_1)\\
 &= {\sum}_{i=1,3,4} {\lambda}_i  \, {\psi}_i(-f_1 \otimes y_{21}+h_2 \otimes y_{22}+ h_4 \otimes y_{12})\\
 &= -{\lambda}_1 \, y_{21}.  
 \end{align*}
On évalue (\ref{relLinaire1}) en $r_2\otimes 1$, on obtient:
$0= -{\lambda}_3 \, {x_{11}}.$\\
On évalue (\ref{relLinaire1}) en $r_3\otimes 1$, on obtient:
$0= {\lambda}_4 \, x_{11}.$\\
On en déduit que $({\lambda}_1, {\lambda}_3, {\lambda}_4)=(0,0,0)$ et donc $\rho^*({\psi}_1)$, $\rho^*({\psi}_3)$ et $\rho^*({\psi}_4)$ forment bien une famille libre dans $\Im(\rho^*)$.  
\end{proof}

\begin{remarque}  \label{Bmoduleexplicite}
D'après \cite[Macaulay2]{Mac2}, une famille de générateurs du $R$-module $\Hom_R(I/I^2,R)$ est donnée dans la table \ref{table2}. 

\begin{table}[ht]
\begin{center}
\center \includegraphics[scale=0.8]{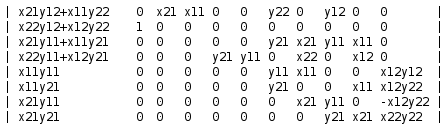}
\end{center}
\caption{\label{table2} Générateurs du $R$-module $\Hom_R(I/I^2,R)$}
\end{table}

\noindent La colonne de gauche donne des générateurs du $R$-module $I/I^2$ et chaque autre colonne représente un $R$-morphisme qui à l'élément de $I/I^2$ de la colonne de gauche associe l'élément de $R$ situé sur la même ligne. On note $\Phi_i$ le morphisme associé à la colonne $i+1$. On vérifie que les quatre morphismes $\Phi_1$, $\Phi_{11}$, $y_{12} \Phi_8+y_{22} \Phi_{10}$, $x_{12} \Phi_7+x_{22} \Phi_9$ sont $G$-équivariants et linéairement indépendants. Ils forment donc une base de l'espace vectoriel ${\Hom}_R^G(I/I^2,R)$. Ensuite, d'après la suite exacte (\ref{complexeCornFlakes2}), le module ${\Hom}_R^G(I/I^2,R)$ s'identifie à un $B'$-sous-module de $\Hom^G(N,R) \cong (V_1^* \otimes V_2) \oplus  \left( D^* \otimes \frac{V_1 \otimes V_2^*}{D}\right)$. Or $\Phi_{11}$, $y_{12} \Phi_8+y_{22} \Phi_{10}$ et $x_{12} \Phi_7+x_{22} \Phi_9$ définissent des éléments de $D^* \otimes \frac{V_1 \otimes V_2^*}{D}$ tandis que $\Phi_1$ définit un élément de $V_1^* \otimes V_2$, donc nécessairement $\left\langle  \Phi_1 \right\rangle$ est l'unique droite $B'$-stable de $V_1^* \otimes V_2$ et donc
$${\Hom}_R^G(I/I^2,R) \cong D^* \oplus \left( D^* \otimes \frac{V_1 \otimes V_2^*}{D}\right)$$
comme $B'$-module.  
\end{remarque}

\noindent On déduit du lemme  \ref{Hlisse4} et de ce qui précède le 

\begin{corollaire} \label{Hcasn2}
$\HH=\HHp$ est une variété lisse de dimension $4$. 
\end{corollaire}

\subsubsection{Construction d'un morphisme \texorpdfstring{$G'$}{G'}-équivariant \texorpdfstring{$\delta:\ \HH \rightarrow\PP(W/\!/G)$}{}}  \label{sectionconstructionmor}  

Le lemme qui suit découle de la théorie classique des invariants de manière analogue au lemme \ref{exi1}.

\begin{lemme} \label{lemmeThClass2}
Le ${k[W]}^G$-module ${k[W]}_{(sl(V))}$ est engendré par ${\Hom}^{G}(sl(V),k[W]_2)$.
\end{lemme}

\begin{lemme} \label{iso3}
On a un isomorphisme de $G'$-modules
$${\Hom}^{G}(sl(V),k[W]_2) \cong V_1 \otimes V_2^* .$$
\end{lemme}

\begin{proof}
On a 
\begin{align*}
{\Hom}^{G}(sl(V),k[W]_2)      &\cong{\Hom}^{G}(sl(V),V_1 \otimes V^* \otimes V \otimes V_2^*) \text{ d'après (\ref{kW2})},\\
                              &\cong V_1 \otimes V_2^* \text{ puisque $V^* \otimes V \cong sl(V) \oplus V_0$.} 
\end{align*} 
\end{proof}


\noindent La proposition \ref{morphismegrass} et le lemme  \ref{iso3} nous donnent un morphisme $G'$-équivariant:
$$\HH \rightarrow \PP(V_1 \otimes V_2^*).$$
Puis $V_1 \otimes V_{2}^{*} =\det(V_1) \otimes V_{1}^{*} \otimes V_{2} \otimes {\det}^{-1}(V_2)$ comme $G'$-module donc
\begin{align} 
\PP (V_1 \otimes V_{2}^{*}) &= \PP (\det(V_1) \otimes V_{1}^{*} \otimes V_{2} \otimes {\det}^{-1}(V_2))\label{maoprhismedelta2} \\
                            &= \PP (V_{1}^{*} \otimes V_{2}) \notag \\
                            &= \PP (W/\!/G) \notag
\end{align}
comme $G'$-variétés. On a donc un morphisme $G'$-équivariant:
\begin{equation} \label{maoprhismedelta}
\delta :\ \HH \rightarrow  \PP (W/\!/G) . 
\end{equation}

\subsubsection{Le morphisme \texorpdfstring{$\gamma \times \delta$}{gamma X delta} est un isomorphisme entre \texorpdfstring{$\HH$}{H} et \texorpdfstring{$Y_0$}{Y0}}  \label{sectionison2}
On rappelle que l'on note 
$$Y_0:=\left\{(f,L)\in W/\!/G \times \PP (W/\!/G)\ |\ f \in L \right\}=\OO_{\PP (W/\!/G)}(-1)$$ 
l'éclatement en $0$ de $W/\!/G=\Hom(V_1,V_2)$. L'opération de $G'$ dans $W/\!/G$ induit une opération de $G'$ dans $W/\!/G \times \PP(W/\!/G)$ qui préserve $Y_0$. On montre alors, en procédant comme dans la preuve du lemme \ref{morphisme_dans_YSLn}, que le morphisme $\gamma \times \delta$ envoie $\HH$ dans $Y_0$. Ensuite, on a le

\begin{lemme} \label{bijens}
Le morphisme $\gamma \times \delta:\ {\HH} \rightarrow Y_0$ est quasi-fini.
\end{lemme}

\begin{proof}
On note comme précédement $Z_0$ l'unique point fixe de $\HH$ pour l'opération de $B'$ et soit $y_0:=(\gamma \times \delta)(Z_0)$. Nous allons montrer dans un premier temps que l'ensemble ${(\gamma \times \delta)}^{-1}(y_0)$ est formé de $Z_0$. Soit $Z \in \HH$ tel que $(\gamma \times \delta)(Z)=y_0$. D'une part $\gamma(Z)=0$, donc $I_Z \supset V_1 \otimes V_2^* \otimes V_0$. D'autre part $\delta(Z)=\delta(Z_0)=D \in \PP(V_1 \otimes V_2^*)$, donc $I_Z \supset D \otimes sl(V)$. Le lemme \ref{l3generators} permet alors de conclure que $Z=Z_0$. Puis
$$E:=\{Z \in \HH \text{ tel que la fibre en } y:=({\gamma \times \delta})(Z) \text{ est de dimension } \geq 1  \}$$
est un fermé ($G'$-stable) de $\HH$ d'après \cite[Exercice II.3.22]{Ha}. D'après le lemme \ref{fixespoints} et ce qui précède, l'ensemble $E$ est vide et le résultat s'ensuit. 
\end{proof}

\begin{proposition} \label{gammaiso}
Le morphisme $\gamma \times \delta:\ \HH  \rightarrow  Y_0$ est un isomorphisme.
\end{proposition}

\begin{proof}
La variété $ Y_0$ est lisse, donc normale et le morphisme $\gamma \times \delta:\ {\HH} \rightarrow  Y_0$ est birationnel, propre et quasi-fini d'après le lemme \ref{bijens}, donc d'après le théorème principal de Zariski, c'est un isomorphisme. 
\end{proof}

\subsubsection{Cas \texorpdfstring{$n_1,n_2 \geq 2$}{n1,n2>2}}  \label{masterpropositioncasn2}
Nous allons utiliser le principe de réduction (voir la section \ref{princreduction}) afin de traiter le cas général du théorème \ref{casn2} en nous ramenant au cas $n_1=n_2=2$.

\begin{notation} \label{notationsKerd}
Pour tout $L \in \PP(W/\!/G)=\PP(\Hom(V_1,V_2)^{\leq 2})$, on note:
\begin{itemize} \renewcommand{\labelitemi}{$\bullet$}
\item $\Ker(L):=\Ker(l)$,
\item $\Im(L):=\Im(l)$, 
\item $\rg(L):=\rg(l)$,
\end{itemize}
où $l$ est n'importe quel élément non-nul de $L$.\\
Si $\rg(L)=2$, on note $L \wedge L:=\left\langle l \wedge l \right\rangle \in \PP(\Hom(\Lambda^2(V_1), \Lambda^2(V_2)))$.\\
Enfin, on note $\iota_1:\ \Gr(2,V_1^*) \rightarrow \PP(\Lambda^2(V_1^*))$ et $\iota_2:\ \Gr(2,V_2) \rightarrow \PP(\Lambda^2(V_2))$ les plongements de Plücker de $\Gr(2,V_1^*)$ et $\Gr(2,V_2)$ respectivement. 
\end{notation}

\noindent On considère la variété
$$\ZZ:=\left\{(f,L,E_1,E_2) \in W/\!/G \times \PP(W/\!/G) \times \Gr(2,V_1^*) \times \Gr(2,V_2) \ \middle| \ \begin{array}{ll}
       f \in L,\\
       E_1^{\perp} \subset \Ker(L),\\
      \Im(L) \subset E_2.
    \end{array} 
\right\}$$ 
alors on a le diagramme
$$\xymatrix{ &  \ZZ \ar@{->>}[ld]_{q_1} \ar@{->>}[rd]^{q_2} \\   Y_0 && \Gr(2,V_1^*) \times \Gr(2,V_2) }$$
où $q_1$ et $q_2$ sont les projections naturelles. La variété $Y_1$ a été définie au début de la section \ref{etudeGL2}. 

\begin{lemme} \label{identificationblowup}
On a un isomorphisme de variétés 
$Y_1 \cong \ZZ$
et via cet isomorphisme, l'éclatement $Y_1 \rightarrow Y_0$ s'identifie au morphisme $q_1:\ \ZZ \rightarrow Y_0$.
\end{lemme}

\begin{proof}
On commence par remarquer que, si $L \in \PP(W/\!/G)$ est tel que $\rg(L)=2$ et si $l$ un élément non-nul de $L$, alors $l \wedge l \in \Hom(\Lambda^2(V_1), \Lambda^2(V_2))^{\leq 1}$ et donc $L \wedge L \in \PP(\Hom(\Lambda^2(V_1), \Lambda^2(V_2))^{\leq 1})$. On note $F_0$ la transformée stricte de $\Hom(V_1,V_2)^{\leq 1}$ via l'éclatement $Y_0 \rightarrow \Hom(V_1,V_2)^{\leq 2}$ et on considère
$$\alpha: \newdir{ >}{{}*!/-5pt/\dir{>}}
  \xymatrix{
     Y_0 \ar@{-->}[r]  & \PP(\Hom(\Lambda^2(V_1),\Lambda^2(V_2))^{\leq 1}) 
    }$$ 
l'application rationnelle définie sur 
$$U:=\{(f,L) \in Y_0\ \mid  \rg(L)=2\}=Y_0 \backslash F_0$$ 
par $\alpha(f,L)=L \wedge L$. On note  
$$\Gamma \subset  Y_0 \times \PP(\Hom(\Lambda^2(V_1),\Lambda^2(V_2))^{\leq 1})$$ 
le graphe de l'application $\alpha$. Alors 
$$\Gamma=\{(f,L,E)\in U \times \PP(\Hom(\Lambda^2(V_1),\Lambda^2(V_2))^{\leq 1}) \mid  L \wedge L =E \} .$$
L'adhérence de $\Gamma$ dans $Y_0 \times \PP(\Hom(\Lambda^2(V_1),\Lambda^2(V_2))^{\leq 1})$ est donnée par
$$\overline{\Gamma}=\left \{(f,L,E)\in Y_0 \times \PP(\Hom(\Lambda^2(V_1),\Lambda^2(V_2))^{\leq 1})\ \middle|  \begin{array}{l}
\forall l  \in L,\ l \wedge l  \in E, \\
\Im(E) \in \iota_2(\Gr(2,V_2)),\\
\Ker(E)^{\perp} \in \iota_1(\Gr(2,V_1^*)).
\end{array} \right \} $$
où $\Im(E):=\Im(g)$ et $\Ker(E):=\Ker(g)$ pour n'importe quel élément non-nul $g \in E$.\\

\noindent \underline{Fait:} le morphisme $\overline{\Gamma} \rightarrow Y_0,\ (f,L,E) \mapsto (f,L)$ est l'éclatement de $Y_0$ le long du fermé $F_0$. \\
On commence par construire un recouvrement de $Y_0$ par des ouverts affines. On a $\PP(W/\!/G) \subset \PP(\Hom(V_1,V_2)) \cong \PP(\MM_{n_2,n_1}(k))$, où l'on rappelle que l'on identifie $\Hom(V_1,V_2)$ avec $\MM_{n_2,n_1}(k)$ via les bases fixées dans la section \ref{lesdiffsituations}. Quels que soient $1 \leq i \leq n_2$ et $1 \leq j \leq n_1$, on note $O_{i,j}$ l'ouvert affine de $\PP(\MM_{n_2,n_1}(k))$ obtenu en fixant la coordonnée $(i,j)$ égale à $1$. Puis, on note $Y_0^{i,j}:=\{(f,l) \in W/\!/G \times (\PP(W/\!/G) \cap O_{i,j})\ |\ f \in \left\langle l \right\rangle\}$ qui est un ouvert affine (éventuellement vide) de $Y_0$. Alors $Y_0= \bigcup_{i,j} Y_0^{i,j}$ est un recouvrement de $Y_0$ par des ouverts affines. Ensuite, on note $\II$ le faisceau d'idéaux de $F_0$ dans $Y_0$. Alors $\II_{| Y_0^{i,j}}$ est l'idéal de $k[Y_0^{i,j}]$ engendré par les mineurs de taille $2$ de la matrice $l$. Et la restriction de $\alpha$ à $Y_0^{i,j}$ est l'application rationnelle $\alpha_{i,j}: \newdir{ >}{{}*!/-5pt/\dir{>}}
  \xymatrix{
     Y_0^{i,j} \ar@{-->}[r]  & \PP(\Hom(\Lambda^2(V_1),\Lambda^2(V_2))^{\leq 1})} $ 
définie sur $U_{i,j}:=\{(f,l) \in Y_0^{i,j} \ |\ \rg(l)=2\}$ par $\alpha_{i,j}(f,l)=\left\langle  l \wedge l \right\rangle$. 
On note $\Gamma_{i,j} \subset Y_0^{i,j} \times \PP(\Hom(\Lambda^2(V_1),\Lambda^2(V_2))^{\leq 1})$ le graphe de l'application $\alpha_{i,j}$, alors d'après \cite[Proposition IV.22]{EH}, le morphisme $\overline{\Gamma_{i,j}} \rightarrow Y_0^{i,j},\ (f,l,E) \mapsto (f,l)$ est l'éclatement de $Y_0^{i,j}$ le long du fermé $F_0 \cap Y_0^{i,j}$. Enfin, ces éclatements se recollent pour donner le morphisme $\overline{\Gamma} \rightarrow Y_0,\ (f,L,E) \mapsto (f,L)$ ce qui montre le fait annoncé.

Il reste à voir que $\overline{\Gamma}$ est isomorphe à $\ZZ$. On a un isomorphisme
$$\begin{array}{ccc}
\PP(\Hom(\Lambda^2(V_1),\Lambda^2(V_2))^{\leq 1}) & \cong & \PP(\Lambda^2(V_1)^*) \times \PP(\Lambda^2(V_2)) \\
 E & \mapsto & (\Ker(E)^{\perp}, \Im(E)) 
\end{array}$$
donné par le plongement de Segre. On en déduit un isomorphisme
$$\overline{\Gamma} \cong \left \{ (f,L,L_1,L_2)\in Y_0 \times \PP(\Lambda^2(V_1)^*) \times \PP(\Lambda^2(V_2))\ \middle| \begin{array}{l}
\forall l  \in L,\ \text{ on a } \left\{ \begin{array}{l}
       \Ker(l \wedge l)^{\perp} \subset L_1, \\
     \Im(l \wedge l)  \subset L_2, 
    \end{array} \right.\\
    L_1 \in \iota_1(\Gr(2,V_1^*)), \\
L_2 \in \iota_2(\Gr(2,V_2)).
\end{array} \right \} .$$
Soit $f \in \Hom(V_1,V_2)^{\leq 2}$, alors $f \wedge f \in \Hom(\Lambda^2(V_1), \Lambda^2(V_2))^{\leq 1}$ et on a deux cas de figure: si $\rg(f) \leq 1$, alors $f \wedge f=0$, sinon $\Im(f \wedge f)=L_2$ pour un certain $L_2 \in \iota_2(\Gr(2,V_2))$ et $\Ker(f \wedge f)^{\perp}=L_1$ pour un certain $L_1 \in \iota_1(\Gr(2,V_1^*))$. Pour $i=1,2$, on note $E_i$ l'antécédent de $L_i$ par $\iota_i$. Alors on a les équivalences
\begin{align*}
\Im(f \wedge f)=L_2 &\Leftrightarrow \Im(f)=E_2, \\  
\Ker(f \wedge f)^{\perp}=L_1 &\Leftrightarrow \Ker(f)^{\perp}=E_1. 
\end{align*}
Il s'ensuit que
$$\overline{\Gamma} \cong \left \{ (f,L,E_1,E_2)\in Y_0 \times \Gr(2,V_1^*) \times \Gr(2,V_2) \mid \ E_1^{\perp} \subset \Ker(L) \text{ et } \Im(L) \subset E_2 \right \}=\ZZ.$$  
\end{proof}

On identifie la variété $Y_1$ à $\ZZ$ grâce au lemme \ref{identificationblowup}. L'opération naturelle de $G'$ dans $ W/\!/G \times \PP(W/\!/G) \times \Gr(2,V_1^*) \times \Gr(2,V_2)$ stabilise $Y_1$. On fixe $(E_1,E_2) \in \Gr(2,V_1^*) \times \Gr(2,V_2)$, alors on a un isomorphisme $\Gr(2,V_1^*) \times \Gr(2,V_2) \cong {G'}/P$ où $P$ est le stabilisateur de $(E_1,E_2)$ dans $G'$. La projection $q_2:\ Y_1 {\rightarrow} G'/P$ est un morphisme $G'$-équivariant et donc, d'après (\ref{iissoo}), on a un isomorphisme $G'$-équivariant 
$$Y_1 \cong G' {\times}^{P} F'' $$
où $F''$ est la fibre schématique du morphisme $q_2$ en $(E_1,E_2)$. On a
\begin{align*} 
F''&= \{(f,L) \in \Hom(V_1/E_1^{\perp},E_2) \times \PP( \Hom(V_1/E_1^{\perp},E_2))\ |\ f \in L\}\\
 &= Bl_{0}(\Hom(V_1/E_1^{\perp},E_2))
\end{align*}
qui est l'éclatement en $0$ de $\Hom(V_1/E_1^{\perp},E_2)$. Et $P$ opère dans $F''$ de la façon suivante:
$$\forall p=(p_1,p_2) \in P, \forall (f,L) \in F'',\ p.(f,L)=(p_2 \circ f \circ p_{1}^{-1}, \left\langle  p_2 \circ l \circ p_{1}^{-1} \right\rangle)$$
où $l$ est n'importe quel élément non nul de $L$.\\
Avec les notations du début de la section \ref{etudeGL2}, on remarque que $Y_1$ est l'espace total du fibré $Bl_0(\Hom(\underline{V_1}/T_1^{\perp},T_2))$ au dessus de $G'/P$; celui-ci s'obtient en éclatant la section nulle dans le fibré $\Hom(\underline{V_1}/T_1^{\perp},T_2)$.\\ 
Par ailleurs, d'après les propositions \ref{reduction1} et \ref{gammaiso}, on a un isomorphisme $G'$-équivariant
$$\HH \cong G' {\times}^{P} {Bl}_{0}(\Hom(V_1/E_1^{\perp},E_2))$$
où l'opération de $P$ dans ${Bl}_{0}(\Hom(V_1/E_1^{\perp},E_2))$ coïncide avec celle de $P$ dans $F''$.
On a donc un isomorphisme $G'$-équivariant:
\begin{equation} \label{ultimeHiso}
 \HH \cong Y_1.
\end{equation}
On reprend les notations de la section \ref{zectionred}: 
\begin{itemize} \renewcommand{\labelitemi}{$\bullet$}
\item $W':=\Hom(V_1/E_1^{\perp},V) \times \Hom(V,E_2) \subset W$,
\item ${\HH}':= {\Hilb}_{h_{W}}^{G}(W') \subset \HH$,
\item ${\gamma}':\ {\HH}' \rightarrow W'/\!/G$ le morphisme de Hilbert-Chow, 
\item ${\delta}':\ {\HH}' \rightarrow \PP(W'/\!/G)$ le morphisme (\ref{maoprhismedelta}).
\end{itemize}
D'après la proposition \ref{fcthilb}, on a $h_{W'}=h_W$, et donc $\HH'$ est simplement le schéma de Hilbert invariant étudié précédement (avec $n_1=n_2=2$). Puis, d'après la proposition \ref{gammaiso}, $ \gamma' \times {\delta}':\ {\HH}' \rightarrow Bl_0(W'/\!/G)$ est un isomorphisme $P$-équivariant et donc le morphisme induit $G' {\times}^P {\HH}' \rightarrow G' {\times}^P Bl_0(W'/\!/G)$ est un isomorphisme $G'$-équivariant. Les observations précédentes et le corollaire \ref{reductionHdiag} nous donnent le diagramme commutatif suivant:
\begin{equation} \label{diagsynthese}
\xymatrix {
    {G' {\times}^P {\HH}'} \ar[rr]^{G' \times^P {\gamma}'} \ar[rd]_{\cong} \ar[dd]_{\psi_1}^{\cong} && {G' {\times}^P W'/\!/G} \ar[dd]^{\phi} \\
    & {G' {\times}^P Bl_0(W'/\!/G)} \ar[ru]_{pr_{1,2}} \ar[dd]_(0.3){\cong} \\
    \HH \ar[rr]_(0.6){\gamma}  \ar[rd]_{\cong} && {W/\!/G} \\
    & Y_1 \ar[ru]_{pr_1} }  
\end{equation}   
où 
$$\begin{array}{ccccc}
pr_{1,2} & : & G' {\times}^P Bl_0(W'/\!/G) & \rightarrow & G' {\times}^P W'/\!/G \\
& & (g',f,L)P & \mapsto & (g',f)P 
\end{array}$$
et 
$$\begin{array}{ccccc}
pr_1 & : & Y_1 & \rightarrow &  W/\!/G  \\
& & (f,L,E_1,E_2) & \mapsto & f .
\end{array}$$  
On identifie $\HH$ à $Y_1$ grâce à l'isomorphisme (\ref{ultimeHiso}), alors $\gamma$ s'identifie à la composition des éclatements $Y_1 \rightarrow Y_0 \rightarrow W/\!/G$. En particulier $\gamma$ est une résolution de $W/\!/G$.

%% file: GL3.tex
\subsection{Etude du cas \texorpdfstring{$\dim(V)=3$}{n=3}}  \label{etudeGL3}

Dans toute cette section, on fixe $n=3$. On a $G \cong GL_3(k)$, $W/\!/G=\Hom(V_1,V_2)^{\leq 3}$ et $\rho:\ \HH \rightarrow \Gr(3,V_1^*) \times \Gr(3,V_2)$ est le morphisme construit dans la section \ref{princreduction}.
Nous allons démontrer le 

\begin{theoreme} \label{casn3} 
Si $n_1,n_2 \geq 3$, alors $\HH$ est connexe et singulier.
\end{theoreme} 

Lorsque $n_1=n_2=3$, la connexité de $\HH$ est donnée par le corollaire \ref{Hconnexe2} et la non-lissité de $\HH$ est donnée par le corollaire \ref{Hcasn3}. Le cas général $n_1,n_2 \geq 3$ se déduit immédiatement du cas particulier $n_1=n_2=3$ par la proposition \ref{reduction1}. 

\begin{remarque}
Si $n_1+n_2 \leq 4$, alors $\HH \cong W/\!/G$ est une variété lisse et $\gamma$ est un isomorphisme d'après le corollaire \ref{cas_facile}. 
\end{remarque}

\begin{remarque}
Les cas particuliers suivants ne sont pas traités par le théorème \ref{casn2}:
\begin{itemize} \renewcommand{\labelitemi}{$\bullet$}
\item $n_1=2$ et $n_2 \geq 3$,
\item $n_2=2$ et $n_1 \geq 3$,
\end{itemize}    
mais les cas particuliers $n_1=1$ ou $n_2=1$ ont été traités dans la section \ref{section_n_1}. 
\end{remarque}

\noindent On suppose dorénavant que $n_1=n_2=3$.

\subsubsection{Points fixes de \texorpdfstring{$\HH$}{H} pour l'opération de \texorpdfstring{$B'$}{B'}} \label{sectionptfixeGL3}

On souhaite montrer que $\HH$ est singulier. D'après le lemme \ref{Hlisse4}, il suffit de déterminer un point fixe de $\HHp$ pour l'opération de $B'$ tel que la dimension de l'espace tangent à $\HH$ en ce point soit différente de la dimension de $\HHp$. On commence donc par déterminer les points fixes de $B'$ dans $\HH$.\\
On rappelle que 
\begin{align*}
{k[W]}_2 &= (S^2(V_1) \otimes S^2(V^{*})) \oplus (S^2(V) \otimes S^2(V_{2}^{*}))  \oplus  ({\Lambda}^2 (V_1) \otimes {\Lambda}^2 (V^{*}))   \\
        &\ \oplus  ({\Lambda}^2 (V) \otimes {\Lambda}^2 (V_{2}^{*})) \oplus (V_1 \otimes V_{2}^{*} \otimes sl(V)) \oplus (V_1 \otimes V_{2}^{*} \otimes V_0) 
\end{align*}
comme $G' \times G$-module (c'est l'isomorphisme (\ref{kW2})). Et on a
\begin{align}
{k[W]}_3&=S^3((V_1 \otimes V^{*}) \oplus (V \otimes V_{2}^{*}))  \label{kW3}  \\
       &=S^3(V_1 \otimes V^{*}) \oplus (S^2(V_1 \otimes V^{*}) \otimes V \otimes V_{2}^{*}) \notag \\
       &\oplus  (V_1 \otimes V^{*} \otimes S^2(V \otimes V_{2}^{*})) \oplus     S^3(V \otimes V_{2}^{*}) \notag \\
       &=S^3(V_1 \otimes V^{*}) \oplus (((S^2(V_1) \otimes S^2(V^{*})) \oplus ({\Lambda}^2 (V_1) \otimes {\Lambda}^2 (V^{*})) ) \otimes V \otimes V_{2}^{*}) \notag\\
       &\oplus  (V_1 \otimes V^{*} \otimes ((S^2(V) \otimes S^2(V_2^{*})) \oplus ({\Lambda}^2 (V) \otimes {\Lambda}^2 (V_2^*)) )) \oplus S^3(V \otimes V_{2}^{*}) \notag \\  
       &=S^3(V_1 \otimes V^{*}) \oplus S^3(V \otimes V_{2}^{*}) \notag \\
       &\oplus  (S^2(V_1) \otimes V_{2}^{*} \otimes (V^{*} \oplus S^{{\epsilon}_1-2{\epsilon}_3}(V)))
       \oplus  ({\Lambda}^2 (V_1) \otimes V_2^* \otimes (V^{*} \oplus S^{{\epsilon}_1-{\epsilon}_2-\epsilon_3}(V))) \notag \\
       &\oplus (V_1 \otimes S^2(V_2^{*}) \otimes (V \oplus S^{2{\epsilon}_1-{\epsilon}_3}(V)))
       \oplus (V_1 \otimes {\Lambda}^2 (V_2^{*}) \otimes (V \oplus S^{{\epsilon}_1+\epsilon_2-{\epsilon}_3}(V))) \notag 
\end{align} 
comme $G' \times G$-module. 

\begin{notation}
On note
\begin{itemize} \renewcommand{\labelitemi}{$\bullet$}
\item $E:= \left\langle e_1 \otimes (f_2^* \wedge f_3^*) ,\ e_2 \otimes (f_2^* \wedge f_3^*)  ,\ e_1 \otimes (f_1^* \wedge f_3^*)    \right\rangle$,
\item $F:= \left\langle (e_1 \wedge e_2) \otimes f_3^*,\ (e_1 \wedge e_2) \otimes f_{2}^* ,\ (e_1 \wedge e_3) \otimes f_3^*  \right\rangle$,
\end{itemize}
qui sont des $B'$-sous-modules de $V_1 \otimes \Lambda^2(V_2^*)$ et $\Lambda^2 (V_1) \otimes V_2^*$ respectivement. \\
On note $D$ l'unique droite $B'$-stable de $V_1 \otimes V_{2}^{*}$. \\
Enfin, on note $I$ l'idéal de $k[W]$ engendré par $(V_1 \otimes V_{2}^{*} \otimes V_0) \oplus (D \otimes sl(V)) \subset {k[W]}_2$ et par $(F \otimes S^{{\epsilon}_1-\epsilon_2-{\epsilon}_3}(V)) \oplus (E \otimes S^{{\epsilon}_1+\epsilon_2-{\epsilon}_3}(V)) \subset k[W]_3$.   
\end{notation}

\begin{remarque}
L'idéal $I$ est homogène, $B' \times G$-stable et contient l'idéal $J$ engendré par les $G$-invariants homogènes de degré positif de $k[W]$. 
\end{remarque}

\begin{theoreme} \label{pointfixeborel2}
L'idéal $I$ est l'unique point fixe de $\HH$ pour l'opération de $B'$.
\end{theoreme}

\begin{proof}
Les points fixes de $\HH$ pour l'opération de $B'$ correspondent exactement aux idéaux homogènes $I_Z$ de $k[W]$ tels que:\\
i) $I_Z$ est $B' \times G$-stable, \\
ii) $k[W]/ I_Z =\bigoplus_{M \in \Irr(G)} M^{\oplus \dim(M)}$ comme $G$-module.\\
Soit $I_Z$ un tel point fixe, alors on peut étudier $I_Z$ degré par degré. En raisonnant comme dans la preuve du théorème \ref{pointfixeborel} pour l'étude des composantes de degré $0$, $1$ et $2$ de $I_Z$, on montre que $I_Z$ contient 
$$(V_1 \otimes V_{2}^{*} \otimes V_0) \oplus (D \otimes sl(V)) \subset {k[W]}_2 .$$
En particulier $I_Z \supset J$. On s'intéresse maintenant à la composante de degré $3$ de $I_Z$.
On a les dimensions suivantes:
$$\left\{
    \begin{array}{l}
\dim(S^{{\epsilon}_1-\epsilon_2-{\epsilon}_3}(V))=\dim(S^{{\epsilon}_1+\epsilon_2-{\epsilon}_3}(V))=6,\\
\dim({\Lambda}^2 (V_1) \otimes V_{2}^{*} )=\dim(V_1 \otimes {\Lambda}^2 (V_2^{*}))=9.
      \end{array}
\right.$$
Donc pour avoir la décomposition souhaitée de $k[W]/{I}_Z$ comme $G$-module, ${k[W]}_3 \cap {I}_Z$ contient nécessairement trois copies de $S^{{\epsilon}_1-\epsilon_2-{\epsilon}_3}(V)$ et trois copies de $S^{{\epsilon}_1+\epsilon_2-{\epsilon}_3}(V)$. Comme ${k[W]}_3 \cap {I}_Z$ est $B'$-stable, il contient $(M_2 \otimes S^{{\epsilon}_1-\epsilon_2-{\epsilon}_3}(V)) \oplus (M_1 \otimes S^{{\epsilon}_1+\epsilon_2-{\epsilon}_3}(V))$ où $M_2$ (resp. $M_1$) est un sous-espace $B'$-stable de dimension $3$ de $V_1 \otimes {\Lambda}^2 (V_2^{*})$ (resp. de ${\Lambda}^2 (V_1) \otimes (V_2^{*})$). On note $E_1:=V_1 \otimes  \left\langle f_2^* \wedge f_3^*  \right\rangle$, $E_2:=E$ et $E_3:= \left\langle e_1  \right\rangle \otimes \Lambda^2 (V_2^*)$ qui sont des $B'$-sous-modules de $V_1 \otimes \Lambda^2 (V_2^*)$, et on note $F_1:=\Lambda^2 (V_1) \otimes  \left\langle f_3^*  \right\rangle$, $F_2:=F$ et $F_3:= \left\langle e_1 \wedge e_2  \right\rangle \otimes V_2^*$ qui sont des $B'$-sous-modules de $\Lambda^2 (V_1) \otimes V_2^*$. On vérifie que les $B'$-modules $E_1$, $E_2$ et $E_3$ (resp. $F_1$, $F_2$ et $F_3$) sont les seuls $B'$-sous-modules de dimension $3$ de $V_1 \otimes \Lambda^2 (V_2^*)$ (resp. de $\Lambda^2 (V_1) \otimes V_2^*$). On a donc a priori neuf façons de choisir des générateurs pour $I_Z$ en degré $3$: pour $1 \leq i,j \leq 3$, on note $I_{i,j}$ l'idéal de $k[W]$ engendré par $(V_1 \otimes V_{2}^{*} \otimes V_0) \oplus (D \otimes sl(V)) \subset {k[W]}_2$ et par $(F_i \otimes S^{{\epsilon}_1-\epsilon_2-{\epsilon}_3}(V)) \oplus (E_j \otimes S^{{\epsilon}_1+\epsilon_2-{\epsilon}_3}(V)) \subset k[W]_3$. Alors $I_{2,2}=I$ et le lemme qui suit montre que nécessairement $I_Z \supset I$. 

\begin{lemme}  \label{exclusion}
Si $I_Z$ est un idéal $G$-stable contenant $I_{i,j}$ avec $(i,j) \neq (2,2)$, alors $I_Z$ ne peut pas avoir la fonction de Hilbert $h_W$.
\end{lemme}

\begin{proof}[\textbf{Preuve du lemme}]
On suppose par exemple que $j=3$ et soit $i$ quelconque. On note $w \in W$ sous la forme 
$$w=\left( \begin{bmatrix}
x_{11}  &x_{12} & x_{13} \\
x_{21} &x_{22} & x_{23} \\
x_{31} &x_{32} & x_{33} 
\end{bmatrix},
\begin{bmatrix}
y_{33} & y_{23} & y_{13}\\
y_{32} & y_{22} & y_{12} \\
y_{31} & y_{21} & y_{11}
\end{bmatrix} \right)$$ 
et on identifie $k[W]$ à $k[x_{ij},y_{ij},\ 1 \leq i,j \leq 3]$. D'après la proposition \ref{KSheadings}, l'idéal ${\left(I_{i,3}/J \right)}^U$ de ${\left( k[W]/J \right)}^U$ contient les éléments suivants: 
$$
\left \{
    \begin{array}{l}
  x_{11}(y_{22}y_{11}-y_{21}y_{12}),\\
  x_{11}(y_{23}y_{11}-y_{13}y_{21}),\\
  x_{11}(y_{12}y_{23}-y_{13}y_{22}),
    \end{array}
\right.$$
qui forment une base du $B' \times T$-sous-module $E_3 \otimes (S^{{\epsilon}_1+\epsilon_2-{\epsilon}_3}(V))^U$ de ${\left( k[W]/J \right)}^U$.\\   
D'après la proposition \ref{decompoiso}, le $G$-module $S^{\epsilon_1+\epsilon_2-2 \epsilon_3}(V)$ apparaît dans $k[W]/J$ uniquement dans $ k[W]_4/(k[W]_4 \cap J) $ et avec une multiplicité égale à $18$, cette multiplicité étant égale à la dimension de l'espace vectoriel engendré dans ${\left( k[W]/J \right)}^U$ par l'ensemble $K$ suivant:
$$ 
K \, \left \{
    \begin{array}{l}
x_{11}^2 (y_{22}y_{11}-y_{12}y_{21}),\\
x_{12}^2 (y_{22}y_{11}-y_{12}y_{21}),\\
x_{13}^2 (y_{22}y_{11}-y_{12}y_{21}),\\
x_{11} x_{12} (y_{22}y_{11}-y_{12}y_{21}),\\
x_{11} x_{13} (y_{22}y_{11}-y_{12}y_{21}),\\
x_{12} x_{13} (y_{22}y_{11}-y_{12}y_{21}),\\
x_{11}^2 (y_{23}y_{11}-y_{13}y_{21}),\\
x_{12}^2 (y_{23}y_{11}-y_{13}y_{21}),\\
x_{13}^2 (y_{23}y_{11}-y_{13}y_{21}),\\
x_{11} x_{12} (y_{23}y_{11}-y_{13}y_{21}),\\
x_{11} x_{13} (y_{23}y_{11}-y_{13}y_{21}),\\
x_{12} x_{13} (y_{23}y_{11}-y_{13}y_{21}),\\
x_{11}^2 (y_{12}y_{23}-y_{13}y_{22}),\\
x_{12}^2 (y_{12}y_{23}-y_{13}y_{22}),\\
x_{13}^2 (y_{12}y_{23}-y_{13}y_{22}),\\
x_{11} x_{12} (y_{12}y_{23}-y_{13}y_{22}),\\
x_{11} x_{13} (y_{12}y_{23}-y_{13}y_{22}),\\
x_{12} x_{13} (y_{12}y_{23}-y_{13}y_{22}).      
   \end{array}
\right. 
$$
Montrons que la multiplicité du $G$-module $S^{\epsilon_1+\epsilon_2-2 \epsilon_3}(V)$ dans $ k[W]_4/(k[W]_4 \cap I_{i,3}) $ est strictement inférieure à $\dim(S^{\epsilon_1+\epsilon_2-2 \epsilon_3}(V))=10$. Pour calculer cette multiplicité, il suffit de calculer la dimension de l'espace vectoriel engendré par $K' \subset K$ dans ${\left( k[W]/J \right)}^U$, où $K'$ est obtenu en supprimant les éléments de $K$ qui appartiennent à l'idéal ${\left(I_{i,3}/J \right)}^U$. On obtient: 
$$ 
K' \, \left \{
    \begin{array}{l}
x_{12}^2 (y_{22}y_{11}-y_{12}y_{21}),\\
x_{13}^2 (y_{22}y_{11}-y_{12}y_{21}),\\
x_{12} x_{13} (y_{22}y_{11}-y_{12}y_{21}),\\
x_{12}^2 (y_{23}y_{11}-y_{13}y_{21}),\\
x_{13}^2 (y_{23}y_{11}-y_{13}y_{21}),\\
x_{12} x_{13} (y_{23}y_{11}-y_{13}y_{21}),\\
x_{12}^2 (y_{12}y_{23}-y_{13}y_{22}),\\
x_{13}^2 (y_{12}y_{23}-y_{13}y_{22}),\\
x_{12} x_{13} (y_{12}y_{23}-y_{13}y_{22}).      
   \end{array}
\right. 
$$
Donc $K'$ engendre un espace vectoriel de dimension $9$, et donc on ne va jamais avoir $10$ copies de $S^{\epsilon_1+\epsilon_2-2 \epsilon_3}(V)$ dans un quotient de $k[W]/I_{i,3}$, ce qui montre le lemme pour $j=3$. Le cas $j=1$ se démontre de manière analogue.  
\end{proof}

D'après le lemme \ref{fixespoints}, le schéma $\HH$ admet au moins un point fixe $I_Z$ pour $B'$ et d'après le lemme \ref{exclusion}, on a nécessairement $I_Z \supset I$. Puis, d'après le corollaire \ref{fcthilbclassn3}, l'idéal $I_Z$ admet la fonction de Hilbert classique suivante:
$$\forall p \in \NN,\ f_{I_Z}(p)= \sum_{\substack{(r_1,r_2,r_3) \in \ZZZ^3 \\ r_1 \geq r_2 \geq r_3 \\ |r_1|+|r_2|+|r_3|=p}}  (\dim(S^{r_1 \epsilon_1+ r_2 \epsilon_2 +r_3 \epsilon_3}(V) ))^2 .$$
D'après la formule des dimensions de Weyl, on a
$$\dim(S^{r_1 \epsilon_1+ r_2 \epsilon_2 +r_3 \epsilon_3}(V))=\frac{1}{2}(1+r_1-r_2)(2+r_1-r_3)(1+r_2-r_3) .$$
Donc
$$f_{I_Z}(p)= \sum_{\substack{(r_1,r_2,r_3) \in \ZZZ^3 \\ r_1 \geq r_2 \geq r_3 \\ |r_1|+|r_2|+|r_3|=p}}  \frac{1}{4}(1+r_1-r_2)^2(2+r_1-r_3)^2(1+r_2-r_3)^2 .$$
Après développement puis simplification, on obtient:
$$\forall p \geq 0,\ f_{I_Z}(p)=1+\frac{122}{35}p+\frac{1654}{315}p^2+\frac{547}{120}p^3+\frac{91}{36}p^4+\frac{37}{40}p^5+\frac{79}{360}p^6+\frac{13}{420}p^7+\frac{1}{504}p^8 .$$
Par ailleurs, un calcul direct avec \cite[Macaulay2]{Mac2} de la fonction de Hilbert classique de l'idéal $I$ nous donne $f_{I}=f_{I_Z}$ et donc $I_Z=I$, ce qui achève la preuve du théorème \ref{pointfixeborel2}.  
\end{proof}

\begin{remarque}
On a $\Stab_{G'}(I)=B'$, donc l'unique orbite fermée de $\HH$ est isomorphe à $G'/B'$.
\end{remarque}

\noindent Le corollaire qui suit découle du lemme \ref{fixespoints} et du théorème \ref{pointfixeborel2}.

\begin{corollaire} \label{Hconnexe2}
Le schéma $\HH$ est connexe.
\end{corollaire}


\subsubsection{Espace tangent de \texorpdfstring{$\HH$}{H} en \texorpdfstring{$I$}{I}} \label{dimTang2}

On note $Z_0:=\Spec(k[W]/I)$. Nous allons démontrer la  

\begin{proposition} \label{dimTangent2}
$\dim(T_{Z_0} \HH)=12.$
\end{proposition}

Pour démontrer cette proposition, nous allons procéder comme pour la proposition \ref{dimTangent}. Cette section est essentiellement calculatoire et on recommande donc au lecteur qui souhaite s'épargner un mal de tête d'en admettre les résultats sans trop s'attarder sur les démonstrations.\\
On identifie $k[W]$ à $k[x_{ij},y_{ij},\ 1 \leq i,j \leq 3]$ comme dans la preuve du lemme \ref{exclusion} et on explicite des bases de certains $G$-modules qui apparaissent dans $k[W]$:\\
$
 \left.
    \begin{array}{l}
f_1:=y_{33} x_{11}+y_{23} x_{21}+y_{13} x_{31}\\
f_2:=y_{33} x_{12}+y_{23} x_{22}+y_{13} x_{32}\\
f_3:=y_{33} x_{13}+y_{23} x_{23}+y_{13} x_{33}\\
f_4:=y_{32} x_{11}+y_{22} x_{21}+y_{12} x_{31}\\
f_5:=y_{32} x_{12}+y_{22} x_{22}+y_{12} x_{32}\\
f_6:=y_{32} x_{13}+y_{22} x_{23}+y_{12} x_{33}\\
f_7:=y_{31} x_{11}+y_{21} x_{21}+y_{11} x_{31}\\
f_8:=y_{31} x_{12}+y_{21} x_{22}+y_{11} x_{32}\\
f_9:=y_{31} x_{13}+y_{21} x_{23}+y_{11} x_{33}
      \end{array}
\right\} \text{ est une base de } V_1 \otimes V_2^* \otimes V_0, \\
 \left.
    \begin{array}{l}
         h_1:=x_{11} y_{11}\\
         h_2:=x_{11} y_{21}\\
         h_3:= x_{11} y_{31}\\
         h_4:=x_{21} y_{11}\\
         h_5:= x_{21} y_{21}\\
         h_6:= x_{21} y_{31}\\ 
         h_7:=x_{31} y_{11}\\
         h_8:= x_{31} y_{21}\\
         h_9:= x_{31} y_{31}
               \end{array}
\right\} \text{ est une base de } D \otimes V^* \otimes V \cong D \otimes (sl(V) \oplus V_0), \\
 \left.
    \begin{array}{l}
s_1:=x_{12} (y_{22} y_{11}-y_{21} y_{12})\\
s_2:=x_{12} (y_{32} y_{11}-y_{31} y_{12})\\
s_3:=x_{22} (y_{32} y_{11}-y_{31} y_{12})\\
s_4:=x_{12} (y_{32} y_{21}-y_{31} y_{22})\\
s_5:=x_{22} (y_{32} y_{21}-y_{31} y_{22})\\
s_6:=x_{32} (y_{32} y_{21}-y_{31} y_{22})\\
s_7:=x_{22} (y_{22} y_{11}-y_{21} y_{12})\\
s_8:=x_{32} (y_{22} y_{11}-y_{21} y_{12})\\
s_9:=x_{32} (y_{32} y_{11}-y_{31} y_{12})
               \end{array}
\right\} \text{ est une base de } e_2 \otimes (f_{n-1}^* \wedge f_n^*) \otimes \Lambda^2(V) \otimes V^* \subset k[W]_3, \\
 \left.
    \begin{array}{l}
t_1:=y_{12} (x_{22} x_{11}-x_{21} x_{12})\\
t_2:=y_{12} (x_{32} x_{11}-x_{31} x_{12})\\
t_3:=y_{22} (x_{32} x_{11}-x_{31} x_{12})\\
t_4:=y_{12} (x_{32} x_{21}-x_{31} x_{22})\\ 
t_5:=y_{22} (x_{32} x_{21}-x_{31} x_{22})\\
t_6:=y_{32} (x_{32} x_{21}-x_{31} x_{22})\\
t_7:=y_{22} (x_{22} x_{11}-x_{21} x_{12})\\
t_8:=y_{32} (x_{22} x_{11}-x_{21} x_{12})\\
t_9:=y_{32} (x_{32} x_{11}-x_{31} x_{12})
               \end{array}
\right\} \text{ est une base de } (e_1 \wedge e_2) \otimes f_{n-1}^* \otimes \Lambda^2(V^*) \otimes V  \subset k[W]_3$.\\
Alors
$$N:= \left\langle  f_1,\ldots,f_9,h_1,\ldots,h_9,s_1,\ldots,s_6,t_1,\ldots,t_6 \right\rangle \subset k[W]$$
est un $G$-module qui engendre l'idéal $I$. On note $R:=k[W]/I$. On a le

\begin{lemme}  \label{inegTang1}
$\dim ({\Hom}_R^G(I/I^2,R)) \geq 12.$
\end{lemme}

\begin{proof}
D'après \cite[Macaulay2]{Mac2}, une famille de générateurs du $R$-module $\Hom_R(I/I^2,R)$ est donnée dans la table \ref{homom_cas_n3}.
\begin{table}[!p] 
\includegraphics[scale=0.8]{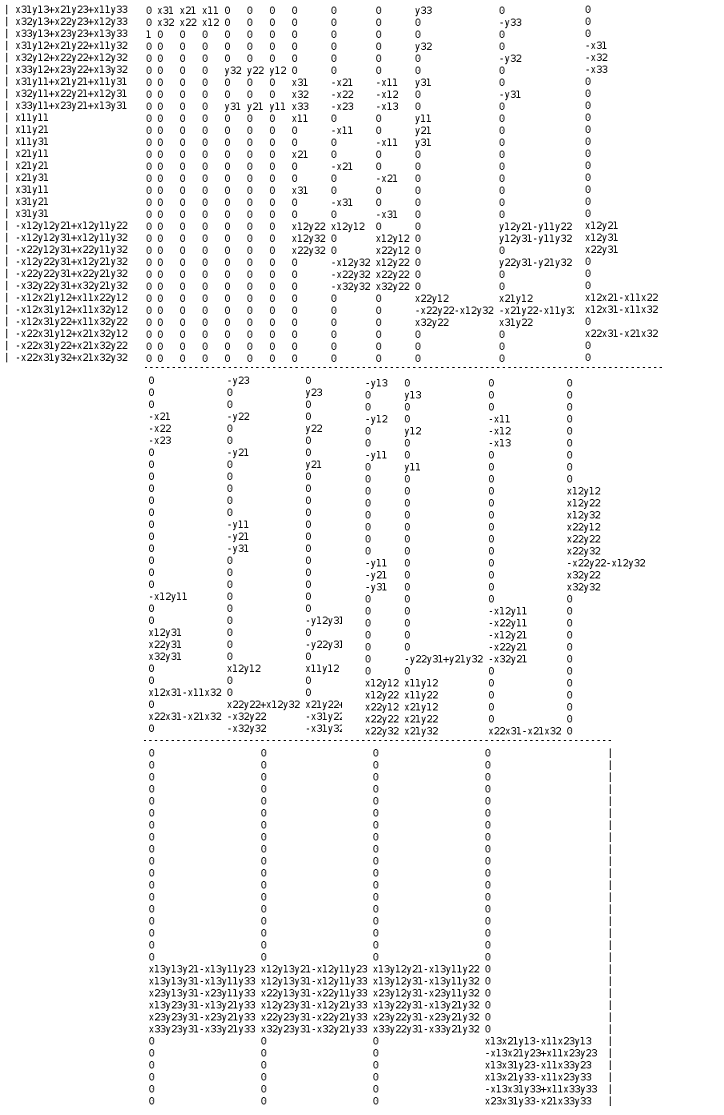}
\caption{\label{homom_cas_n3} Générateurs du $R$-module $\Hom_R(I/I^2,R)$}
\end{table}
On note $\Phi_i$ le morphisme donné par la $i+1$-ème colonne de cette matrice. On vérifie que les douze morphismes suivants sont $G$-équivariants et linéairement indépendants:
\begin{itemize} \renewcommand{\labelitemi}{$\bullet$}
   \item $\Phi_1$      
   \item $\Phi_{20}$,   
   \item $\Phi_{21}$,  
   \item $\Phi_{22}$,  
   \item $\Phi_{23}$,  
   \item $\Phi_{24}$, 
   \item $-x_{13} \Phi_{12}+x_{23} \Phi_{16}+x_{33} \Phi_{18}$, 
   \item $ y_{13} \Phi_{13}+y_{23} \Phi_{14}+y_{33} \Phi_{19}$,  
   \item $y_{12} \Phi_8-y_{22} \Phi_9-y_{32} \Phi_{10}$, 
   \item $y_{13} \Phi_8-y_{23} \Phi_9-y_{33} \Phi_{10}$, 
   \item $x_{12} \Phi_{11}-x_{22} \Phi_{15}-x_{32} \Phi_{17}$, 
   \item $x_{13} \Phi_{11}-x_{23} \Phi_{15}-x_{33} \Phi_{17}$. 
\end{itemize}  
Il s'ensuit que ${\Hom}_R^G(I/I^2,R)$ est de dimension au moins $12$.
\end{proof}

On reprend les notations de la section \ref{ConnexitéetTangence}. On a $\dim(N)=29$ et donc d'après le lemme \ref{InegTang}, on a $\dim(T_{Z_0} \HH)=29- \rg(\rho^*)$. Puis, d'après le lemme \ref{inegTang1}, on a $\rg(\rho^*) \leq 17$ et donc il suffit de montrer que $\rg(\rho^*) \geq 17$ pour prouver la proposition \ref{dimTangent2}.\\  
Pour $i=1, \ldots ,9$, on définit ${\psi}_i \in {\Hom}_R^G(R {\otimes} N,R)$ par 
$$ \left\{
    \begin{array}{ll}
        {\psi}_i(f_j \otimes 1)={\delta}_i^j & \text{ pour } j=1,\ldots,9, \\
        {\psi}_i( h_j \otimes 1)=0 & \text{ pour } j=1,\ldots,9,\\
        {\psi}_i( s_j \otimes 1)={\psi}_i( t_j \otimes 1)=0 & \text{ pour } j=1,\ldots,6,\\
    \end{array}
\right.$$
où ${\delta}_i^j$ est le symbole de Kronecker.\\
Ensuite, on définit les neuf éléments $\phi_1, \ldots, \phi_5, \gamma_1, \gamma_2, \delta_1, \delta_2$ de $\Hom^G(N,R)$ dans la table \ref{homom_cas_n333} de la façon suivante: pour chaque colonne $2$ à $10$, la première ligne indique le nom du morphisme et chaque autre ligne donne l'image dans $R$ de l'élément de $N$ donné dans la première colonne. Et les images de $f_1, \ldots, f_9, h_1, s_1, t_1$ déterminent entièrement chaque morphisme par $G$-équivariance. 
\begin{table}
$$\begin{array}{|l|c|c|c|c|c|}
  \hline
   & \phi_1 & \phi_2  & \phi_3 & \phi_4 & \phi_5   \\
  \hline
  f_i,\ 1 \leq i \leq 9  & 0 & 0  & 0 & 0 & 0  \\
   \hline
   h_1:=x_{11} y_{11}   & x_{11} y_{13} & x_{12} y_{13}  & x_{13} y_{11} & x_{13} y_{12} & x_{13} y_{13} \\
    \hline
   s_1:=x_{12} (y_{22} y_{11}-y_{21} y_{12}) & 0 & 0  & 0 & 0 & 0  \\
    \hline
   t_1:=y_{12} (x_{22} x_{11}-x_{21} x_{12}) & 0 & 0  & 0 & 0 & 0  \\
    \hline
\end{array}$$  
$$\begin{array}{|c|c|c|c|}
  \hline
     \gamma_1 & \gamma_2  & \delta_1 & \delta_2  \\
  \hline
   0 & 0  & 0 & 0   \\
   \hline
    0 & 0 & 0 & 0\\
    \hline
     x_{11} (y_{22} y_{13}-y_{23} y_{12}) &  x_{13} (y_{22} y_{13}-y_{23} y_{12}) & 0 & 0\\
    \hline
    0 & 0 &  y_{11} (x_{22} x_{13}-x_{23} x_{12}) &  y_{13} (x_{22} x_{13}-x_{23} x_{12})\\
    \hline
\end{array}$$ 
\vspace{0.1 cm}  
\caption{\label{homom_cas_n333}}
\end{table}

\noindent Enfin, on identifie ces morphismes à des éléments de $\Hom_R^G(R \otimes N,R)$ via l'isomorphisme $\Hom^G(N,R) \cong \Hom_R^G(R \otimes N,R)$.
Par ailleurs, \cite[Macaulay2]{Mac2} nous donne les relations suivantes entre les générateurs du $R$-module $I/I^2$ définis précédement:
$$
 \left \{
    \begin{array}{l}
   r_1:=-h_1 \otimes y_{21}+h_2 \otimes y_{11},\\
   r_2:=-f_1 \otimes y_{21}+h_2\otimes y_{33}+ h_5 \otimes y_{23}+h_8 \otimes y_{13}, \\
   r_3:=-f_4 \otimes y_{21}+h_2\otimes y_{32}+ h_5 \otimes y_{22}+h_8 \otimes y_{12}, \\
   r_4:=-f_7 \otimes y_{21}+h_2\otimes y_{31}+ h_5 \otimes y_{21}+h_8 \otimes y_{11}, \\
   r_5:=-f_8 \otimes x_{21}+h_4\otimes y_{32}+ h_5 \otimes y_{22}+h_6 \otimes y_{12}, \\
   r_6:=-f_9 \otimes x_{21}+h_4\otimes y_{33}+ h_5 \otimes y_{23}+h_6 \otimes y_{13}, \\   
   r_7:=f_6 \otimes x_{12} y_{11}-f_9 \otimes x_{12} y_{12}-s_1 \otimes x_{23}-s_7 \otimes x_{13}, \\
   r_8:=-f_1 \otimes x_{22} y_{12}+f_2 \otimes x_{21} y_{12}+t_1 \otimes y_{33}-t_2 \otimes y_{13}, \\
   r_9:=f_5 \otimes x_{12} y_{11}-f_9 \otimes x_{12} y_{12}-s_1 \otimes x_{22}-s_7 \otimes x_{12}, \\
   r_{10}:=-h_1 \otimes x_{32} y_{22}-h_2 \otimes (x_{22} y_{22}+x_{12} y_{32})+s_1 \otimes x_{31}-t_7 \otimes y_{21}+ t_8 \otimes y_{11}. 
      \end{array}
      \right. $$

\begin{remarque}
Ces relations ne suffisent pas à engendrer le module des relations entre les générateurs de $I/I^2$.
\end{remarque}

\noindent Le lemme qui suit achève la preuve de la proposition \ref{dimTangent2}:

\begin{lemme} 
$\rg(\rho^*) \geq 17.$
\end{lemme}

\begin{proof}
Soit $({\lambda}_1, {\lambda}_2, {\lambda}_4,\ldots,{\lambda}_9, \alpha_1, \ldots, \alpha_5, \epsilon_1, \epsilon_2, \beta_1,\beta_2) \in k^{17}$ tel que 
\begin{equation} \label{relLinaire2}
 \sum_{\substack{1 \leq i \leq 9\\ i \neq 3}} {\lambda}_i \, \rho^*({\psi}_i)+\sum_{1 \leq j \leq 5} {\alpha}_j \, \rho^*({\phi}_j)+\epsilon_1 \rho^*(\gamma_1)+ \epsilon_2 \rho^*(\gamma_2)+\beta_1 \rho^*(\delta_1)+ \beta_2 \rho^*(\delta_2)=0 .
\end{equation}
On évalue (\ref{relLinaire2}) en $r_1 \otimes 1$, on obtient:
$$\alpha_2 x_{12} (y_{11} y_{23}-y_{13} y_{21}) + \alpha_4 x_{13} (y_{11} y_{22}-y_{12} y_{21}) + \alpha_5 x_{13} (y_{11} y_{23}-y_{13} y_{21})=0 .$$
Donc $\alpha_2=\alpha_4=\alpha_5=0$.\\
On évalue (\ref{relLinaire2}) en $r_2\otimes 1,\ r_3\otimes 1,\ r_4\otimes 1,\ r_5\otimes 1,\ r_6\otimes 1$ successivement, on obtient $\lambda_1=\lambda_4=\lambda_7=\lambda_8=\lambda_9=0$.\\
On évalue (\ref{relLinaire2}) en $r_7\otimes 1$, on obtient $\lambda_6=\epsilon_1=\epsilon_2=0$.\\
On évalue (\ref{relLinaire2}) en $r_8\otimes 1$, on obtient $\lambda_2=\beta_1=\beta_2=0$.\\
On évalue (\ref{relLinaire2}) en $r_9\otimes 1$, on obtient $\lambda_5=0$.\\
On évalue (\ref{relLinaire2}) en $r_{10}\otimes 1$, on obtient $\alpha_1=\alpha_3=0$. \\
Donc $\left \{ \rho^*({\psi}_i),\ 1 \leq i \leq 9, i \neq 3 \right \} \cup \left \{ \rho^*(\phi_1), \ldots, \rho^*(\phi_5), \rho^*(\gamma_1),\rho^*(\gamma_2), \rho^*(\delta_1), \rho^*(\delta_2) \right \}$ est une famille libre de cardinal $17$ dans $\Im(\rho^*)$. 
\end{proof}

\noindent On déduit du lemme  \ref{Hlisse4} et de la proposition \ref{dimTangent2} le 

\begin{corollaire} \label{Hcasn3}
Le schéma $\HH$ est singulier en $Z_0$.
\end{corollaire}

%% file: symplectique_GLn_1.tex
\section{Cas de \texorpdfstring{$GL(V)$}{GLV} opérant dans \texorpdfstring{$\mu^{-1}(0)$}{la fibre en 0 de mu}} \label{posisimpy}

On considère le cas particulier suivant:
\begin{itemize} \renewcommand{\labelitemi}{$\bullet$}
\item $V_1=V_2=:E$ et $d:=\dim(E)$, 
\item $W:=\Hom(E,V) \times \Hom(V,E)=\Hom(E,V) \times \Hom(E,V)^*$,
\item $G':=GL(E)$ vu comme sous-groupe de $GL(V_1) \times GL(V_2)$ via le plongement diagonal,
\item $B':=(B_1 \times B_2) \cap GL(E)$, où $B_1$ et $B_2$ ont été définis dans la section \ref{lesdiffsituations}.
\end{itemize}
L'opération de $GL(V_1) \times GL(V_2) \times G$ dans $W$ induit une opération de $G' \times G$ dans $W$. On note 
$$\gg'=\End(E)$$ 
l'algèbre de Lie de $G'$ et 
$$\gg'^{\leq n}=\{f \in \End(E)\ |\ \rg(f) \leq n\}.$$
On a vu dans la section \ref{description_quotient} que $W/\!/G \cong \gg'^{\leq n}$ comme $G'$-variété. 
Nous allons étudier
$$\HHm:=\mathrm{Hilb}_{h_s}^{G}(\mu^{-1}(0))$$
le schéma de Hilbert invariant pour l'opération de $G$ dans $\mu^{-1}(0)$ et $\HHmp$ sa composante principale. Dans la section \ref{appzmoment}, on décrit l'application moment $\mu:\ W \rightarrow \End(V)$ ainsi que sa fibre en $0$. Dans la section \ref{sectionavecX}, on étudie le morphisme de passage au quotient $\mu^{-1}(0) \rightarrow \mu^{-1}(0)/\!/G$. On appelle $\mu^{-1}(0)/\!/G$ la réduction symplectique de $W$. On verra que ${\mu}^{-1}(0)/\!/G$ est irréductible et on pourra donc parler de la fonction de Hilbert $h_s$ de la fibre générique de $\nu$. Dans la section \ref{MropRRED}, on établit un résultat de réduction qui permet de ramener la détermination de $\HHmp$ à celle de $\Hilb_{h_W}^{G}(W)$; ce dernier a été étudié dans la section \ref{GLngénéral}. Dans la section \ref{reductibilité_cas_symp}, on montre que $\HHm$ admet plusieurs composantes irréductibles lorsque $d \geq 2n$. Enfin, on étudie en détail $\HHm$ lorsque $n=1$ dans la section \ref{GLsympN1}.

\subsection{L'application moment}  \label{appzmoment}
Des rappels concernant les variétés symplectiques, les applications moment et les résolutions symplectiques sont donnés dans la section \ref{appendice2}.

Soit $\phi: (u_1,u_2) \mapsto \tr(u_2 \circ u_1)$ qui est une forme quadratique sur $W$ non-dégénérée et $G' \times G$-invariante et soit $\Omega$ la forme symplectique sur $W$ définie par:
\begin{equation} \label{defsymp}
\forall u_1,v_1 \in \Hom(E,V),\ \forall u_2,v_2 \in \Hom(V,E),\ \Omega((u_1,u_2),(v_1,v_2)):=\phi(v_1,u_2)-\phi(u_1,v_2).
\end{equation}
Alors $G' \times G$ opère symplectiquement dans $(W,\Omega)$, c'est-à dire que l'opération de $G' \times G$ préserve $\Omega$. On vérifie qu'une application moment pour l'opération de $G$ dans $(W,\Omega)$ est donnée par:
\begin{equation} \label{defmu}
\begin{array}{ccccc}
\mu & : & W & \to & \gg^* \\
& & (u_1,u_2) & \mapsto & (\phi:\ h \mapsto \phi(h \circ u_1,u_2)) \\
\end{array}
\end{equation}
où l'on rappelle que l'on note $\gg=\End(V)$ l'algèbre de Lie de $G$. Puis, on a l'isomorphisme $G$-équivariant
\begin{align*}
\gg &\cong \gg^*\\
h & \mapsto (h' \mapsto \tr(h \circ h'))
\end{align*}
et donc $\mu:\ W \rightarrow \gg$ est donnée par:
\begin{equation}  \label{defMoment}
\forall (u_1,u_2) \in W,\ \mu(u_1,u_2)=u_1 \circ u_2.
\end{equation}

\begin{remarque}
On fera par la suite l'abus de désigner cette application $\mu$ comme l'application moment pour l'opération de $G$ dans $(W,\Omega)$ bien qu'il en existe une infinité qui se déduisent de $\mu$ par translation par un élément de $\gg^G \cong k$.
\end{remarque}

L'application moment coïncide avec le morphisme de passage au quotient $W \rightarrow W/\!/G'$. En particulier, $\mu^{-1}(0)$ est le nilcône $\NNN(W,G')$ (cf définition \ref{defnilk}) et on déduit de \cite[Theorem 9.1]{KS} que le schéma $\mu^{-1}(0)$ est toujours réduit. La $G$-équivariance et la $G'$-invariance de $\mu$ impliquent que $\mu^{-1}(0)$ est un fermé $G' \times G$-stable de $W$, on pourra donc considérer le quotient $\mu^{-1}(0)/\!/G$.
On a bien sûr 
$$\mu^{-1}(0)=\left\{(u_1,u_2) \in W \ |\ u_1 \circ u_2=0\right\}.$$

Déterminons les composantes irréductibles de $\mu^{-1}(0)$ et leurs dimensions. Soit $m \in \{0,\ldots,d\}$, on définit le fermé
$$X'_m:=\left\{(u_1,u_2)\in W \ \middle| \ 
    \begin{array}{l}
       \Im(u_2) \subset \Ker(u_1),\\
       \rg(u_2) \leq \min(n,m),\\
       \dim(\Ker(u_1)) \geq \max(d-n,m).   
    \end{array}
\right\}.$$ 

\noindent La preuve de la proposition \ref{compirredfibzero2} (resp. du corollaire \ref{fibzero2}) est analogue à celle de la proposition \ref{compirredfibzero} (resp. du corollaire \ref{fibzero}).

\begin{proposition} \label{compirredfibzero2}
Chaque $X'_m$ est un fermé irréductible de $\mu^{-1}(0)$ et les composantes irréductibles de $\mu^{-1}(0)$ sont:
$$\left\{
    \begin{array}{ll}
        X'_0, \ldots ,X'_{d} &\text{ si } d \leq n,\\
        X'_{d-n}, \ldots ,X'_n &\text{ si } n < d < 2n,\\
        X'_n &\text{ si } d \geq 2n.
    \end{array}
\right.$$
\end{proposition}

\begin{corollaire} \label{fibzero2}
Si $m \leq n$ ou $m \geq d-n$, alors $\dim(X_m)=m(d-m)+dn$ et donc
$$
\dim(\mu^{-1}(0))= \left\{
    \begin{array}{ll}
            nd+\frac{1}{4}{d}^2   &\text{ si }  d < 2n \text{ et $d$ est pair,}\\
            nd+\frac{1}{4}({d}^2-1) &\text{ si }  d < 2n \text{ et $d$ est impair,} \\
            2nd-n^2 &\text{ si } d \geq 2n.
    \end{array}
\right.
$$
\end{corollaire}

\subsection{Etude du morphisme de passage au quotient}  \label{sectionavecX}
Dans cette section, on décrit géométriquement la réduction symplectique ${\mu}^{-1}(0)/\!/G$. On note 
$$N:=\min \left(E\left( \frac{d}{2} \right),n \right)$$
où l'on rappelle que $E(x)$ désigne la partie entière inférieure de $x$. On introduit les notations suivantes qui nous seront utiles pour les preuves de la proposition \ref{descQuotient} et du lemme \ref{fibreUnsymp1}.

\begin{notation}  \label{BipBip}
Pour $l=0,\ldots,N$, on note\\
\begin{itemize} \renewcommand{\labelitemi}{$\bullet$}
\item $u_1^l \in \Hom(E,V)$ le morphisme défini par la matrice $\begin{bmatrix}
0_{l,d-l}  &I_l \\
0_{n-l,d-l}  & 0_{n-l,l} 
\end{bmatrix},$
\item $u_2^l \in \Hom(V,E)$ le morphisme défini par la matrice $\begin{bmatrix}
I_l  &0_{l,n-l} \\
0_{d-l,l}  & 0_{d-l,n-l} 
\end{bmatrix},$
\item $f^l \in \End(E)^{\leq N}$ l'endomorphisme défini par la matrice $\begin{bmatrix}
0_{l,d-l}  &I_l \\
0_{d-l,d-l}  & 0_{d-l,l} 
\end{bmatrix}.$
\end{itemize}
\end{notation}

On va utiliser des résultats sur la combinatoire des orbites nilpotentes dans les algèbres de Lie semi-simples pour lesquels on renvoie à \cite{CoMc}. On rappelle que l'on a une correspondance bijective entre les partitions $(d_1 \geq \ldots \geq d_k)$ de $n'$ et les orbites nilpotentes $\OO_{[d_1,\ldots,d_k]}$ de $\gg'$ (\cite[§3.1]{CoMc}). On a vue que $\mu^{-1}(0)$ est réduit, donc $\mu^{-1}(0)/\!/G$ est réduit et on a la 

\begin{proposition}  \label{descQuotient}
$\mu^{-1}(0)/\!/G=\overline{\OO_{[2^N,1^{d-2N}]}}.$ 
\end{proposition}

\begin{proof}
Si $f \in \mu^{-1}(0)/\!/G \subset \gg'^{\leq n}$, alors il existe $(u_1,u_2) \in \mu^{-1}(0)$ tel que $f=u_2 \circ u_1$ et donc $f \circ f=(u_2 \circ u_1) \circ (u_2 \circ u_1)=u_2 \circ (u_1 \circ u_2) \circ u_1=0$. 
Donc $f$ est conjuguée à $f^l$ pour un certain $l \leq N$ et on a donc l'inclusion $\mu^{-1}(0)/\!/G \subset \overline{\OO_{[2^N,1^{d-2N}]}}$. Réciproquement, soit $f \in \overline{\OO_{[2^N,1^{d-2N}]}}$; quitte à conjuguer $f$ par un élément de $G'$, on peut supposer que $f=f^l$ pour un certain $l \leq N$. On vérifie alors que $u_1^l \circ u_2^l=0$ et $u_2^l \circ u_1^l=f^l$, ce qui achève la preuve de cette proposition. 
\end{proof}   

\begin{remarque}
Nous verrons dans la section \ref{appendice2} que la variété $\mu^{-1}(0)/\!/G$ est symplectique et admet toujours des résolutions symplectiques.  
\end{remarque}

\begin{corollaire}
Le quotient $\mu^{-1}(0)/\!/G$ se décompose en $N+1$ orbites pour l'opération de $G':$ 
$$U_i:=\OO_{[2^i,1^{d-2i}]},\ \text{ pour } i=0, \ldots, N.$$
\end{corollaire}

\noindent Les adhérences de ces orbites sont imbriquées de la façon suivante:
$$\{0\}=\overline{U_0} \subset \cdots \subset \overline{U_N}=\mu^{-1}(0)/\!/G.$$
La géométrie de l'adhérence d'une orbite nilpotente dans $\gg'$ est décrite dans \cite{KP2} et \cite{KP4}. En particulier $\mu^{-1}(0)/\!/G$ est normale (\cite{KP2}), de dimension $2N(d-N)$ (\cite[Corollary 6.1.4]{CoMc}) et son lieu singulier est $\overline{U_{N-1}}$ (\cite[§3.2]{KP4}). On sait que $\nu$ est plat sur $U_N$ qui est un ouvert dense de $\mu^{-1}(0)/\!/G$ et donc $\HHmp=\overline{\gamma^{-1}(U_N)}$. 

\begin{remarque}  \label{cas_debile}
Si $d=1$, alors $\HHm$ est le point réduit correspondant au sous-schéma fermé $\mu^{-1}(0) \subset W$ et $\gamma$ est un isomorphisme.
\end{remarque}

\noindent La dimension de la fibre de $\nu$ en un point de $U_N$ est:
$$\left\{
    \begin{array}{ll}
          nd-\frac{1}{4} {d}^2 &\text{ si } d <2n \text{ et $d$ est pair, } \\
          nd-\frac{1}{4} ({d}^2-1) &\text{ si } d <2n \text{ et $d$ est impair,}  \\
          n^2 &\text{ si } d \geq 2n.
    \end{array}
\right.
$$

\begin{notation}
Si $d<2n$, on note 
$$H:=\left\{ \begin{bmatrix}
M  &0_{n-N,N} \\
0_{N,n-N}  & I_N 
\end{bmatrix},\ M \in GL_{n-N}(k)\right\} \cong GL_{n-N}(k)$$
qui est un sous-groupe réductif de $G$. 
\end{notation}

\begin{lemme} \label{fibreUnsymp1}
La fibre de $\nu$ en un point de $U_N$ est isomorphe à 
$$\left\{
    \begin{array}{ll}
        G/H  &\text{ si } d <2n \text{ et $d$ est pair,} \\
        G    &\text{ si } d \geq 2n.
    \end{array}
\right.
$$
\end{lemme}

\begin{proof}
On rappelle que $u_1^N$, $u_2^N$ et $f^N$ ont été défini dans la notation \ref{BipBip}. On reprend ici des arguments similaires à ceux utilisés lors de la preuve du lemme \ref{orbfermee}.
Si $d <2n$ et $d$ est pair, alors $N=\frac{d}{2}$ et on vérifie que $\Stab_G((u_1^N,u_2^N))=H$. Or $H$ est un sous-groupe réductif de $G$, on a donc l'équivalence: 
$$ G.(u_1^N,u_2^N) \text{ est fermée dans } \mu^{-1}(0) \Leftrightarrow C_G(H).(u_1^N,u_2^N) \text{ est fermée dans } \mu^{-1}(0).$$ 
Puis $C_G(H)=\left\{ \begin{bmatrix}
M  &0 \\
0  &\lambda I_{n-N} \end{bmatrix},\ M\in GL_{N}(k),\ \lambda \in \Gm \right\}$. Donc 
$$C_G(H).(u_1^N,u_2^N)=\left\{ \left(\begin{bmatrix}
0   &M \\
0   &0 \end{bmatrix},\begin{bmatrix}
M^{-1} &0 \\
0     &0 \end{bmatrix} \right)\ ,\ M\in {GL}_{N}(k)\right\}$$ 
est un fermé de $\mu^{-1}(0)$, et donc $G.(u_1^N,u_2^N)$ est l'unique orbite fermée contenue dans $\nu^{-1}(f^N)$. 
Enfin $\dim(G/H)=n^2-(n-N)^2=N(2n-N)$ ce qui est aussi la dimension de la fibre générique de $\nu$. Donc $\nu^{-1}(f^N) \cong G/H$.\\
Si $d \geq 2n$, alors $N=n$ et a on vérifie que $\Stab_G((u_1^N,u_2^N))=Id$ et donc $\nu^{-1}(f^N)$ contient une unique orbite fermée isomorphe à $G$. Or $\dim(G)=n^2$ est la dimension de la fibre générique de $\nu$. Donc $\nu^{-1}(f^N) \cong G$. 
\end{proof}

\begin{remarque}
Si $d <2n$ et $d$ impair, alors la situation se complique car la fibre de $\nu$ en un point de l'orbite ouverte n'est plus irréductible. 
On considérera donc uniquement par la suite les cas $d \geq 2n$ et ($d <2n$ et $d$ pair). 
\end{remarque}

\begin{proposition} \label{fctH2}
La fonction de Hilbert $h_s$ de la fibre générique du morphisme de passage au quotient $\nu:\ \mu^{-1}(0) \rightarrow \mu^{-1}(0)/\!/G$ est donnée par:
$$\forall M \in \Irr(G),\ h_s(M)=\left\{
    \begin{array}{ll}
       \dim(M)  &\text{ si } d \geq 2n,  \\
       \dim(M^H)    &\text{ si } d <2n \text{ et $d$ est pair.} 
    \end{array}
\right.
$$
\end{proposition}

La fibre $\mu^{-1}(0)$ est un sous-schéma fermé $G' \times G$-stable de $W$, donc d'après \cite[Lemma 3.3]{Br}, le schéma $\HHm$ est un sous-schéma fermé $G'$-stable de $\Hilb_{h_s}^G(W)$. Ensuite, d'après les propositions \ref{fcthilb} et \ref{fctH2}, les fonctions de Hilbert $h_W$ et $h_s$ coïncident lorsque ($d \geq 2n$) ou ($d<n$ et $d$ pair). Donc, dans ces deux cas, $\HHm$ s'identifie à un sous-schéma fermé de $\Hilb_{h_W}^G(W)$ et ce dernier a été étudié dans la section \ref{GLngénéral}. 

%% file: symplectique_GLn_2.tex
\subsection{Réduction au cas classique pour la composante principale} \label{MropRRED}

\subsubsection{Construction d'un morphisme équivariant vers \texorpdfstring{$\Gr(N,E) \times \Gr(d-N,E)$}{des Gr}}  
On souhaite déterminer $\HHm$ en procédant comme dans la section \ref{GLngénéral}, c'est-à-dire en utilisant le principe de réduction pour se ramener à un schéma de Hilbert invariant plus simple. Dans la section \ref{princreduction}, on a construit un morphisme $G'$-équivariant 
\begin{equation} \label{rosymp}
\Hilb_{h_s}^G(W) \rightarrow \Gr(h_s(V),E) \times \Gr(d-h_s(V^*),E).
\end{equation}
D'après la proposition \ref{fctH2}, on a $h_s(V)=h_s(V^*)=N$, donc la restriction du morphisme (\ref{rosymp}) à $\HHm$ donne un morphisme $G'$-équivariant 
\begin{equation} \label{mmmorphisme_red}
\rho_s:\ \HHm \rightarrow \Gr(N,E) \times \Gr(d-N,E).
\end{equation} 
On note
$$O_i:=\{(L_1,L_2) \in \Gr(N,E) \times \Gr(d-N,E)\ |\ \dim(L_1 \cap L_2)=N-i\} \text{ pour } i=0, \ldots, N.$$
Alors les $O_i$ sont les $N+1$ orbites pour l'opération de $G'$ dans $\Gr(N,E) \times \Gr(d-N,E)$ et on a:
$$\overline{O_0} \subset \overline{O_1} \subset \cdots \subset \overline{O_N}=\Gr(N,E) \times \Gr(d-N,E).$$ 
En particulier, $O_N$ est l'unique orbite ouverte et 
$$O_0=\FF_{N,d-N}:=\{(L_1,L_2) \in \Gr(N,E) \times \Gr(d-N,E)\ |\ L_1 \subset L_2\},$$ 
qui est une variété de drapeaux partiels, est l'unique orbite fermée. La variété $\Gr(N,E) \times \Gr(d-N,E)$ n'étant pas $G'$-homogène, on ne peut pas ramener l'étude de $\HHm$ à l'étude de la fibre de $\rho_s$ en un point comme nous l'avons fait dans la section \ref{princreduction}. Cependant, nous allons voir que $\rho_s$ envoie la composante principale $\HHmp$ dans $O_0$. Nous allons donc pouvoir ramener l'étude de $\HHmp$ à celle de la fibre de $\rho_s$ en un point de $O_0$. La définition qui suit nous sera utile ultérieurement: 

\begin{definition}  \label{pulbacktauto}
On considère le diagramme suivant
$$\xymatrix{ & \FF_{N,d-N} \ar@{->>}[ld]_{p_1} \ar@{->>}[rd]^{p_2} \\   \Gr(N,E) && \Gr(d-N,E) }$$ 
où $p_1$ et $p_2$ sont les projections naturelles. 
On note $T_1$ (resp. $T_2$) le tiré en arrière du fibré tautologique de $\Gr(N,E)$ (resp. de $\Gr(d-N,E)$).    
\end{definition}


\subsubsection{Réduction au cas classique pour la composante principale de \texorpdfstring{$\HHm$}{Hs}}   

On suppose que ($d \geq 2n$) ou bien que ($d<2n$ et $d$ pair). Soient $x_0:=(L_1,L_2) \in O_0$ et $P:=\Stab_{G'}(x_0)$ le sous-groupe parabolique de $G'$ qui stabilise $x_0$. On note 
\begin{itemize} \renewcommand{\labelitemi}{$\bullet$}
\item $W':=\Hom(E/L_2,V) \times \Hom(V,L_1)$ qui est un $P \times G$-module qui s'identifie naturellement à un sous-espace vectoriel de $W$ contenu dans $\mu^{-1}(0)$,
\item $\nu':\ W' \rightarrow W'/\!/G$ le morphisme de passage au quotient,
\item $h_{W'}$ la fonction de Hilbert de la fibre générique de $\nu'$,
\item $\HH':=\Hilb_{h_{W'}}^{G}(W')$ et $\HH'^{\mathrm{prin}}$ sa composante principale.
\end{itemize} 
On remarque que $h_s=h_{W'}$. Le but de cette section est de démontrer la 

\begin{proposition} \label{reduction3}
On a un isomorphisme de $G'$-variétés
$$ \HHmp \cong G' {\times}^{P} \HH'^{\mathrm{prin}}.$$ 
\end{proposition}

\noindent Nous aurons besoin du

\begin{lemme}  \label{versX0}
Le morphisme $\rho_s$ envoie $\HHmp$ dans $O_0$.
\end{lemme}

\begin{proof}
La preuve est analogue à celle du lemme \ref{morphisme_dans_YSLn}.
\end{proof}

Le morphisme $\rho_s$ est $G'$-équivariant, il munit donc $\HHmp$ d'une structure de fibration $G'$-homogène. Soit $F_s$ la fibre schématique du morphisme $\rho_s$ en $x_0$. L'opération de $P$ dans $\HHmp$ induite par l'opération de $G'$ dans $\HHmp$ préserve $F_s$ par $G'$-équivariance. D'après (\ref{iissoo}), on dispose d'un isomorphisme $G'$-équivariant: 
$$\HHmp \cong G' {\times}^{P} F_s.$$ 
Donc, pour déterminer $\HHmp$, on est ramené à déterminer $F_s$ comme $P$-schéma. On commence par déterminer $F'_s$ la fibre schématique de ${\rho_s}:\ \HHm \rightarrow  \Gr(N,E) \times \Gr(d-N,E)$ en $x_0$ comme $P$-schéma.

\begin{lemme} \label{fibrehil2s}
La fibre $F'_s$ est isomorphe au schéma de Hilbert invariant $\HH'$ et l'opération de $P$ dans $F'_s$ coïncide, via cet isomorphisme, avec l'opération de $P$ dans $\HH'$ induite par l'opération de $P$ dans $W'$.
\end{lemme}

\begin{proof}
La preuve est analogue à celle du lemme \ref{fibrehil}.
\end{proof}

\noindent Comme $\HHmp$ est une variété de dimension $2N(d-N)$, on en déduit que $F_s$ est une variété de dimension $N^2$. 
Ensuite, d'après le lemme \ref{fibrehil2s}, la fibre $F_s$ est isomorphe à une sous-variété de $\HH'^{\mathrm{prin}}$. Or $\dim(\HH'^{\mathrm{prin}})=N^2$ donc on a un isomorphisme de $P$-variétés:
$$ F_s \cong \HH'^{\mathrm{prin}}$$
et la proposition \ref{reduction3} s'ensuit.

\begin{remarque}
Le schéma $\HH'$ est $P$-stable et s'identifie à un sous-schéma fermé de $\HHm$, donc on a l'inclusion de $G'$-schémas $G' \times^{P} \HH' \subset \HHm$. 
\end{remarque}

D'après la proposition \ref{reduction3}, la variété $\HHmp$ est lisse si et seulement si $\HH'^{\mathrm{prin}}$ est lisse, et alors $\gamma$ est une résolution de $\mu^{-1}(0)/\!/G$. On va utiliser les résultats de la section \ref{GLngénéral} pour démontrer le    

\begin{corollaire}  \label{symplisssse}
\begin{enumerate}
\item Si $d$ est pair et $d \leq n+1$, alors 
$$\HHmp \cong \Hom(\underline{E}/T,T)$$ 
où $T$ est le fibré tautologique de $\Gr(N,E)$ et $\underline{E}$ est le fibré trivial de fibre $E$ au dessus de $\Gr(N,E)$.
\item Si $n=1$ et $d \geq 3$, alors 
$$\HHmp \cong \Hom(\underline{E}/T_2,T_1)$$ 
où $T_1$, $T_2$ et $\underline{E}$ sont les fibrés au dessus de $\FF_{1,d-1}$ de la définition \ref{pulbacktauto}. 
\item Si $n=2$ et $d \geq 4$, alors 
$$\HHmp \cong Bl_0(\Hom(\underline{E}/T_2,T_1))$$ 
est l'éclatement de la section nulle du fibré $\Hom(\underline{E}/T_2,T_1))$, où $T_1$, $T_2$ et $\underline{E}$ sont les fibrés au dessus de $\FF_{2,d-2}$ de la définition \ref{pulbacktauto}. 
\end{enumerate}  
\noindent En particulier, dans chacun de ces cas la variété $\HHmp$ est lisse et donc le morphisme de Hilbert-Chow $\gamma$ est une résolution de $\mu^{-1}(0)/\!/G$.
\end{corollaire}

\begin{proof}
Les cas  ($d$ est pair et $d  \leq n+1$) et ($n=1$ et $d \geq 3$) découlent du corollaire \ref{cas_facile}. \\
Le cas $n=2$ et $d \geq 4$ découle du théorème \ref{casn2}.
\end{proof}

\begin{remarque}
Lorsque ($n=1$ et $d \geq 2$) ou ($d=2$ et $n \geq 2$), on a $\mu^{-1}(0)/\!/G=U_1 \cup \{0\}$. La proposition \ref{enoncepgene} s'applique et on retrouve les mêmes résultats que ceux donnés par le corollaire \ref{symplisssse}.  
\end{remarque}

\subsection{Réductibilité de \texorpdfstring{$\HHm$}{Hs}}  \label{reductibilité_cas_symp}

On suppose dans cette section que $d \geq 2n$, c'est-à-dire que $N=n$. Nous allons démontrer la 

\begin{proposition}  \label{noIrred}
Le schéma $\HHm$ est réductible. 
\end{proposition}

Soient
$$x_n:=( \left\langle  e_1, \ldots, e_n \right\rangle , \left\langle  e_{n+1}, \ldots, e_{d} \right\rangle) \in O_n$$ 
un point de l'orbite ouverte de $\Gr(n,E) \times \Gr(d-n,E)$ et 
$$W'':=\Hom(E/\left\langle  e_{n+1}, \ldots, e_d \right\rangle,V) \times \Hom(V,\left\langle  e_1, \ldots, e_n \right\rangle)$$ 
qui s'identifie naturellement à un $G$-sous-module de $W$. On a un isomorphisme de $G$-modules:
\begin{equation} \label{idenWW}
W'' \cong \Hom(\left\langle  e_1, \ldots, e_n \right\rangle,V) \times \Hom(V,\left\langle  e_1, \ldots, e_n \right\rangle).  
\end{equation}
On peut alors munir $W''$ de la forme symplectique $\Omega$ définie par \ref{defsymp} et considérer $\mu'$ l'application moment pour l'opération de $G$ dans $(W'',\Omega)$ comme nous l'avons fait pour $(W,\Omega)$ dans la section \ref{appzmoment}.\\ 
Le lemme qui suit se démontre comme le lemme \ref{fibrehil}:

\begin{lemme} \label{fibrehil3}
La fibre schématique $F_s''$ du morphisme ${\rho_s}:\ \HHm \rightarrow \Gr(n,E) \times \Gr(d-n,E)$ en $x_n$ est isomorphe au schéma de Hilbert invariant ${\Hilb}_{h_s}^{G} (\mu'^{-1}(0))$.
\end{lemme}

\begin{remarque}
La fonction de Hilbert $h_s$ qui apparaît dans le lemme \ref{fibrehil3} ne coïncide pas en général avec la fonction de Hilbert de la fibre générique du morphisme de passage au quotient $\mu'^{-1}(0) \rightarrow \mu'^{-1}(0)/\!/G$. 
\end{remarque}

D'après le lemme \ref{versX0}, le morphisme ${\rho_s}$ envoie la composante principale $\HHmp$ dans $O_0$. Donc, pour montrer la proposition \ref{noIrred}, il suffit de s'assurer que ${\Hilb}_{h_s}^{G} (\mu'^{-1}(0))$ est non vide. Pour ce faire, on identifie $W''$ avec $\Hom(E',V) \times \Hom(V,E')$ via l'isomorphisme (\ref{idenWW}), où l'on note $E':=\left\langle  e_1, \ldots, e_n \right\rangle$. Alors $W''$ est naturellement muni d'une structure de $GL(E') \times G$-module et on a 
\begin{align}  
{k[W'']}_2 &= (S^2(E') \otimes S^2(V^{*})) \oplus (S^2(V) \otimes S^2(E'^*)) \label{kW4} \\
       &\ \oplus  ({\Lambda}^2 (E') \otimes {\Lambda}^2 (V^{*})) \oplus  ({\Lambda}^2 (V) \otimes {\Lambda}^2 (E'^*)) \notag \\
       &\ \oplus ((sl(E') \oplus M_0) \otimes (sl(V) \oplus V_0)) \text{ comme $GL(E') \times G$-module,} \notag
\end{align} 
où $M_0$ désigne le $GL(E')$-module trivial.

\noindent Ensuite, soit $J_0$ l'idéal de $k[W'']$ engendré par $M_0 \otimes (sl(V) \oplus V_0) \subset k[W'']_2$. Alors $\mu'^{-1}(0)=\Spec(k[W'']/J_0)$ et donc on a une correspondance bijective, donnée par le morphisme de passage au quotient, entre les idéaux de $k[W'']$ contenant l'idéal $J_0$ et les idéaux de $k[\mu'^{-1}(0)] \cong k[W'']/J_0$: 
\begin{equation} \label{IcontientJ}
\begin{array}{ccccc}
\pi & : & k[W''] & \to & k[\mu'^{-1}(0)] \\
& & I_1 & \mapsto & I_2:=\pi(I_1). 
\end{array}
\end{equation}

\begin{notation}
On note $I_0$ l'idéal de $k[W'']$ engendré par $(sl(E') \otimes V_0) \oplus (M_0 \otimes V_0) \oplus (M_0 \otimes sl(V)) \subset {k[W'']}_2$. 
\end{notation}

L'idéal $I_0$ est homogène, $GL(E') \times G$-stable et contient l'idéal engendré par les $G$-invariants (resp. par les $GL(E')$-invariants) homogènes de degré positif de $k[W'']$. D'après ce qui précède, l'idéal $I_0$ s'identifie naturellement à un idéal de $k[\mu'^{-1}(0)]$ et on a la 

\begin{proposition}   \label{pppfixe}
$I_0 \in {\Hilb}_{h_s}^{G} (\mu'^{-1}(0)).$
\end{proposition}

\begin{proof}
Il suffit de vérifier que $I_0$ a la bonne fonction de Hilbert, c'est-à-dire que 
$$k[W'']/I_0 \cong \bigoplus_{M \in \Irr(G)}  M^{\oplus \dim(M)}$$ comme $G$-module. Des arguments similaires à ceux utilisés dans les preuves des propositions \ref{KSheadings} et \ref{decompoiso} nous donnent la décomposition de $k[W'']/I_0$ comme $GL(E') \times G$-module:
$$k[W'']/I_0 \cong \bigoplus_{\lambda \in \Lambda+} S^{\lambda}(E'^*) \otimes S^{\lambda}(V).$$
Comme $\dim(V)=\dim(E')$, on a $\dim(S^{\lambda}(V))=\dim(S^{\lambda}(E'^*))$ pour tout $\lambda \in \Lambda+$ et le résultat s'ensuit.
\end{proof}

\begin{remarque}
L'idéal $I_0$ est un point fixe de ${\Hilb}_{h_s}^{G}(\mu'^{-1}(0))$ pour l'opération de $GL(E')$. On en déduit que $\HHm$ admet une composante irréductible distincte de $\HHmp$ de dimension supérieure ou égale à $\dim(O_n)=2n(d-n)$.
\end{remarque}



\subsection{Etude du cas \texorpdfstring{$\dim(V)=1$}{}} \label{GLsympN1}

Dans cette section, on fixe $n=1$ et $d \geq 2$. Alors $G=\Gm$ est le groupe multiplicatif, $\mu^{-1}(0)/\!/G=\overline{\OO_{[2,1^{d-2}]}}$ et $\rho_s:\ \HHm \rightarrow \PP(E) \times \Gr(d-1,E)$ est le morphisme (\ref{mmmorphisme_red}). On a vu que $\HHmp$ est une variété lisse (corollaire \ref{symplisssse}). Nous allons utiliser les résultats des sections précédentes et de la section \ref{section_n_1} pour déterminer $\HHm$ comme $G'$-schéma. On  a la 

\begin{proposition}  \label{casSympn1}
On a un isomorphisme $G'$-équivariant
$$\HHm \cong \left \{(f,L) \in \overline{\OO_{[2,1^{d-2}]}} \times \PP (\gg'^{\leq 1}) \ \mid  \ f \in L  \right \}.$$
Le schéma $\HHm$ est donc la réunion de deux composantes irréductibles lisses $C_1$ et $C_2$ de dimension $2d-2$:
\begin{itemize} \renewcommand{\labelitemi}{$\bullet$}
\item $C_1:=\left \{(f,L) \in \overline{\OO_{[2,1^{d-2}]}} \times \PP (\overline{\OO_{[2,1^{d-2}]}}) \ \mid  \ f \in L \right \}=\HHmp$ et $\gamma:\ \HHmp \rightarrow \overline{\OO_{[2,1^{d-2}]}}$ est l'éclatement en $0$ de $\overline{\OO_{[2,1^{d-2}]}}$, 
\item $C_2:=\left \{(0,L) \in \overline{\OO_{[2,1^{d-2}]}} \times \PP ({\gg'}^{\leq 1}) \right \} \cong  \PP ({\gg'}^{\leq 1}).$
\end{itemize}
Et l'intersection de ces deux composantes irréductibles est: \\
$C_1 \cap C_2=\left \{(0,L) \in \overline{\OO_{[2,1^{d-2}]}} \times \PP (\overline{\OO_{[2,1^{d-2}]}}) \right \} \cong \PP (\overline{\OO_{[2,1^{d-2}]}})$.
\end{proposition}

\begin{proof}
D'après la proposition \ref{Hcasn11111}, on a une immersion fermée
$$\gamma \times \rho_s:\ \HHm \hookrightarrow Y:=\left \{(f,L) \in \overline{\OO_{[2,1^{d-2}]}} \times \PP ({\gg'}^{\leq 1}) \ \mid  \ f \in L  \right \}.$$ 
On vérifie que $Y$ est la réunion des deux fermés irréductibles $C_1$ et $C_2$ qui sont tout deux de dimension $2d-2$. Le morphisme $\gamma \times \rho_s$ envoie $\HHmp$ dans $C_1$; ce sont deux variétés de même dimension, donc $\gamma \times \rho_s:\ \HHmp \rightarrow C_1$ est un isomorphisme. Puis, on a vu dans la section \ref{reductibilité_cas_symp} que $\HHm$ admet une deuxième composante irréductible de dimension au moins $2d-2$, ce qui est la dimension de $C_2$ et donc $\gamma \times \rho_s$ est un isomorphisme entre cette seconde composante de $\HH$ et $C_2$. Enfin, on vérifie que $C_1 \cap C_2$ est bien ce qui est annoncé dans la proposition.
\end{proof}

\begin{remarque}
On vérifie que la composante $C_2$ de $\HHm$ est formée des idéaux homogènes de $k[\mu^{-1}(0)]$.
\end{remarque}

Lorsque $n \geq 2$ et $d \geq 2n$, les choses se compliquent car des composantes irréductibles de $\HHm$ de grande dimension peuvent apparaître. Par exemple, pour $n=2$ et $d \geq 4$, on vérifie que la composante formée des idéaux homogènes de $k[\mu^{-1}(0)]$ est de dimension $4d-5$, alors que la composante principale est de dimension $4d-8$. De plus, on a montré l'existence d'au moins deux composantes irréductibles pour $\HHm$, mais il peut exister encore d'autres composantes irréductibles a priori.

%% file: position_probleme_On.tex
\chapter{Cas des autres groupes classiques} \label{chp3}

\section{Cas de \texorpdfstring{$O(V)$}{O(V)} opérant dans \texorpdfstring{$V^{\oplus  n'}$}{n'V}} \label{HclassOngeneral}

On se place dans la situation $3$: on a $G:=O(V)$, $G':=GL(V')$, $W:={\Hom}(V',V)$ et l'opération de $G' \times G$ dans $W$ est donnée par (\ref{actionSLLn}).

\subsection{Etude du morphisme de passage au quotient} \label{description_quotientOn}

Les résultats essentiels de  cette section sont les propositions \ref{descriptiongeofibOn} et \ref{ouvertplatitudeOn} qui décrivent les fibres et l'ouvert de platitude de $\nu$. On suit dans cette section le même cheminement que dans la section \ref{description_quotient}.

D'après le premier théorème fondamental pour $O(V)$ (voir \cite[§11.2.1]{Pro}) l'algèbre des invariants $k[W]^G$ est engendrée par les $(i\ |\ j)$, où pour chaque couple $(i,j)$, $\ 1\leq i \leq j \leq n'$, on définit la forme bilinéaire symétrique $(i\ |\ j)$ sur $W \cong V^{\oplus n'}$ par: 
\begin{equation} \label{Oninvariants}
\forall v_1,\ldots,v_{n'}\in V,\ (i\ |\ j):(v_1,\ldots,v_{n'}) \mapsto \phi(v_i,v_j) 
\end{equation} 
où l'on note $\phi$ la forme polaire associée à la forme quadratique $q$ définie par (\ref{defQ}).
On a le morphisme naturel $G' \times G$-équivariant
\begin{equation}  \label{appznatO}
\Hom(V',V)    \rightarrow  \Hom(S^2(V'),S^2(V)), w \mapsto S^2(w).
\end{equation}
D'après \cite[§19.5]{FH}, il existe un $G$-module irréductible $M$ (différent de $V_0$) tel que $S^2(V) \cong V_0 \oplus M$ comme $G$-module, et la représentation triviale $V_0$ est engendrée par la forme quadratique $q$. Le morphisme de passage au quotient $\nu$ est obtenu en composant le morphisme (\ref{appznatO}) et le morphisme $G$-invariant 
$$  \Hom(S^2(V'),S^2(V))  \rightarrow   \Hom(S^2(V'),V_0) \cong S^2(V'^*)$$ 
induit par la projection $S^2(V) \rightarrow V_0$. On a donc
$$\begin{array}{lccl}
\nu:  &\Hom(V',V)   & \rightarrow  & S^2(V'^*) \\
        & w  & \mapsto      &  \leftexp{t}{w}  w 
\end{array}$$ 
où l'on note $\leftexp{t}{w}$ la transposée du morphisme $w$.  \\
Et donc 
$$W/\!/G=S^2(V'^*)^{\leq n}:=\left\{ Q \in  S^2(V'^*) \  \mid \ \rg(Q) \leq n \right\} $$
est une variété déterminantielle symétrique.\\
Si $n' \leq n$, alors $W/\!/G=S^2(V'^*)$ est un espace affine. Sinon, c'est une variété normale (\cite[§3.2, Théorème 2]{SB}), de dimension $n'n-\frac{1}{2}n(n-1)$, de Cohen-Macaulay (\cite[§3.4, Théorème 4]{SB}) et singulière le long du fermé $S^2(V'^*)^{\leq n-1}$ (\cite[§6.3]{Wey}). De plus, $W/\!/G$ est de Gorenstein si et seulement si $n'-n$ est impair (\cite[Corollary 6.3.7]{Wey}).\\
On pose 
$$N:=\min(n',n).$$
L'opération de $G'$ dans $W$ induit une opération dans $W/\!/G$ telle que $\nu$ soit $G'$-équivariant: 
$$\forall Q \in W/\!/G \subset S^2(V'^*),\ \forall g' \in G',\ g'.Q=\leftexp{t}{g'^{-1}}  Q g'^{-1} . $$
Pour cette opération, $W/\!/G$ se décompose en $N+1$ orbites
$$U_i:=\left\{  Q \in  S^2(V'^*) \  \mid \ \rg(Q) = i \right\}$$
pour $i=0,\ldots,N$. Les adhérences de ces orbites sont imbriquées de la façon suivante: 
$$ \{0\}=\overline{U_0} \subset \overline{U_1} \subset \cdots \subset \overline{U_N}=W/\!/G .$$ 
En effet, pour chaque $i=0,\ldots,N$, on a $\overline{U_i}=S^2(V'^*)^{\leq i}$. En particulier, l'orbite $U_N$ est un ouvert dense de $W/\!/G$.

On rappelle que l'on note $\NNN(W,G)$ le nilcône, c'est-à-dire la fibre schématique en $0$ du morphisme $\nu$. Dans \cite{KS}, Kraft et Schwarz ont montré certaines propriétés géométriques du nilcône: il est irréductible si $n'<\frac{n}{2}$ et il est réduit si et seulement si $n' \leq \frac{n}{2}$. Pour simplifier, on considère dorénavant $\NNN(W,G)$ muni de sa structure réduite. Nous allons déterminer les composantes irréductibles de $\NNN(W,G)$ ainsi que leurs dimensions.\\
Soit $m \in \NN$, on note $\OG(m,V)$ le schéma (projectif) des sous-espaces isotropes pour la forme quadratique $q$ et de dimension $m$ dans $V$ lorsqu'elle existe. On rappelle que $q$ est non dégénérée, donc la dimension de n'importe quel sous-espace de $V$ totalement isotrope maximal pour l'inclusion est $E(\frac{n}{2})$. Il s'ensuit que le schéma $\OG(m,V)$ existe si et seulement si $m \in \{0,\ldots,E(\frac{n}{2})\}$, et dans ce cas 
$$\dim \left(\OG\left(m,V\right) \right)=m(n-m)-\frac{1}{2}m(m+1) . $$ 
Le schéma $\OG(m,V)$ est réduit et homogène sous $G$. Il est irréductible, sauf lorsque $n$ est pair et $m=\frac{n}{2}$, auquel cas $\OG(m,V)$ est la réunion de deux composantes irréductibles isomorphes, chacune homogène pour l'opération de $SO(V)$ mais échangées par n'importe quel élément de $O(V) \backslash SO(V)$.  

\begin{proposition} \label{nilcôneOn}
\begin{itemize}
\item Si $n$ est pair et $n' \geq \frac{n}{2}$ alors $\NNN(W,G)$ est la réunion de deux variétés, chacune de dimension $\frac{1}{2} n \, n' + \frac{1}{8} n^{2}-\frac{1}{4}n$.  
\item Si $n$ est impair et $n' \geq \frac{n-1}{2}$ alors $\NNN(W,G)$ est une variété de dimension $\frac{1}{2} n' (n-1) + \frac{1}{8} n^{2}-\frac{1}{8}$. 
\item Si $n' < \frac{n}{2}$, alors $\NNN(W,G)$ est une variété de dimension $n n' - \frac{1}{2} n' (n'+1)$.
\end{itemize}
\end{proposition}   

\begin{proof}
On a 
$$\NNN(W,G)=\left\{ w \in \Hom(V',V) \  \mid \ \leftexp{t}{w} w=0 \right\} .$$ 
Or
\begin{align*}
\leftexp{t}{w}  w =0 &\Leftrightarrow \forall x,y \in V,\ \leftexp{t}{x} \leftexp{t}{w} w y=0 \\
 &\Leftrightarrow \Im(w) \text{ est un sous-espace vectoriel isotrope de } V . 
\end{align*} 
On pose $\alpha:=\min\left(n',E(\frac{n}{2})\right)$ et soit
$$\ZZ:=\left\{(w,L) \in \Hom(V',V) \times \OG(\alpha,V) \  \mid \ \Im(w) \subset L\right\}$$ 
alors on a le diagramme suivant:
$$\xymatrix{ &  \ZZ \ar@{->>}[ld]_{p_1} \ar@{->>}[rd]^{p_2} \\  \NNN(W,G)  && \OG(\alpha,V) }$$ 
où les $p_i$ sont les projections naturelles. 
On fixe $L_0 \in \OG(\alpha,V)$, alors $\ZZ$ est un fibré vectoriel $G$-homogène au dessus de $\OG(\alpha,V)$ de fibre en $L_0$ isomorphe à $\Hom(V',L_0)$. Et donc
\begin{align*}
\dim(\ZZ) &=\dim \left(\OG\left(\alpha,V\right)\right)+\dim\left(\Hom \left(V',L_0\right) \right)\\
       &=\alpha(n-\alpha)-\frac{1}{2}\alpha(\alpha+1)+\alpha n'\\
       &=\alpha(n'+n-\alpha)-\frac{1}{2}\alpha(\alpha+1) .
\end{align*}
Soient $\ZZ':=\left\{(w,L) \in \ZZ\  \mid \ \rg(w)=\alpha \right\}$ et $X:=\{w \in \NNN(W,G)\ \mid \ \rg(w)=\alpha \}$ qui sont des ouverts denses de $\ZZ$ et $\NNN(W,G)$ respectivement. Alors $p_1:\ \ZZ' \rightarrow X$ est un isomorphisme. Si $n$ est impair ou $n'<\frac{n}{2}$, alors $\ZZ$ est irréductible et donc $\NNN(W,G)$ est irréductible. En revanche, si $n$ est pair et $n' \geq \frac{n}{2}$, alors $\ZZ$ est la réunion de deux variétés de même dimension et il en va de même pour $\NNN(W,G)$. Dans tous les cas, on a $\dim(\NNN(W,G))=\dim(\ZZ)$ et le résultat s'ensuit.   
\end{proof}

Nous allons maintenant nous intéresser à la description géométrique des fibres de $\nu$ au dessus de chaque orbite $U_i$.

\begin{notation} \label{defjOn}
Soit $ \ 0 \leq r \leq N$, on note 
$$J_r =\begin{bmatrix}
I_r \ \ &0_{r,n'-r} \\
0_{n'-r,r}  \ &0_{n'-r,n'-r} 
\end{bmatrix}$$
où $I_r$ est la matrice identité de taille $r$. La matrice $J_r$ est symétrique de rang $r$ et donc s'identifie naturellement à un élément de $U_r$. 
\end{notation}

\noindent On fixe $r \in \{0, \ldots,N\}$ et on définit $w_r:= \begin{bmatrix}
I_r &0_{r,n'-r} \\
0_{n-r,r}   &0_{n-r,n'-r} \end{bmatrix} \in W$ et $G_r$ le stabilisateur de $w_r$ dans $G$. On vérifie que
$$G_r =\left\{ \begin{bmatrix}
I_r   &0 \\
0     &M \end{bmatrix},\ M\in O_{n-r}(k) \right\} \cong O_{n-r}(k)$$  
et $V=r V_0 \oplus E_r$ comme $G_r$-module, où $E_r$ désigne la représentation standard de $G_r$ et $V_0$ la représentation triviale de $G_r$. On note $N_{w_r}$ la représentation slice de $G_r$ en $w_r$ (voir la définition \ref{slice}), alors on a le

\begin{lemme} \label{slice_expliciteOn}
On a un isomorphisme de $G_r$-modules 
$$ N_{w_r} \cong (n'-r) E_r \oplus r(n'-\frac{1}{2}(r-1)) V_0 .$$
\end{lemme}

\begin{proof}
La preuve est analogue à celle du lemme \ref{slice_explicite}.
\end{proof}

Soient $F_1$ et $F_2$ des espaces vectoriels de dimensions $n'-r$ et $r(n'-\frac{1}{2}(r-1))$ respectivement et dans lesquels $G_r$ opère trivialement. D'après le lemme \ref{slice_expliciteOn}, on a un isomorphisme de $G_r$-modules 
$$ N_{w_r} \cong \Hom(F_1, E_r) \times F_2  $$ 
et le morphisme de passage au quotient $\nu_N:\ N_{w_r} \rightarrow N_{w_r}/\!/G_r$ est donné par:
$$\begin{array}{lrcl}
 \nu_N:\  &\Hom(F_1, E_r) \times F_2 & \rightarrow  & S^2(F_1^*) \times F_2 . \\
        & (w,x)  & \mapsto      &  ( \leftexp{t}{w} w , x)  
\end{array}$$ 
Donc $\NNN(N_{w_r},G_r):= \nu_{N}^{-1}(\nu_N(0))=\nu_{N}^{-1}(0) \cong {\nu'}_{N}^{-1}(0)$, comme schéma, avec
 $$\begin{array}{lrcl}
 {\nu'_N}:\  &\Hom(F_1, E_r)  & \rightarrow  & S^2(F_1^*)  .\\
        & w  & \mapsto     &  \leftexp{t}{w}w
\end{array}$$ 

On en déduit la

\begin{proposition} \label{descriptiongeofibOn}
Avec les notations précédentes, on a un isomorphisme $G$-équivariant
$$\nu^{-1}(J_r) \cong G \times^{G_r} {\nu'}_{N}^{-1}(0) .$$ 
En particulier, si l'on note $H:=G_N$, on a 
$$\nu^{-1}(J_N) \cong \left\{
    \begin{array}{ll}
        G &\text{ si } n' \geq n, \\
        G/H &\text{ si } n'<n.
    \end{array}
\right.$$ 
\end{proposition}

\begin{proof}
La preuve est analogue à celle de la proposition \ref{descriptiongeofib}.
\end{proof}

\begin{corollaire} \label{dimfibreOn}
Soit $r \in \{0, \ldots,N\}$, alors la dimension de la fibre du morphisme $\nu$ en $J_r$ vaut: 
\begin{itemize} \renewcommand{\labelitemi}{$\bullet$}
\item $n' n-\frac{1}{2}n'(n'+1)$ lorsque $2 n'-r < n$,
\item $\frac{1}{2} n' (n-r)+\frac{1}{8}{(r+n)}^{2}-\frac{1}{4}(n+r)$  lorsque $2 n'-r \geq n$ et $n-r$ est pair,
\item $\frac{1}{2} n' (n-r-1)+\frac{1}{8}{(r+n)}^{2}-\frac{1}{8}$ lorsque $2 n'-r \geq n$ et $n-r$ est impair.
\end{itemize}
\end{corollaire}

Pour chaque couple $(n,n')$, le corollaire \ref{dimfibreOn} permet d'une part de calculer la dimension de la fibre générique de $\nu$, d'autre part de déterminer l'ouvert de platitude de $\nu$. En procédant comme pour la proposition \ref{ouvertplatitude}, on montre la

\begin{proposition} \label{ouvertplatitudeOn}
La dimension de la fibre générique et l'ouvert de platitude de $\nu$ sont donnés par le tableau suivant:\\
\begin{center}
\begin{tabular}{|c|c|c|}
  \hline
  configuration & dim. de la fibre générique & ouvert de platitude \\
  \hline
  $n' < n$ &     $n' n -\frac{1}{2}n'(n'+1)$ & $U_{n'} \cup \cdots \cup U_{\max(2 n'-n-1, 0)}$\\
  $n' = n$ &     $\frac{1}{2}n(n-1)$       & $U_{n} \cup U_{n-1} $\\ 
  $n' > n$ &     $\frac{1}{2}n(n-1)$       & $U_{n} $\\ 
  \hline    
\end{tabular}
\end{center}
\vspace*{0.5mm}
\end{proposition}

\begin{corollaire} \label{ouverttoutplatOn}
Le morphisme $\nu$ est plat sur $W/\!/G$ tout entier si et seulement si $n \geq 2n'-1$ et dans ce cas $W/\!/G=S^2(V'^*)$. 
\end{corollaire}

\noindent Le corollaire qui suit est une conséquence de la proposition \ref{chow} et du corollaire \ref{ouverttoutplatOn}.

\begin{corollaire} \label{cas_facileOn}
Si $n \geq 2n'-1$, alors $\HH \cong S^2(V'^*)$ et $\gamma$ est un isomorphisme.
\end{corollaire}

On s'intéresse, dans la proposition qui suit (et qui se démontre comme la proposition \ref{fcthilb}), à la fonction de Hilbert de la fibre générique de $\nu$. On note comme précédement $H:=G_N \cong O_{n-N}(k)$ le stabilisateur de $w_N$ dans $G$.

\begin{proposition} \label{fcthilbOn}
La fonction de Hilbert de la fibre générique du morphisme $\nu$ est donnée par:\\
$$\forall M \in \Irr(G),\ h_W(M)= \left\{
    \begin{array}{ll}
        \dim(M) &\text{ si } n' \geq n, \\
        \dim(M^{H}) &\text{ si } n'<n. 
    \end{array}
\right.$$
\end{proposition}

\begin{remarque}
Si $n=1$, alors $G \cong \ZZZ_2$ qui est le groupe d'ordre $2$. Les mêmes arguments que ceux utilisés dans la preuve du théorème \ref{Hcasn11111} permettent de montrer que $\HH=\HHp$ est une variété lisse isomorphe à l'éclatement en $0$ de $W/\!/G$ (mais cette méthode revient ici à utiliser un marteau pour écraser une mouche).
\end{remarque}

%% file: O2.tex
\subsection{Etude du cas \texorpdfstring{$\dim(V)=2$}{n=2}}   \label{casO2}

Dans toute cette section, on fixe $n=2$. Alors on a $G \cong O_2(k)$, $W/\!/G=S^2(V'^*)^{\leq 2}$ et $\rho:\ \HH \rightarrow \Gr(2,V'^*)$ est le morphisme de la section \ref{redGrass2}. On note 
$$Y_0:=\left\{(Q,L)\in W/\!/G \times \PP (W/\!/G)\ \mid \ Q \in L \right\}=\OO_{\PP(W/\!/G)}(-1)$$ 
l'éclatement de $W/\!/G$ en l'origine et $Y_1$ l'éclatement de $Y_0$ le long de la transformée stricte de $S^2(V'^*)^{\leq 1}$. Nous verrons que $Y_1$ est isomorphe à l'éclatement de la section nulle du fibré $S^2\left( T \right)$ au dessus de $\Gr(2,V'^*) $ où $T$ désigne le fibré tautologique de $\Gr(2,V'^*)$.    

Nous allons démontrer le

\begin{theoreme} \label{casn2O2}
\begin{itemize} 
\item Si $n'=2$, alors $\HH \cong Y_0$ et $\gamma$ est l'éclatement de $W/\!/G$ en l'origine. 
\item Si $n'> 2$, alors $\HH \cong Y_1$ et $\gamma$ est une résolution des singularités de $W/\!/G$.
\end{itemize}
En particulier $\HH$ est toujours une variété lisse.
\end{theoreme}

\begin{remarque}
Si $n'=1$, alors $\HH \cong W/\!/G$ et $\gamma$ est un isomorphisme d'après le corollaire \ref{cas_facileOn}.
\end{remarque}

La stratégie adoptée pour démontrer le théorème \ref{casn2O2} est la même que celle utilisée pour démontrer le théorème \ref{casn2}. On commence par établir le cas particulier $n'=2$, puis on établit le cas $n'>2$ à la fin de la section à l'aide du principe de réduction (voir la section \ref{princreduction}) en se ramenant au cas $n'=2$.

\subsubsection{Représentations de $O_2(k)$}

Des rappels concernant la théorie des représentations de $SO_n(k)$ et $O_n(k)$ ont été effectués dans la section \ref{rappelsthedesreps}. Cependant cette théorie est particulièrement simple pour $n=2$. En effet, d'après \cite[Exercise 18.2]{FH}, on a un isomorphisme de groupes algébriques $SO_2(k) \cong \Gm$ et donc d'après (\ref{OnSonScinde}), on a $O_2(k) \cong \Gm \ltimes \ZZZ_2$. Les $\Gm$-modules irréductibles sont tous de dimension $1$ et si $M$ est un $\Gm$-module irréductible, il existe un unique $d \in \ZZZ$ tel que:
$$ \forall t \in \Gm,\ \forall m \in M,\ t.m=t^d m.$$
On note alors $M=\Gamma_{d \epsilon_1}$. Les $\Gm$-modules irréductibles sont donc paramétrés par les entiers relatifs: 
$$d \in \ZZZ \leftrightarrow \Gamma_{d \epsilon_1}$$
et les $O_2(k)$-modules irréductibles sont:
\begin{itemize}
\item le module trivial $\Gamma_{0}$,
\item le module signe noté $\epsilon$,
\item les modules $\Gamma_{i \epsilon_1} \oplus \Gamma_{-i \epsilon_1}$, où $i >0$, qui sont tous de dimension $2$. 
\end{itemize}
En particulier, la représentation standard de $O_2(k)$ (et de $SO_2(k)$) est isomorphe à $\Gamma_{- \epsilon_1} \oplus \Gamma_{\epsilon_1}$.\\ 
Ensuite, par définition de la forme quadratique $q$, on a $Mat_{\BB}(q)=\begin{bmatrix}
  1     & 0 \\          
  0    & 1   
  \end{bmatrix} $
où $\BB$ est la base de $V$ fixée dans la section \ref{lesdiffsituations}. 
On va effectuer un changement de base dans $V$. Soit $\BB_0:=\{v_1, v_2\}$ la base de $V$ telle que  $Mat_{\BB_0}(q)=\begin{bmatrix}          
  0     & 1 \\          
  1    & 0   
  \end{bmatrix}$. 
Dans la base $\BB_0$, on a 
  $$SO(V) \cong \left\{ \begin{bmatrix}           
  \alpha     & 0  \\          
  0    & {\alpha}^{-1}   \end{bmatrix},\ \alpha \in \Gm \right\} \cong \Gm$$ 
et $\Gamma_{\epsilon_1} \cong  \left\langle v_1 \right\rangle,\  \Gamma_{- \epsilon_1} \cong  \left\langle v_2 \right\rangle$. Dans toute la section \ref{casO2}, on travaillera dans les bases $\BB_0$ et $\BB'$ de $V$ et $V'$ respectivement. Nous aurons besoin par la suite d'expliciter des bases de certains $B' \times G$-sous-module de $k[W]$ et ce changement de base va nous permettre d'avoir des générateurs "simples" pour ces modules.

\subsubsection{Points fixes de \texorpdfstring{$\HH$}{H} pour l'opération de \texorpdfstring{$B'$}{B'}}   \label{pointsfixesO2}
On suppose pour le moment que $n'=2$.
Nous allons déterminer les points fixes de $B'$ dans $\HH$. 
On commence par écrire la décomposition de ${k[W]}_2$ comme $G' \times G$-module:
\begin{equation} \label{kW2O2}
{k[W]}_2 \cong S^2(W^*) \cong (S^2(V') \otimes (\Gamma_{2 \epsilon_1} \oplus \Gamma_{-2 \epsilon_1} \oplus \Gamma_0)) \oplus ({\Lambda}^{2}(V') \otimes \epsilon).
\end{equation}

\begin{notation} \label{l3generators2}
On note:
\begin{itemize} \renewcommand{\labelitemi}{$\bullet$}
\item $J$ l'idéal engendré par les $G$-invariants homogènes de degré positif de $k[W]$,
\item $D:=<e_1.e_1>$ l'unique droite $B'$-stable de $S^2(V')$,
\item $I$ l'idéal de $k[W]$ engendré par $(S^2(V')\otimes \Gamma_0) \oplus (D \otimes (\Gamma_{2 \epsilon_1} \oplus \Gamma_{-2 \epsilon_1})) \subset {k[W]}_2$.
\end{itemize}
\end{notation}

\begin{remarque}
On a $J \cap k[W]_2=S^2(V') \otimes \Gamma_0$ comme $G' \times G$-module et ce module engendre l'idéal $J$.
L'idéal $I$ est homogène, $B' \times G$-stable et contient l'idéal $J$.
\end{remarque}

\begin{theoreme} \label{pointfixeborelO2}
L'idéal $I$ est l'unique point fixe de $\HH$ pour l'opération de $B'$.
\end{theoreme}

\begin{proof}
On raisonne comme dans la preuve du théorème \ref{pointfixeborel} en considérant $I_Z$ un point fixe de $B'$ et en étudiant $k[W]/I_Z$ composante par composante:\\
$\bullet$ Composantes de degré $0$ et $1$:\\
On a bien sûr $I_Z \cap {k[W]}_0=\{0\}$ et $I_Z \cap {k[W]}_1 \neq {k[W]}_1$.\\  
$\bullet$ Composante de degré $2$: on utilise la décomposition (\ref{kW2O2}).\\
Pour avoir la décomposition souhaitée de $k[W]/I_Z$ comme $G$-module, on a nécessairement ${k[W]}_2 \cap I_Z \supseteq 3 \Gamma_0 \oplus  (\Gamma_{2 \epsilon_1} \oplus \Gamma_{-2 \epsilon_1})$. 
En effet, le $G$-module $k[W]/I_Z$ contient déjà une copie de la représentation triviale (qui provient de la composante de degré $0$), il ne peut donc pas en contenir d'autre. Ensuite, ${k[W]}_2$ contient $3$ copies de $\Gamma_{2 \epsilon_1} \oplus \Gamma_{-2 \epsilon_1}$ qui est un $G$-module de dimension $2$, donc ${k[W]}_2 \cap I_Z$ contient au moins une copie de $\Gamma_{2 \epsilon_1} \oplus \Gamma_{-2 \epsilon_1}$. Comme $k[W]_2 \cap I_Z$ est $B'$-stable, il contient $D \otimes ((\Gamma_{2 \epsilon_1} \oplus \Gamma_{-2 \epsilon_1}))$ car $D$ est l'unique droite $B'$-stable de $S^2(V')$. Il s'ensuit que $I_Z$ contient $(S^2(V') \otimes V_0) \oplus (D \otimes ((\Gamma_{2 \epsilon_1} \oplus \Gamma_{-2 \epsilon_1})))$ et donc $I_Z \supset I$.
Le lemme qui suit implique que cette inclusion est en fait une égalité et achève ainsi la démonstration du théorème \ref{pointfixeborelO2}:

\begin{lemme} \label{l3generatorsO2}
L'idéal $I$ a pour fonction de Hilbert $h_W$.
\end{lemme}

\noindent \begin{itshape} \textbf{Preuve du lemme:} \end{itshape}
Il faut montrer que 
$$k[W]/I \cong \bigoplus_{M \in \Irr(G)} M^{\oplus \dim(M)}$$  
comme $G$-module. On a l'inclusion d'idéaux $J \subset I$, d'où l'isomorphisme de $B' \times G$-modules 
\begin{equation} \label{isoquotientIJ} 
k[W]/I \cong \frac{k[W]/J}{I/J}.
\end{equation} 
On commence par déterminer la décomposition de $k[W]/J$ en $G$-modules irréductibles. On note les éléments $w \in W$ sous forme matricielle via le choix des bases $\BB_0$ et $\BB'$: $w=\begin{bmatrix}           
  x_1     & y_1 \\          
  x_2    & y_2   \end{bmatrix}$ et on identifie $k[W]$ à $k[x_1,x_2,y_1,y_2]$.\\
Alors
$$J=(x_1 x_2, y_1 y_2, x_1 y_2+x_2 y_1)$$
et $k[W]/J$ a pour base les images des monômes
$$\{x_1^p y_1^q, x_2^p y_2^q, x_1 y_2;\ p,q \geq 0\} \subset k[W]$$
dans $k[W]/J$. On a donc une décomposition de $k[W]/J$ en droites $SO(V)$-stables. On en déduit la décomposition de $k[W]/J$ en $G$-modules irréductibles:
$$\begin{tabular}{|c|c|c|}
  \hline
  représentation & droite(s) des vecteurs de plus haut poids & multiplicité dans $k[W]/J$ \\
  \hline
  $\Gamma_0$ &              $ \left\langle  1 \right\rangle$  & $1$\\
  $\epsilon$ &            $ \left\langle x_1 y_2  \right\rangle$ & $1$\\  
  $\Gamma_{j \epsilon_1} \oplus \Gamma_{-j \epsilon_1}$  & $ \left\langle x_1^{p} y_1^{j-p}  \right\rangle,\ j\geq 1,\ 0 \leq p \leq j$ & $j+1$\\  
  \hline    
\end{tabular}$$

\noindent Ensuite, on a $I=J + (x_1^2,x_2^2)$ donc l'idéal $I/J$ est engendré par les images des éléments $x_1^2$ et $x_2^2$ modulo $J$. Donc, l'image du monôme $x_1^{p} y_1^{j-p}$ dans $k[W]/J$ est dans $I/J$ si et seulement si $j \geq 2$ et $p \geq 2$. On en déduit, via l'isomorphisme (\ref{isoquotientIJ}), que chaque $G$-module irréductible apparaît avec une multiplicité égale à sa dimension dans $k[W]/I$. 
\end{proof}

\begin{remarque}
On a $\Stab_{G'}(I)=B'$, donc l'unique orbite fermée de $\HH$ est isomorphe à $G'/B' \cong \PP^1$.
\end{remarque}

\noindent Le corollaire qui suit découle du lemme \ref{fixespoints} et du théorème \ref{pointfixeborelO2}:

\begin{corollaire} \label{HconnexeO2}
Le schéma $\HH$ est connexe.
\end{corollaire}

\subsubsection{Espace tangent de \texorpdfstring{$\HH$}{H} en \texorpdfstring{$Z_0$}{Z0}} \label{espacTangO2}

\noindent On note $Z_0:=\Spec(k[W]/I)$. Nous allons démontrer la  

\begin{proposition} \label{dimTangentO2}
$\dim(T_{Z_0} \HH)=3 .$
\end{proposition}

On identifie $k[W]$ à $k[x_1,x_2,y_1,y_2]$ comme dans la preuve du lemme \ref{l3generatorsO2} et on explicite des bases de certains $B' \times G$-modules qui apparaissent dans $k[W]_2$:\\

$
 \left.
    \begin{array}{l}
        f_1:=x_1 x_2 \\
        f_2:=y_1 y_2 \\
        f_3:=x_1 y_2+ x_2 y_1 
      \end{array}
\right \} \text{ est une base de } S^2(V') \otimes \Gamma_0, 
$

$
 \left.
    \begin{array}{l}
        h_1:=x_1^2 \\
        h_2:=x_2^2 \\ 
        h_3:=y_1^2 \\
        h_4:=y_2^2 \\
        h_5:=x_1 y_1 \\
        h_6:=x_2 y_2 
      \end{array}
\right \} \text{ est une base de } S^2(V') \otimes (\Gamma_{2 \epsilon_1} \oplus \Gamma_{-2 \epsilon_1}). 
$\\

On reprend les notations de la section \ref{ConnexitéetTangence}. Soit $R:=k[W]/I$ et soit 
$$N:= \left\langle  f_1,f_2,f_3,h_1,h_2 \right\rangle \subset k[W]$$ 
qui est un $B' \times G$-module qui engendre l'idéal $I$. D'après \cite[Macaulay2]{Mac2}, les relations entre les générateurs ci-dessus du $R$-module $I/I^2$ sont données dans la table \ref{table4}.

\begin{table}[ht]
\begin{center}
\center \includegraphics[scale=0.8]{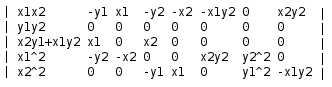}
\end{center}
\caption{ \label{table4} Relations entre les générateurs de $I/I^2$}
\end{table}

\noindent En particulier, on a les relations suivantes données par les colonnes $2$ et $5$ respectivement:
$$ \left \{
    \begin{array}{l}
r_1:=x_1 \otimes f_3-y_1 \otimes f_1 -y_2 \otimes h_1, \\
r_2:=x_2 \otimes f_1-x_1 \otimes h_2.    
 \end{array}
      \right.$$
On a $\dim(N)=5$ et donc d'après le lemme \ref{InegTang}, on a $\dim(T_{Z_0} \HH)=5- \rg(\rho^*)$. 
D'après le lemme \ref{fixespoints}, la variété $\HHp$ contient au moins un point fixe pour l'opération de $B'$, donc $Z_0 \in \HHp$ et donc $\dim(T_{Z_0} \HH) \geq \dim(\HHp)=3$. Donc, pour montrer la proposition \ref{dimTangentO2}, il suffit de montrer le

\begin{lemme} 
$\rg(\rho^*) \geq 2 .$
\end{lemme}

\begin{proof} 
Pour $i=1,2,3$, on définit $\psi_i \in \Hom_R^G(R \otimes N,R)$ par 
$$
\left\{
    \begin{array}{ll}
        \psi_i(h_j \otimes 1)=0 \ \text{ pour }j=1,2, \\
        \psi_i(f_j \otimes 1)={\delta}_i^j\ \text{ pour }j=1,2,3,
    \end{array}
\right.
$$
où $\delta_i^j$ est le symbole de Kronecker. Les $\psi_i$ forment une famille libre de $\Hom_R^{G}(R \otimes N,R)$, nous allons voir que $\{ \rho^*(\psi_1),\, \rho^*(\psi_3) \}$ est une famille libre de $\Hom_R^G(R \otimes \RRR,R)$ ce qui démontrera le lemme. Soient ${\lambda}_1$, ${\lambda}_3 \in k$ tels que 
$$ {\lambda}_1 \, \rho^*({\psi}_1)+ {\lambda}_3 \, \rho^*({\psi}_3)=0.$$
On évalue cette relation en $r_2$, on obtient:
$${\lambda}_1 \, {\psi}_1(\rho(1 \otimes r_2))+ {\lambda}_3 \, {\psi}_3(\rho(1 \otimes r_2))={\lambda}_1 \, {\psi}_1(r_2)+ {\lambda}_3 \, {\psi}_3(r_2)={\lambda}_1 x_2=0.$$
De même, on évalue cette relation en $r_1$, on obtient:
$${\lambda}_3 \, x_1- {\lambda}_1 \, y_1=0.$$
On en déduit que $({\lambda}_1, {\lambda}_3)=(0,0)$ et donc $\{ \rho^*(\psi_1),\, \rho^*(\psi_3) \}$ est bien une famille libre.
\end{proof}

\begin{remarque}
D'après \cite[Macaulay2]{Mac2}, une famille de générateurs du $R$-module $\Hom_R(I/I^2,R)$ est donnée dans la table \ref{table5}. 
\begin{table}[ht]
\begin{center}
\center \includegraphics[scale=0.8]{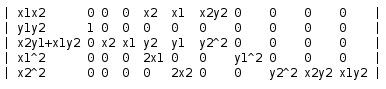}
\end{center}
\caption{\label{table5} Générateurs du $R$-module $\Hom_R(I/I^2,R)$}
\end{table}

\noindent On note $\Phi_i$ le morphisme associé à la colonne $i+1$. On vérifie que les trois morphismes $\Phi_1$, $\Phi_7+\Phi_8$ et $\frac{1}{2}(y_1 \Phi_4+y_2 \Phi_5)$ sont $G$-équivariants et linéairement indépendants. Ils forment donc une base de l'espace vectoriel ${\Hom}_R^G(I/I^2,R)$ et on en déduit, comme dans la remarque \ref{Bmoduleexplicite}, que 
$${\Hom}_R^G(I/I^2,R) \cong D^* \oplus \left( D^* \otimes \frac{S^2(V')}{D}\right)$$
comme $B'$-module.   
\end{remarque}

\noindent On déduit du lemme  \ref{Hlisse4} et et de ce qui précède le 

\begin{corollaire} \label{Hcasn2O2}
$\HH=\HHp$ est une variété lisse de dimension $3$. 
\end{corollaire}

\subsubsection{Construction d'un morphisme équivariant \texorpdfstring{$\delta:\ \HH \rightarrow \PP(W/\!/G)$}{}} \label{construction_delta}

Le lemme qui suit découle de la théorie classique des invariants comme les lemmes \ref{exi1} et \ref{lemmeThClass2}.

\begin{lemme} 
Le ${k[W]}^{G}$-module ${k[W]}_{(\Gamma_{2 \epsilon_1} \oplus \Gamma_{-2 \epsilon_1})}$ est engendré par ${\Hom}^{G}(\Gamma_{2 \epsilon_1} \oplus \Gamma_{-2 \epsilon_1},{k[W]}_2)$.
\end{lemme}

Puis, on a l'isomorphisme de $G'$-modules
$$\Hom^{G}((\Gamma_{2 \epsilon_1} \oplus \Gamma_{-2 \epsilon_1}),{k[W]}_2) \cong S^2(V').$$ 
La proposition \ref{morphismegrass} donne un morphisme $G'$-équivariant 
$$\HH \rightarrow \Gr(2,S^2(V'^*)).$$
Et $V'^* \otimes \det(V') \cong V'$, donc $S^2(V'^*) \otimes \det^2(V) \cong S^2(V')$ et donc 
$$\Gr(2,S^2(V'^*)) \cong \PP ( S^2(V')) \cong \PP ( S^2(V'^*))=\PP(W/\!/G)$$ 
comme $G'$-variété. On obtient donc un morphisme $G'$-équivariant 
\begin{equation} \label{maoprhismedeltaO2}
\delta:\ \HH \rightarrow \PP (W/\!/G).
\end{equation}

On rappelle que l'on note 
$$Y_0:=\left\{(Q,L)\in W/\!/G \times \PP (W/\!/G)\ |\ Q \in L \right\}=\OO_{\PP (W/\!/G)}(-1)$$ 
l'éclatement en $0$ de $W/\!/G$. On vérifie alors que le morphisme $\gamma \times \delta$ envoie ${\HH}$ dans $Y_0$. Puis, en procédant comme pour la proposition \ref{gammaiso}, on montre que le morphisme $\gamma \times \delta:\ \HH \rightarrow Y_0$ est un isomorphisme $G'$-équivariant.

\subsubsection{Cas \texorpdfstring{$n'>2$}{n'>2}}  \label{masterpropositioncasn2O2}

Pour tout $L \in \PP(W/\!/G)$, on définit $\Im(L)$ comme dans la notation \ref{notationsKerd}.
On considère la variété
$$\ZZ:=\left\{(Q,L,E) \in W/\!/G \times \PP(W/\!/G) \times \Gr(2,V'^*) \ \mid \ Q \in L \text{ et } \Im(L) \subset E  \right\}.$$       
On a le diagramme
$$\xymatrix{ &  \ZZ \ar@{->>}[ld]_{q_1} \ar@{->>}[rd]^{q_2} \\   Y_0 && \Gr(2,V'^*) }$$
où $q_1$ et $q_2$ sont les projections naturelles. En particulier, avec les notations du début de la section, on vérifie que $\ZZ \cong S^2\left( T \right)$.


\begin{lemme} \label{identificationblowupO2}
On a un isomorphisme de variétés 
$Y_1 \cong \ZZ$
et via cet isomorphisme, l'éclatement $Y_1 \rightarrow Y_0$ s'identifie au morphisme $q_1:\ \ZZ \rightarrow Y_0$.
\end{lemme}

\begin{proof}
La preuve est analogue à celle du lemme \ref{identificationblowup}.
\end{proof}

On raisonne alors comme dans la section \ref{masterpropositioncasn2}: on identifie $Y_1$ à $\ZZ$ grâce au lemme \ref{identificationblowupO2}, on montre l'existence d'un isomorphisme $G'$-équivariant $\HH \cong Y_1$, on identifie $\HH$ à $Y_1$ via cet isomorphisme et on montre que $\gamma$ est la composition des éclatements $Y_1 \rightarrow Y_0 \rightarrow W/\!/G$. En particulier, $\gamma$ est toujours une résolution de $W/\!/G$.

%% file: O3.tex
\subsection{Etude du cas \texorpdfstring{$\dim(V)=3$}{n=3}}  \label{monnio}

Dans toute cette section, on fixe $n=3$. On a $G \cong O_3(k)$, $W/\!/G=S^2(V'^*)^{\leq 3}$ et $\rho:\ \HH \rightarrow \Gr(3,V'^*)$ est le morphisme de la section \ref{redGrass2}. Nous allons démontrer le 

\begin{theoreme} \label{casn3O3} 
Si $n' \geq 3$, alors le schéma $\HH$ est singulier.
\end{theoreme} 

\begin{remarque}
Si $n' \leq 2$, alors $\HH \cong W/\!/G$ est une variété lisse et $\gamma$ est un isomorphisme d'après le corollaire \ref{cas_facileOn}. 
\end{remarque}

Lorsque $n'=3$, la non-lissité de $\HH$ est donnée par le corollaire \ref{HpaslisseO3}. Le cas général $n' \geq 3$ se déduit du cas particulier $n'=3$ grâce au principe de réduction (proposition \ref{reduction1}). Nous allons procéder comme pour la démonstration du théorème \ref{casn3}, cependant les choses vont s'avérer un peu plus compliquées ici car nous allons voir que $\HH$ possède deux points fixes pour $B'$.

\subsubsection{Représentations de \texorpdfstring{$O_3(k)$}{O(V)}}  \label{threpO3}

Comme pour le cas $n=2$, la théorie des représentation de $O_n(k)$ est particulièrement simple lorsque $n=3$. D'après \cite[§10.4]{FH}, on a un isomorphisme de groupes algébriques $SO_3(k) \cong {PSL}_2(k)$, où ${PSL}_2(k):={SL}_2(k)/\{\pm I_2\}$. Et donc, $O_3(k) \cong {PSL}_2(k) \times \ZZZ_2$ d'après (\ref{OnSonScinde}).\\
Les représentations irréductibles de ${SL}_2(k)$ sont paramétrées par les entiers naturels:
$d \in \NN \leftrightarrow V(d)$, où $V(d):={k[x,y]}_d$ est l'ensemble des polynômes homogènes de degré $d$. En particulier, on a $\dim(V(d))=d+1$. Les représentations irréductibles de $SO_3(k)$ sont donc paramétrées par les entiers naturels pairs: 
$d \in 2 \NN  \leftrightarrow V(d).$
On note $M_0$ (resp. $\epsilon$) la représentation triviale (resp. la représentation signe) de $\ZZZ_2$. Les représentations irréductibles de $O_3(k)$ sont:\\ 
$\bullet \ V(d) \otimes M_0$, avec $d$ pair, de dimension $d+1$,\\
$\bullet \ V(d) \otimes \epsilon$, avec $d$ pair, de dimension $d+1$.\\
Si $g \in O_3(k)$, on note $\overline{g}$ l'image de $g$ dans $O_3(k)/SO_3(k) \cong \ZZZ_2$ et alors $\epsilon(\overline{g})=\det(g)$. Pour alléger les notations, on fera l'abus d'écrire $V(d)$ pour désigner la représentation $V(d) \otimes M_0$. Avec ces notations, $V(0)$ est la représentation triviale de $O_3(k)$ et $SO_3(k)$, $V(2)$ est la représentation standard de $SO_3(k)$ et $V(2) \otimes \epsilon$ est la représentation standard de $O_3(k)$.\\
On rappelle enfin que l'on dispose de la formule de Clebsch-Gordan (\cite[Exercise 11.11]{FH}) pour décomposer les produits tensoriels de représentations de $SL_2(k)$.

\subsubsection{Points fixes de \texorpdfstring{$\HH$}{H} pour l'opération de \texorpdfstring{$B'$}{B'}} \label{sectionptfixeO3}

On suppose dorénavant que $n'=3$. On souhaite montrer que $\HH$ est singulier. Pour ce faire, on commence par déterminer les points fixes de $B'$ dans $\HH$. 
On a:
\begin{align}
{k[W]}_1 \cong &V' \otimes V(2) \otimes \epsilon, \label{kW1O3}\\
{k[W]}_2 \cong &(S^2(V') \otimes (V(4) \oplus V(0))) \oplus ({\Lambda}^2(V') \otimes V(2)), \label{kW2O3} \\
{k[W]}_3 \cong &(S^3(V') \otimes (V(6) \oplus V(2)) \otimes \epsilon) \oplus  (S^{2,1}(V') \otimes (V(4) \oplus V(2)) \otimes \epsilon) \label{kW3O3} \\ 
               &\oplus (\Lambda^3(V') \otimes V(0) \otimes \epsilon), \notag 
\end{align}
comme $G'\times G$-modules.

On reprend les notations de la section \ref{rappelsthedesreps} et on considère les différents poids présents dans le $G'$-module $S^{2,1}(V')$:
$$\xymatrix{ & 2 \epsilon_1+\epsilon_2 \\ \epsilon_1+2 \epsilon_2 \ar[ru] && 2\epsilon_1+ \epsilon_3 \ar[lu] \\ & \epsilon_1+\epsilon_2+ \epsilon_3 \ar[ru] \ar[lu] \\ 2 \epsilon_2+\epsilon_3 \ar[ru] &&  \epsilon_1+2 \epsilon_3 \ar[lu] \\ & \epsilon_2+2 \epsilon_3 \ar[ru] \ar[lu] }$$
Le poids $\epsilon_1+\epsilon_2+\epsilon_3$ est présent avec multiplicité $2$ et tous les autres poids avec multiplicité $1$. Les flèches indiquent comment opère $B'$ dans $S^{2,1}(V')$: deux poids $\lambda_1$ et $\lambda_2$ peuvent être reliés par une suite de flèches si et seulement si il existe un élément $b' \in B'$ tel que, si l'on note $v_i$ un vecteur de poids $\lambda_i$, alors $b'.v_1= \alpha v_2 + \ldots$ pour un certain $\alpha \neq 0$. On note $\PPP$ le plan $B'$-stable de $S^{2,1}(V')$ engendré par les deux droites $T'$-stables associées aux poids $2 \epsilon_1+\epsilon_2$ et $2 \epsilon_1+\epsilon_3$. On vérifie que le plan $T'$-stable associé au poids $\epsilon_1+\epsilon_2+\epsilon_3$ contient une unique droite $L$ telle que pour tout $v \in L$, pour tout $b' \in B',\ b'.v \in \PPP \oplus L$. On note $E_1:=\PPP \oplus L$ et $E_2$ le sous-espace de $S^{2,1}(V')$ engendré par les trois droites $T'$-stables associées aux poids $2 \epsilon_1+\epsilon_2$, $2 \epsilon_1+\epsilon_3$ et $\epsilon_1+2\epsilon_2$. Les espaces $E_1$ et $E_2$ sont les deux seuls $B'$-sous-module de $S^{2,1}(V')$ de dimension $3$ contenant $\PPP$.

\begin{notation}  \label{espacesdim3O3}
On note:
\begin{itemize} \renewcommand{\labelitemi}{$\bullet$}
\item $J$ l'idéal engendré par les $G$-invariants homogènes de degré positif de $k[W]$,
\item $D:=\left\langle  e_1.e_1 \right\rangle \subset S^2(V')$ l'unique droite $B'$-stable de $S^2(V')$,
\item $I_i$  l'idéal engendré par  $(D \otimes V(4)) \oplus (S^2(V') \otimes V(0)) \subset k[W]_2$ \text{ et par } $(E_i \otimes V(4)) \subset {k[W]}_3$ pour $i=1,2$.
\end{itemize}
\end{notation}

\begin{remarque}
On a $J \cap k[W]_2=S^2(V') \otimes V(0)$ comme $G' \times G$-module et ce module engendre l'idéal $J$.
Les idéaux $I_1$ et $I_2$ sont homogènes, $B' \times G$-stable et contiennent $J$.
\end{remarque}

\noindent Nous allons montrer le

\begin{theoreme} \label{pointfixeborel32}
Les idéaux $I_1$ et $I_2$ sont les deux points fixes de $B'$ dans $\HH$. 
\end{theoreme}

\noindent Pour ce faire, nous aurons besoin du

\begin{lemme} \label{decompo_Si2}
Pour chaque $i \geq 1$, on a les isomorphismes de $G$-modules suivants:
\begin{enumerate}
\item        $S^i(V) \cong V(2i) \oplus V(2i-4) \oplus \cdots \oplus V(0) \text{ si $i$ est pair,}$ 
\item        $S^i(V) \cong (V(2i) \oplus V(2i-4) \oplus \cdots \oplus V(2)) \otimes \epsilon \text{ si $i$ est impair,}$ 
\item        $S^{i,1}(V) \cong (V(2i) \oplus V(2i-2) \oplus \cdots \oplus V(2)) \otimes \epsilon \text{ si $i$ est pair,}$
\item        $S^{i,1}(V) \cong V(2i) \oplus V(2i-2) \oplus \cdots \oplus V(2) \text{ si $i$ est impair.}$
\end{enumerate}
\end{lemme}

\begin{proof}
On remarque que $O(V)/SO(V) \cong \ZZZ_2$ opère dans $S^i(V)$ trivialement lorsque $i$ est pair, et par le signe lorsque $i$ est impair. Il suffit donc de montrer que, pour chaque $i \geq 1$, on a les isomorphismes de $SO(V)$-modules suivants:
\begin{enumerate}
\item  $S^i(V) \cong V(2i) \oplus V(2i-4) \oplus \cdots \oplus V(0) \text{ si $i$ est pair,}$   
\item  $S^i(V) \cong V(2i) \oplus V(2i-4) \oplus \cdots \oplus V(2) \text{ si $i$ est impair,}$ 
\item  $S^{i,1}(V) \cong V(2i) \oplus V(2i-2) \oplus \cdots \oplus V(2)$.
\end{enumerate}
Les isomorphismes $(1)$ et $(2)$ sont donnés dans \cite[Exercise 11.14]{FH}. Ensuite, soit $i \geq 1$ un entier pair, alors
$$S^i(V) \otimes V \cong S^{i+1}(V) \oplus S^{i,1}(V)$$ 
comme $GL(V)$-module. On utilise l'isomorphisme $(1)$ et la formule de Clebsch-Gordan pour décomposer $S^i(V) \otimes V$ en $SO(V)$-modules irréductibles:
\begin{align*}
S^i(V) \otimes V \cong &(V(2i) \oplus V(2i-4) \oplus \cdots \oplus V(0)) \otimes V(2)\\
 \cong &(V(2i) \oplus V(2i-2) \oplus \cdots \oplus V(2)) \\
 &\ \oplus (V(2i+2) \oplus V(2i-2) \oplus V(2i-6) \oplus \cdots \oplus V(2)).     
\end{align*}
On utilise alors l'isomorphisme $(2)$ pour décomposer $S^{i+1}(V)$ et le résultat s'ensuit. La démarche est analogue dans le cas où $i$ est impair. 
\end{proof}  

\begin{proof}[\textbf{Preuve du théorème \ref{pointfixeborel32}}]
On raisonne comme dans la preuve du théorème \ref{pointfixeborel} en considérant $I_Z$ un point fixe de $B'$ et en étudiant $k[W]/I_Z$ composante par composante:\\
$\bullet$ Composantes de degré $0$ et $1$:\\
On a bien sûr $I_Z \cap {k[W]}_0=\{0\}$ et $I_Z \cap {k[W]}_1 \neq {k[W]}_1$.\\ 
$\bullet$ Composante de degré $2$: on utilise la décomposition (\ref{kW2O3}).\\
Pour avoir la décomposition souhaitée de $k[W]/I_Z$ comme $G$-module, on a nécessairement ${k[W]}_2 \cap I_Z \supseteq 6 V(0) \oplus V(4)$. 
En effet, le $G$-module $k[W]/I_Z$ contient déjà une copie de la représentation triviale (qui provient de la composante de degré $0$), il ne peut donc pas en contenir d'autre. Ensuite, ${k[W]}_2$ contient $6$ copies de $V(4)$ qui est un $G$-module de dimension $5$, donc ${k[W]}_2 \cap I_Z$ contient au moins une copie de $V(4)$. Comme $k[W]_2 \cap I_Z$ est $B'$-stable, il contient $D \otimes V(4)$ car $D$ est l'unique droite $B'$-stable de $S^2(V')$. Il s'ensuit que $I_Z$ contient $(S^2(V') \otimes V(0)) \oplus (D \otimes V(4))$.\\
$\bullet$ Composante de degré $3$: on utilise la décomposition (\ref{kW3O3}).\\
On remarque que ${k[W]}_3$ contient huit copies de $V(4) \otimes \epsilon$ qui est un $G$-module de dimension $5$, donc nécessairement $I_Z \cap {k[W]}_3 \supset 3 V(4) \otimes \epsilon$. On reprend les notations du début de la section \ref{sectionptfixeO3}.
On vérifie que l'idéal engendré par $(D \otimes V(4)) \oplus (S^2(V') \otimes V(0)) \subset k[W]_2$ contient $\PPP \otimes V(4) \otimes \epsilon$. Comme $k[W]_3 \cap I_Z$ est $B'$-stable, il contient $E_i \otimes V(4) \otimes \epsilon$, pour $i=1$ ou $2$, car on a vu que $E_1$ et $E_2$ sont les seuls espaces $B'$-stables de dimension $3$ dans $S^{2,1}(V')$ contenant $\PPP$. Donc ${I}_Z$ contient nécessairement $I_1$ ou $I_2$. Le lemme qui suit achève la preuve du théorème \ref{pointfixeborel32}:  

\begin{lemme} \label{bonnefcthilbO3}
Les idéaux $I_1$ et $I_2$ ont pour fonction de Hilbert $h_W$.
\end{lemme}

\noindent \begin{itshape} \textbf{Preuve du lemme:} \end{itshape} 
On considère l'idéal $J$ défini dans la notation \ref{espacesdim3O3}. D'après le lemme \ref{decompo_Si2},  pour chaque $i \geq 2$, on a les inclusions:
\begin{equation} \label{decompoInclu}
\left\{
    \begin{array}{l}
      {k[W]}_i/({J \cap {k[W]}_i}) \supset (S^i(V') \otimes V(2i)) \oplus (S^{i-1,1}(V') \otimes V(2i-2)) \text{ si $i$ est pair,}\\
      {k[W]}_i/({J \cap {k[W]}_i}) \supset (S^i(V') \otimes V(2i) \otimes \epsilon) \oplus (S^{i-1,1}(V') \otimes V(2i-2) \otimes \epsilon) \text{ si $i$ est impair.} 
       \end{array}
\right.
\end{equation}
En effet, l'idéal $J$ est engendré par les copies de $V(0)$ dans ${k[W]}_2$ et donc, si $p$ est un entier tel que $V(2p) \subset J \cap {k[W]}_i$, alors nécessairement $p \leq i-2$. On dispose donc, pour chaque $i \geq 1$, d'une minoration de la dimension de ${k[W]}_i/({J \cap {k[W]}_i})$ par 
$$Q(i):=\dim((S^i(V') \otimes V(2i)) \oplus (S^{i-1,1}(V') \otimes V(2i-2)))=3i^3+\frac{5}{2}i^2+\frac{3}{2}i+2 .$$
Ensuite, \cite[Macaulay2]{Mac2} nous fournit le polynôme de Hilbert $P_J$ et la fonction de Hilbert classique $f_J$ de l'idéal $J$:
$$
 \left\{
    \begin{array}{l}
        P_J(X)=3X^3+\frac{5}{2}X^2+\frac{3}{2}X+2,\\
        \forall n \notin \{0,3\},\ f_J(n)=P_J(n) \text{ et } f_J(0)=1,\ f_J(3)=111.
    \end{array}
\right.
$$
Il s'ensuit que les inclusions (\ref{decompoInclu}) sont en fait des égalités pour $i=2$ et $i \geq 4$. Pour $i=3$, on a 
$${k[W]}_3/({J \cap {k[W]}_3}) \cong (S^3(V') \otimes V(6) \otimes \epsilon) \oplus (S^{2,1}(V') \otimes V(4) \otimes \epsilon) \oplus \Lambda^3(V') \otimes \epsilon .$$ 
On note $K$ l'idéal de $k[W]$ engendré par $(D \otimes V(4)) \oplus (S^2(V') \otimes V(0)) \subset {k[W]}_2$. Alors $(K \cap k[W]_2)/(J \cap k[W]_2) \cong D \otimes V(4)$ comme $B' \times G$-module. Donc ${(K/J)}^{U}$ est un idéal de ${(k[W]/J)}^{U}$ engendré par un unique élément de ${({k[W]}_2/ {k[W]}_2 \cap J)}^{U}$ et donc, pour tout $i \geq 1$, la multiplicité du $G$-module $V(2i) \otimes \left( \epsilon^{\otimes i} \right)$ dans $k[W]_i/(K \cap k[W]_i)$ est égale à
$$\dim(S^i(V'))-\dim(S^{i-2}(V'))=2i+1=\dim(V(2i) \otimes \epsilon^{\otimes i}) .$$
Ensuite, on fixe $j \in \{1,2\}$, alors $I_j$ est l'idéal de $k[W]$ engendré par $K$ et par le $B'\times G$-module $E_j \otimes V(4) \otimes \epsilon \subset k[W]_3$, donc le nombre de copies de $V(2i) \otimes \left(\epsilon^{\otimes i}\right)$ dans $k[W]_i/(K \cap k[W]_i)$ est le même que dans $k[W]_i/(I_j \cap k[W]_i)$. On calcule le polynôme de Hilbert $Q_j$ et la fonction de Hilbert $g_j$ de l'idéal $I_j$ avec \cite[Macaulay2]{Mac2}:
$$
 \left\{
    \begin{array}{l}
        Q_j(X)=8X^2+2,\\
        \forall n \geq 4,\ g_j(n)=Q_j(n) \text{ et } g_j(0)=1,\ g_j(1)=9,\ g_j(2)=34,\ g_j(3)=75.
    \end{array}
\right.
$$ 
Et donc, la multiplicité du $G$-module $V(2i)$ dans $k[W]/I_j$ est:
$$
 \left\{
    \begin{array}{l}
        2i+1=\dim(V(2i)) \text{ si $i$ est pair,}\\
       \frac{g(i+1)-{(2i+3)}^2}{2i+1}=2i+1=\dim(V(2i)) \text{ si $i$ est impair,}
    \end{array}
\right.
$$ 
et la multiplicité du $G$-module $V(2i) \otimes \epsilon$ dans $k[W]/I_j$ est:
$$
 \left\{
    \begin{array}{l}
\frac{g(i+1)-{(2i+3)}^2}{2i+1}=2i+1=\dim(V(2i)) \text{ si $i$ est pair,}\\
2i+1=\dim(V(2i) \otimes \epsilon) \text{ si $i$ est impair.}   
 \end{array}
\right.
$$ 
Il s'ensuit que l'idéal $I_j$ a la bonne fonction de Hilbert. 
\end{proof}


\begin{remarque}
On a $\Stab_{G'}(I_1)=\Stab_{G'}(I_2)=B'$, donc chacune des deux orbites fermées de $\HH$ est isomorphe à $G'/B'$.
\end{remarque} 

\subsubsection{\texorpdfstring{$I_2$}{I2} est un point de \texorpdfstring{$\HHp$}{Hp}}  \label{reductibilitéO3}

Le but de cette section est de montrer la

\begin{proposition}  \label{reductiO3}
L'idéal $I_2$ appartient à $\HHp$.
\end{proposition}

Cette proposition nous permettra ensuite de montrer que $\HH$ est singulier en $I_2$. 
On note $w \in W$ sous la forme $w=\begin{bmatrix}
x_1 &y_1 & z_1 \\
x_2 &y_2 & z_2 \\
x_3 &y_3 & z_3 
\end{bmatrix}$ et on identifie $k[W]$ à $k[x_1,x_2,x_3,y_1,y_2,y_3,z_1,z_2,z_3]$. On commence par expliciter des bases de certains $B' \times G$-modules qui apparaissent dans $k[W]$:\\ 
$\left.
    \begin{array}{l}
        x_1^2+x_2^2+x_3^2 \\
        y_1^2+y_2^2+y_3^2 \\
        z_1^2+z_2^2+z_3^2 \\
        x_1 y_1+x_2 y_2+x_3 y_3\\
        x_1 z_1+x_2 z_2+x_3 z_3\\
        z_1 y_1+z_2 y_2+z_3 y_3
      \end{array}
\right \} \text{ est une base de } S^2(V') \otimes V(0) \subset k[W]_2,\\
\left.
    \begin{array}{l}
        x_1^2 \\
        x_2^2\\
        x_3^2\\
        x_1 x_2 \\
        x_2 x_3\\
        x_1 x_3 
      \end{array}
\right \} \text{ est une base de }  D \otimes S^2(V)  \subset k[W]_2,\\
\left.
    \begin{array}{l}
x_1(x_1y_2-x_2y_1)\\
x_2(x_1y_2-x_2y_1)\\
x_3(x_1y_2-x_2y_1)\\
x_1(x_1y_3-x_3y_1)\\
x_2(x_1y_3-x_3y_1)\\
x_3(x_1y_3-x_3y_1)\\
x_1(x_2y_3-x_3y_2)\\
x_2(x_2y_3-x_3y_2)\\
x_3(x_2y_3-x_3y_2)\\  
x_1(x_1z_2-x_2z_1)\\
x_2(x_1z_2-x_2z_1)\\
x_3(x_1z_2-x_2z_1)\\
x_1(x_1z_3-x_3z_1)\\
x_2(x_1z_3-x_3z_1)\\
x_3(x_1z_3-x_3z_1)\\
x_1(x_2z_3-x_3z_2)\\
x_2(x_2z_3-x_3z_2)\\
x_3(x_2z_3-x_3z_2)\\   
y_1(x_1y_2-x_2y_1)\\
y_2(x_1y_2-x_2y_1)\\
y_3(x_1y_2-x_2y_1)\\
y_1(x_1y_3-x_3y_1)\\
y_2(x_1y_3-x_3y_1)\\
y_3(x_1y_3-x_3y_1)\\
y_1(x_2y_3-x_3y_2)\\
y_2(x_2y_3-x_3y_2)\\
y_3(x_2y_3-x_3y_2)    
      \end{array}
\right \} \text{ est une famille génératrice de }  E_2 \otimes S^{2,1}(V). 
$\\

\noindent On en déduit des générateurs de l'idéal $I_2$:
\begin{align*}
I_2= (       &y_1^2+y_2^2+y_3^2, z_1^2+z_2^2+z_3^2, x_1^2+x_2^2+x_3^2, x_2^2, x_3^2, x_1 x_2, x_2 x_3, x_1 x_3, \\
             & x_1 y_1+x_2 y_2+x_3 y_3, x_1 z_1+x_2 z_2+x_3 z_3, z_1 y_1+z_2 y_2+z_3 y_3,\\
             & y_3(x_1y_2-x_2y_1), y_2(x_1y_3-x_3y_1), y_3(x_1y_3-x_3y_1),\\
             &y_2(x_2y_3-x_3y_2), y_3(x_2y_3-x_3y_2) ) . 
\end{align*}

D'après la proposition \ref{chow}, le morphisme de Hilbert-Chow est un isomorphisme au dessus de $U_3 \subset W/\!/G$, donc il existe un unique $Z_{Id} \in \HH$ tel que $\gamma(Z_{Id})=I_d$. On note $L$ l'idéal de $Z_{Id}$ dans $k[W]$. D'après le lemme \ref{diagcom}, on a:
\begin{align*}
L=(&x_1^2+x_2^2+x_3^2-1,\ y_1^2+y_2^2+y_3^2-1,\ z_1^2+z_2^2+z_3^2-1, \\
         &x_1 y_1+x_2 y_2+x_3 y_3, x_1 z_1+x_2 z_2+x_3 z_3, z_1 y_1+z_2 y_2+z_3 y_3).   
\end{align*}

Soit $I$ un idéal de $k[W]$ correspondant à un point de $\HH$. On va construire une famille plate d'idéaux ${(L_t)}_{t \in \Aff}$ de $k[W]$ telle que $L_1=L$ et $L_0=I_2$. Par définition de la composante principale $\HHp$ de $\HH$, on en déduira que $I_2$ est un point de $\HHp$.\\ 
On commence par calculer une base de Gröbner $\GG_L$ de l'idéal $L$ à l'aide de \cite[Macaulay2]{Mac2}:
\begin{align*}
\GG_L=\{ &z_1^2+z_2^2+z_3^2-1, y_1z_1+y_2z_2+y_3z_3, x_1z_1+x_2z_2+x_3z_3,\\ 
              &y_1^2+y_2^2+y_3^2-1, x_1y_1+x_2y_2+x_3y_3, x_3^2+y_3^2+z_3^2-1,\\ 
              &x_2x_3+y_2y_3+z_2z_3, x_1x_3+y_1y_3+z_1z_3, x_2^2+y_2^2+z_2^2-1,\\
              &x_1x_2+y_1y_2+z_1z_2, x_1^2-y_2^2-y_3^2-z_2^2-z_3^2+1, x_2y_1y_2-x_1y_2^2-x_3z_1z_3+x_1z_3^2,\\
              &y_2z_1z_2-y_1z_2^2+y_3z_1z_3-y_1z_3^2+y_1, x_2z_1z_2-x_1z_2^2+x_3z_1z_3-x_1z_3^2+x_1,\\
              &x_2y_1z_2-x_1y_2z_2+x_3y_1z_3-x_1y_3z_3 , y_2^2z_1+y_3^2z_1-y_1y_2z_2-y_1y_3z_3-z_1,\\
              &x_2y_2z_1+x_3y_3z_1-x_1y_2z_2-x_1y_3z_3, x_3y_2y_3-x_2y_3^2+x_3z_2z_3-x_2z_3^2+x_2, \\
              &x_3y_1y_3-x_1y_3^2+x_3z_1z_3-x_1z_3^2+x_1 ,x_2y_1y_3-x_1y_2y_3+x_2z_1z_3-x_1z_2z_3,\\
              &x_3y_2^2-x_2y_2y_3+x_3z_2^2-x_2z_2z_3-x_3 , x_3y_1y_2-x_1y_2y_3+x_3z_1z_2-x_1z_2z_3,\\
              &x_3y_2z_1z_3-x_2y_3z_1z_3-x_3y_1z_2z_3+x_1y_3z_2z_3+x_2y_1z_3^2-x_1y_2z_3^2-x_2y_1+x_1y_2,\\
              &y_3^2z_2^2-2y_2y_3z_2z_3+y_2^2z_3^2-y_2^2-y_3^2-z_2^2-z_3^2+1, \\
              &x_3y_3z_2^2-x_3y_2z_2z_3-x_2y_3z_2z_3+x_2y_2z_3^2-x_2y_2-x_3y_3,\\
              &y_3^2z_1z_2-y_2y_3z_1z_3-y_1y_3z_2z_3+y_1y_2z_3^2-y_1y_2-z_1z_2, \\
              &x_3y_3z_1z_2-x_2y_3z_1z_3-x_3y_1z_2z_3+x_2y_1z_3^2-x_2y_1 \}.
\end{align*}

\noindent On fixe $(n_1,n_2,n_3) \in {\ZZZ}^3$ et soit 
\begin{equation}  \label{ssgroupeparam2}
\begin{array}{ccccc}
\theta_{n_1,n_2,n_3} & : & \Gm & \to & T' \\
& & t & \mapsto & (t^{n_1},t^{n_2},t^{n_3}) 
\end{array}
\end{equation}
un sous-groupe à un paramètre du tore $T'$. Par définition de l'opération de $T'$ dans $W$, on a:
$$\forall t \in \Gm,\ \forall w= \begin{bmatrix}       
  x_1    & y_1   &z_1 \\          
  x_2    & y_2   &z_2 \\
  x_3    & y_3   &z_3 \end{bmatrix} \in W,\ {\theta}_{n_1,n_2,n_3}(t).w= \begin{bmatrix}         
  t^{-n_1} x_1    & t^{-n_2} y_1  &t^{-n_3} z_1 \\          
  t^{-n_1} x_2    & t^{-n_2} y_2   &t^{-n_3} z_2 \\
  t^{-n_1} x_3    & t^{-n_2} y_3   &t^{-n_3} z_3 \end{bmatrix}.$$ 
Pour $t \in \Gm$ fixé, l'opération de ${\theta}_{n_1,n_2,n_3}(t)$ dans $W$ induit un automorphisme ${\sigma}_{n_1,n_2,n_3,t}$ de $k[W]$ défini par:
$$\forall i=1,2,3,\ \left\{
    \begin{array}{l}
        {\sigma}_{n_1,n_2,n_3,t}(x_i)=t^{n_1} x_i, \\
        {\sigma}_{n_1,n_2,n_3,t}(y_i)=t^{n_2} y_i,\\
        {\sigma}_{n_1,n_2,n_3,t}(z_i)=t^{n_3} z_i.
    \end{array}
\right.$$
On note $L_t$ l'image de l'idéal $L$ par cet automorphisme. 
Soit $g \in k[W]$, on l'écrit $g=\sum_i u_i m_i$, où les $m_i$ sont des monômes de $k[W]$ et les $u_i$ sont des éléments de $\Gm$. On définit $$\tilde{g}:=t^{m}g(t^{n_1}x_1,\ldots,t^{n_3}z_3) \in k[W]$$ 
avec $m$ le plus petit entier relatif tel que, lorsque l'on décompose $\tilde{g}$ en une somme de monômes:
$\tilde{g}=\sum_i u_i t^{N_i} m_i'$, chaque entier $N_i$ qui apparaît est supérieur ou égal à $0$.
On dispose alors d'une description explicite de l'idéal $L_t$:
$$L_t=\{ \tilde{g}\ |\ g \in L\}.$$
D'après \cite[Exercise 15.25]{Ei}, si $\{g_1,\ldots,g_r\}$ est une base de Gröbner de l'idéal $L$, alors $L_t=(\tilde{g_1},\ldots,\tilde{g_r})$. En particulier, on a $L_1=L$. On note $L_0$ l'idéal de $k[W]$ obtenu en posant $t=0$ pour chaque élément de $L_t$. Alors, d'après \cite[Theorem 15.17]{Ei}, la famille d'idéaux ${(L_t)}_{t \in \Aff}$ est une famille plate au dessus de $\Aff$. \\
Pour démontrer que $I_2 \in \HHp$, il suffit donc de trouver un triplet $(n_1,n_2,n_3) \in {\ZZZ}^3$ tel que l'on ait $L_0=I_2$. On vérifie que le triplet $(-3,-2,-1)$ convient. 

\begin{remarque}
On souhaiterait procéder de manière analogue pour montrer que $I_1 \in \HHp$, malheureusement il n'existe aucun triplet $(n_1,n_2,n_3)$ tel que l'on ait $L_0=I_1$. 
\end{remarque}

\subsubsection{Espace tangent de \texorpdfstring{$\HH$}{H} en \texorpdfstring{$I_2$}{I2}} \label{dimTang2O3}

On note $Z_2:=\Spec(k[W]/I_2)$. On a la  

\begin{proposition} \label{dimTangent2O3}
$\dim(T_{Z_2} \HH)\geq 7.$
\end{proposition}

\begin{proof}
D'après \cite[Macaulay2]{Mac2}, une famille d'éléments du $R$-module $\Hom_R(I/I^2,R)$ est donnée dans la table \ref{homom_cas_n3OO3}.  

\begin{table}[ht]
\begin{center}
\center \includegraphics[scale=0.8]{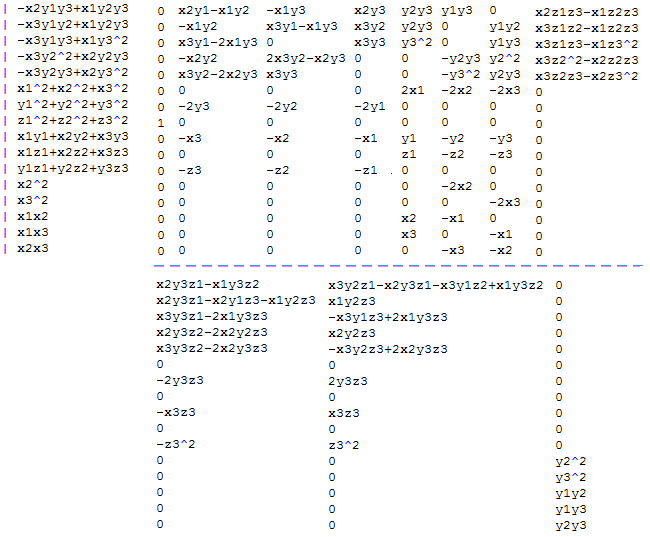}
\end{center}
\caption{\label{homom_cas_n3OO3} Quelques éléments du $R$-module $\Hom_R(I/I^2,R)$}
\end{table}

\noindent On note $\Phi_i$ le morphisme donné par la $i+1$-ème colonne de cette matrice. On vérifie alors que les sept morphismes suivants sont $G$-équivariants et linéairement indépendants:
\begin{itemize} \renewcommand{\labelitemi}{$\bullet$}
   \item $\Phi_1,$      
   \item $z_3 \Phi_2+ z_2 \Phi_3+z_1 \Phi_4$,   
   \item $y_1 \Phi_5- y_2 \Phi_6-y_3 \Phi_7$,  
   \item $z_1 \Phi_5- z_2 \Phi_6-z_3 \Phi_7$,  
   \item $\Phi_8$,
   \item $\Phi_9+ \Phi_{10}$,
   \item $\Phi_{11}$.
\end{itemize}  
Il s'ensuit que ${\Hom}_R^G(I_2/I_2^2,R)$ est de dimension au moins $7$ et le résultat découle alors de l'isomorphisme $T_{Z_2} \HH \cong \Hom_{R}^{G}(I_2/I_2^2,R)$.
\end{proof}

\begin{remarque}
Une étude des relations entre les générateurs du $R$-module $I/I^2$ similaire à celle réalisée dans la section \ref{dimTang2} permet de montrer que  $\dim(T_{Z_2} \HH) \leq 8$. Je ne sais malheureusement pas si $\dim(T_{Z_2} \HH)$ vaut $7$ ou $8$.  
\end{remarque}

\begin{corollaire}  \label{HpaslisseO3}
Le schéma $\HH$ est singulier en $Z_2$.
\end{corollaire}

\begin{proof}
On a $\dim(\HHp)=6$ et donc $\HH$ est lisse en $Z_2$ si et seulement si $\dim(T_{Z_2} \HH)=6$, ce qui est faux d'après la proposition \ref{dimTangent2O3}.  
\end{proof}

%% file: position_probleme_SOn.tex
\section{Cas de \texorpdfstring{$SO(V)$}{SO(V)} opérant dans \texorpdfstring{$V^{\oplus  n'}$}{n'V}}  \label{HclassSOngeneral}

On se place dans la situation $4$: on a $G:=SO(V)=O(V) \cap SL(V)$, $G':=GL(V')$, $W:=\Hom(V',V)$ et l'opération de $G' \times G$ dans $W$ est donnée par (\ref{actionSLLn}).

Lorsque $n=2$, on a $G \cong \Gm$ et alors l'étude du schéma de Hilbert invariant pour l'opération de $G$ dans $W$ a déjà été effectuée dans la section \ref{section_n_1}. On supposera donc que $n \geq 3$ par la suite.

\subsection{Etude du morphisme de passage au quotient}  \label{SSOONN}

D'après le premier théorème fondamental pour $SO(V)$ (voir \cite[§11.2.1]{Pro}) l'algèbre des invariants $k[W]^G$ est engendrée par les $(i\ |\ j)$, $1 \leq i \leq j \leq n'$, définis par (\ref{Oninvariants}), et par les $[i_1, \ldots, i_n]$, $1 \leq i_1 < \ldots < i_n \leq n'$, définis par (\ref{SLinvariants} ). 
Le morphisme de passage au quotient $\nu:\ W \rightarrow W/\!/G$ s'exprime donc en fonction des morphismes de passage au quotient $W \rightarrow W/\!/O(V)$ et $W \rightarrow W/\!/SL(V)$ étudiés dans les sections \ref{description_quotientOn} et \ref{casSln} respectivement. Plus précisément
$$\begin{array}{lrcl}
 \nu:  &\Hom(V',V)   & \rightarrow  & S^2(V'^*) \times \Lambda^n (V'^*) \\
        & w  & \mapsto      &  (\leftexp{t}{w}w,\ L_1 \wedge \ldots \wedge L_n) 
\end{array}$$
où les $L_i$ sont définis dans la notation \ref{lignescolonnes}. On a donc un diagramme commutatif
\begin{equation} \label{diagOvSOv}
\xymatrix{ & W/\!/SL(V) \\ W \ar[r]^{\nu} \ar[rd] \ar[ru] & W/\!/G \ar[d]^{p_1} \ar[u]_{p_2} \\ & W/\!/O(V) }
\end{equation}
où les $p_i$ sont les projections naturelles et les deux flèches diagonales sont les morphismes de passage au quotient.\\
Si $n'<n$, alors $\Lambda^n (V'^*)=\{0\}$, donc $W/\!/G=W/\!/O(V)$ et $\nu$ est le morphisme de passage au quotient $W \rightarrow W/\!/O(V)$ de la section \ref{description_quotientOn}. \\
Si $n'=n$, alors $\Lambda^n (V'^*) \cong \Aff$ et $\nu:\ w \mapsto (\leftexp{t}{w}w, \det(w))$. On a donc 
$$W/\!/G=\left\{ (Q,x) \in  S^2(V'^*) \times \Aff  \  \mid  \  \det(Q)=x^2 \right\}$$ 
et donc $p_1:\ W/\!/G \rightarrow S^2(V'^*)$ est un revêtement double de $S^2(V'^*)$ dont le lieu de ramification est $S^2(V'^*)^{\leq n-1}$. On vérifie, grâce au critère jacobien, que le lieu singulier de $W/\!/G$ est $\{(Q,x) \in W/\!/G \ |\ \rg(Q) \leq n-2\}$.\\
Enfin, si $n'>n$ et $D:=O(V)/G \cong \ZZZ_2$. On a l'inclusion $p_1^*:\ k[W]^{O(V)} \hookrightarrow k[W]^{G}$ qui induit un isomorphisme $k[W]^{O(V)} \cong (k[W]^{G})^D$. Donc $W/\!/O(V) \cong (W/\!/G)/\!/D$ et $p_1$ s'identifie au morphisme de passage au quotient $W/\!/G \rightarrow (W/\!/G)/\!/D$. Or $D$ est un groupe d'ordre $2$ qui opère librement dans l'ouvert $\{(Q,x) \in W/\!/G \mid \ \rg(Q)=n \text{ et } x \neq 0\}$ en multipliant $x$ par $\pm 1$ et qui opère trivialement dans le complémentaire. Il s'ensuit que $p_1$ est un revêtement double de $W/\!/O(V)$ dont le lieu de ramification est $\{(Q,0) \in W/\!/G \ \mid \ \rg(Q) \leq n-1\}$. J'ignore malheureusement quel est le lieu singulier de $W/\!/O(V)$ lorsque $n'>n$. Néanmoins, des calculs pour des petites valeurs de $n$ suggèrent que le lieu singulier est $\{(Q,0) \in W/\!/G \ \mid \ \rg(Q) \leq n-2\}$. \\

\begin{lemme} \label{revuniv}
On suppose que $n' > n$ et on note $U:=\{ Q \in S^2(V'^*) \ |\ \rg(Q)=n\}$ l'orbite ouverte de $W/\!/O(V)$, alors $p_1:\ p_1^{-1}(U) \rightarrow U$ est le revêtement universel.
\end{lemme}

\begin{proof}
Soit $\tilde{U}$ l'image réciproque de $U$ par le morphisme de passage au quotient $q:\ W \rightarrow W/\!/O(V)$. D'après \cite[§2.1, Theorem 6]{SB}, le morphisme $q$ est un $O(V)$-fibré principal au dessus de l'ouvert $U$. On a donc la suite exacte de groupes suivante:
\begin{equation}
\pi_1(\tilde{U}) \longrightarrow \pi_1(U) \longrightarrow \pi_0(O(V)) \longrightarrow \pi_0(\tilde{U}). 
\end{equation}
Or, $\tilde{U}=\{ w \in W\ |\ \rg(w)=n\}$, donc la codimension de $\tilde{U}$ dans $W$ est $n'-n+1 \geq 2$, et donc $\pi_1(\tilde{U}) \cong  \pi_1(W) \cong \{0\}$. On en déduit que $\pi_1(U) \cong \pi_0(O(V)) \cong \ZZZ_2$ et le résultat en découle.     
\end{proof}

Si $n' \geq n$, alors la variété $W/\!/G$ est normale (\cite[§3.2, Théorème 2]{SB}) et de Gorenstein (\cite[§4.4, Théorème 4]{SB}) puisque $G$ est semi-simple et connexe. 
Si $n'<n$, alors $W/\!/G=S^2(V'^*)$ est un espace affine.\\
On note 
$$N:=\min(n',n) .$$
L'opération de $G'$ dans $W$ induit une opération dans $W/\!/G$ telle que $\nu$ soit $G'$-équivariant. Pour cette opération, on vérifie que $W/\!/G$ se décompose en $N+1$ orbites:
$$U_i:=\left\{  (Q,x) \in  W/\!/G \  \mid \ \rg(Q) = i \right\}$$
pour $i=0,\ldots,N$. Les adhérences de ces orbites sont imbriquées de la façon suivante: 
$$ \{0\}=\overline{U_0} \subset \overline{U_1} \subset \cdots \subset \overline{U_N}=W/\!/G.$$ 
En effet, pour chaque $i=0,\ldots,N-1$, on a $\overline{U_i}=\{ (Q,0) \ \mid \ Q \in S^2(V'^*)^{\leq i} \} \cong S^2(V'^*)^{\leq i}$ et $p_1:\ U_N \rightarrow \{Q \in S^2(V'^*)\ |\ \rg(Q)=N\}$ est $G'$-équivariant donc $U_N$ est l'unique orbite ouverte de $W/\!/G$. La discussion précédente et le corollaire \ref{dimfibreOn} permettent de déduire la

\begin{proposition} \label{ouvertplatitudeSOn}
La dimension de la fibre générique et l'ouvert de platitude de $\nu$ sont donnés dans le tableau suivant:
\begin{center}
\begin{tabular}{|c|c|c|}
  \hline
  configuration & dim. de la fibre générique & ouvert de platitude \\
  \hline
  $n' < n$ &     $n' n -\frac{1}{2}n'(n'+1)$ & $U_{n'} \cup \cdots \cup U_{\max(2 n'-n-1, 0)}$\\
  $n' = n$ &     $\frac{1}{2}n(n-1)$       & $U_{n} \cup U_{n-1} $\\ 
  $n' > n$ &     $\frac{1}{2}n(n-1)$       & $U_{n} $\\ 
  \hline    
\end{tabular}
\end{center}
\vspace*{0.5mm}
\end{proposition}

\begin{corollaire} \label{ouvert_platitudeSOn}
Le morphisme $\nu$ est plat sur $W/\!/G$ tout entier si et seulement si $n \geq 2n'-1$ et dans ce cas $W/\!/G=S^2(V'^*)$. 
\end{corollaire}

Le corollaire qui suit est une conséquence immédiate de la proposition \ref{chow} et du corollaire \ref{ouvert_platitudeSOn}:

\begin{corollaire} \label{cas_facileSOn}
Si $n \geq 2n'-1$, alors $\HH \cong S^2(V'^*)$ et $\gamma$ est un isomorphisme.
\end{corollaire}

\begin{notation}
Si $n'<n$, on note 
$$H:=\left\{ \begin{bmatrix}
M  &0_{n-n',n'} \\
0_{n',n-n'}  & I_{n'} 
\end{bmatrix},\ M \in SO_{n-n'}(k)\right\} \cong SO_{n-n'}(k)$$
qui est un sous-groupe de $G$. 
\end{notation}

\noindent Pour déterminer la fonction de Hilbert de la fibre générique de $\nu$, on a besoin du

\begin{lemme} \label{fibgeneSOn}
La fibre de $\nu$ en un point de $U_N$ est isomorphe à 
$$\left\{
    \begin{array}{ll}
        G/H & \text{ si } n'<n, \\
         G & \text{ si } n' \geq n
    \end{array}
\right.$$ 
\end{lemme}

\begin{proof}
La preuve est analogue à celle du lemme \ref{fibreUnsymp1}.
\end{proof}

\noindent Grâce au lemme \ref{fibgeneSOn} et à la proposition \ref{ouvertplatitudeSOn}, on montre la

\begin{proposition} \label{fctHilbSOn}
La fonction de Hilbert de la fibre générique de $\nu$ est donnée par:
$$\forall M \in \Irr(G),\ h_W(M)=\left\{
    \begin{array}{ll}
       \dim(M^H) \text{ si } n'<n, \\
       \dim(M) \text{ si } n' \geq n.
    \end{array}
\right.$$
\end{proposition}

%% file: SO3.tex
\subsection{Etude du cas \texorpdfstring{$\dim(V)=3$}{n=3}}

Dans toute cette section, on fixe $n'=n=3$. On a $G \cong SO_3(k)$, le quotient $W/\!/G$ est un revêtement double de $S^2(V'^*)$ ramifié en $S^2(V'^*)^{\leq 2}$.
Nous allons démontrer le 

\begin{theoreme} \label{casn3SO3} 
Le schéma $\HH$ est connexe et singulier.
\end{theoreme} 

\begin{remarque}
Si $n' \leq 2$, alors $\HH \cong W/\!/G$ est une variété lisse et $\gamma$ est un isomorphisme d'après le corollaire \ref{cas_facileSOn}. 
\end{remarque}

La connexité de $\HH$ est donnée par la proposition \ref{Hconnexe4} et la non-lissité de $\HH$ est donnée par le corollaire \ref{Hcasn3SO3}. Remarquons que, contrairement aux exemples traités précédement, on ne dispose pas du principe de réduction lorsque $G=SO(V)$ et donc les propriétés géométriques de $\HH$ lorsque $n'>3$ ne se déduisent pas du cas $n'=3$ a priori.
On renvoie au début de la section \ref{threpO3} pour des rappels concernant les représentations de $SO_3(k)$.

\subsubsection{Points fixes de \texorpdfstring{$\HH$}{H} pour l'opération de \texorpdfstring{$B'$}{B'}} \label{sectionptfixeSO3}

Nous allons procéder comme pour les démonstrations des théorèmes \ref{casn3} et \ref{casn3O3}. On commence par déterminer les points fixes de $B'$ dans $\HH$. On a:
\begin{align}
&{k[W]}_1 \cong V' \otimes V(2), \label{kW1SO3}\\
&{k[W]}_2 \cong (S^2(V') \otimes (V(4) \oplus V(0))) \oplus ({\Lambda}^2(V') \otimes V(2)), \label{kW2SO3} 
\end{align}
comme $G'\times G$-modules.

\begin{notation}  \label{espacesdim3SO3}
On note:
\begin{itemize} \renewcommand{\labelitemi}{$\bullet$}
\item $J$ l'idéal engendré par les $G$-invariants homogènes de degré positif de $k[W]$,
\item $D_1:=\left\langle  e_1 \right\rangle \subset V'$ l'unique droite $B'$-stable de $V'$,
\item $D_2:=\left\langle  e_1.e_1 \right\rangle \subset S^2(V')$ l'unique droite $B'$-stable de $S^2(V')$,
\item $I_1$  l'idéal engendré par  $D_1 \otimes V(2) \subset {k[W]}_1$ et par $(S^2(V') \otimes V(0)) \subset {k[W]}_2$, 
\item $I_2$ l'idéal engendré par  $(D_2 \otimes V(4)) \oplus (S^2(V') \otimes V(0)) \oplus ({\Lambda}^2 (V') \otimes V(2)) \subset {k[W]}_2$.
\end{itemize}
\end{notation}

\begin{remarque}
On a $J \cap k[W]_2=S^2(V') \otimes V(0)$ et $J \cap k[W]_3=\Lambda^3(V') \otimes V(0)$ comme $G' \times G$-modules et ces modules engendrent l'idéal $J$.
Les idéaux $I_1$ et $I_2$ sont homogènes, $B' \times G$-stables et contiennent l'idéal $J$.
\end{remarque}

\noindent Nous allons montrer le

\begin{theoreme} \label{pointfixeborel36}
Les idéaux $I_1$ et $I_2$ sont les deux points fixes de $B'$ dans $\HH$. 
\end{theoreme}

\noindent Pour ce faire, nous aurons besoin du

\begin{lemme} \label{decompo_SiSO3}
Pour chaque $i \geq 1$, on a les isomorphismes de $G$-modules suivants:
\begin{enumerate}
\item  $S^i(V) \cong V(2i) \oplus V(2i-4) \oplus \cdots \oplus V(0) \text{ si $i$ est pair,}$   
\item  $S^i(V) \cong V(2i) \oplus V(2i-4) \oplus \cdots \oplus V(2) \text{ si $i$ est impair,}$ 
\item  $S^{i,1}(V) \cong V(2i) \oplus V(2i-2) \oplus \cdots \oplus V(2)$.
\end{enumerate}
\end{lemme}

\begin{proof}
C'est une conséquence immédiate du lemme \ref{decompo_Si2}.
\end{proof}

\begin{proof}[\textbf{Preuve du théorème \ref{pointfixeborel36}}]
On raisonne comme dans la preuve du théorème \ref{pointfixeborel} en considérant $I_Z$ un point fixe de $B'$ et en étudiant $k[W]/I_Z$ composante par composante:\\
$\bullet$ Composante de degré $0$: on a bien sûr $I_Z \cap {k[W]}_0=\{0\}$.\\  
$\bullet$ Composante de degré $1$: on utilise la décomposition (\ref{kW1SO3}).\\
Si $I_Z \cap {k[W]}_1={k[W]}_1$, alors $k[W]/ I_Z=V(0)$ ne donne pas la décomposition souhaitée comme $G$-module.\\ 
Si $I_Z \cap {k[W]}_1$ contient deux copies de $V(2)$, alors $k[W]/I_Z$ est un quotient de $k[V(2)]$ comme $G$-module. Mais $S^2(V(2))=V(4) \oplus V(0)$ et ${k[W]}_0/(I_Z \cap {k[W]}_0)=V(0)$ donc nécessairement ${k[W]}_2/(I_Z \cap {k[W]}_2) \subset V(4)$. Il s'ensuit que, pour chaque $i \geq 0$,  ${k[W]}_i/(I_Z \cap {k[W]}_i) \subset V(2i)$ et donc $I_Z$ ne peut pas avoir la bonne fonction de Hilbert. Donc $I_Z \cap {k[W]}_1$ contient au plus une copie de $V(2)$.\\

\underline{$1^{er}$ cas}: $I_Z \cap {k[W]}_1=\{0\}.$ \\
$\bullet$ Composante de degré $2$: on utilise la décomposition (\ref{kW2SO3}). \\
On a ${k[W]}_2 \cap I_Z \supset (S^2(V') \otimes  V(0)) \oplus ({\Lambda}^2 (V') \otimes V(2))$. En effet, le $G$-module $k[W]/I_Z$ contient déjà une copie de la représentation triviale $V(0)$ (l'image des constantes), il ne peut donc pas en contenir d'autres. De même, comme $I_Z \cap {k[W]}_1=\{0\}$, le $G$-module ${k[W]}/I_Z$ contient déjà $3$ copies de la représentation standard $V(2)$ qui est de dimension $3$ et donc il ne peut pas en contenir d'autres. Ensuite, ${k[W]}_2$ contient $6$ copies de $V(4)$ qui est de dimension $5$, donc nécessairement, $I_Z \cap {k[W]}_2$ contient une copie de $V(4)$. Comme $I_Z \cap {k[W]}_2$ est $B'$-stable, il contient $D_2 \otimes V(4)$ car $D_2$ est l'unique droite $B'$-stable de $S^2(V')$. Donc $I_Z$ contient $I_2$. Le lemme qui suit montre que cette inclusion est nécessairement une égalité:

\begin{lemme} \label{pfixeI1}
L'idéal $I_2$ a pour fonction de Hilbert $h_W$.
\end{lemme}

\begin{proof}[\textbf{Preuve du lemme:}]
On note $J_2 \subset I_2$ l'idéal engendré par $(S^2(V') \otimes V(0)) \oplus ({\Lambda}^2 (V') \otimes V(2)) \subset {k[W]}_2$, alors $k[W]_2/(J_2 \cap k[W]_2) \cong S^2(V') \otimes V(4)$. Mais $k[W] \cong k[V(2)^{\oplus 3}] \cong k[V(2)]^{\otimes 3}$ comme $G$-module et $V(2) \otimes V(2) \cong V(4) \oplus V(2) \oplus V(0)$, donc deux copies quelconques de la représentation standard de $G$ dans ${k[W]}_1$ sont orthogonales dans $k[W]_2/J_2$, c'est-à-dire que l'image du $G$-module $V(2) \otimes V(2) \subset k[W]_2$ dans $k[W]/J_2$ est isomorphe à $V(4)$.  
On déduit alors de \cite[Lemme 4.1]{Br5} que l'on a un isomorphisme de $G' \times G$-modules 
$$k[W]/J_2 \cong \bigoplus_{i \geq 0} S^i(V') \otimes V(2i).$$ 
Donc $(I_2 \cap k[W]_2)/(J_2 \cap k[W]_2) \cong D_2 \otimes V(4)$ comme $B' \times G$-module. Soit $U$ le sous-groupe unipotent de $G$ défini dans la section \ref{lesdiffsituations}. Alors ${(I_2/J_2)}^{U}$ est un idéal de ${(k[W]/J_2)}^{U}$ engendré par un unique élément de ${({k[W]}_2/ {k[W]}_2 \cap J_2)}^{U}$ et donc la multiplicité du $G$-module $V(2i)$ dans $k[W]/I_2$ est:\\
$$\dim(S^i(V'))-\dim(S^{i-2}(V'))=2i+1=\dim(V(2i)).$$
Et donc l'idéal $I_2$ a la bonne fonction de Hilbert. 
\end{proof}


\underline{$2^{nd}$ cas}: $I_Z \cap {k[W]}_1=V(2)$.\\
Comme $I_Z \cap {k[W]}_1$ est $B'$-stable, on a $I_Z \cap {k[W]}_1 = D_1 \otimes V(2)$ car $D_1$ est l'unique droite $B'$-stable de $V'$. 
On a nécessairement $I_Z \cap {k[W]}_2 \supset S^2(V') \otimes  V(0)$ d'après l'étude de la composante de degré $0$ de $I_Z$. Donc $I_Z$ contient $I_1$. Le lemme qui suit montre que cette inclusion est en fait une égalité:  

\begin{lemme} \label{boubouI1SO3}
L'idéal $I_1$ a pour fonction de Hilbert $h_W$.
\end{lemme}

\noindent \textbf{ \textit{Preuve du lemme:}}
On note $V'':=V'/D_1$ et $W':=\Hom(V'',V)$ alors $\dim(V'')=2$ et 
$${k[W']}_2 \cong (S^2(V'') \otimes (V(4) \oplus V(0))) \oplus ({\Lambda}^2 (V'') \otimes V(2))$$ 
comme $B' \times G$-module.\\
On note $J_1$ l'idéal de $k[W']$ engendré par $S^2(V'') \otimes V(0) \subset {k[W']}_2$, alors $k[W]/{I_1} \cong k[W']/{J_1}$ comme $B' \times G$-module. 
D'après le lemme \ref{decompo_SiSO3},  pour chaque $i \geq 2$, on a 
\begin{equation} \label{minoration}
{k[W']}_i/({J_1 \cap {k[W']}_i}) \supset (S^i(V'') \otimes V(2i)) \oplus (S^{i-1,1}(V'') \otimes V(2i-2)). 
\end{equation}
En effet, l'idéal $J_1$ est engendré par les copies de $V(0)$ dans ${k[W']}_2$ et donc, si $p$ est un entier tel que $V(2p) \subset J_1 \cap {k[W]}_i$, alors nécessairement $p \leq i-2$. On dispose donc, pour chaque $i \geq 1$, d'une minoration de la dimension de ${k[W']}_i/({J_1 \cap {k[W']}_i})$ par 
$$Q(i):=\dim((S^i(V'') \otimes V(2i)) \oplus (S^{i-1,1}(V'') \otimes V(2i-2)))=4i^2+2.$$
Ensuite, un calcul direct avec \cite[Macaulay2]{Mac2} donne le polynôme de Hilbert $P_{J_1}$ et la fonction de Hilbert $f_{J_1}$ de l'idéal $J_1$:
$$
 \left\{
    \begin{array}{l}
        P_{J_1}(X)=4X^2+2,\\
        \forall n \geq 1,\ f_{J_1}(n)=P_{J_1}(n) \text{ et } f_{J_1}(0)=1.
    \end{array}
\right.
$$
Il s'ensuit que l'inclusion (\ref{minoration}) est en fait une égalité. On est maintenant en mesure de calculer les multiplicités de chaque $G$-module irréductible qui apparaît dans $k[W]/{I_1}$:
pour chaque $i \geq 0$, la représentation $V(2i)$ apparaît $i+1$ fois dans ${k[W]}_i/({I_1 \cap {k[W]}_i})$ et $i$ fois dans ${k[W]}_{i+1}/({I_1 \cap {k[W]}_{i+1}})$. Et donc la multiplicité de $V(2i)$ dans $k[W]/{I_1}$ est $2i+1=\dim(V(2i))$.
\end{proof}

\begin{remarque}
On vérifie que $\Stab_{G'}(I_1)=\Stab_{G'}(I_2)$ est le sous-groupe parabolique de $G'$ qui stabilise la droite $D_1 \subset V'$. Il est donc maximal, de dimension $7$, et les deux orbites fermées de $\HH$ sont isomorphes à $\PP^2$. 
\end{remarque}

\subsubsection{Connexité de \texorpdfstring{$\HH$}{H}} \label{grobner}

\noindent Nous allons montrer la 
\begin{proposition} \label{Hconnexe4}
Le schéma $\HH$ est connexe.
\end{proposition}

D'après le lemme \ref{fixespoints}, il suffit de montrer que les deux points fixes $I_1$ et $I_2$ de $B'$ sont dans $\HHp$. Pour ce faire, nous allons procéder comme dans la section \ref{reductibilitéO3} et construire deux familles plates d'idéaux ${(L_t)}_{t \in \Aff}$ de $k[W]$ telles que la première famille vérifie $L_0=I_1$ et la seconde famille vérifie $L_0=I_2$.

On note $w \in W$ sous la forme $w=\begin{bmatrix}
x_1 &y_1 & z_1 \\
x_2 &y_2 & z_2 \\
x_3 &y_3 & z_3 
\end{bmatrix}$ et on identifie $k[W]$ à $k[x_1,x_2,x_3,y_1,y_2,y_3,z_1,z_2,z_3]$. On commence par expliciter des bases de certains $B' \times G$-modules qui apparaissent dans $k[W]$:\\ 
$
 \left.
    \begin{array}{l}
        x_1 \\
        x_2 \\
        x_3 
      \end{array}
\right \} \text{ est une base de } D_1 \otimes V(2) \subset k[W]_1,\\
\left.
    \begin{array}{l}
        x_1^2+x_2^2+x_3^2 \\
        y_1^2+y_2^2+y_3^2 \\
        z_1^2+z_2^2+z_3^2 \\
        x_1 y_1+x_2 y_2+x_3 y_3\\
        x_1 z_1+x_2 z_2+x_3 z_3\\
        z_1 y_1+z_2 y_2+z_3 y_3
      \end{array}
\right \} \text{ est une base de } S^2(V') \otimes V(0) \subset k[W]_2,\\
\left.
    \begin{array}{l}
        x_2 y_3-x_3 y_2 \\
        x_1 y_3-x_3 y_1 \\
        x_1 y_2-x_2 y_1 \\
        x_2 z_3-x_3 z_2 \\
        x_1 z_3-x_3 z_1 \\
        x_1 z_2-x_2 z_1 \\
        z_2 y_3-z_3 y_2 \\
        z_1 y_3-z_3 y_1 \\
        z_1 y_2-z_2 y_1 \\
     \end{array}
\right \} \text{ est une base de } {\Lambda}^2(V') \otimes V(2) \subset k[W]_2,\\
\left.
    \begin{array}{l}
        x_1^2 \\
        x_2^2\\
        x_3^2\\
        x_1 x_2 \\
        x_2 x_3\\
        x_1 x_3 
      \end{array}
\right \} \text{ est une base de }  D_2 \otimes S^2(V)  \subset k[W]_2.$\\

On en déduit des générateurs des idéaux $I_1$ et $I_2$:
\begin{align*}
I_1=(&x_1,\ x_2,\ x_3,\ y_1^2+y_2^2+y_3^2,\ z_1^2+z_2^2+z_3^2,\ z_1 y_1+z_2 y_2+z_3 y_3),\\
I_2=( &x_1^2-x_3^2,\ x_2^2,\ x_1 x_2 ,\ x_2 x_3,\ x_1 x_3,\\
      & x_1^2+x_2^2+x_3^2,\ y_1^2+y_2^2+y_3^2,\ z_1^2+z_2^2+z_3^2, \\
      & x_1 y_1+x_2 y_2+x_3 y_3,\ x_1 z_1+x_2 z_2+x_3 z_3,\ z_1 y_1+z_2 y_2+z_3 y_3,\\
      & x_2 y_3-x_3 y_2,\ x_1 y_3-x_3 y_1,\ x_1 y_2-x_2 y_1, \\
      & x_2 z_3-x_3 z_2,\ x_1 z_3-x_3 z_1,\ x_1 z_2-x_2 z_1, \\
      & z_2 y_3-z_3 y_2,\ z_1 y_3-z_3 y_1,\ z_1 y_2-z_2 y_1).
\end{align*}

D'après la proposition \ref{chow}, le morphisme de Hilbert-Chow est un isomorphisme au dessus de $U_3 \subset W/\!/G$, donc il existe un unique $Z_{Id} \in \HH$ tel que $\gamma(Z_{Id})=(Id,1)$. On note $L$ l'idéal de $Z_{Id}$ dans $k[W]$. D'après le lemme \ref{diagcom}, on a:
\begin{align*}
L=(&x_1^2+x_2^2+x_3^2-1,\ y_1^2+y_2^2+y_3^2-1,\ z_1^2+z_2^2+z_3^2-1, \\
   &x_1 y_1+x_2 y_2+x_3 y_3, x_1 z_1+x_2 z_2+x_3 z_3, z_1 y_1+z_2 y_2+z_3 y_3, \\
   &x_1(y_2 z_3-y_3 z_2)-x_2(y_1 z_3-y_3 z_1)+x_3(y_1 z_2-y_2 z_1)-1).   
\end{align*}
On calcule une base de Gröbner de l'idéal $L$ avec \cite[Macaulay2]{Mac2}:
$$
 \left.
    \begin{array}{l}
      y_3 z_2-y_2 z_3+x_1\\
      y_3 z_1-y_1 z_3-x_2\\
      y_2 z_1-y_1 z_2+x_3\\
      x_3 z_2-x_2 z_3-y_1\\
      x_3 z_1-x_1 z_3+y_2\\
      x_2 z_1-x_1 z_2-y_3\\
      x_3 y_2-x_2 y_3+z_1\\
      x_3 y_1-x_1 y_3-z_2\\
      x_2 y_1-x_1 y_2+z_3\\
      y_1^2+y_2^2+y_3^2-1\\
      z_1^2+z_2^2+z_3^2-1\\
      x_2^2+y_2^2+z_2^2-1\\
      x_3^2+y_3^2+z_3^2-1\\
      y_1 z_1+y_2 z_2+y_3 z_3\\
      x_1 z_1+x_2 z_2+x_3 z_3\\
      x_1 y_1+x_2 y_2+x_3 y_3\\
      x_2 x_3+y_2 y_3+z_2 z_3\\
      x_1 x_3+y_1 y_3+z_1 z_3\\
      x_1 x_2+y_1 y_2+z_1 z_2\\
      x_1^2-y_2^2-y_3^2-z_2^2-z_3^2+1  
     \end{array}
\right \} \text{ base de Gröbner de l'idéal $L$.} 
$$
On fixe $(n_1,n_2,n_3) \in {\ZZZ}^3$ et soit $\theta_{n_1,n_2,n_3}$ un sous-groupe à un paramètre du tore $T'$ comme défini par (\ref{ssgroupeparam2}). 
Pour tout $t \in \Gm$, on définit alors l'idéal $L_t$ comme dans la section \ref{reductibilitéO3}.

Pour démontrer la proposition \ref{Hconnexe4}, il suffit de trouver deux triplets $(n_1,n_2,n_3) \in {\ZZZ}^3$ tels que l'on obtienne  $L_0=I_1$ dans un cas et $L_0=I_2$ dans l'autre cas.\\
On pose $(n_1,n_2,n_3)=(-3,-1,-1)$, alors pour chaque $t \in \Gm$, on obtient les générateurs suivants de l'idéal $L_t$: 
$$
 \left.
    \begin{array}{l}
      t(y_3 z_2-y_2 z_3)+x_1\\
      t(y_3 z_1-y_1 z_3)-x_2\\
      t(y_2 z_1-y_1 z_2)+x_3\\
      x_3 z_2-x_2 z_3-t^3 y_1\\
      x_3 z_1-x_1 z_3+t^3 y_2\\
      x_2 z_1-x_1 z_2- t^3 y_3\\
      x_3 y_2-x_2 y_3+t^3 z_1\\
      x_3 y_1-x_1 y_3-t^3 z_2\\
      x_2 y_1-x_1 y_2+t^3 z_3\\
      y_1^2+y_2^2+y_3^2-t^2\\
      z_1^2+z_2^2+z_3^2-t^2\\
      x_2^2+t^4 y_2^2+t^4 z_2^2-t^6\\ 
      x_3^2+t^4 y_3^2+t^4 z_3^2-t^6\\ 
      y_1 z_1+y_2 z_2+y_3 z_3\\ 
      x_1 z_1+x_2 z_2+x_3 z_3\\ 
      x_1 y_1+x_2 y_2+x_3 y_3\\ 
      x_2 x_3+t^4(y_2 y_3+z_2 z_3)\\ 
      x_1 x_3+t^4(y_1 y_3+z_1 z_3)\\ 
      x_1 x_2+t^4(y_1 y_2+z_1 z_2)\\ 
      x_1^2+t^4(-y_2^2-y_3^2-z_2^2-z_3^2)+t^6  
     \end{array}
\right \} \text{ générateurs de l'idéal $L_t$.}  
$$
Si l'on pose $t=0$, alors on retrouve les générateurs de l'idéal $I_1$, autrement dit $L_0=I_1$. De même, si l'on considère $(n_1,n_2,n_3)=(-3,-2,-2)$, alors on obtient une famille plate d'idéaux ${(L_t)}_{t \in \Aff}$ telle que $L_0=I_2$. La proposition \ref{Hconnexe4} s'ensuit.

\subsubsection{Dimensions des espaces tangents de \texorpdfstring{$\HH$}{H} en \texorpdfstring{$Z_1$}{Z1} et \texorpdfstring{$Z_2$}{Z2}}  \label{tangSO3}

Pour $i \in \{1,2\}$, on note $R_i:=k[W]/I_i$ et $Z_i:=\Spec(k[W]/I_i)$. Nous allons montrer la

\begin{proposition} \label{dimTangentSO3} 
On a $\dim(T_{Z_1} \HH)=6$ et $\dim(T_{Z_2} \HH)=8$. 
\end{proposition}

$\bullet$ Montrons que $\dim(T_{Z_1} \HH)=6$. 
On vérifie que
$$I_1/I_1^2 \cong ((D_1 \otimes V(2)) \oplus (S^2(V'/D_1) \otimes V(0))) \otimes R_1 $$
comme $R_1,B' \times G$-module. Le résultat découle alors du lemme \ref{dimRmod}.\\

$\bullet$ Montrons que $\dim(T_{Z_2} \HH)=8$. 
On procède comme dans la preuve de la proposition \ref{dimTangent2}. Soient
$$
\left\{
    \begin{array}{l}
f_1:=x_1^2+x_2^2+x_3^2,\\
f_2:=y_1^2+y_2^2+y_3^2,\\
f_3:=z_1^2+z_2^2+z_3^2,\\
f_4:=x_1 y_1+x_2 y_2+x_3 y_3,\\
f_5:=y_1 z_1+y_2 z_2+y_3 z_3,\\
f_6:=x_1 z_1+x_2 z_2+x_3 z_3,\\
g_{11}:=x_2 y_3-x_3 y_2,\\
g_{12}:=x_3 y_1-x_1 y_3,\\
g_{13}:=x_1 y_2-x_2 y_1,\\
g_{21}:=x_2 z_3-x_3 z_2,\\
g_{22}:=x_3 z_1-x_1 z_3,\\
g_{23}:=x_1 z_2-x_2 z_1,\\
g_{31}:=z_3 y_2-y_3 z_2,\\
g_{32}:=z_1 y_3-z_3 y_1,\\
g_{33}:=z_2 y_1-z_1 y_2,\\
h_1:=x_1^2-x_3^2,\\
h_2:=x_1 x_2,\\
h_3:=x_2^2,\\
h_4:=x_2 x_3,\\
h_5:=x_1 x_3,
   \end{array}
\right.
$$
des générateurs de l'idéal $I_2$ et soit $N \subset k[W]$ le $B' \times G$-module engendré par ces générateurs. 

\begin{lemme}  \label{inek}
$\dim(\Hom_{R_2}^G(I_2/I_2^2,R_2)) \geq 8.$
\end{lemme}

\begin{proof}
D'après \cite[Macaulay2]{Mac2}, une famille de générateurs du $R_2$-module $\Hom_{R_2}(I_2/I_2^2,R_2)$ est donnée dans la table \ref{table8}.
\begin{table}[ht]
\begin{center}
\includegraphics[scale=0.8]{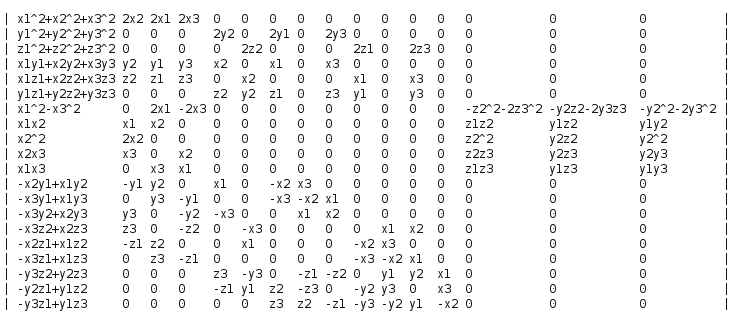}
\end{center}
\caption{\label{table8} Générateurs du $R_2$-module $\Hom_{R_2}(I_2/I_2^2,R_2)$}
\end{table}

\noindent On note $\Phi_i$ le morphisme donné par la $i+1$-ème colonne de cette matrice. On vérifie que les huit morphismes suivants sont $G$-équivariants et linéairement indépendants:
\begin{itemize} \renewcommand{\labelitemi}{$\bullet$}
\item $\Phi_{7}$,
\item $\Phi_{10}$,
\item $\Phi_{12}$,
\item $\Phi_{13}$,
\item $\Phi_{14}$,
\item $\Phi_{15}$,
\item $y_2 \Phi_{1}+y_1 \Phi_2+y_3 \Phi_3$,
\item $z_2 \Phi_{1}+z_1 \Phi_2+z_3 \Phi_3$.
\end{itemize}  
Il s'ensuit que ${\Hom}_{R_2}^G(I_2/I_2^2,R_2)$ est de dimension au moins $8$. 
\end{proof}

On reprend les notations de la section \ref{ConnexitéetTangence}. On a $\dim(N)=20$ et donc, d'après le lemme \ref{InegTang}, on a $\dim(T_{Z_0} \HH)=20- \rg(\rho^*)$. D'après le lemme \ref{inek}, il suffit de montrer que $\rg(\rho^*) \geq 12$ pour montrer la proposition \ref{dimTangentSO3}.\\
Pour $i=1,\ldots, 6$, on définit ${\psi}_i \in {\Hom}_{R_2}^G(R_2 \otimes N,R_2)$ par 
$$
\left\{
    \begin{array}{ll}
        {\psi}_i(f_j \otimes 1)={\delta}_i^j, &\text{ pour }  j=1,\ldots,6,\\
        {\psi}_i( h_j \otimes 1)=0, &\text{ pour }   j=1,\ldots,5, \\
        {\psi}_i( g_{jk} \otimes 1)=0, &\text{ pour }   j,k=1,\ldots,3,
    \end{array}
\right.
$$
où ${\delta}_i^j$ est le symbole de Kronecker. 
De même, pour $1 \leq j$,$k \leq 3$, on définit ${\phi}_{jk} \in {\Hom}_{R_2}^G(R_2 \otimes N,R_2)$ par 
$$
\left\{
    \begin{array}{ll}
        {\phi}_{jk}(f_i \otimes 1)=0, &\text{ pour }  i=1,\ldots,6,\\
        {\phi}_{jk}( h_i \otimes 1)=0, &\text{ pour }   i=1,\ldots,5, \\
        {\phi}_{jk}( g_{p,l} \otimes 1)= &\left\{
    \begin{array}{ll}
        x_l {\delta}_j^p &\text{ si } k=1,\\
        y_l {\delta}_j^p &\text{ si }  k=2,\\
        z_l {\delta}_j^p &\text{ si }  k=3.\\
     \end{array}
            \right.
    \end{array}
\right.
$$ 
On identifie ces morphismes à des éléments de $\Hom^G(N,R)$ via l'isomorphisme $\Hom^G(N,R) \cong \Hom_R^G(R \otimes N,R)$.
Enfin, soient
$$
 \left \{
    \begin{array}{l}
r_1:=-y_1 \otimes f_1+2x_1 \otimes f_4-y_1 \otimes h_1-2y_2 \otimes h_2+y_1 \otimes h_3-2y_3 \otimes h_5,\\
r_2:=y_2 \otimes f_1-y_2 \otimes h_1-y_2 \otimes h_3 -2 y_3 \otimes h_4+2 x_3 \otimes g_{11},\\
r_3:=z_1 \otimes f_1-z_1 \otimes h_1-z_1 \otimes h_3 -2 z_3 \otimes h_4+2 x_3 \otimes g_{21},\\  
r_4:=-z_2 \otimes f_4+y_2 \otimes f_6 +z_3 \otimes g_{11}-y_3 \otimes g_{21}+x_1 \otimes g_{33},\\
r_5:=x_3 \otimes f_2-y_3 \otimes f_4 +y_2 \otimes g_{11}+y_1 \otimes g_{12},\\   
r_6:=x_3 \otimes f_3-z_3 \otimes f_6 +z_2 \otimes g_{21}+z_1 \otimes g_{22},\\ 
r_7:=x_2 \otimes f_5-y_2 \otimes f_6 -z_3 \otimes g_{11}+z_1 \otimes g_{13},
       \end{array}
      \right. $$
des relations entre les générateurs du $R_2$-module $I_2/I_2^2$. Le lemme qui suit achève la preuve de la proposition \ref{dimTangentSO3}:

\begin{lemme} 
$\rg(\rho^*) \geq 12.$
\end{lemme}

\begin{proof}
Les ${\psi}_i$ et les ${\phi}_{jk}$ forment une famille libre de ${\Hom}_{R_2}^G(R_2 {\otimes} N,R_2)$. Nous allons voir que la famille 
$$\left\{ \rho^*({\psi}_i),\ i=1,\ldots,6 \right \} \cup \left \{  \rho^*({\phi}_{jk}),\ j,k \in \{ 1,2,3\} \right\}$$ 
engendre un sous-espace de dimension au moins $12$ dans  
${\Hom}_{R_2}^G(R_2 \otimes \RRR,R_2)$.\\
Soient ${\lambda}_1,\ldots,{\lambda}_6, \gamma_{11}, \ldots, \gamma_{33} \in k$ tels que 
\begin{equation}  \label{relaSO3}
\sum_{i=1}^{6} \lambda_i \rho^*(\psi_i) + \sum_{j,k=1}^{3}  \gamma_{jk} \rho^*(\phi_{jk})=0.
\end{equation}
On évalue (\ref{relaSO3}) en $r_1 \otimes 1$, on obtient:
$$ (2 {\lambda}_4 x_1 -{\lambda}_1 y_1=0)  \Rightarrow ({\lambda}_1={\lambda}_4=0).$$
On évalue (\ref{relaSO3}) en $r_2\otimes 1$, on obtient:
$$(2 \gamma_{11} x_3 x_1 + 2 \gamma_{12} x_3 y_1+ 2 \gamma_{13} x_3 z_1=0) \Rightarrow (\gamma_{12}=\gamma_{13}=0). $$
On évalue (\ref{relaSO3}) en $r_3\otimes 1$, on obtient:
$$(2 \gamma_{21} x_3 x_1 + 2 \gamma_{22} x_3 y_1+ 2 \gamma_{23} x_3 z_1=0) \Rightarrow (\gamma_{22}=\gamma_{23}=0). $$
On évalue (\ref{relaSO3}) en $r_4\otimes 1$, on obtient:
$$(y_2 {\lambda}_6 + \gamma_{11} z_3 x_1 - \gamma_{21} y_3 x_1+ \gamma_{31} x_3 x_1+ \gamma_{32} y_3 x_1+ \gamma_{33} z_3 x_1=0) \Rightarrow ({\lambda}_6=0,\ \gamma_{11}=-\gamma_{33} \text{ et } \gamma_{21}=\gamma_{32}).$$ 
De même, si on évalue (\ref{relaSO3}) en $r_5\otimes 1$, $r_6\otimes 1$ et $r_7\otimes 1$, on obtient ${\lambda}_2={\lambda}_3={\lambda}_5=0$. 
\end{proof} 

\begin{corollaire} \label{Hcasn3SO3}
Le schéma $\HH$ est lisse en $Z_1$ et singulier en $Z_2$.
\end{corollaire}

\begin{proof}
On a vu dans la section \ref{grobner} que $\HH$ est connexe et en fait, on a même montré que $Z_1$ et $Z_2$ appartiennent à $\HHp$. Donc $\HH$ est singulier en $Z_i$ si et seulement $\dim(T_{Z_i}\HH) > \dim(\HHp)$. Or, on a $\dim(\HHp)=6$ et la proposition \ref{dimTangentSO3} permet alors de conclure.
\end{proof}

%% file: position_probleme_Sp2n.tex
\section{Cas de \texorpdfstring{$Sp(V)$}{Sp(V)} opérant dans \texorpdfstring{$V^{\oplus  n'}$}{n'V}}  \label{schcasSymp}

On se place dans la situation $5$: on suppose que $n$ est pair et on a $G:=Sp(V)$, $G':=GL(V')$, $W:={\Hom}(V',V)$ et l'opération de $G' \times G$ dans $W$ est donnée par (\ref{actionSLLn}).\\
Lorsque $n=2$, on a $G \cong SL_2(k)$ et alors l'étude du schéma de Hilbert invariant pour l'opération de $G$ dans $W$ a déjà été effectuée dans la section \ref{casSln}.

\subsection{Etude du morphisme de passage au quotient}  \label{etudequotientSp2n}

Les résultats essentiels de  cette section sont les propositions \ref{descriptiongeofibSpn} et \ref{ouvertplatitudeSpn} qui décrivent les fibres et l'ouvert de platitude de $\nu$. On suit dans cette section le même cheminement que dans les sections \ref{description_quotient} et \ref{description_quotientOn}.

D'après le premier théorème fondamental pour $Sp(V)$ (voir \cite[§11.2.1]{Pro}) l'algèbre des invariants $k[W]^G$ est engendrée par les $\left\langle  i\ |\ j \right\rangle$, où pour chaque couple $(i,j)$, $\ 1\leq i < j \leq n'$, on définit la forme bilinéaire antisymétrique $\left\langle  i\ |\ j \right\rangle$ sur $W \cong V^{\oplus n'}$ par: 
\begin{equation} \label{GLinvariants}
\forall v_1,\ldots,v_{n'}\in V,\ \left\langle  i\ |\ j \right\rangle:\ (v_1,\ldots,v_{n'}) \mapsto \Omega(v_i,v_j)
\end{equation}
où $\Omega$ est la forme symplectique définie par (\ref{defFormSymp}).
On note $A:=\begin{bmatrix}
0  & 1 \\
-1 & 0 
\end{bmatrix}$ et $J:=\begin{bmatrix}
A & &\\
 & \ddots & \\
 && A 
\end{bmatrix} \in \MM_n(k)$ la matrice avec $\frac{n}{2}$ blocs $A$ sur la diagonale et des $0$ partout ailleurs. Alors $J$ est la matrice de $\Omega$ relativement à la base $\BB$. On a le morphisme naturel $G' \times G$-équivariant
\begin{equation} \label{appznatSp}
\Hom(V',V)    \rightarrow  \Hom(\Lambda^2(V'),\Lambda^2(V)),\ w \mapsto \Lambda^2(w).
\end{equation}
D'après \cite[§17.2]{FH}, on a $\Lambda^2(V) \cong V_0 \oplus \Gamma_{\epsilon_1+\epsilon_2}(V)$ comme $G$-module si $n \geq 4$ (sinon, $\Lambda^2(V) \cong V_0$) et la représentation triviale $V_0$ est engendrée par la forme symplectique $\Omega$. Le morphisme de passage au quotient $\nu$ est obtenu en composant le morphisme (\ref{appznatSp}) et le morphisme $G$-invariant 
$$  \Hom(\Lambda^2(V'),\Lambda^2(V))  \rightarrow   \Hom(\Lambda^2(V'),V_0) \cong \Lambda^2(V'^*)$$ 
induit par la projection $\Lambda^2(V) \rightarrow V_0$. On a donc
$$\begin{array}{lrcl}
\nu:  &\Hom(V',V)   & \rightarrow  & \Lambda^2(V'^*) \\
        & w  & \mapsto      &  \leftexp{t}{w}J w 
\end{array}$$ 
et 
$$W/\!/G=\Lambda^2(V'^*)^{\leq n}:=\left\{ Q \in  \Lambda^2(V'^*) \  \mid \ \rg(Q) \leq n \right\} $$
est une variété déterminantielle antisymétrique.

\begin{remarque}
Si $Q \in W/\!/G$, alors $Q$ s'identifie naturellement à une application bilinéaire antisymétrique et donc son rang est pair.
\end{remarque}

Si $n' \leq n$, alors $W/\!/G=\Lambda^2(V'^*)$ est un espace affine. Sinon, c'est une variété normale (\cite[§3.2, Théorème 2]{SB}), de dimension $n'n-\frac{1}{2}n(n+1)$, singulière le long du fermé $\Lambda^2(V'^*)^{\leq n-2}$ (\cite[6.4]{Wey}) et de Gorenstein (\cite[§4.4, Théorème 4]{SB}) car $G$ est semi-simple et connexe. On pose 
$$N:=\min \left( E \left(\frac{n'}{2}\right),\frac{n}{2} \right).$$
L'opération de $G'$ dans $W$ induit une opération dans $W/\!/G$ telle que $\nu$ soit $G'$-équivariant: 
$$\forall Q \in W/\!/G \subset \Lambda^2(V'^*),\ \forall g' \in G',\ g'.Q=\leftexp{t}{g'^{-1}} Q g'^{-1}  .$$
Pour cette opération, $W/\!/G$ se décompose en $N+1$ orbites:
$$U_i:=\left\{  Q \in  \Lambda^2(V'^*) \  \mid \ \rg(Q) = 2i \right\}$$
pour $i=0,\ldots,N$. Les adhérences de ces orbites sont imbriquées de la façon suivante: 
$$ \{0\}=\overline{U_0} \subset \overline{U_1} \subset \cdots \subset \overline{U_N}=W/\!/G  .$$ 
En effet, pour chaque $i=0,\ldots,N$, on a $\overline{U_i}=\Lambda^2(V'^*)^{\leq i}$. En particulier, l'orbite $U_N$ est un ouvert dense de $W/\!/G$. 

Dans \cite{KS}, Kraft et Schwarz montrent que $\NNN(W,G)$ est toujours irréductible et réduit (\cite[Theorem 9.1]{KS}). Nous allons retrouver le fait que le nilcône $\NNN(W,G)$ est irréductible et déterminer sa dimension. Soit $m \in \NN$, on note $\IG(m,V)$ la variété (projective) des sous-espaces isotropes pour $\Omega$ de dimension $m$ dans $V$. On rappelle que $\Omega$ est non dégénérée, donc la dimension de n'importe quel sous-espace de $V$ totalement isotrope maximal pour l'inclusion est $\frac{n}{2}$. Il s'ensuit que la variété $\IG(m,V)$ existe si et seulement si $m \in \{0,\ldots,\frac{n}{2}\}$, et dans ce cas 
$$\dim \left(\IG\left(m,V\right) \right)=m(n-m)-\frac{1}{2}m(m-1)  .$$ 
La variété $\IG(m,V)$ est homogène sous $G$.

\begin{proposition} \label{nilconeSpn}
Le nilcône $\NNN(W,G)$ est une variété de dimension
\begin{itemize} \renewcommand{\labelitemi}{$\bullet$}
\item $nn'-\frac{1}{2}n'(n'-1)$ si $n' \leq n(\Omega)$,
\item $\frac{1}{2}nn'+\frac{1}{8}n^2+\frac{1}{4}n$ sinon.
\end{itemize}
\end{proposition}   

\begin{proof}
La preuve est analogue à celle de la proposition \ref{nilcôneOn}.
\end{proof}  

On s'intéresse maintenant à la description géométrique des fibres de $\nu$ au dessus de chaque orbite $U_i$.  

\begin{notation} 
Pour $0 \leq r \leq N$, on note
$$J_{2r} =\begin{bmatrix}
A & 0& \cdots & \cdots & \cdots & \cdots & 0\\
0 & A & &&&& \vdots \\
\vdots && \ddots &&&& \vdots \\
\vdots &&&A&&& \vdots \\
\vdots &&&&0&& \vdots \\
\vdots &&&&& \ddots & \vdots \\
0 & \cdots & \cdots & \cdots & \cdots & \cdots & 0
\end{bmatrix}$$
la matrice avec $r$ blocs $A$ sur la diagonale et des $0$ partout ailleurs. La matrice $J_{2r}$ est antisymétrique de rang $2r$ et donc s'identifie naturellement à un élément de $U_r$. 
\end{notation}

On fixe $r \in \{0, \ldots,N\}$ et on définit $w_r:= \begin{bmatrix}
I_{2r} &0_{2r,n'-2r} \\
0_{n-2r,2r}   &0_{n-2r,n'-2r} \end{bmatrix} \in W$ et $G_r$ le stabilisateur de $w_r$ dans $G$. 
On vérifie alors que
$$G_r =\left\{ \begin{bmatrix}
I_{2r}   &0 \\
0     &M \end{bmatrix},\ M\in Sp_{n-2r}(k) \right\} \cong Sp_{n-2r}(k)$$  
et $V=2r V_0 \oplus E_r$ comme $G_r$-module, où $E_r$ désigne la représentation standard de $G_r$ et $V_0$ la représentation triviale de $G_r$. On note $N_{w_r}$ la représentation slice de $G$ en $w_r$ (voir la définition \ref{slice}), alors on a le

\begin{lemme} \label{slice_expliciteSpn}
On a un isomorphisme de $G_r$-modules 
$$ N_{w_r} \cong (n'-2r) E_r \oplus r(2n'-2r-1) V_0 .$$
\end{lemme}

\begin{proof}
La preuve est analogue à celle du lemme \ref{slice_explicite}.
\end{proof}

Soient $F_1$ et $F_2$ des espaces vectoriels de dimensions $n'-2r$ et $r(2n'-2r-1)$ respectivement et dans lesquels $G_r$ opère trivialement. D'après le lemme \ref{slice_expliciteSpn}, on a un isomorphisme de $G_r$-modules 
$$ N_{w_r} \cong \Hom(F_1, E_r) \times F_2 $$
et le morphisme de passage au quotient $\nu_N:\ N_{w_r} \rightarrow N_{w_r}/\!/G_r$ est donné par:
$$\begin{array}{lrcl}
 \nu_N:\  &\Hom(F_1, E_r) \times F_2 & \rightarrow  & \Lambda^2(F_1^*) \times F_2 \\
        & (w,x)  & \mapsto     &  ( \leftexp{t}{w} J' w , x)  
\end{array}$$ 
où $J' \in \MM_{n-2r}(k)$ est la matrice avec $\frac{n}{2}-r$ blocs $A$ sur la diagonale et des $0$ partout ailleurs.
Et donc 
$$\NNN(N_{w_r},G_r):= {\nu}_{N}^{-1}({\nu}_N(0))={\nu}_{N}^{-1}(0) \cong {{\nu}_{N}'}^{-1}(0)$$ 
avec
$$\begin{array}{lrcl}
 {\nu'_N}:\  &\Hom(F_1, E_r)  & \rightarrow  & \Lambda^2(F_1^*) \\
        & w  & \mapsto      &  \leftexp{t}{w} J' w.
\end{array}$$

\begin{proposition} \label{descriptiongeofibSpn}
Avec les notations précédentes, on a un isomorphisme $G$-équivariant
$${\nu}^{-1}(J_{2r}) \cong G {\times}^{G_r} {{\nu}_{N}'}^{-1}(0) .$$ 
En particulier, si l'on note $H:=G_N$, on a 
$${\nu}^{-1}(J_N) \cong \left\{
    \begin{array}{ll}
        G &\text{ si } n' \geq n, \\
        G/H &\text{ si } n'<n \text{ et $n'$ pair},\\
        G \times^H E_r &\text{ si } n'<n \text{ et $n'$ impair}.\\
    \end{array}
\right.$$ 
\end{proposition}

\begin{corollaire} \label{dimfibreSpn}
Soit $r \in \{0, \ldots,N\}$, alors la dimension de la fibre du morphisme $\nu$ au dessus de $J_{2r}$ vaut: 
\begin{itemize} \renewcommand{\labelitemi}{$\bullet$}
\item $n' n-\frac{1}{2}n'(n'-1)$ lorsque $n'-r \leq \frac{n}{2}$,
\item $\frac{1}{2} n' (n-2r)+\frac{1}{8}{(n+2r)}^{2}+\frac{1}{4}(n+2r)$ sinon.
\end{itemize}
\end{corollaire}

On rappelle que $N:=\min \left( E \left( \frac{n'}{2} \right), \frac{n}{2} \right)$. En procédant comme pour la proposition \ref{ouvertplatitude}, on montre la

\begin{proposition} \label{ouvertplatitudeSpn}
La dimension de la fibre générique et l'ouvert de platitude de $\nu$ sont donnés dans le tableau suivant:\\
\begin{center}
\begin{tabular}{|c|c|c|}
  \hline
  configuration & dim. de la fibre générique & ouvert de platitude \\
  \hline
  $n' < n$ &     $n' n -\frac{1}{2}n'(n'-1)$ & $U_{N} \cup \cdots \cup U_{\max(n'-\frac{n}{2}-1, 0)}$\\
  $n' = n$ &     $\frac{1}{2}n(n+1)$       & $U_N \cup U_{N-1} $\\  
  $n' > n$ &     $\frac{1}{2}n(n+1)$       & $U_N $\\  
  \hline    
\end{tabular}
\end{center}
\vspace*{0.5mm}
\end{proposition}

\begin{corollaire} \label{ouverttoutplatSpn}
Le morphisme $\nu$ est plat sur $W/\!/G$ tout entier si et seulement si $n+2 \geq 2n'$ et dans ce cas $W/\!/G=\Lambda^2(V'^*)$. 
\end{corollaire}

\noindent Le corollaire qui suit est une conséquence de la proposition \ref{chow} et du corollaire \ref{ouverttoutplatSpn}:

\begin{corollaire} \label{cas_facileSpn}
Si $n+2 \geq 2n'$, alors $\HH \cong \Lambda^2(V'^*)$ est une variété lisse et $\gamma$ est un isomorphisme.
\end{corollaire}

On s'intéresse, dans la proposition qui suit (et qui se démontre comme la proposition \ref{fcthilb}), à la fonction de Hilbert de la fibre générique de $\nu$. On note comme précédement $H:=G_N \cong Sp_{n-2N}(k)$ le stabilisateur de $w_N$ dans $G$.

\begin{proposition} \label{fcthilbSpn}
La fonction de Hilbert de la fibre générique de $\nu$ est donnée par:
$$\forall M \in \Irr(G),\ h_W(M)= \left\{
    \begin{array}{ll}
        \dim(M) &\text{ si } n' \geq n, \\
        \dim(M^{H}) &\text{ si } n'<n \text{ et $n'$ est pair}, \\ 
        \dim \left( (M \otimes k[E_r])^H \right) &\text{ si } n'<n \text{ et $n'$ est impair}.
    \end{array}
\right.$$
\end{proposition}

%% file: nilcone_Sp2n.tex
\subsection{Description de l'algèbre du nilcône}  \label{sectionJSymp}

Dans cette section, on décrit l'algèbre du nilcône de $\nu$ comme $G' \times G$-module lorsque $n'=n$ en procédant comme dans la section \ref{sectionJ}. 

\begin{notation} 
On note $J$ l'idéal engendré par les $G$-invariants homogènes de degré positif de $k[W]$. 
\end{notation}

L'idéal $J$ est $G' \times G$-stable par définition. Avec les notations de la section \ref{rappelsthedesreps}, on a
\begin{equation} \label{kW2Sp4}
{k[W]}_2 \cong (S^2(V') \otimes \Gamma_{2 \epsilon_1}(V)) \oplus ({\Lambda}^{2}(V') \otimes (V_0 \oplus \Gamma_{\epsilon_1+\epsilon_2}(V))).
\end{equation}
comme $G' \times G$-module. Donc $J \cap {k[W]}_2 ={\Lambda}^{2}(V') \otimes V_0$ comme $G' \times G$-module et ce module engendre l'idéal $J$. 

On rappelle que l'on a défini les sous-groupes $B', T', U'$ de $G'$ (resp. $B, T, U$ de $G$) dans la section \ref{lesdiffsituations}. Soient $N':=\frac{n}{2}$ et $1 \leq p \leq N'$. D'après le formule de Cauchy-Littlewood (\cite[§8.3, Corollary 3]{Fu}), on a un isomorphisme de $G' \times GL(V)$-modules
\begin{equation} 
k[W]_p \cong \bigoplus_{|\lambda|=p} S^{\lambda}(V') \otimes S^{\lambda}(V)^*
\end{equation}
où $|\lambda|:= \sum_i r_i$. Puis, d'après \cite[Theorem 17.5]{FH}, le $GL(V)$-module $\Lambda^p V^*$ contient une unique copie du $G$-module $\Gamma_{\epsilon_1+\ldots+\epsilon_p}(V)$. On fixe $x_p \in k[W]_p^{U' \times U}$ un vecteur de plus haut poids du $G' \times G$-module $\Lambda^p V' \otimes \Gamma_{\epsilon_1+\ldots+\epsilon_p}(V)$. La proposition qui suit est prouvée dans \cite[§9]{KS}:

\begin{proposition} \label{KSheadingsSp}
On a un isomorphisme de $T' \times T$-algèbres
$$ k[x_1, \ldots,x_{N'}] \cong (k[W]/J)^{U' \times U}.$$
\end{proposition}

\begin{proposition} \label{decompoisoSp}
Soit $\lambda=r_1 \epsilon_1+ \ldots + r_{N'} \epsilon_{N'} \in \Lambda+$, alors la composante isotypique associée au $G$-module $\Gamma_{\lambda}(V)$ dans $k[W]/J$ est 
$$S^{\lambda}(V') \otimes \Gamma_{\lambda}(V).$$
De plus, la représentation $\Gamma_{\lambda}(V)$ apparaît dans $k[W]_p/(J \cap k[W]_p)$ si et seulement si $p=\sum_i r_i$. 
\end{proposition}

\begin{proof}
La démarche est la même que dans la preuve de la proposition \ref{decompoiso}. Les monômes forment une base de $k[x_1, \ldots,x_{N'}]$ et chaque monôme engendre une droite $T' \times T$-stable. Soit $p_1,\ldots,p_{N'}$ des entiers positifs, alors le poids $(\lambda',\lambda)$ du monôme $x_1^{p_1} \ldots x_{N'}^{p_{N'}}$ est 
$$(\ p_{N'}\epsilon_{N'}+(p_{N'}+p_{N'-1}) \epsilon_{N'-1}+ \ldots+(p_{N'}+\ldots+p_1) \epsilon_1,\ p_{N'}\epsilon_{N'}+ \ldots+(p_{N'}+\ldots+p_1) \epsilon_1)$$ 
et le poids $\lambda$ détermine uniquement ce monôme. Le résultat annoncé s'ensuit. 
\end{proof}

On en déduit le

\begin{corollaire} \label{fcthilbclassn3Sp4}
Si $I_Z$ un idéal $G$-stable homogène de $k[W]$  contenant $J$ et qui a pour fonction de Hilbert $h$, alors la fonction de Hilbert classique de $I$ est donnée par:
$$\forall p \geq 0,\ f_I(p)= \sum_{\substack{ r_1 \geq \ldots \geq r_{N'} \geq 0, \\ r_1+ \ldots+r_{N'}=p }}  h(\Gamma_{r_1 \epsilon_1+ \ldots+r_{N'}\epsilon_{N'}}(V)) \dim(\Gamma_{r_1 \epsilon_1+ \ldots+r_{N'}\epsilon_{N'}}(V)) .$$
\end{corollaire}

%% file: Sp4.tex
\subsection{Etude du cas \texorpdfstring{$\dim(V)=4$}{n=4}}  \label{etudeSp4}

Dans toute cette section, on fixe $n=4$. On a $G \cong Sp_4(k)$, $W/\!/G=\Lambda^2(V'^*)^{\leq 4}$ et $\rho:\ \HH \rightarrow \Gr(4,V'^*)$ est le morphisme de la section \ref{redGrass2}. On note 
$$Y_0:=\left\{(Q,L)\in W/\!/G \times \PP (W/\!/G)\ \mid \ Q \in L \right\}=\OO_{\PP(W/\!/G)}(-1)$$ 
l'éclatement de $W/\!/G$ en l'origine et $Y_1$ l'éclatement de $Y_0$ le long de la transformée stricte de $\Lambda^2(V'^*)^{\leq 2}$.
Nous verrons à la fin de cette section que $Y_1$ est isomorphe à l'éclatement de la section nulle du fibré $\Lambda^2\left(T \right)$ au dessus de $\Gr(4,V'^*)$, où $T$ désigne le fibré tautologique de $\Gr(4,V'^*)$. 

Nous allons démontrer le

\begin{theoreme} \label{casn2Sp4}
\begin{itemize}  
\item Si $n'=4$, alors $\HH \cong Y_0$ et $\gamma$ est l'éclatement de $W/\!/G$ en l'origine. 
\item Si $n'> 4$, alors $\HH \cong Y_1$ et $\gamma$ est une résolution des singularités de $W/\!/G$.
\end{itemize}
En particulier $\HH$ est toujours une variété lisse.
\end{theoreme} 

\begin{remarque}
Si $n'\leq 3$, alors $\HH \cong W/\!/G$ et $\gamma$ est un isomorphisme d'après le corollaire \ref{cas_facileSpn}. 
\end{remarque}

Nous allons démontrer le théorème \ref{casn2Sp4} en procédant comme pour les théorèmes \ref{casn2} et \ref{casn2O2}.

\subsubsection{Points fixes de \texorpdfstring{$\HH$}{H} pour l'opération de \texorpdfstring{$B'$}{B'}}   \label{pointsfixesSp4}


On suppose pour le moment que $n'=4$ et on commence par déterminer les points fixes de $B'$ dans $\HH$.

\begin{notation} \label{l3generatorsSp4}
On note  
\begin{itemize} \renewcommand{\labelitemi}{$\bullet$}
\item $D:=\left\langle  e_1 \wedge e_2 \right\rangle$ l'unique droite $B'$-stable de $\Lambda^2(V')$,
\item $I$ l'idéal de $k[W]$ engendré par $({\Lambda}^{2}(V') \otimes V_0) \oplus (D \otimes \Gamma_{\epsilon_1+\epsilon_2}(V)) \subset {k[W]}_2$.
\end{itemize}
\end{notation}

\begin{remarque}
L'idéal $I$ est homogène, $B' \times G$-stable et contient l'idéal $J$.
\end{remarque}

\begin{theoreme} \label{pointfixeborelSp4}
L'idéal $I$ est l'unique point fixe de $B'$ dans $\HH$.
\end{theoreme}

\begin{proof}
On raisonne comme dans la preuve du théorème \ref{pointfixeborel} en considérant $I_Z$ un point fixe de $B'$ et en étudiant $k[W]/I_Z$ composante par composante.\\
$\bullet$ Composantes de degré $0$ et $1$:\\
On a bien sûr $I_Z \cap {k[W]}_0=\{0\}$ et $I_Z \cap {k[W]}_1 \neq {k[W]}_1$.\\  
$\bullet$ Composante de degré $2$: on utilise la décomposition (\ref{kW2Sp4}).\\
Pour avoir la décomposition souhaitée de $k[W]/I_Z$ comme $G$-module, on a nécessairement ${k[W]}_2 \cap I_Z \supseteq 6 V_0 \oplus \Gamma_{\epsilon_1+\epsilon_2}(V)$. 
En effet, le $G$-module $k[W]/I_Z$ contient déjà une copie de la représentation triviale (qui provient de la composante de degré $0$), il ne peut donc pas en contenir d'autre. Ensuite, ${k[W]}_2$ contient $6$ copies de $\Gamma_{\epsilon_1+\epsilon_2}(V)$ qui est un $G$-module de dimension $5$, donc ${k[W]}_2 \cap I_Z$ contient au moins une copie de $\Gamma_{\epsilon_1+\epsilon_2}(V)$. Comme $k[W]_2 \cap I_Z$ est $B'$-stable, il contient $D \otimes \Gamma_{\epsilon_1+\epsilon_2}(V)$ car $D$ est l'unique droite $B'$-stable de $\Lambda^2(V')$. Il s'ensuit que $I_Z$ contient $(\Lambda^2(V') \otimes V_0) \oplus (D \otimes \Gamma_{\epsilon_1+\epsilon_2}(V))$ et donc $I_Z \supset I$. En particulier, $I_Z \supset J$ et donc d'après le corollaire \ref{fcthilbclassn3Sp4}, la fonction de Hilbert classique de $I_Z$ est donnée par:
$$\forall p \geq 0,\ f_{I_Z}(p)= \sum_{\substack{r_1 \geq r_2 \geq 0 \\ r_1 + r_2=p}}  \dim(\Gamma_{r_1 \epsilon_1+r_2 \epsilon_2}(V))^2 .$$
La formule des dimensions de Weyl permet de déduire que
$$\forall p \geq 0,\ f_{I_Z}(p)=\frac{1}{15120} p^9+\frac{1}{560} p^8+\frac{3}{140} p^7+\frac{3}{20} p^6+\frac{97}{144} p^5+\frac{481}{240} p^4+\frac{29683}{7560} p^3+\frac{4069}{840} p^2+\frac{473}{140} p+1.$$
Par ailleurs, un calcul direct avec \cite[Macaulay2]{Mac2} nous donne la fonction de Hilbert classique $f_I$ de l'idéal $I$:
$$ \forall p \geq 0,\ f_I(p)=f_{I_Z}(p).$$
Donc $I_Z=I$ est l'unique point fixe de $B'$.
\end{proof}

\begin{remarque}
On a $\Stab_{G'}(I)=\left \{ \begin{bmatrix} M_1 & M_2 \\ 0 & M_3 \end{bmatrix};\ M_1, M_3 \in GL_2(k),\ M_2 \in \MM_2(k) \right \}$, donc l'unique orbite fermée de $\HH$ est isomorphe à $\Gr(2,V')$. 
\end{remarque}

\noindent Le corollaire qui suit découle du lemme \ref{fixespoints} et du théorème \ref{pointfixeborelSp4}.

\begin{corollaire} \label{HconnexeSp4}
Le schéma $\HH$ est connexe.
\end{corollaire}

\subsubsection{Espace tangent de \texorpdfstring{$\HH$}{H} en \texorpdfstring{$Z_0$}{Z0}}  \label{tanggSp4}

On note $Z_0:=\Spec(k[W]/I)$. Nous allons démontrer la  

\begin{proposition} \label{dimTangentSp4}
$\dim(T_{Z_0} \HH)=6.$
\end{proposition}

On note $w \in W$ sous la forme $w=\begin{bmatrix}
x_1 &y_1 & z_1 &t_1 \\
x_2 &y_2 & z_2 &t_2\\
x_3 &y_3 & z_3 &t_3 \\
x_4 & y_4 & z_4 &t_4
\end{bmatrix}$ et on identifie $k[W]$ à $k[x_i,y_i,z_i,t_i,\ 1 \leq i \leq 4]$. On explicite des bases de certains $B' \times G$-modules qui apparaissent dans $k[W]_2$:\\ 

$
 \left.
    \begin{array}{l}
f_1:=-x_3y_1-x_4y_2+x_1y_3+x_2y_4 \\
f_2:=-x_3z_1-x_4z_2+x_1z_3+x_2z_4 \\
f_3:=-x_3t_1-x_4t_2+x_1t_3+x_2t_4 \\
f_4:=-y_3z_1-y_4z_2+y_1z_3+y_2z_4 \\
f_5:=-y_3t_1-y_4t_2+y_1t_3+y_2t_4 \\
f_6:=-z_3t_1-z_4t_2+z_1t_3+z_2t_4
      \end{array}
\right \} \text{ est une base de } \Lambda^2(V') \otimes V_0, 
$

$
 \left.
    \begin{array}{l}
h_1:=x_1y_2-x_2y_1 \\
h_2:=x_1y_3-x_3y_1 \\
h_3:=x_1y_4-x_4y_1 \\
h_4:=x_2y_3-x_3y_2 \\
h_5:=x_3y_4-x_4y_3 \\
h_6:=x_2y_4-x_4y_2 \\

      \end{array}
\right \} \text{ est une base de } D \otimes \Lambda^2(V). 
$

\ \\
On reprend les notations de la section \ref{ConnexitéetTangence}. Soient $R:=k[W]/I$ et 
$$N:= \left\langle  f_1,\ldots,f_6,h_1,\ldots,h_5 \right\rangle \subset k[W]$$ 
qui est un $B' \times G$-module qui engendre l'idéal $I$. D'après \cite[Macaulay2]{Mac2}, les relations entre les générateurs ci-dessus du $R$-module $I/I^2$ sont données dans la table \ref{table7}.
\begin{table}[ht]
\begin{center}
\center \includegraphics[scale=0.72]{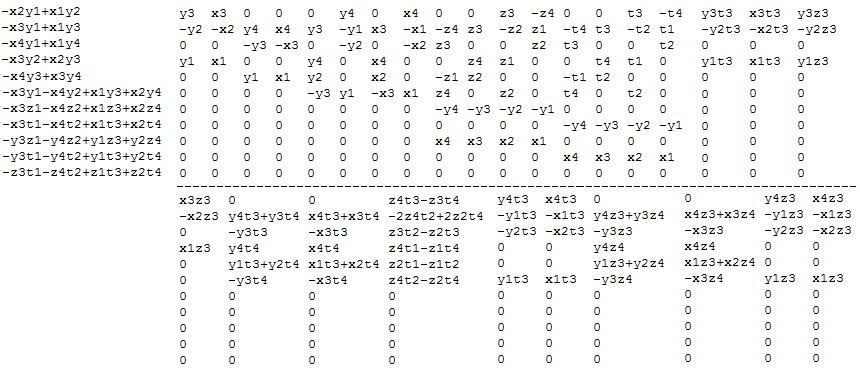}
\end{center}
\caption{ \label{table7} Relations entre les générateurs de $I/I^2$}
\end{table}

\noindent En particulier, on a les relations suivantes données par les colonnes 10 et 14 respectivement:
$$ \left \{
    \begin{array}{l}
r_1:=-z_3 \otimes h_2+z_3 \otimes h_3 -z_1 \otimes h_5 +z_4 \otimes f_1 -y_4 \otimes f_2+x_4 \otimes f_4, \\
r_2:=t_4 \otimes h_2+t_3 \otimes h_3 -t_1 \otimes h_5 +t_4 \otimes f_1 -y_4 \otimes f_3+x_4 \otimes f_5.  
   \end{array}
      \right.$$
On a $\dim(N)=11$ et donc d'après le lemme \ref{InegTang}, on a $\dim(T_{Z_0} \HH)=11- \rg(\rho^*)$. 
D'après le lemme \ref{fixespoints}, la variété $\HHp$ contient au moins un point fixe pour l'opération de $B'$, donc $Z_0 \in \HHp$ et donc $\dim(T_{Z_0} \HH) \geq \dim(\HHp)=6$. Donc pour montrer la proposition \ref{dimTangentSp4}, il suffit de montrer le

\begin{lemme} 
$\rg(\rho^*) \geq 5.$
\end{lemme}

\begin{proof} 
Pour $i=1,\ldots,6$, on définit $\psi_i \in \Hom_R^G(R \otimes N,R)$ par: 
$$
\left\{
    \begin{array}{ll}
        \psi_i(1 \otimes f_j)={\delta}_i^j &\text{ pour }j=1,\ldots,6,\\
        \psi_i(1 \otimes h)=0  &\text{ pour tout } h \in D \otimes \Gamma_{\epsilon_1+\epsilon_2}(V),        
    \end{array}
\right.
$$
où $\delta_i^j$ est le symbole de Kronecker.  
La famille $\{ \psi_1, \ldots, \psi_6\}$ est libre dans $\Hom_R^{G}(R \otimes N,R)$. Nous allons voir que 
$\{ \rho^*(\psi_1),\ldots, \rho^*(\psi_5)\}$ 
est une famille libre de $\Hom_R^G(R \otimes \RRR,R)$, ce qui démontrera le lemme. \\
Soient ${\lambda}_1, \ldots, {\lambda}_5 \in k$ tels que
\begin{equation}  \label{relaziSp4} 
{\lambda}_1 \, \rho^*({\psi}_1)+ \ldots + {\lambda}_5 \, \rho^*({\psi}_5)=0.
\end{equation}
On évalue (\ref{relaziSp4}) en $r_1 \otimes 1$, on obtient:
$$\lambda_1 z_4- \lambda_2 y_4 + \lambda_4 x_4=0 \Rightarrow \lambda_1=\lambda_2=\lambda_4=0.$$
De même, on évalue (\ref{relaziSp4}) en $r_2 \otimes 1$, on obtient:
$$-\lambda_3 y_4+ \lambda_5 x_4=0 \Rightarrow \lambda_3=\lambda_5=0.$$
Et donc $\{ \rho^*(\psi_1),\ldots, \rho^*(\psi_5)\}$ est bien une famille libre.
\end{proof}

\begin{remarque}
D'après \cite[Macaulay2]{Mac2}, une famille de générateurs du $R$-module $\Hom_R(I/I^2,R)$ est donnée dans la table \ref{table9}. 

\begin{table}[h]
\begin{center}
\center \includegraphics[scale=0.8]{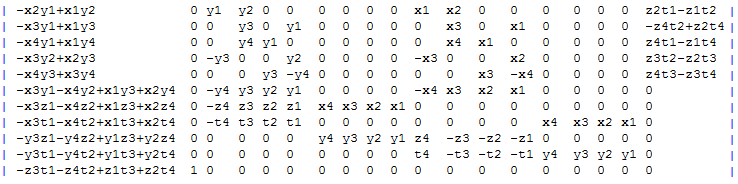}
\end{center}
\caption{\label{table9} Générateurs du $R$-module $\Hom_R(I/I^2,R)$}
\end{table}

\noindent On note $\Phi_i$ le morphisme associé à la colonne $i+1$. On vérifie que les six morphismes suivants sont $G$-équivariants et linéairement indépendants:
\begin{itemize} \renewcommand{\labelitemi}{$\bullet$}
\item $\Phi_1$, 
\item  $-z_2 \Phi_2+z_1 \Phi_3-z_4 \Phi_4-z_3 \Phi_5$,
\item  $-t_2 \Phi_2+t_1 \Phi_3-t_4 \Phi_4-t_3 \Phi_5$,
\item  $-z_2 \Phi_{10}+z_1 \Phi_{11}-z_4 \Phi_{12}-z_3 \Phi_{13}$,
\item  $-t_2 \Phi_{10}+t_1 \Phi_{11}-t_4 \Phi_{12}-t_3 \Phi_{13}$.
\item $\Phi_{18}$.
\end{itemize}
Ils forment donc une base de l'espace vectoriel ${\Hom}_R^G(I/I^2,R)$ et on en déduit, comme dans la remarque \ref{Bmoduleexplicite}, que 
$${\Hom}_R^G(I/I^2,R) \cong D^* \oplus \left( D^* \otimes \frac{\Lambda^2(V')}{D}\right)$$
comme $B'$-module.   
\end{remarque}

\noindent On déduit du lemme  \ref{Hlisse4} et de ce qui précède le 

\begin{corollaire} \label{Hcasn2Sp4}
$\HH=\HHp$ est une variété lisse de dimension $6$. 
\end{corollaire}

\subsubsection{Construction d'un morphisme équivariant \texorpdfstring{$\delta:\ \HH \rightarrow \PP(W/\!/G)$}{}} \label{construction_deltaSp4}

Le lemme qui suit découle de la théorie classique des invariants de manière analogue au lemme \ref{exi1}.

\begin{lemme} 
Le ${k[W]}^{G}$-module ${k[W]}_{(\Gamma_{\epsilon_1+\epsilon_2}(V))}$ est engendré par ${\Hom}^{G}(\Gamma_{\epsilon_1+\epsilon_2}(V),{k[W]}_2)$.
\end{lemme}

Puis, on a l'isomorphisme de $G'$-modules
$$\Hom^{G}(\Gamma_{\epsilon_1+\epsilon_2}(V),{k[W]}_2) \cong \Lambda^2(V').$$ 
La proposition \ref{morphismegrass} donne un morphisme $G'$-équivariant 
$$\HH \rightarrow \Gr(5,\Lambda^2(V'^*)).$$
Et $\Lambda^2(V'^*) \otimes \det(V') \cong \Lambda^2(V')$, donc 
$$\Gr(5,\Lambda^2(V'^*)) \cong \PP (\Lambda^2(V')) \cong \PP(\Lambda^2(V'^*))=\PP(W/\!/G)$$ 
comme $G'$-variété. On a donc un morphisme $G'$-équivariant 
\begin{equation} \label{maoprhismedeltaSp4}
\delta:\ \HH \rightarrow \PP (W/\!/G).
\end{equation}

On rappelle que l'on note 
$$Y_0:=\left\{(Q,L)\in W/\!/G \times \PP (W/\!/G)\ |\ Q \in L \right\}=\OO_{\PP (W/\!/G)}(-1)$$ 
l'éclatement en $0$ de $W/\!/G$. On vérifie que le morphisme $\gamma \times \delta$ envoie $\HH$ dans $Y_0$. Puis, en procédant comme pour la proposition \ref{gammaiso}, on montre que le morphisme $\gamma \times \delta:\ \HH \rightarrow Y_0$ est un isomorphisme $G'$-équivariant.

\subsubsection{Cas \texorpdfstring{$n'>4$}{n'>4}}  \label{masterpropositioncasn2Sp4}


On considère la variété
$$\ZZ:=\left\{(Q,L,E) \in W/\!/G \times \PP(W/\!/G) \times \Gr(4,V'^*) \ \mid \ Q \in L \text{ et } \Im(L) \subset E  \right\}$$       
où $\Im(L)$ est défini dans la notation \ref{notationsKerd}.
On a le diagramme
$$\xymatrix{ &  \ZZ \ar@{->>}[ld]_{q_1} \ar@{->>}[rd]^{q_2} \\   Y_0 && \Gr(4,V'^*) }$$
où $q_1$ et $q_2$ sont les projections naturelles. En particulier, avec les notations du début de la section \ref{etudeSp4}, on vérifie que $\ZZ \cong \Lambda^2\left(T \right)$. Ensuite, on montre en procédant comme pour le lemme \ref{identificationblowup} que $Y_1 \cong \ZZ$
et que, via cet isomorphisme, l'éclatement $Y_1 \rightarrow Y_0$ s'identifie au morphisme $q_1:\ \ZZ \rightarrow Y_0$. On raisonne alors comme dans la section \ref{masterpropositioncasn2}: on identifie $Y_1$ à $\ZZ$, on montre l'existence d'un isomorphisme $G'$-équivariant $\HH \cong Y_1$, on identifie $\HH$ à $Y_1$ via cet isomorphisme et enfin on montre que $\gamma$ est la composition des éclatements $Y_1 \rightarrow Y_0 \rightarrow W/\!/G$. En particulier, $\gamma$ est toujours une résolution de $W/\!/G$.

%% file: symplectique_On_1.tex
\section{Cas de \texorpdfstring{$O(V)$}{O(V)} et \texorpdfstring{$SO(V)$}{SO(V)} opérant dans \texorpdfstring{$\mu^{-1}(0)$}{}}  \label{sektionsympOn}

On se place à nouveau dans la situation $3$ et on considère le cas particulier où l' espace vectoriel $V'$ est de dimension paire $n'=2d$. Plus précisément, soit $E$ un espace vectoriel de dimension $d$ et soit $V':=E \oplus E^*$. Alors 
\begin{align*}
W&=\Hom(V',V)\\
 &=\Hom(E \oplus E^*,V)\\
 &=\Hom(E,V) \times \Hom(V,E) \text{ puisque $V \cong V^*$ comme $G$-module,}\\
 &=\Hom(E,V) \times \Hom(E,V)^*.
\end{align*}
On a $G=O(V)$ et on note $G^0=SO(V)$ la composante neutre de $G$ qui opère dans $W$ par restriction de l'opération de $G$. Soient $A:=\begin{bmatrix}
0  & 1 \\
-1 & 0 
\end{bmatrix}$ et $J':=\begin{bmatrix}
A & &\\
 & \ddots & \\
 && A 
\end{bmatrix} \in \MM_{n'}(k)$ la matrice avec $d$ blocs $A$ sur la diagonale et des $0$ partout ailleurs. On note $\Omega'$ la forme symplectique sur $V'$ définie par $J'$ et $G':=Sp(V')$ le sous-groupe de $GL(V')$ qui préserve $\Omega'$.  
L'opération de $GL(V') \times G$ dans $W$ induit une opération de $G' \times G$ dans $W$. On note 
$$\gg':=\{ M \in \End(V') \ |\ J'M + \leftexp{t}{M}J'=0\}$$
l'algèbre de Lie de $G'$ et
$${\gg'}^{\leq n}:=\{ M \in \gg' \ |\ \rg(M) \leq n\}.$$
Le groupe $G'$ opère dans $\gg'$ par l'action adjointe et on a un isomorphisme $G'$-équivariant
$$\begin{array}{ccc}
 \gg' & \cong & S^2(V'^*) \\
 M & \mapsto & J'M 
\end{array}$$
qui permet d'identifier $\gg'^{\leq n}$ au quotient $W/\!/G=S^2(V'^*)^{\leq n}$ étudié dans la section \ref{description_quotientOn}. 

Notre but est d'étudier les schémas de Hilbert invariants
$$\HHmg:=\mathrm{Hilb}_{h_s}^{G}(\mu^{-1}(0))$$
et 
$$\HHmgz:=\mathrm{Hilb}_{h_s^0}^{G^0}(\mu^{-1}(0)).$$
L'étude de ces deux schémas étant fortement liée, nous faisons le choix de les considérer simultanément. On procède comme dans la section \ref{posisimpy}. Dans la section \ref{appzmomentOn}, on décrit l'application moment $\mu:\ W \rightarrow \gg$, où $\gg$ désigne l'algèbre de Lie de $G$, ainsi que sa fibre en $0$. Dans la section \ref{sectionavecXOn}, on étudie les morphismes de passage au quotient $\nu:\ \mu^{-1}(0) \rightarrow \mu^{-1}(0)/\!/G$ et $\nu_0:\ \mu^{-1}(0) \rightarrow \mu^{-1}(0)/\!/G^0$. Nous verrons que les réductions symplectiques ${\mu}^{-1}(0)/\!/G$ et ${\mu}^{-1}(0)/\!/G^0$ sont irréductibles et on pourra donc parler de la fonction de Hilbert $h_s$ (resp. $h_s^0$) de la fibre générique de $\nu$ (resp. de $\nu_0$). On note $\HHmgp$ et $\HHmgzp$ les composantes principales de $\HHmg$ et $\HHmgz$ respectivement. Dans la section \ref{MropRREDOn}, on établit un résultat de réduction qui permet de ramener la détermination de $\HHmgp$ au schéma de Hilbert invariant étudié dans la section \ref{HclassOngeneral}. Le résultat essentiel de cette section est le corollaire \ref{symplisssseOn}. 

\subsection{L'application moment}  \label{appzmomentOn}
On considère $\Omega$ la forme symplectique sur $W$ définie dans la section \ref{appzmoment}, alors $G' \times G$ opère symplectiquement dans $(W, \Omega)$. D'après \cite[§4.2]{Tanja}, l'application moment pour l'opération de $G$ dans $(W,\Omega)$ est donnée par:
$$\begin{array}{cccc}
\mu : & W & \to & \gg \\
 & w & \mapsto & -\frac{1}{2}wJ' \leftexp{t}{w}.\\
\end{array}$$
Et donc
$$\mu^{-1}(0)=\{ w \in W \ |\ wJ' \leftexp{t}{w}=0\}.$$

\begin{remarque}
On vérifie que $Sp(V')$ est le plus grand sous-groupe de $GL(V')$ qui stabilise $\mu^{-1}(0)$ dans $W$ et ce constat justifie a posteriori notre choix pour $G'$. 
\end{remarque}

L'application moment coïncide (à un scalaire près) avec le morphisme de passage au quotient $W \rightarrow W/\!/G'$. En particulier, $\mu^{-1}(0)$ est le nilcône $\NNN(W,G')$ et on déduit de \cite[Theorem 9.1]{KS} que le schéma $\mu^{-1}(0)$ est une variété normale. La preuve de la proposition qui suit est analogue à celle de la proposition \ref{nilconeSpn}. 

\begin{proposition} \label{dimmudezero3}
La variété $\mu^{-1}(0)$ est de dimension
\begin{itemize} \renewcommand{\labelitemi}{$\bullet$}
\item $2dn-\frac{1}{2}n(n-1)$ si $n \leq d$,
\item $dn+\frac{1}{2}d^2+\frac{1}{2}d$ sinon.
\end{itemize}\end{proposition}

\subsection{Etude des morphismes de passage au quotient}  \label{sectionavecXOn}
Soient
$$\nu:\ {\mu}^{-1}(0) \rightarrow {\mu}^{-1}(0)/\!/G$$ 
et 
$$\nu_0:\ {\mu}^{-1}(0) \rightarrow {\mu}^{-1}(0)/\!/G^0$$ 
les morphismes de passage au quotient. Dans cette section, on décrit géométriquement les variétés ${\mu}^{-1}(0)/\!/G$ et ${\mu}^{-1}(0)/\!/G^0$. On note $h_s$ (resp. $h_s^0$) la fonction de Hilbert de la fibre générique de $\nu$ (resp. de $\nu_0$) et
$$N:=\min \left(d,n \right).$$ 

On rappelle que l'on a une correspondance bijective entre les partitions $(d_1 \geq \ldots \geq d_k)$ de $n'=2d$ dans lesquelles les $d_i$ impairs sont présents avec multiplicité paire et les orbites nilpotentes $\OO_{[d_1,\ldots,d_k]}$ de $\gg'$ (voir \cite[§5.1]{CoMc}). La proposition qui suit se démontre comme la proposition \ref{descQuotient}. 

\begin{proposition}  \label{descQuotientOn}
$\mu^{-1}(0)/\!/G =\overline{\OO_{[2^N,1^{2(d-N)}]}}.$
\end{proposition}

\begin{corollaire}
Le quotient $\mu^{-1}(0)/\!/G$ se décompose en $N+1$ orbites pour l'opération de $G':$ 
$$U_i:=\OO_{[2^i,1^{2(d-i)}]},\ \text{ pour } i=0, \ldots, N.$$
\end{corollaire}

Les adhérences de ces orbites sont imbriquées de la façon suivante:
\begin{equation} \label{orbImb}
\{0\}=\overline{U_0} \subset \cdots \subset \overline{U_N}=\mu^{-1}(0)/\!/G.
\end{equation}

La géométrie de l'adhérence d'une orbite nilpotente dans $\gg'$ est étudiée dans \cite{KP3}. En particulier, la variété $\mu^{-1}(0)/\!/G$ est normale (car $\mu^{-1}(0)$ est normale), de dimension $N(2d+1-N)$ (\cite[Corollary 6.1.4]{CoMc}) et son lieu singulier est $\overline{U_{N-1}}$ (\cite[Theorem 2]{KP3}). 

On s'intéresse maintenant à la géométrie de la variété $\mu^{-1}(0)/\!/G^0$. On considère le morphisme $G'$-équivariant $p_1:\ \mu^{-1}(0)/\!/G^0 \rightarrow \mu^{-1}(0)/\!/G$ qui apparaît dans le diagramme (\ref{diagOvSOv}). Si $d<n$, alors $\mu^{-1}(0) \subset \Hom(V',V)^{\leq d}$ et donc $p_1$ est un isomorphisme. En revanche, si $d \geq n$, alors $\mu^{-1}(0) \cap \{w \in W \ |\ \rg(w)=n\}$ est non vide et donc $p_1$ est un revêtement double de $\mu^{-1}(0)/\!/G$ dont le lieu de ramification est $\overline{\OO_{[2^{n-1},1^{2(d-n+1)}]}}$ (en particulier les dimensions de $\mu^{-1}(0)/\!/G$ et $\mu^{-1}(0)/\!/G^0$ coïncident). Malheureusement, j'ignore quel est le lieu singulier de $\mu^{-1}(0)/\!/G^0$, on sait seulement qu'il est de codimension au moins 2.  
D'après \cite[§6.1]{CoMc}, on a $\pi_1(\OO_{[2^n,1^{2(d-n)}]}) \cong \ZZZ_2$  et donc, si $d \geq n$, le morphisme $p_1$ est le revêtement universel de $\OO_{[2^n,1^{2(d-n)}]}$.\\
Dans tous les cas, la variété quotient $\mu^{-1}(0)/\!/G^0$ est normale et se décompose en $N+1$ orbites $U_i$, $i=0, \ldots, N$, qui sont imbriquées comme dans (\ref{orbImb}). On fait ici l'abus d'utiliser les mêmes notations pour désigner les orbites de $\mu^{-1}(0)/\!/G$ et de $\mu^{-1}(0)/\!/G^0$ mais on sera attentif à toujours préciser le contexte.

\begin{remarque}
Nous verrons dans la section \ref{appendice2} que les variétés $\mu^{-1}(0)/\!/G$ et $\mu^{-1}(0)/\!/G^0$ sont symplectiques. Nous verrons aussi que $\mu^{-1}(0)/\!/G$ admet des résolutions symplectiques si et seulement si $d \leq n$.  
\end{remarque}

La dimension de la fibre générique de $\nu$ (resp. de $\nu_0$) est:
$$\left\{
    \begin{array}{ll}
   \frac{1}{2}n(n-1) &\text{ si } d \geq n,\\
    dn-\frac{1}{2}d(d+1) &\text{ si } d<n.
    \end{array}
\right.
$$


\begin{notation}
Si $d<n$, on note 
$$H:=\left\{ \begin{bmatrix}
M  &0_{n-d,d} \\
0_{d,n-d}  & I_d 
\end{bmatrix},\ M \in O_{n-d}(k)\right\} \cong O_{n-d}(k)$$
(qui est un sous-groupe réductif de $G$) et $H^0\cong SO_{n-d}(k)$ la composante neutre de $H$ (qui est un sous-groupe réductif de $G^0$). 
\end{notation}

\begin{lemme}
La fibre de $\nu$ (resp. $\nu_0$) en un point de $U_N$ est isomorphe à 
$$\left\{
    \begin{array}{ll}
        G   \ (\text{resp. } G^0)  &\text{ si } d \geq n,\\
        G/H  \ (\text{resp. } G^0/H^0) &\text{ si } d <n.
        
    \end{array}
\right.
$$\end{lemme}

\noindent On en déduit la

\begin{proposition} \label{fctH2Osymp}
La fonction de Hilbert $h_s$ (resp. $h_s^0$) de la fibre générique de $\nu$ (resp. de $\nu_0$) est donnée par:
$$\forall M \in \Irr(G),\ h_s(M)=\left\{
    \begin{array}{ll}
       \dim(M)  &\text{ si } d \geq n,  \\
       \dim(M^H)    &\text{ si } d <n. 
    \end{array}
\right.
$$
$$\left( resp.\ \forall M \in \Irr(G^0),\ h_s^0(M)=\left\{
    \begin{array}{ll}
       \dim(M)  &\text{ si } d \geq n,  \\
       \dim(M^{H^0})    &\text{ si } d <n. 
    \end{array}
\right. \right)
$$
\end{proposition}

La fibre $\mu^{-1}(0)$ est un sous-schéma fermé $G' \times G$-stable de $W$, donc le schéma $\HHmg$ est un sous-schéma fermé $G'$-stable de $\Hilb_{h_s}^G(W)$. Ensuite, d'après les propositions \ref{fcthilbOn} et \ref{fctH2Osymp}, les fonctions de Hilbert $h_W$ et $h_s$ coïncident lorsque $d \geq n$ ou $d<\frac{n}{2}$ et alors $\HHmg$ s'identifie à un sous-schéma fermé de $\Hilb_{h_W}^G(W)$; ce dernier a été étudié dans la section \ref{HclassOngeneral}. De même, lorsque $d \geq n$ ou $d<\frac{n}{2}$, le schéma $\HHmgz$ s'identifie à un sous-schéma fermé de $\Hilb_{h_W}^{G^0}(W)$; ce dernier a été étudié dans la section \ref{HclassSOngeneral}.  

\begin{remarque}
Il n'y a pas de relation d'inclusion entre $\HHmg$ et $\HHmgz$ en général.
\end{remarque}

%% file: symplectique_On_2.tex
\subsection{Construction de morphismes équivariants et réduction} \label{MropRREDOn}
On souhaite déterminer $\HHmgp$ en procédant comme nous l'avons fait pour $GL(V)$ dans la section \ref{MropRRED}, c'est-à-dire en utilisant le principe de réduction pour se ramener à une variété plus simple. On n'a malheureusement pas de résultat semblable pour $\HHmgzp$ puisque, comme nous l'avons vu dans la section \ref{zectionred}, on ne dispose pas du principe de réduction lorsque $G=SO(V)$.\\
Dans la section \ref{redGrass2}, on a construit un morphisme $G'$-équivariant 
\begin{equation} \label{rosympOn}
\Hilb_{h_s}^G(W) \rightarrow \Gr(h_s(V),V'^*).
\end{equation}
D'après la proposition \ref{fctH2Osymp}, on a $h_s(V)=N$. La restriction du morphisme (\ref{rosympOn}) à $\HHmg$ donne donc un morphisme $G'$-équivariant 
\begin{equation}  \label{yeah}
\rho_s:\ \HHmg \rightarrow \Gr(N,V'^*).
\end{equation}
On rappelle que l'on note $\Omega'$ la forme symplectique sur $V'$ définie par la matrice $J'$ au début de la section \ref{sektionsympOn}. On a $V' \cong V'^*$ comme $G'$-module et donc $\Omega'$ s'identifie à une forme symplectique sur $V'^*$. Soient 
$$N':=E \left( \frac{N}{2} \right)$$ 
et
$$O_i:=\{L \in \Gr(N,V'^*)\ |\ {\Omega'}_{|L} \text{ est de rang } 2i\} \text{ où } i=0,\ldots, N'.$$
Alors les $O_i$ sont les orbites pour l'opération de $G'$ dans $\Gr(N,V'^*)$ et on a:
$$O_0 \subset \overline{O_1} \subset \cdots \subset \overline{O_{N'}}=\Gr(N,V'^*).$$ 
En particulier, $O_{N'}$ est l'unique orbite ouverte et 
$$O_0 = \IG(N,V'^*),$$ 
qui est la grassmannienne isotrope introduite dans la section \ref{etudequotientSp2n}, est l'unique orbite fermée de $\Gr(N,V'^*)$. 


Soient: 
\begin{itemize} \renewcommand{\labelitemi}{$\bullet$}
\item $L_0 \in O_0$ et $P:=\Stab_{G'}(L_0)$ le sous-groupe parabolique de $G'$ qui préserve $L_0$,
\item $W':=\Hom(V'/L_0^{\perp},V)$ et $\nu':\ W' \rightarrow W'/\!/G$ le morphisme de passage au quotient,
\item $h_{W'}$ la fonction de Hilbert de la fibre générique de $\nu'$,
\item $\text{$G$-}\HH':=\Hilb_{h_{W'}}^{G}(W')$ et $\text{$G$-}\HH'^{\mathrm{prin}}$ sa composante principale.
\end{itemize}
On remarque que $h_s=h_{W'}$. En procédant comme pour la proposition \ref{reduction3}, on montre la

\begin{proposition}  \label{isoOnSonsymp}
On a isomorphisme de $G'$-variétés $\HHmgp \cong G' \times^P (\text{$G$-}\HH'^{\mathrm{prin}})$.
\end{proposition}

\noindent On déduit de la proposition \ref{isoOnSonsymp} et des résultats de la section \ref{HclassOngeneral} le

\begin{corollaire}  \label{symplisssseOn}
$\bullet$ Si $d \leq \frac{n+1}{2}$, alors 
$$\HHmgp \cong S^2(T),$$ 
où $T$ est le fibré tautologique de $\IG(d,V'^*)$.\\
$\bullet$ Si $d \geq n=2$, alors 
$$\HHmgp \cong Bl_0(S^2(T)),$$ 
où $T$ est le fibré tautologique de $\IG(2,V'^*)$ et $Bl_0(S^2(T))$ désigne l'éclatement de la section nulle du fibré $S^2(T)$.
\end{corollaire}

\begin{proof}
Le cas $d \leq \frac{n+1}{2}$ découle du corollaire \ref{cas_facileOn}. \\
Le cas $d \geq n=2$ découle du théorème \ref{casn2O2}.
\end{proof}

\begin{remarque}
Si $n=1$, alors on a $\HHmgp \cong S^2(T)$, où $T$ est le fibré tautologique de $\IG(1,V'^*)=\PP(V'^*)$.
\end{remarque}

\begin{remarque}  \label{rkobs}
Dans la section \ref{reductibilité_cas_symp}, nous avons montré l'existence d'une deuxième composante irréductible pour $\HHm$ en nous ramenant à montrer qu'un certain schéma de Hilbert invariant $\Hilb_{h_s}^{G}(\mu'^{-1}(0))$ était toujours non-vide. On ne peut malheureusement pas obtenir de résultat analogue ici. La raison est que, dans notre situation, le schéma $\Hilb_{h_s}^{G}(\mu'^{-1}(0))$ peut être vide. Par exemple si $G=O_2(k)$, nous allons voir dans la section \ref{MMO} que $\HHmg$ est une variété.    
\end{remarque}

\begin{remarque}
Si $d<\frac{n}{2}$, alors comme nous l'avons suggéré dans le remarque \ref{rkOninutile}, on peut construire un morphisme $G'$-équivariant $\HHmgz \rightarrow \Gr(d,V'^*)$ et obtenir un résultat de réduction pour $SO(V)$ analogue à celui obtenu pour $O(V)$ dans la proposition \ref{isoOnSonsymp}. Néanmoins, lorsque $d<\frac{n}{2}$, on a 
$$\text{$G$-}\HH'^{\mathrm{prin}} \cong W'/\!/O(V) \cong W'/\!/SO(V) \cong \text{$G^0$-}\HH'^{\mathrm{prin}}$$
et on n'obtient donc rien de nouveau dans ce cas.  
\end{remarque}

\subsection{Etude du cas \texorpdfstring{$\dim(V)=2$}{O2(k)}}  \label{MMO}

Dans cette section, on fixe $n=2$ et $d \geq 2$. Alors $\mu^{-1}(0)/\!/G=\overline{\OO_{[2^2,1^{2d-4}]}}$ et $\rho_s:\ \HHmg \rightarrow \Gr(2,V'^*)$ est le morphisme (\ref{yeah}). On a vu que $\Gr(2,V'^*)$ se décompose en deux orbites $O_0 \cup O_1$ pour l'opération de $G'$.

\begin{lemme}
Le morphisme $\rho_s$ envoie $\HHmg$ dans $O_0$.
\end{lemme} 

\begin{proof}
Soit $L_1\in O_1$ un point de l'orbite ouverte, on va montrer que la fibre schématique de $\rho_s$ en $L_1$ est vide.
On note $V'':=V'/L_1^{\perp}$ et $W'':=\Hom(V'',V)$, alors $W'' \cong V \oplus V^*$ comme $G$-module. On munit $W''$ de la forme symplectique $\Omega$ définie par (\ref{defsymp}) et on considère $\mu':\ W'' \rightarrow \gg$ l'application moment pour l'opération de $G$ dans $(W'',\Omega)$ comme nous l'avons fait pour $(W,\Omega)$ dans la section \ref{appzmomentOn}. En procédant comme pour le lemme \ref{fibrehil}, on montre que la fibre schématique de $\rho_s$ en $L_1$ est isomorphe au schéma de Hilbert invariant $\Hilb_{h_s}^{G}(\mu'^{-1}(0))$.\\
Il s'agit donc de montrer que ce schéma de Hilbert invariant est vide. Dans le cas contraire, il existe un idéal homogène $I_0$ de $k[\mu'^{-1}(0)]$ tel que $k[\mu'^{-1}(0)]/I_0 \cong \bigoplus_{M \in \Irr(G)}  M^{\oplus \dim(M)}$ comme $G$-module. On note $I$ l'image réciproque de $I_0$ par le morphisme de passage au quotient $k[W''] \rightarrow k[\mu'^{-1}(0)]$. On a
$$k[W'']_2 \cong (S^2(V'') \otimes (V_0 \oplus (\Gamma_{2 \epsilon_1} \oplus \Gamma_{-2 \epsilon_1}))) \oplus (\Lambda^2(V'') \otimes \epsilon)$$
comme $GL(V'') \times G$-module. L'idéal $I$ contient nécessairement $(S^2(V'') \otimes V_0) \oplus (\Lambda^2(V'') \otimes \epsilon)$ puisque $k[W'']/I$ contient déjà une copie de $V_0$ (l'image des constantes dans le quotient) et  $\Lambda^2(V'') \cong V_0$ comme $Sp(V'')$-module. Donc $k[W'']_2/(I \cap k[W'']_2)$ est un quotient du $G' \times G$-module $S^2(V'') \otimes (\Gamma_{2 \epsilon_1} \oplus \Gamma_{-2 \epsilon_1})$. Mais alors, un raisonnement analogue à celui effectué dans la preuve du lemme \ref{pfixeI1} permet de montrer qu'un tel quotient ne contient aucune copie de la représentation signe $\epsilon$. Le résultat s'ensuit puisque le $G$-module $k[\mu'^{-1}(0)]/I_0 \cong k[W'']/I$ est supposé contenir chaque représentation irréductible de $G$ avec une multiplicité égale à sa dimension.     
\end{proof}

Donc, avec les notations de la section \ref{MropRREDOn}, on a un morphisme $G'$-équivariant $\HHmg \rightarrow G'/P$, et donc d'après (\ref{iissoo}) on a un isomorphisme $G'$-équivariant
$$\HHmg \cong G' \times^P F$$
où $F$ est la fibre schématique de $\rho_s$ en $L_0$. Ainsi, pour déterminer $\HHmg$, il suffit de déterminer $F$. En procédant comme pour le lemme \ref{fibrehil}, on montre le  

\begin{lemme}
La fibre $F$ est isomorphe au schéma de Hilbert invariant $\text{$G$-}\HH'$ et l'opération de $P$ dans $F$ coïncide, via cet isomorphisme, avec l'opération de $P$ dans $\text{$G$-}\HH'$ induite par l'opération de $P$ dans $W'$.
\end{lemme}

\noindent D'où la 

\begin{proposition}
On a un isomorphisme $G'$-équivariant
$\HHmg \cong G' \times^P (\text{$G$-}\HH')$. 
\end{proposition}

Enfin, on a vu dans la section \ref{casO2} que $\text{$G$-}\HH'$ est une variété, donc $\HHmg$ est une variété. En particulier, on a $\HHmg=\HHmgp$ et donc, d'après le corollaire \ref{symplisssseOn}, on a $\HHmg \cong Bl_0(S^2(T))$.

%% file: symplectique_Spn_1.tex
\section{Cas de \texorpdfstring{$Sp(V)$}{Sp(V)} opérant dans \texorpdfstring{$\mu^{-1}(0)$}{}}  \label{sektionsympSpn}

On se place à nouveau dans la situation $5$ et on considère le cas particulier $W=\Hom(E,V) \times \Hom(E,V)^*$ comme dans la section \ref{sektionsympOn}. Soient $O(V')$ le sous-groupe de $GL(V')$ qui préserve une forme quadratique non-dégénérée $q'$ sur $V'$ et $G':=SO(V')$ sa composante neutre. L'opération de $GL(V') \times G$ dans $W$ induit une opération de $G' \times G$ dans $W$. On note 
$$\gg':=\Lambda^2(V')$$
l'algèbre de Lie de $G'$ et
$${\gg'}^{\leq n}:=\Lambda^2(V')^{\leq n}.$$
Le quotient $W/\!/G={\gg'}^{\leq n}$ a été étudié dans la section \ref{etudequotientSp2n}.

Le but de cette section est d'étudier le schéma de Hilbert invariant 
$\HHm:=\Hilb_{h_{s}}^{G}(\mu^{-1}(0)).$ 
Nous allons procéder comme dans les sections \ref{posisimpy} et \ref{sektionsympOn}. Lorsque $d>n$ ou ($d \leq n$ et $d$ est impair), nous verrons que $\mu^{-1}(0)/\!/G$ est irréductible et donc $h_s$ est bien définie. En revanche, lorsque $d \leq n$ et $d$ est pair, nous verrons que $\mu^{-1}(0)/\!/G$ est la réunion de deux composantes irréductibles. Cette situation ne s'était jamais produite jusqu'à présent et on ne peut alors plus parler de la fibre générique du morphisme de passage au quotient. Pour surmonter cette difficulté, on va considérer indépendamment les deux composantes irréductibles de $\mu^{-1}(0)/\!/G$ et associer à chacune d'elles un schéma de Hilbert invariant. Les principaux résultats de cette section sont la proposition \ref{reduction37} et son corollaire \ref{resuss}. Le cas particulier $n=2$ et $d=3$ a été traité par Becker dans \cite{Tanja2}.

\subsection{L'application moment} \label{appzmomentSpn}
On considère $\Omega$ la forme symplectique sur $W$ définie dans la section \ref{appzmoment}. Le groupe $G' \times G$ opère symplectiquement dans $(W,\Omega)$ et l'application moment pour l'opération de $G$ dans $(W,\Omega)$ est donnée par \cite[Proposition 3.1]{Tanja}:
$$\begin{array}{ccccc}
\mu & : & W & \to & \gg \\
& & w & \mapsto & \frac{1}{2}w (\leftexp{t}{w}) J \\
\end{array}$$
et donc
$$\mu^{-1}(0)= \{ w \in W \ |\ w (\leftexp{t}{w})=0\}.$$

\begin{remarque}
On vérifie que le plus grand sous-groupe de $GL(V')$ qui stabilise $\mu^{-1}(0)$ dans $W$ est $O(V')$. Cependant, pour des raisons d'ordre pratique (étude des orbites nilpotentes de $\gg'$), on préfère considérer l'opération de $SO(V')$.
\end{remarque}

Le morphisme de passage au quotient $W \rightarrow W/\!/O(V')$ est donné par $w \mapsto w (\leftexp{t}{w})$. On en déduit que les schémas $\mu^{-1}(0)$ et $\NNN(W,O(V'))$ sont isomorphes. En particulier, le schéma $\mu^{-1}(0)$ est réduit si et seulement si $d \geq n$ et il est irréductible et normal lorsque $d>n$ (\cite[Theorem 9.1]{KS}). La preuve de la proposition qui suit est analogue à celle de la proposition \ref{nilcôneOn}. 

\begin{proposition} \label{dimmudezero33}
Le schéma $\mu^{-1}(0)$ est 
\begin{itemize} 
\item une variété de dimension $2dn-\frac{1}{2}n(n+1)$ si $d>n$,
\item la réunion de deux fermés irréductibles de dimension $dn+\frac{1}{2}d(d-1)$ sinon.
\end{itemize}
\end{proposition}

\subsection{Etude du morphisme de passage au quotient}  \label{morppquotientSpn}

On note 
$$\nu:\ \mu^{-1}(0) \rightarrow \mu^{-1}(0)/\!/G$$ 
le morphisme de passage au quotient et 
$$N:=\min(d,n).$$ 
Dans cette section, nous allons décrire géométriquement le quotient $\mu^{-1}(0)/\!/G$ et voir que, contrairement aux situations étudiées jusqu'à présent, il n'est pas toujours irréductible.

On a une correspondance entre certaines partitions de $n'=2d$ et les orbites nilpotentes de $\gg'$ (voir \cite[§5.1]{CoMc}). Soit $(d_1 \geq \ldots \geq d_k)$ une partition de $n'$ telle que au moins l'un des $d_i$ est impair et les $d_i$ pairs sont présents avec une multiplicité paire. A une telle partition, on associe une unique orbite nilpotente de $\gg'$ que l'on note $\OO_{[d_1,\ldots,d_k]}$. 
On peut aussi considérer une partition "très paire" $(d_1 \geq \ldots \geq d_k)$ de $n'$, c'est-à-dire une partition telle que tous les $d_i$ sont pairs et présents avec multiplicité paire. A une telle partition, il correspond cette fois deux orbites nilpotentes distinctes que l'on note $\OO_{[d_1, \ldots, d_k]}^{I}$ et $\OO_{[d_1, \ldots, d_k]}^{I\!I}$. Ces deux orbites sont échangées par l'opération de n'importe quel élément de $O(V') \backslash SO(V')$. Le résultat qui suit est démontré dans \cite[Proposition 3.6]{Tanja}:

\begin{proposition}
On a l'égalité ensembliste
$$
\mu^{-1}(0)/\!/G = \left\{
    \begin{array}{ll}
        \overline{\OO_{[2^n,1^{2(d-n)}]}} & \text{ si } d>n,\\
        \overline{\OO_{[2^{d-1},1^2]}} & \text{ si } d < n \text{ et } d \text{ est impair},\\
        \overline{\OO_{[2^d]}^{I}} \cup \overline{\OO_{[2^d]}^{I\!I}} & \text{ si } d \leq n \text{ et } d \text{ est pair}.
    \end{array}
\right.
$$
\end{proposition}   

\begin{remarque}
Le schéma $\mu^{-1}(0)/\!/G$ n'est pas toujours réduit. Par exemple si $n=2$ et $d=1$, on vérifie que $\mu^{-1}(0)/\!/G \cong \Spec(k[x]/(x^2))$ qui n'est pas réduit. 
Dorénavant, on considère toujours $\mu^{-1}(0)$ et $\mu^{-1}(0)/\!/G$ munis de leurs structures réduites pour simplifier.  
\end{remarque}

\begin{corollaire}
Les orbites pour l'opération de $G'$ dans le quotient $\mu^{-1}(0)/\!/G$ sont:
\begin{itemize} \renewcommand{\labelitemi}{$\bullet$}
\item $U_i:=\OO_{[2^i,1^{2(d-i)}]}$ pour $i=0,2,4, \ldots,n$,  si $d>n$, 
\item $U_i:=\OO_{[2^i,1^{2(d-i)}]}$ pour $i=0,2,4, \ldots, d-1$, si $d<n$ et $d$ impair,
\item $U_i:=\OO_{[2^i,1^{2(d-i)}]}$ pour $i=0,2,4, \ldots, d-2$ et $U_d^{I}:=\OO_{[2^d]}^{I}$, $U_d^{I\!I}:=\OO_{[2^d]}^{I\!I}$, si $d \leq n$ et $d$ pair.
\end{itemize}
\end{corollaire}

Les adhérences de ces orbites sont imbriquées de la façon suivante:
\begin{equation*}
\left\{
    \begin{array}{ll}
        \{0\}=\overline{U_0} \subset \overline{U_2} \subset \cdots  \subset \overline{U_n} &\text{ lorsque $d>n$,} \\
         \{0\}=\overline{U_0} \subset \overline{U_2} \subset \cdots \subset \overline{U_{d-1}} &\text{ lorsque $d<n$ et $d$ est impair,} \\
         \{0\}=\overline{U_0} \subset \overline{U_2} \subset \cdots \subset \overline{U_{d-2}} \subset \overline{U_d^{I}} \text{ et } \overline{U_{d-2}} \subset \overline{U_d^{I\!I}}&\text{ lorsque $d \leq n$ et $d$ est pair.}    
    \end{array}
\right.
\end{equation*}

Donc, si $d \leq n$ et $d$ pair, le quotient $\mu^{-1}(0)/\!/G$ est la réunion des adhérences des deux orbites nilpotentes $U_d^{I}$ et $U_d^{I\!I}$. L'intersection des adhérences de ces deux orbites est $\overline{U_{d-2}}$. En revanche, lorsque ($d>n$) ou ($d \leq n$ et $d$ est impair), le quotient $\mu^{-1}(0)/\!/G$ est irréductible. La géométrie de l'adhérence d'une orbite nilpotente dans $\gg'$ est étudiée dans \cite{Hes} et \cite{KP3}.

\begin{proposition}
Si $d>n$, alors la variété $\mu^{-1}(0)/\!/G$ est normale, de dimension $2dn-n(n+1)$ et son lieu singulier est $\overline{U_{n-2}}$.\\
Si $d \leq n$, alors chaque composante irréductible de $\mu^{-1}(0)/\!/G$ est normale et de dimension $d(d-1)$. Le lieu singulier de $\mu^{-1}(0)/\!/G$ est $\overline{U_{d-2}}$ (resp. $\overline{U_{d-3}}$) lorsque $d$ est pair (resp. $d$ est impair). 
\end{proposition} 

\begin{proof}
La normalité des composantes irréductibles de $\mu^{-1}(0)/\!/G$ est donnée par \cite[Criterion 2]{Hes}. Leur dimension est donnée par \cite[Corollary 6.1.4]{CoMc}. 
Enfin, le lieu singulier est donné par \cite[Theorem 2]{KP3}.  
\end{proof}

\begin{remarque}
Nous verrons dans la section \ref{appendice2} que les composantes irréductibles de $\mu^{-1}(0)/\!/G$ sont symplectiques et mais qu'elles admettent des résolutions symplectiques si et seulement si $d \leq n+1$. 
\end{remarque}

Si $d<n$ et $d$ impair, alors la fibre générique de $\nu$ est réductible ce qui rend plus compliqué la détermination de sa fonction de Hilbert. On exclura donc toujours ce cas par la suite. Il reste deux cas de figure à considérer:\\
$\bullet$ \textit{Cas $d>n$}.
La fibre de $\nu$ en un point de l'orbite ouverte est irréductible de dimension $\frac{1}{2}n(n+1)$. On note $h_s$ la fonction de Hilbert de la fibre générique de $\nu$,
$$\HHm:=\Hilb_{h_s}^G(\mu^{-1}(0))$$ 
et $\HHmp$ la composante principale de $\HHm$.\\
$\bullet$ \textit{Cas $d\leq n$ et $d$ pair.}
On note $X^I$, $X^{I\!I}$ les deux composantes irréductibles de $\mu^{-1}(0)$ et $Y^I:=\overline{U_d^{I}}$, $Y^{I\!I}:=\overline{U_d^{I\!I}}$ les deux composantes irréductibles de $\mu^{-1}(0)/\!/G$. Quitte à échanger $X^I$ et $X^{I\!I}$, on a $X^I/\!/G=Y^I$ et $X^{I\!I}/\!/G=Y^{I\!I}$. On note $\nu_{I}:\ X^I \rightarrow Y^I$ et $\nu_{I\!I}:\ X^{I\!I} \rightarrow Y^{I\!I}$ les morphismes de passage au quotient. On vérifie que la dimension de la fibre de $\nu_I$ (resp. de $\nu_{I\!I}$) en un point de l'orbite ouverte de $Y^I$ (resp. de $Y^{I\!I}$) est $dn-\frac{1}{2}d(d-1)$. Le groupe $O(V')$ opère transitivement sur $U_d^I \cup U_d^{I\!I}$ donc les fibres génériques de $\nu_{I}$ et $\nu_{I\!I}$ sont isomorphes. En particulier, elles ont la même dimension et la même fonction de Hilbert $h_s$. On note 
$$
\left\{
    \begin{array}{l}
\HHm:=\Hilb_{h_s}^G(\mu^{-1}(0)),\\
\HHmx:=\Hilb_{h_s}^G(X^I),\\
\HHmy:=\Hilb_{h_s}^G(X^{I\!I}).
   \end{array}
\right.
$$
Enfin, on note $\HHmxp$ et $\HHmyp$ les composantes principales de $\HHmx$ et $\HHmy$ respectivement. Comme $X^I$ et $X^{I\!I}$ sont des fermés de $\mu^{-1}(0)$, on a l'inclusion $\HHmx \cup \HHmy \subset \HHm$, mais celle-ci est stricte a priori.

\begin{notation}
Si $d < n$ et $d$ est pair, on note 
$$H:=\left\{ \begin{bmatrix}
M  &0_{n-d,d} \\
0_{d,n-d}  & I_d 
\end{bmatrix},\ M \in Sp_{n-d}(k) \right\} \cong Sp_{n-d}(k)$$
qui est un sous-groupe réductif de $G$. 
\end{notation}

\begin{lemme} \label{fibreUnsymp}
Si $d>n$, la fibre de $\nu$ en un point de $U_n$ est isomorphe à $G$.\\
Si $d=n$, la fibre de $\nu_I$ (resp. de $\nu_{I\!I}$) en un point de $U_n^{I}$ (resp. de $U_n^{I\!I}$) est isomorphe à $G$.\\
Si $d<n$ et $d$ est pair, la fibre de $\nu_I$ (resp. de $\nu_{I\!I}$) en un point de $U_n^{I}$ (resp. de $U_n^{I\!I}$) est isomorphe à $G/H$. 
\end{lemme}

\begin{proof}
La preuve est analogue à celle du lemme \ref{fibreUnsymp1}.
\end{proof}

\noindent D'où la

\begin{proposition} \label{fctHSp2} 
Soit $M \in \Irr(G)$, alors
\begin{itemize} \renewcommand{\labelitemi}{$\bullet$}
\item $h_s(M)=\dim(M)$ si $d \geq n$,
\item $h_s(M)=\dim(M^H)$ si $d<n$ et $d$ est pair. 
\end{itemize}
\end{proposition}

Le fermé $\mu^{-1}(0)$ est $G' \times G$-stable dans $W$, donc le schéma $\HHm$ est un sous-schéma fermé $G'$-stable de $\Hilb_{h_s}^{G}(W)$. Ensuite, d'après les propositions \ref{fcthilbSpn} et \ref{fctHSp2}, les fonctions de Hilbert $h_W$ et $h_s$ coïncident lorsque $d \geq n$ ou $d < \frac{n}{2}$ et alors $\HHm$ s'identifie à un sous-schéma fermé de $\Hilb_{h_W}^{G}(W)$; ce dernier a été étudié dans la section \ref{schcasSymp}. De même, lorsque $d=n$ ou ($d <\frac{n}{2}$ et $d$ est pair), les schémas $\HHmx$ et $\HHmy$ s'identifient à des sous-schémas fermés de $\Hilb_{h_W}^{G}(W)$.

%% file: symplectique_Spn_2.tex
\subsection{Construction d'un morphisme équivariant et réduction}  \label{sssympSpn}
Dans cette section, on procède comme dans les sections \ref{MropRRED} et \ref{MropRREDOn} pour  montrer que l'étude des composantes principales $\HHmp$, $\HHmxp$ et $\HHmyp$ se ramène à l'étude du schéma de Hilbert invariant effectuée dans la section \ref{schcasSymp}.

Dans la section \ref{redGrass2}, on a construit un morphisme $G'$-équivariant 
\begin{equation} \label{rosympSpn}
\Hilb_{h_s}^G(W) \rightarrow \Gr(h_s(V),V'^*).
\end{equation}
D'après la proposition \ref{fctHSp2}, on a $h_s(V)=N$. La restriction du morphisme (\ref{rosympSpn}) à $\HHm$ donne un morphisme $G'$-équivariant 
\begin{equation}  \label{zorrro}
\rho_s:\ \HHm \rightarrow \Gr(N,V'^*).
\end{equation} 
On rappelle que l'on note $q'$ la forme quadratique non-dégénérée sur $V'$ préservée par $O(V')$. On a $V' \cong V'^*$ comme $O(V')$-module et donc $q'$ s'identifie à une forme quadratique non-dégénérée sur $V'^*$. Pour $i=0, \ldots,N$, on note
$$O_i:=\{L \in \Gr(N,V'^*)\ |\ {q'}_{|L} \text{ est de rang } i\}.$$
Si $d>n$, alors les $O_i$ sont les $n+1$ orbites pour l'opération de $G'$ dans $\Gr(n,V')$.
En revanche, si $d \leq n$ et $d$ est pair, alors les $O_i$ sont des $G'$-orbites pour $i=1,\ldots,d$ mais  $O_0=\OG(d,V')$, qui est la grassmannienne isotrope introduite dans la section \ref{description_quotientOn}, est la réunion de deux $G'$-orbites que l'on note $O_0^I$ et $O_0^{I\!I}$ et qui sont échangées par l'opération de n'importe quel élément de $O(V') \backslash SO(V')$. 
Dans tous les cas, on a:
$$ \OG(N,V')=\overline{O_0} \subset \overline{O_1} \subset \cdots \subset \overline{O_N}=\Gr(N,V').$$

\noindent Ensuite, soient: 
\begin{itemize} \renewcommand{\labelitemi}{$\bullet$}
\item $L_0 \in \OG(N,V'^*)$ et $P:=\Stab_{G'}(L_0)$ le sous-groupe parabolique de $G'$ qui préserve $L_0$,
\item $W':=\Hom(V'/L_0^{\perp},V)$ et $\nu':\ W' \rightarrow W'/\!/G$ le morphisme de passage au quotient,
\item $h_{W'}$ la fonction de Hilbert de la fibre générique de $\nu'$,
\item $\HH':=\Hilb_{h_{W'}}^{G}(W')$ et $\HH'^{\mathrm{prin}}$ sa composante principale.
\end{itemize}
On remarque que $h_s=h_{W'}$. Par ailleurs, lorsque $d \leq n$ et $d$ est pair, on vérifie que le morphisme $\rho_s$ envoie $\HHmxp$ dans l'une des deux composantes irréductibles de $O_0$ et $\HHmyp$ dans l'autre composante. Quitte à échanger $O_0^{I}$ et $O_0^{I\!I}$, on peut supposer que $\rho_s$ envoie $\HHmxp$ dans $O_0^{I}$ et $\HHmyp$ dans $O_0^{I\!I}$.
En procédant comme pour la proposition \ref{reduction3}, on montre la

\begin{proposition} \label{reduction37}
On a les isomorphismes de $G'$-variétés suivants:
\begin{itemize}
\item si $d>n$, alors 
$$\HHmp \cong G' {\times}^{P} \HH'^{\mathrm{prin}},$$
\item si ($d \leq n$ et $d$ est pair) et $L_0 \in O_0^I$ (resp. $L_0 \in O_0^{I\!I}$), alors 
$$\HHmxp \cong G' {\times}^{P} \HH'^{\mathrm{prin}} \text{(resp. $\HHmyp \cong G' {\times}^{P} \HH'^{\mathrm{prin}}$).}$$
\end{itemize}
\end{proposition}

\noindent On déduit de la proposition \ref{reduction37} et des résultats de la section \ref{schcasSymp} le

\begin{corollaire}  \label{resuss}
\begin{itemize}  \renewcommand{\labelitemi}{$\bullet$}
\item Si $d \leq \frac{n}{2}+1$ et $d$ est pair, alors 
$$\HHmxp \cong \Lambda^2(T_I) \text{ et } \HHmyp \cong \Lambda^2(T_{I\!I}),$$ 
où $T_I$ (resp. $T_{I\!I}$) est le fibré tautologique de $O_0^{I}$ (resp. de $O_0^{I\!I}$).
\item Si $d=n=4$, alors 
$$\HHmxp \cong Bl_0(\Lambda^2(T_I)) \text{ et } \HHmyp \cong Bl_0(\Lambda^2(T_{I\!I})),$$ 
où $Bl_0(\Lambda^2(T_I))$ (resp. $Bl_0(\Lambda^2(T_{I\!I}))$) désigne l'éclatement de la section nulle du fibré $\Lambda^2(T_I)$ (resp. du fibré $\Lambda^2(T_{I\!I})$).
\item Si $d>n=2$, alors 
$$\HHmp \cong  \Lambda^2(T),$$ 
où $T$ est le fibré tautologique de $\OG(2,V'^*)$.
\item Si $d>n=4$, alors 
$$\HHmp \cong Bl_0(\Lambda^2(T))$$ 
qui est l'éclatement de la section nulle du fibré $\Lambda^2(T)$, où $T$ est le fibré tautologique de $\OG(4,V'^*)$.
\end{itemize}
\end{corollaire}

\begin{proof}
Les cas $d \leq \frac{n}{2}+1$ et $d>n=2$ découlent du corollaire \ref{cas_facileSpn}.\\ 
Les cas $d \geq n=4$ découlent du théorème \ref{casn2Sp4}.
\end{proof}

\begin{remarque}
On ne peut malheureusement pas procéder comme nous l'avons fait dans la section \ref{reductibilité_cas_symp} pour déterminer si $\HHm$ est ou non réductible (pour les mêmes raisons que celles données dans la remarque \ref{rkobs}). A priori, $\HHm$ peut admettre plusieurs composantes irréductibles. Cependant, si $n=2$ nous allons voir dans la section \ref{cassympn2} que $\HHm$ est toujours une variété.  
\end{remarque}

\subsection{Etude du cas \texorpdfstring{$\dim(V)=2$}{Sp2(k)}} \label{cassympn2}

Dans cette section, on fixe $d \geq n=2$. Alors $G \cong Sp_2(k)=SL_2(k)$,\\ 
$\mu^{-1}(0)/\!/G= \left\{
    \begin{array}{ll}
           \overline{\OO_{[2^2,1^{2d-4}]}} & \text{ si } d \geq 3,\\
      \overline{\OO_{[2^2]}^{I}} \cup \overline{\OO_{[2^2]}^{I\!I}} & \text{ si } d=2,\\
     \end{array}
\right.$
et $\rho_s:\ \HHm \rightarrow \Gr(2,V'^*)$ est le morphisme (\ref{zorrro}). Nous allons utiliser les résultats des sections précédentes et de la section \ref{casSln} pour déterminer $\HHm$ comme $G'$-schéma.  

\begin{proposition}  \label{wxcv}
Si $d \geq 3$, alors $\HHm=\HHmp$ est une variété lisse isomorphe à 
$$Bl_0(\overline{\OO_{[2^2,1^{2d-4}]}}):=\left \{(f,L) \in \overline{\OO_{[2^2,1^{2d-4}]}} \times \PP (\overline{\OO_{[2^2,1^{2d-4}]}}) \ \mid  \ f \in L  \right \}$$ 
et le morphisme de Hilbert-Chow $\gamma$ s'identifie à l'éclatement $Bl_0(\overline{\OO_{[2^2,1^{2d-4}]}}) \rightarrow \overline{\OO_{[2^2,1^{2d-4}]}}$.\\
Si $d=2$, alors $\HHm=\HHmxp \cup \HHmyp$ est la réunion de deux variétés lisses isomorphes à $Bl_0(\overline{\OO_{[2^2]}^{I}})$ et $Bl_0(\overline{\OO_{[2^2]}^{I\!I}})$ respectivement. De plus, l'intersection $\HHmxp \cap \HHmyp$ est l'ensemble des idéaux homogènes de $\HHm$. Enfin, la restriction de $\gamma$ sur $\HHmxp$ (resp. sur $\HHmyp$) s'identifie à l'éclatement de $\overline{\OO_{[2^2]}^{I}}$ (resp. de $\overline{\OO_{[2^2]}^{I\!I}}$) en $0$.    
\end{proposition}

\begin{proof}
Les cas $d=2$ et $d \geq 3$ se traitent de manière analogue. On considère le cas où $d \geq 3$. D'après le théorème \ref{SLLn}, on a une immersion fermée
$$\gamma \times \rho_s:\ \HHm \hookrightarrow Y:=\left \{(f,L) \in \overline{\OO_{[2^2,1^{2d-4}]}} \times \PP ({\gg'}^{\leq 2}) \ \mid  \ f \in L  \right \}.$$ 
On vérifie que $Y$ est la réunion de deux variétés $C_1$ et $C_2$ définies par:
\begin{align*}
C_1&:=Bl_0(\overline{\OO_{[2^2,1^{2d-4}]}}), \\
C_2&:=\left \{(f,L) \in \overline{\OO_{[2^2,1^{2d-4}]}} \times \PP ({\gg'}^{\leq 2}) \ |\ f=0 \right \} \cong \PP ({\gg'}^{\leq 2}).
\end{align*}
Ces deux variétés sont de dimension $4d-6$ et $4d-4$ respectivement. Le morphisme $\gamma \times \rho_s$ envoie $\HHmp$ dans $C_1$; ce sont deux variétés de même dimension, donc $\gamma \times \rho_s:\ \HHmp \rightarrow C_1$ est un isomorphisme.\\
Ensuite, on vérifie que la composante $C_2$ est formée des idéaux homogènes de $k[W]$ (c'est une conséquence de la proposition \ref{decompoisoSp}). On a un isomorphisme $G'$-équivariant 
\begin{equation}  \label{isomonn}
\PP(\gg'^{\leq 2}) \cong \Gr(2,V'^*)
\end{equation} 
qui permet d'identifier tout point de $\PP(\gg'^{\leq 2})$ avec un sous-espace de $V'^*$ de dimension $2$. Via cette identification, l'idéal $I_{L}$ de $k[W]$ associé au point $(0,L) \in C_2$ est l'idéal engendré par les $G$-invariants homogènes de degré positif de $k[W]$ ainsi que par le $G$-module $L^{\perp} \otimes V \subset k[W]_1$. Nous allons montrer que $I_L$ est un point de $\HHm$ si et seulement si $L \in \OG(2,V'^*)$. Le résultat en découlera puisque $\PP (\overline{\OO_{[2^2,1^{2d-4}]}})$ s'identifie à $\OG(2,V'^*)$ via l'isomorphisme (\ref{isomonn}) et que $\left \{(f,L) \in \overline{\OO_{[2^2,1^{2d-4}]}} \times \PP (\overline{\OO_{[2^2,1^{2d-4}]}})  \ |\ f=0 \right\}$ est une sous-variété de $C_1$.\\  
On note $W':=\Hom(V'/L^{\perp},V)$, alors
$$k[W']_2 \cong S^2(V'/L^{\perp}) \otimes S^2(V) \oplus \Lambda^2(V'/L^{\perp}) \otimes \Lambda^2(V)$$ 
comme $G$-module. On note $I'$ l'idéal de $k[W']$ engendré par $\Lambda^2(V'/L^{\perp}) \otimes \Lambda^2(V) \subset k[W']_2$. On a 
$$k[W']/I' \cong k[W]/I_L \cong \bigoplus_{M \in \Irr(G)}  M^{\oplus \dim(M)}$$ 
comme $G$-module. Donc
\begin{align*}
I_L \in \HHm & \Leftrightarrow I_L \supset V'_0 \otimes S^2(V) \\
             &  \Leftrightarrow  {q'}_{|L}=0 \\
             & \Leftrightarrow L \in \OG(2,V'^*)
\end{align*} 
et le résultat s'ensuit.     
\end{proof}

\begin{remarque}
Lorsque $d=3$ et $n=2$, Becker avait déjà montré dans \cite{Tanja2} que $\HHm=\HHmp$ est une variété lisse. 
\end{remarque}

\begin{remarque}
Dans la preuve de la proposition \ref{wxcv}, on a montré que si $d > n=2$, alors les idéaux homogènes de $\HHm$ sont contenus dans $\HHmp$. En reprenant les arguments de la fin de la preuve, on vérifie que cette assertion est vraie plus généralement si $d > n \geq 2$. 
\end{remarque}

%% file: perspectives.tex
\chapter*{Questions ouvertes}   
\addcontentsline{toc}{chapter}{Questions ouvertes} 

Dans ce mémoire, nous avons déterminé plusieurs familles infinies de schémas de Hilbert invariants pour l'opération des groupes algébriques classiques dans certains schémas affines. Plus précisément, nous avons étudié d'une part le cas dit "classique": \\ 
$\bullet$ $G=GL_n(k)$ opère dans $W=V^{\oplus n_1} \oplus V^{* \oplus n_2}$, ou bien\\
$\bullet$ $G=SL_n(k)$, $O_n(k)$, $SO_n(k)$ ou $Sp_{2n}(k)$ opère dans $W=V^{\oplus n'}$,\\  
où $V$ est la représentation standard de $G$, et $V^*$ est sa duale;\\
d'autre part le cas dit "symplectique":\\
$\bullet$ $G=GL_n(k)$, $O_n(k)$, $SO_n(k)$ ou $Sp_{2n}(k)$ opère dans la fibre en $0$ de l'application moment. \\
On note $h_W$ la fonction de Hilbert de la fibre générique du morphisme $W \rightarrow W/\!/G$ et 
$$\HH:=\Hilb_{h_W}^G(W),$$ 
le schéma de Hilbert invariant associé au $G$-schéma affine $W$ et à la fonction de Hilbert $h_W$. Nous avons pu déterminer $\HH$ pour des petites valeurs de $n$, mais le cas général reste ouvert et paraît très compliqué. Les schémas de Hilbert invariants semblent donc respecter la loi de Murphy qui se traduit ici par: «Plus $n$ est grand, plus $\HH$ est pathologique».

\section*{Cas classique}

Soit $G'$ le sous-groupe algébrique réductif de $\Aut^G(W)$ défini dans la section \ref{lesdiffsituations}, et $B'$ un sous-groupe de Borel de $G'$. Si $n'=n=3$ et  $G=GL_n(k)$, $O_n(k)$ ou $SO_n(k)$, on a montré que $\HH$ est singulier (théorèmes \ref{casn3}, \ref{casn3O3} et \ref{casn3SO3}) en déterminant l'espace tangent aux points fixes de $B'$. 

\begin{question1}
Est-ce que $\HH$ est réduit, irréductible?
\end{question1}
\begin{question2}
Est-ce que la composante principale $\HHp$ est lisse?
\end{question2}

On remarque que ces deux questions sont liées. En effet, si $\HHp$ est lisse alors nécessairement $\HH$ est réductible et/ou non-réduit. Pour répondre à la question 1, une stratégie possible serait de calculer l'espace de base de la déformation verselle aux points fixes de $B'$ (voir par exemple \cite[§3]{Stevens}). Lorsque $G=O_3(k)$ ou $GL_3(k)$, je pense que la composante des idéaux homogènes de $\HH$ donne une composante irréductible distincte de $\HHp$, et donc que $\HH$ est réductible. \\
Dans tous les exemples de ce mémoire (à l'exception du cas  $G=O_3$ qui reste à déterminer), le schéma $\HH$ est connexe. Il est de plus toujours singulier dès lors que $W//G$ admet au moins $4$ orbites pour l'opération de $G'$. D'où la
 
\begin{question3}
Si $G=GL_n(k)$, $O_n(k)$, $SO_n(k)$, et si $n \geq 4$ (ou si $G=Sp_{2n}(k)$, et si $n \geq 3$), est-ce que $\HH$ est toujours connexe et singulier? 
\end{question3}

Des calculs effectués pour $GL_4(k)$ et $Sp_6(k)$ (et qui ne figurent pas dans ce mémoire par égard pour le lecteur) laissent penser que la réponse à la question 3 est positive. Pour montrer la connexité, on peut essayer de procéder comme dans la section \ref{grobner}: on détermine les points fixes de $B'$ puis on utilise des bases de Gröbner pour construire des chaînes de courbes rationnelles entre les points fixes. Pour montrer que $\HH$ est singulier, il suffit (par un argument de semi-continuité) de calculer la dimension de l'espace tangent aux points fixes de $B'$. Malheureusement, la méthode calculatoire employée dans ce mémoire est inapplicable lorsque $n$ est grand. Quelles que soient les méthodes employées pour répondre à la question 3, il sera sans doute nécessaire de commencer par répondre à la 

\begin{question4}
Quels sont les points fixes de $B'$ dans $\HH$?
\end{question4} 

Plus précisément, on aimerait avoir une description de chacun de ces points fixes en termes de $B' \times G$-sous-modules de $k[W]$. 

\begin{conjecture1}
Si $I$ est un point fixe de $B'$ dans $\HH$, alors l'idéal $I$ est engendré par ses composantes homogènes de degré $\leq n$. 
\end{conjecture1}

\begin{conjecture2}
Si $G=GL_n(k)$, alors $\HH$ admet un unique point fixe de $B'$, et $B'$ est son groupe d'isotropie. 
\end{conjecture2}

Par ailleurs, dans le cas du schéma de Hilbert multigradué de Haiman et Sturmfels (\cite{HS}), de nombreux exemples pathologiques ont pu être déterminés en partie grâce à une description combinatoire de celui-ci (\cite[...]{Cart2,Cart,Aholt}). Une telle description combinatoire pour $\HH$ (ou au moins pour les points fixes de $B'$) serait très précieuse mais reste à expliciter.

Signalons enfin que quelques cas particuliers ont été mis sous silence dans ce mémoire. Par exemple, si $G=GL_3(k)$, les cas ($n_1=2$ et $n_2 \geq 3$) et
($n_2=2$ et $n_1 \geq 3$) n'ont pas été traités. La raison en est que la fonction de Hilbert est alors compliquée, et l'on ne peut donc pas appliquer le principe de réduction pour se ramener à un cas plus simple déjà traité. On peut toutefois supposer que les méthodes employées pour traiter les autres cas s'appliquent encore.

\section*{Cas symplectique}

Soient $W=V^{\oplus d} \oplus V^{* \oplus d}$, où $d$ est entier positif, $\mu^{-1}(0)$ la fibre en $0$ de l'application moment $\mu:\ W \rightarrow Lie(G)$, $h_s$ la fonction de Hilbert de la fibre générique du morphisme $\mu^{-1}(0) \rightarrow  \mu^{-1}(0)//G$, et 
$$\HHm:=\Hilb_{h_s}^G(\mu^{-1}(0)).$$ 
On a vu que $\HHm$ est toujours réductible lorsque $G=GL_n(k)$ et $d \geq 2n$ (proposition \ref{noIrred}).

\begin{question5}
Si $G=O_n(k)$, $SO_n(k)$ ou $Sp_{2n}(k)$ et si $d \geq n$, est-ce que $\HHm$ est réductible? 
\end{question5}
\begin{question6}
Est-ce que la composante principale $\HHmp$ est toujours lisse?
\end{question6}

Lorsque $n \geq 3$, je pense que $\HHm$ est réductible par analogie avec le cas des groupes finis. Pour le prouver, il suffirait par exemple de montrer que la fibre du morphisme $\rho$ (défini dans les sections \ref{MropRREDOn} et \ref{sssympSpn}) en un point de l'orbite ouverte est non-vide. Malheureusement, je n'ai pas été capable de déterminer cette fibre en général. 
Quant à la lissité de $\HHmp$, le principe de réduction nous dit qu'elle est équivalente à la lissité de $\HHp$ dans le cas classique. 
 
Comme dans le cas classique, on a été amené à négliger certains cas. Par exemple si $G=GL_n(k)$, $d<2n$ et $d$ est impair, alors on ne connaît pas $h_s$; je n'ai donc pas pu déterminer des exemples de schémas de Hilbert invariants pour cette configuration (à l'exception du cas trivial $d=1$).
 
Enfin, il serait intéresser de déterminer $\HHm$ pour l'opération de $SL_n(k)$ dans $\mu^{-1}(0) \subset V^{\oplus d} \oplus V^{* \oplus d}$. Ce cas n'a pas été traité dans ce mémoire (excepté pour $n=2$) essentiellement par manque de temps.

\section*{Déterminer la composante principale à l'aide de la théorie des plongements des espaces homogènes sphériques}

Dans le cas classique, la variété $W/\!/G$ est normale et, dans chaque cas, il est connu qu'elle admet une orbite ouverte pour l'opération de $B'$. De même, dans le cas symplectique, les composantes irréductibles de $\mu^{-1}(0)/\!/G$ sont normales et elles admettent une orbite ouverte pour l'opération de $B'$ d'après \cite[§4.2 et §4.3]{Pany}. Autrement dit, tous les quotients que l'on a regardé dans ce mémoire sont des \textit{variétés sphériques} (ou des réunions de variétés sphériques dans la section \ref{sektionsympSpn}). 
On rappelle qu'une $G'$-variété $X$ est dite \textit{sphérique} si $X$ est normale et si $X$ contient une orbite ouverte pour l'opération de $B'$. En particulier, un espace $G'$-homogène $U$ est sphérique s'il contient une orbite ouverte pour l'opération de $B'$. On dit que $X$ est un \textit{plongement sphérique} de $U$ si $X$ est une variété sphérique qui admet une $G'$-orbite ouverte isomorphe à $U$. A tout plongement sphérique, on peut associer un objet combinatoire appelé \textit{éventail colorié} qui classifie le plongement. En particulier, si $G'=B'$ est un tore, alors on retrouve la théorie des variétés toriques. Pour une introduction à la théorie des plongements sphériques, le lecteur peut par exemple consulter \cite{Tim}. 

Tous les quotients qui apparaissent dans ce mémoire sont des plongements sphériques connus (variétés déterminantielles, adhérences d'orbites nilpotentes,...) et on peut donc aisément déterminer l'éventail colorié associé. Par ailleurs, le morphisme de Hilbert-Chow $\gamma:\ \HHp \rightarrow W/\!/G$ est $G'$-équivariant, birationnel et propre. Donc $\HHp$ contient une orbite ouverte qui est un espace homogène sphérique. On souhaite utiliser la théorie des plongements sphériques pour déterminer $\HHp$. Malheureusement, on ne sait pas si $\HHp$ est normale et il faut donc considérer sa normalisation $\widetilde{\HHp}$; celle-ci est un plongement sphérique de l'orbite ouverte de $W/\!/G$. Le morphisme $\widetilde{\gamma}:  \widetilde{\HHp} \rightarrow W/\!/G$ obtenu en composant $\gamma$ avec le morphisme de normalisation $\widetilde{\HHp} \rightarrow \HHp$ vérifie les mêmes propriétés que $\gamma$. Si $\HHp$ admet une unique orbite fermée $F$ isomorphe à $G'/B'$, on montre que l'éventail colorié de $\widetilde{\HHp}$ est entièrement déterminé par celui  de $W/\!/G$. En général, la situation est plus compliquée, mais on peut encore espérer décrire (au moins partiellement) l'éventail colorié de $\widetilde{\HHp}$ à partir de celui de $W/\!/G$.

Cette méthode offre une approche possible pour étudier la composante principale $\HHp$ lorsque $\HH$ est réductible ou non-réduit. Cependant, il semble que l'on ne puisse pas faire mieux que de décrire la normalisation de $\HHp$. D'où la 

\begin{question7}
Est-ce que $\HHp$ est normale?
\end{question7}

On peut bien sûr reprendre tout ce qui précède en remplaçant $W$ par $\mu^{-1}(0)$ et $\HHp$ par $\HHmp$. Je pense que cette approche du schéma de Hilbert invariant par la théorie des plongements des espaces homogènes sphériques est prometteuse et j'espère la développer ultérieurement.

%% file: appendiceA1.tex
\chapter{Résolutions crépantes de \texorpdfstring{$W/\!/G$}{W//G} et \texorpdfstring{$\text{ résolutions symplectiques de }\mu^{-1}(0)/\!/G$}{résolutions symplectiques de W///G}} \label{appendiceA}

\section{Résolutions crépantes de \texorpdfstring{$W/\!/G$}{W//G}} \label{appendice1}

\subsection{Généralités}

Soit $X$ une variété de Cohen-Macaulay, alors $X$ admet un faisceau dualisant (unique à isomorphisme près) que l'on note $\omega_X$. On dit que $X$ est de Gorenstein si $X$ est de Cohen-Macaulay et si $\omega_X$ est inversible. Si de plus $X$ est lisse, alors son faisceau dualisant coïncide avec son faisceau canonique $\OO_X(K_X)$. Toutes les variétés quotients $W/\!/G$ étudiées dans les situations 1 à 5 sont de Cohen-Macaulay, mais elles ne sont pas toujours de Gorenstein. De plus, ce sont toutes des cônes affines, c'est-à-dire que $\Gm$ opère dans $W/\!/G$ avec un unique point fixe (appelé sommet du cône et noté $0$) et pas d'autre orbite fermée. Le lemme qui suit nous fournit un critère pour déterminer si un cône affine est de Gorenstein. 

\begin{lemme} \label{critereGor}
Soit $X$ un cône affine de Cohen-Macaulay. Alors on a l'équivalence
$$X \text{ est de Gorenstein } \Leftrightarrow \omega_X \cong \OO_X.$$
\end{lemme}

\begin{proof}
L'implication $\Leftarrow$ est immédiate. Montrons l'autre implication.
L'opération de $\Gm$ dans $X$ munit la $k$-algèbre $R:=k[X]$ d'une graduation:
$$R= \bigoplus_{i \geq 0} R_i$$
où $\Gm$ opère dans $R_i$ par $t.f=t^if$ pour tout $t \in \Gm$ et pour tout $f \in R_i$.\\
De même, le $R$-module $M:=\omega_X(X)$ est un $R$-module gradué:
$$M= \bigoplus_{j \geq 0} M_j \text{ avec pour tous } i,j \geq 0,\ R_i.M_j \subset M_{i+j}.$$
On souhaite montrer que $M \cong R$ comme $R$-module. On note $R_+:=\bigoplus_{i \geq 1} R_i$ l'idéal maximal homogène de $R$ et soit $\psi:\ M \rightarrow M/R_+ M \cong k$ le morphisme de passage au quotient. 
Il existe $\tilde{m} \in M$, homogène, tel que $\psi(\tilde{m})=1$. On considère l'application 
$$\begin{array}{ccccc}
\phi & : & R & \to & M \\
& & r & \mapsto & r.\tilde{m} \\
\end{array}$$
Alors $\phi$ est clairement un morphisme injectif de $R$-modules gradués et $\phi$ induit un isomorphisme $R/R_+ \rightarrow M/R_+M$. Donc, d'après le lemme de Nakayama gradué (\cite[Exercise 4.6]{Ei}), le morphisme $\phi$ est surjectif, donc c'est un isomorphisme de $R$-modules et le résultat s'ensuit. 
\end{proof}

Soient $X$ une variété de Gorenstein et $f:\ Y \rightarrow X$ une résolution des singularités de $X$. La résolution $f$ est dite crépante si $f$ préserve le faisceau dualisant de $X$, c'est-à-dire si $f^*(\omega_X) \cong \omega_{Y}$, où $\omega_X$ (resp.  $\omega_{Y}$) est le faisceau dualisant de $X$ (resp. de $Y$). Signalons qu'une variété $X$ n'admet pas toujours de résolution crépante et lorsqu'elle en admet, il peut en exister plusieurs non isomorphes. On conseille au lecteur qui aimerait en savoir plus sur les résolutions crépantes de consulter l'article d'exposition \cite{Re}. 

On souhaite déterminer les liens entre le morphisme de Hilbert-Chow $\gamma:\ \HH \rightarrow W/\!/G$ et les résolutions crépantes (lorsqu'elles existent) de $W/\!/G$ dans les situations 1,2,3 et 5. Nous aurons besoin du

\begin{lemme}  \label{pascrepant3}
Soient $X$ une variété singulière, normale et de Gorenstein, et $Y_1$, $Y_2$ des variétés lisses telles que l'on ait:
\begin{equation*}
\xymatrix{ Y_2  \ar[d]^{f_2}  \ar@/_2pc/[dd]_{g}  \\
          Y_1   \ar[d]^{f_1}  \\
          X }
\end{equation*}
où $g$ et $f_1$ sont des résolutions de $X$.\\
Si $g$ est crépante, alors $f_1$ est crépante et $f_2$ est un isomorphisme. 
\end{lemme} 

\begin{proof}
La formule de ramification nous donne les égalités suivantes dans les groupes des classes de diviseurs $Cl(Y_1)$ et $Cl(Y_2)$ respectivement:
\begin{align*} 
K_{Y_1} &= f_1^*(K_X) + E_1,\\
K_{Y_2} &= f_2^*(K_{Y_1}) + E_2,
\end{align*}
où $E_1$ (resp. $E_2$) est un diviseur à support dans le lieu exceptionnel de $f_1$ (resp. de $f_2$). On en déduit que
\begin{equation} \label{discrep13}
K_{Y_2} = g^*(K_X)+ f_2^*(E_1) + E_2.
\end{equation}
Or, $g$ est crépante par hypothèse, donc 
\begin{equation} \label{discrep14}
f_2^*(E_1) + E_2=0. 
\end{equation}
Puis, le poussé en avant préserve l'équivalence linéaire, donc
$${f_2}_*(f_2^*(E_1) + E_2)={f_2}_*(f_2^*(E_1))+{f_2}_*(E_2)=E_1.$$
Mais
$${f_2}_*(f_2^*(E_1) + E_2)={f_2}_*(0)=0$$ 
d'après (\ref{discrep14}).\\
Il s'ensuit que $E_1=0$, et donc $f_1$ est crépante. L'égalité (\ref{discrep14}) implique alors que $E_2=0$. Enfin, $f_2$ est un morphisme birationnel et surjectif entre deux variété lisses dont le diviseur de ramification est trivial, c'est donc un isomorphisme d'après \cite[§1.41]{Deb}.  
\end{proof}

\begin{remarque}
Ce lemme justifie pourquoi les résolutions crépantes sont parfois qualifiées de \textit{résolutions minimales}.
\end{remarque}

\subsection{Situation 1}

On se place dans la situation 1 et on reprend les notations de la section \ref{casSln}. On suppose que la variété $W/\!/G$ est singulière, c'est-à-dire que $1<n<n'-1$. On a vu que $W/\!/G$ est de Gorenstein et que $\gamma:\ \HH \rightarrow W/\!/G$ est une résolution des singularités de $W/\!/G$. Malheureusement, on a la 

\begin{proposition}  \label{res_crepantes_SLn}
La résolution $\gamma$ n'est pas crépante.
\end{proposition} 

\begin{proof}
On identifie $\gamma:\HH  \rightarrow W/\!/G$ à l'éclatement $f:\ Bl_0(W/\!/G) \rightarrow W/\!/G$ grâce au théorème \ref{SLLn}, et on considère le diagramme 
$$\xymatrix{ &  Bl_0(W/\!/G) \ar@{->>}[ld]^{f} \ar@{->>}[rd]_{p} \\   W/\!/G && \PP(W/\!/G) \ar@/_/@{.>}[lu]_{\sigma} }$$
où $p$ est la projection naturelle et $\sigma$ est la section nulle. Le diviseur exceptionnel de $f$, noté $D_f$, s'identifie à $\PP(W/\!/G)$ grâce à $\sigma$.
La formule d'adjonction (\cite[Proposition 8.20]{Ha}) donne:
$$\omega_{D_f} \cong \omega_{Bl_0(W/\!/G)} \otimes \OO(D_f) \otimes \OO_{D_f}.$$
On rappelle que $Bl_0(W/\!/G)$ est naturellement plongée dans $\Lambda^n V'^* \times \PP(\Lambda^n V'^*)$. D'après \cite[Proposition 6.18]{Ha}, on a $\OO(-D_f) \cong \II_{D_f}$, où $\II_{D_f}$ est le faisceau d'idéaux de $D_f$ dans $Bl_0(W//G)$. Mais $D_f$ est l'image de la section nulle dans $Bl_0(W/\!/G)$, donc ${\II_{D_f}}_{|D_f} \cong \OO_{D_f}(1)$. Il s'ensuit que 
$$\omega_{D_f} \cong \omega_{Bl_0(W/\!/G)} \otimes \OO_{D_f}(-1).$$
On suppose que $f$ est crépante, alors 
$$\omega_{Bl_0(W/\!/G)} \cong f^*(\omega_{W/\!/G}) \cong f^*(\OO_{W/\!/G}) \cong \OO_{Bl_0(W/\!/G)}$$ 
où le deuxième isomorphisme découle du lemme \ref{critereGor}.\\
On a donc $\omega_{D_f} \cong \OO_{D_f}(-1)$. Mais $D_f \cong \PP(W/\!/G)=\Gr(n,V'^*)$ et donc 
$$\omega_{D_f} \cong \omega_{\Gr(n,V'^*)} \cong \OO_{\Gr(n,V'^*)}(-n').$$ 
Or, $n'>1$, ce qui contredit l'hypothèse que $f$ est crépante. 
\end{proof}

\subsection{Situation 2}  \label{GroGLn}

On se place dans la situation 2 et on reprend les notations de la section \ref{GLngénéral}. On suppose que la variété déterminantielle $W/\!/G=\Hom(V_1,V_2)^{\leq n}$ est singulière, c'est-à-dire que $n_1,n_2 >n$. On a vu que $W/\!/G$ est de Cohen-Macaulay et que $\gamma:\ \HH \rightarrow W/\!/G$ est une résolution des singularités de $W/\!/G$ lorsque $n \leq 2$.
On s'intéresse aux résolutions crépantes (éventuelles) de $W/\!/G$, donc la première chose à faire est de déterminer pour quelles valeurs de $n,n_1,n_2$ la variété $W/\!/G$ est de Gorenstein. La proposition qui suit est un résultat connu (\cite[Theorem 5.5.6]{Sv}). On en redonne ici une preuve dans le but d'être complet.

\begin{proposition}  \label{det_gorenstein}
La variété $W/\!/G$ est de Gorenstein si et seulement si $n_1=n_2$.
\end{proposition}

\begin{proof}
Le groupe multiplicatif $\Gm$ opère dans $W/\!/G={\Hom}(V_1,V_2)^{\leq n}$ par:
$$\forall \alpha \in \Gm,\ \forall f \in {\Hom}(V_1,V_2)^{\leq n},\ \alpha.f:=\alpha f$$
Donc $W/\!/G$ est un cône affine et donc, d'après le lemme \ref{critereGor}, il suffit de déterminer pour quelles valeurs de $n,\ n_1,\ n_2$ le faisceau dualisant $\omega_{W/\!/G}$ est trivial.\\ 
On note $i$ l'immersion ouverte $\ U_n \hookrightarrow W/\!/G$. Si $\omega_{W/\!/G}$ est trivial, alors $\omega_{U_n}=i^* \omega_{W/\!/G}$ est trivial. Réciproquement, si $\omega_{U_n}$ est trivial, alors comme $(W/\!/G) \backslash U_n$ est de codimension $\geq 2$, le faisceau réflexif $\omega_{W/\!/G}$ est trivial. On est donc ramené à déterminer pour quelles valeurs de $n,n_1,n_2$ le fibré canonique $\omega_{U_n}$ est trivial.\\
Soient $J_n =\begin{bmatrix}
I_n  &0_{n,n_1-n} \\
0_{n_2-n,n}   &0_{n_2-n,n_1-n} 
\end{bmatrix} \in U_n$ et 
\begin{align*}
H:&=\Stab_{G'}(J_n) \\
  &=\left\{ \left( \begin{bmatrix}
A  &0 \\
B_1   &C_1 \end{bmatrix},\begin{bmatrix}
A  &B_2 \\
0  &C_2 \end{bmatrix}\right) \middle| 
    \begin{array}{l}
      A \in GL_n(k), C_1 \in GL_{n_1-n}(k), C_2 \in GL_{n_2-n}(k),\\
      B_1 \in M_{n_1-n,n}(k), B_2 \in M_{n,n_2-n}(k) 
    \end{array} \right\}.
\end{align*}
Le fibré $\omega_{U_n}$ est $G'$-linéarisé sur $U_n \cong G'/H$, donc l'espace total du fibré $\omega_{U_n}$ est $G'$-isomorphe à la variété 
$$E_{\chi_0}:=G' \times^H \Aff$$ 
où $H$ opère $\Aff$ par un caractère $\chi_0 \in \XX(H)$; celui-ci est le caractère de $H$ dans $\Lambda^m (T_{J_n} U_n)$, où $T_{J_n} U_n$ est l'espace tangent à $U_n$ en $J_n$ et $m:=\dim(U_n)$.\\

\noindent \underline{Fait:} le caractère $\chi_0 \in \XX(H)$ est donné par 
$$\begin{array}{ccccc}
\chi_0 & : & H & \longrightarrow & \Gm \\
& & \left( \begin{bmatrix}
A  &0 \\
B_1   &C_1 \end{bmatrix},\begin{bmatrix}
A  &B_2 \\
0  &C_2 \end{bmatrix}\right) & \mapsto & \det(A)^{n_1-n_2} \det(C_1)^{-n} \det(C_2)^{n}. \\
\end{array}$$

Comme $\chi_0$ est entièrement déterminé par sa restriction à n'importe quel sous-groupe de Levi de $H$, on fixe
$$L:=\left\{ \left( \begin{bmatrix}
A  &0 \\
0   &C_1 \end{bmatrix},\begin{bmatrix}
A  &0 \\
0  &C_2 \end{bmatrix}\right)\ \middle| \ A \in GL_n(k), C_1 \in GL_{n_1-n}(k), C_2 \in GL_{n_2-n}(k) \right\}$$
un tel sous-groupe de Levi. L'espace tangent $T_{J_n} U_n$ s'identifie à $\gg'/\hh$ comme $L$-module, où $\gg'$ et $\hh$ sont les algèbres de Lie de $G'$ et $H$ respectivement. On commence donc par déterminer $\gg'$ et $\hh$ comme $L$-modules. Soient 
\begin{align*}
&E_1:=\left\langle e_1, \ldots, e_n \right\rangle \subset V_1,\\
&E_2:=\left\langle f_1, \ldots, f_n \right\rangle \subset V_2,\\
&F_1:=\left\langle e_{n+1}, \ldots, e_{n_1} \right\rangle \subset V_1,\\
&F_2:=\left\langle f_{n+1}, \ldots, f_{n_2} \right\rangle \subset V_2,
\end{align*}
qui sont des $L$-sous-modules de $V_1$ et $V_2$. On a les isomorphismes de $L$-modules suivants: 
\begin{align*}
&V_1 \cong E_1 \oplus F_1,\\
&V_2 \cong E_2 \oplus F_2,\\
&E_1 \cong E_2,
\end{align*}
et donc
\begin{align*}
\gg' \cong &\End(V_1) \oplus \End(V_2) \\
         \cong & \End(E_1) \oplus \End(F_1) \oplus \Hom(E_1,F_1) \oplus \Hom(F_1,E_1) \\
               & \oplus \End(E_2) \oplus \End(F_2) \oplus \Hom(E_2,F_2) \oplus \Hom(F_2,E_2), \\
\hh \cong &\End(E_1) \oplus \Hom(E_1,F_1) \oplus \End(F_1) \\
              &\oplus \Hom(F_2,E_2) \oplus \End(F_2).
\end{align*}
On en déduit que 
$$T_{J_n} U_n \cong \End(E_1) \oplus \Hom(E_2,F_2) \oplus \Hom(F_1,E_1)$$ 
comme $L$-module. D'où
\begin{align*}
\Lambda^m (T_{J_n} U_n) \cong &\Lambda^{n^2} (\End(E_1)) \otimes \Lambda^{n(n_2-n)} (\Hom(E_2,F_2)) \otimes \Lambda^{n(n_1-n)} (\Hom(F_1,E_1))\\
                       \cong  & (\Lambda^n (E_2^*))^{\otimes n_2-n} \otimes (\Lambda^{n_2-n} (F_2))^{\otimes n} \otimes (\Lambda^n (E_1))^{\otimes n_1-n} \otimes (\Lambda^{n_1-n} (F_1^*))^{\otimes n}\\
                       \cong  &(\Lambda^n (E_2^*))^{\otimes n_2-n} \otimes (\Lambda^n (E_1))^{\otimes n_1-n} \otimes (\Lambda^{n_2-n} (F_2))^{\otimes n} \otimes (\Lambda^{n_1-n} (F_1^*))^{\otimes n}\\
                       \cong  &(\Lambda^n (E_1))^{\otimes n_1-n_2} \otimes (\Lambda^{n_2-n} (F_2))^{\otimes n} \otimes (\Lambda^{n_1-n} (F_1^*))^{\otimes n}
\end{align*} 
comme $L$-module. Le fait annoncé s'ensuit.\\ 

On revient maintenant à la démonstration de la proposition \ref{det_gorenstein}. On cherche pour quelles valeurs de $n,n_1,n_2$ la $G'$-variété $E_{\chi_0}$ est l'espace total du fibré trivial de $U_n$. D'après \cite[Proposition 3.2]{KKV}, on a la suite exacte de groupes suivante:
$$ X(G) \stackrel{res}{\longrightarrow} X(H) \stackrel{\epsilon}{\longrightarrow} \Pic(G'/H) \stackrel{p^*}{\longrightarrow} \Pic(G')$$
avec
\begin{itemize} \renewcommand{\labelitemi}{$\bullet$}
\item $\Pic(G')$ et $\Pic(G'/H)$ les groupes de Picard de $G'$ et $G'/H$ respectivement,
\item $X(G')$ et $X(H)$ les groupes des caractères de $G'$ et de $H$ respectivement,
\item $res$ l'homomorphisme de restriction,
\item $p^*$ l'homomorphisme induit par le morphisme de passage au quotient $p:\ G' \rightarrow G'/H$,
\item $\epsilon$ l'homomorphisme canonique qui à un caractère $\chi \in X(H)$ associe le fibré en droites $E_{\chi}$ qui est défini comme précédement.
\end{itemize} 
Or $\Pic(G')=0$ donc $\Pic(G'/H) \cong X(H)/{X(G')}_{|H}$. Il s'ensuit que $\omega_{U_n}$ est trivial si et seulement si $\chi_0 \in {X(G')}_{|H}=\left\langle  \det(A) \det(C_1), \det(A) \det(C_2) \right\rangle$, c'est-à-dire si et seulement si $n_1=n_2$.      
\end{proof}


On suppose dorénavant que $n_1=n_2$, alors la variété $W/\!/G$ est de Gorenstein d'après la proposition \ref{det_gorenstein}, et l'on va en construire des résolutions crépantes. On considère le diagramme suivant:
$$\xymatrix{ &  R_1 \ar[ld]_{p_{12}} \ar[rd]^{p_{11}} && R_2 \ar[ld]_{p_{21}} \ar[rd]^{p_{22}} & \\   
            \Gr(n,V_1^*) &&   W/\!/G && \Gr(n,V_2)  }$$
où
\begin{align*}
R_1:&=\left \{(M,L_1) \in W/\!/G \times \Gr(n,V_1^*) \ \mid  \ L_1^{\perp} \subset \Ker(M) \right \},\\
R_2:&=\left \{(M,L_2) \in W/\!/G \times \Gr(n,V_2) \ \mid  \ \Im(M) \subset L_2 \right \},
\end{align*}
et les $p_{ij}$ sont les projections naturelles. On note $T_1$ le fibré tautologique de $\Gr(n,V_1^*)$, alors $R_1$ s'identifie à l'espace total du fibré vectoriel $\Hom(\underline{V_1}/T_1^{\perp},\underline{V_2}) \cong T_1 \otimes \underline{V_2}$ au dessus de $\Gr(n,V_1^*)$. De même, si l'on note $T_2$ le fibré tautologique de $\Gr(n,V_2)$, alors $R_2$ s'identifie à l'espace total du fibré vectoriel $\Hom(\underline{V_1},T_2) \cong \underline{V_1^*} \otimes T_2$ au dessus de $\Gr(n,V_2)$. 

Les variétés $R_1$ et $R_2$ sont lisses et donc les morphismes $p_{11}$ et $p_{21}$ sont des résolutions des singularités de $W/\!/G$.     
Le lieu exceptionnel $X_1:=p_{11}^{-1}(\overline{U_{n-1}})$ (resp. $X_2:=p_{21}^{-1}(\overline{U_{n-1}})$) de $p_{11}$ (resp. de $p_{21}$) est de codimension $n_2-n+1 \geq 2$ (resp. $n_1-n+1\geq 2$). Donc les résolutions $p_{11}$ et $p_{21}$ sont crépantes.

On considère maintenant $R:=R_1 \times_{W/\!/G} R_2$. Alors 
\begin{align*}
R& = \left\{(M,L_1,L_2) \in W/\!/G \times \Gr(n,V_1^*) \times \Gr(n,V_2)\ |\ L_1^{\perp} \subset \Ker(M) \text{ et } \Im(M) \subset L_2 \right\}\\
  &\cong \Hom(\underline{V_1}/T_1^{\perp},T_2) \\
  &\cong T_1 \otimes T_2
\end{align*} 
qui est l'espace total d'un fibré vectoriel au dessus de $\Gr(n,V_1^*) \times \Gr(n,V_2)$, donc une variété lisse. On a le diagramme cartésien: 
\begin{eqnarray} \label{res_crep_casnn}
\xymatrix{
 &R \ar@{->>}[dd]^\phi \ar@{->>}[ld]_{\phi_1} \ar@{->>}[rd]^{\phi_2} & \\
    R_1 \ar@{->>}[rd]_(0.4){p_{11}} & & R_2 \ar@{->>}[ld]^(0.4){p_{21}} \\
     &W/\!/G &
  }
\end{eqnarray}
où $\phi,\ \phi_1$ et $\phi_2$ sont les projections naturelles. En particulier, $\phi$ est une résolution de $W/\!/G$ qui domine $p_{11}$ et $p_{21}$.

\begin{proposition}
Si $n \leq 2$, alors la résolution $\gamma:\ \HH \rightarrow W/\!/G$ n'est pas crépante.
\end{proposition}

\begin{proof}
Les morphismes $\phi_1$ et $\phi_2$ sont surjectifs mais ne sont jamais des isomorphismes (car $n<n_1,\,n_2$), donc d'après le lemme \ref{pascrepant3}, la résolution $\phi$ n'est jamais crépante. Donc, toujours d'après le lemme \ref{pascrepant3}, si une résolution de $W/\!/G$ se factorise par $\phi$, alors elle n'est pas crépante.\\
Si $n=1$, alors $\gamma$ s'identifie à $\phi$ grâce à la proposition \ref{Hcasn11111}.\\
Si $n=2$, alors d'après le théorème \ref{casn2}, la résolution $\gamma$ s'identifie à $\phi \circ q$, où $q$ est l'éclatement de la section nulle dans $R \cong T_1 \otimes T_2$.\\
Dans les deux cas, on en déduit que $\gamma$ n'est pas crépante.
\end{proof}

\subsection{Situation 3}  \label{GroOn}

On se place dans la situation 3 et on reprend les notations de la section \ref{HclassOngeneral}. On suppose que la variété déterminantielle symétrique $W/\!/G=S^2(V'^*)^{\leq n}$ est singulière et de Gorenstein, c'est-à-dire que $n'-n=2p+1$ pour un certain entier $p$. Lorsque $n=2$, on a vu que $\gamma:\ \HH \rightarrow W/\!/G$ est une résolution. Nous allons montrer la

\begin{proposition} \label{moniose2}
Si $n=2$, alors $\gamma$ n'est pas crépante.
\end{proposition}

\begin{proof}
On commence par construire une résolution de $S^2(V'^*)^{\leq n}$ pour un $n$ général. Soient 
$$R:=\{(Q,L) \in S^2(V'^*)^{\leq n} \times \Gr(n,V'^*)\ |\ \Im(Q) \subset L\},$$ 
et $p_1$, $p_2$ les projections naturelles sur $S^2(V'^*)^{\leq n}$ et $\Gr(n,V'^*)$ respectivement. Le morphisme $p_2$ permet d'identifier $R$ à l'espace total du fibré vectoriel $S^2(T)$, où $T$ est le fibré tautologique de $\Gr(n,V'^*)$. Il s'ensuit que $R$ est une variété lisse et donc $p_1$ est une résolution de $S^2(V'^*)^{\leq n}$.\\
D'après le théorème \ref{casn2O2}, lorsque $n=2$ on a
\begin{equation*}
\xymatrix{ \HH  \ar[d]^{q}  \ar@/_2pc/[dd]_{\gamma}  \\
          R   \ar[d]^{p_1}  \\
          S^2(V'^*)^{\leq n} }
\end{equation*}  
où $q$ est l'éclatement de la section nulle dans $R \cong S^2(T)$. En particulier, $q$ n'est pas un isomorphisme et donc, d'après le lemme \ref{pascrepant3}, la résolution $\gamma$ n'est pas crépante.
\end{proof}

\begin{remarque}
On peut montrer que $p_1:\ R \rightarrow S^2(V'^*)^{\leq n}$ n'est jamais crépante. En particulier, si $n=1$, alors $\HH \cong R$ et $\gamma$, qui s'identifie à $p_1$, n'est donc pas crépante. 
\end{remarque}

\subsection{Situation 5}  \label{GroSpn}

On se place dans la situation 5 et on reprend les notations de la section \ref{schcasSymp}. On suppose que la variété déterminantielle antisymétrique $W/\!/G=\Lambda^2(V'^*)^{\leq n}$ est singulière, c'est-à-dire que $n'>n$. On a vu que $\Lambda^2(V'^*)^{\leq n}$ est toujours de Gorenstein, donc lorsque $\gamma:\ \HH \rightarrow W/\!/G$ est une résolution, il est naturel de se demander si celle-ci est crépante. Si $n=2$, alors $Sp_2(k)=SL_2(k)$ et donc $\gamma$ n'est pas crépante d'après la proposition \ref{res_crepantes_SLn}. Il reste à traiter le cas $n=4$. La proposition qui suit se démontre de manière analogue à la proposition \ref{moniose2} (il suffit essentiellement de remplacer partout $S^2(V'^*)^{\leq n}$ par $\Lambda^2(V'^*)^{\leq n}$).  

\begin{proposition} \label{moniose3}
Si $n=4$, alors $\gamma$ n'est pas crépante.
\end{proposition}

\begin{remarque}
J'ignore si les variétés quotient $W/\!/G$ qui apparaissent dans ce mémoire admettent ou non des résolutions crépantes (à l'exception de $R_1$ et de $R_2$ pour la variété déterminantielle). 
\end{remarque}

%% file: appendiceA2.tex
\section{Résolutions symplectiques de \texorpdfstring{$\mu^{-1}(0)/\!/G$}{W///G}} \label{appendice2}

\subsection{Généralités}

L'article d'exposition \cite{FuBB} fournit une introduction détaillée aux variétés symplectiques et à leurs résolutions. 
Soit $X$ une variété normale dont la partie lisse $X_{\mathrm{reg}}$ admet une forme symplectique $\Omega$ telle que pour n'importe quelle résolution $f:\ Y \rightarrow X$, la $2$-forme $f^*(\Omega)$ s'étend en une $2$-forme sur $Y$ tout entier. On dit alors que $(X, \Omega)$ est une variété symplectique. Si de plus $f^*(\Omega_X)$ s'étend en une forme symplectique, alors la résolution $f$ est dite symplectique. Si $(X,\Omega)$ est une variété symplectique, il doit être souligné que $X$ n'admet pas toujours de résolution symplectique, et que lorsqu'elle en admet, il peut en exister plusieurs non-isomorphes.

Soit maintenant $(X, \Omega)$ une variété symplectique dans laquelle un groupe réductif $G$ opère symplectiquement, c'est-à-dire que l'opération de $G$ préserve $\Omega$. Il existe une application $G$-équivariante $\mu:\ X \rightarrow \gg^*$, où $\gg$ est l'algèbre de Lie de $G$, telle que:
$$\forall x \in X,\ \forall \xi \in T_x X,\ \forall M \in \gg, D\mu_x(\xi)(M)=\Omega(\xi,Mx).$$
Une telle application $\mu$ est appelée application moment pour l'opération de $G$ dans $(X,\Omega)$ et est définie à un élément près de ${(\gg^*)}^G$.

On rappelle le résultat suivant dû à Kostant, Kirillov et Souriau (voir par exemple \cite[§2.2]{FuBB} pour plus de détails):

\begin{proposition}
La normalisation de l'adhérence d'une orbite nilpotente dans une algèbre de Lie semi-simple complexe est une variété symplectique.
\end{proposition}

On en déduit que les composantes irréductibles des quotients $\mu^{-1}(0)/\!/G$ étudiés dans les sections \ref{posisimpy}, \ref{sektionsympOn} et \ref{sektionsympSpn} sont des variétés symplectiques. On note $p_1:\ \mu^{-1}(0)/\!/G^0 \rightarrow \mu^{-1}(0)/\!/G$ le morphisme considéré dans la section \ref{sectionavecXOn}. Si $d<n$, alors $p_1$ est un isomorphisme. Si $d \geq n$, alors c'est un revêtement double ramifié. 

\begin{lemme}
La variété $\mu^{-1}(0)/\!/G^0$ est symplectique.
\end{lemme}

\begin{proof}
Le cas $d<n$ est trivial, on peut donc supposer que $d \geq n$. On considère le diagramme commutatif suivant:
\begin{equation} \label{klekle}
\xymatrix{
    X' \ar[r]^{h} \ar[d]_{f'} & X \ar[d]^{f} \\
    \mu^{-1}(0)/\!/G^0 \ar[r]_{p_1} & \mu^{-1}(0)/\!/G
  }  
  \end{equation}
où $f$ est une résolution de $\mu^{-1}(0)/\!/G$ et $f'$ est une résolution de $\mu^{-1}(0)/\!/G^0$ qui fait commuter le diagramme (une telle résolution existe toujours). On tire en arrière par $p_1$ la forme symplectique $\Omega$ sur $\OO_{[2^n,1^{d-2n}]} \subset \mu^{-1}(0)/\!/G$. On obtient une forme symplectique $\Omega'$ sur $p_1^{-1}(\OO_{[2^n,1^{d-2n}]})$ qui s'étend sur l'ouvert lisse de $\mu^{-1}(0)/\!/G^0$ de manière unique. En effet, le complémentaire de l'orbite ouverte est de codimension au moins $2$. Ensuite, $f \circ h$ est une résolution de la variété symplectique $\mu^{-1}(0)/\!/G$, donc $(f \circ h)^*(\Omega)$ se prolonge en une $2$-forme non sur $X'$ tout entier. Comme le diagramme (\ref{klekle}) commute, $(f \circ h)^*(\Omega)$ coïncide avec $f'^*(\Omega')$ sur l'ouvert $f'^{-1}(p_1^{-1}(\OO_{[2^n,1^{d-2n}]}))$ et le résultat s'ensuit.
\end{proof}     

Les résolutions symplectiques des adhérences d'orbites nilpotentes ont été étudiées par Fu et Namikawa dans \cite{FuB}, \cite{Fu2}, \cite{Fu3}, \cite{FuNa} et \cite{Nam}. Malheureusement, les résolutions symplectiques des revêtements finis de degré pair d'orbites nilpotentes sont peu connues, et j'ignore si $\mu^{-1}(0)/\!/SO(V)$ admet ou non des résolutions symplectiques. Dans ce qui suit, on détermine les liens entre le morphisme de Hilbert-Chow $\gamma:\ \HHmp \rightarrow \mu^{-1}(0)/\!/G$ (lorsque celui-ci est une résolution) et les résolutions symplectiques (éventuelles) de $\mu^{-1}(0)/\!/G$.

\subsection{Situation 2}

On reprend les notations de la section \ref{posisimpy}. On a vu que $\mu^{-1}(0)/\!/G=\overline{\OO_{[2^N,1^{d-2N}]}}$ et donc $\mu^{-1}(0)/\!/G$ admet toujours une résolution symplectique (\cite[Corollary 3.16]{FuB}). Plus précisément, on distingue deux cas:

\noindent  $\bullet$ Si $N=\frac{d}{2}$ (c'est-à-dire si $d$ est pair et si $\frac{d}{2}\leq n$), alors $\mu^{-1}(0)/\!/G$ admet une unique résolution symplectique (à isomorphisme près) $\phi:\ R \rightarrow \mu^{-1}(0)/\!/G$ où
$$R:=\{ (f,L) \in \mu^{-1}(0)/\!/G \times \Gr(N,E) \ |\ \Im(f) \subset L \subset \Ker(f) \}$$
et $\phi$ est la première projection (\cite[§2]{Fu3}). On note $T$ le fibré tautologique de $\Gr(N,E)$ et $\underline{E}$ le fibré trivial de fibre $E$ au dessus de $\Gr(N,E)$, alors $R \cong \Hom(\underline{E}/T,T)$ qui est le fibré cotangent de $\Gr(N,E)$. Le corollaire \ref{symplisssse} nous permet de déduire la

\begin{proposition}  \label{dewey}
Si $d$ est pair et $d \leq n+1$, alors $\gamma:\ \HHmp \rightarrow \mu^{-1}(0)/\!/G$ est l'unique résolution symplectique de $\mu^{-1}(0)/\!/G$.\\
En revanche, si $n=2$ et $d=4$, alors $\HHmp \cong Bl_0(\Hom(\underline{E}/T,T))$, et donc $\gamma$ n'est pas la résolution symplectique de $\mu^{-1}(0)/\!/G$ (mais $\gamma$ se factorise par celle-ci). 
\end{proposition}

\noindent $\bullet$ Si $N<\frac{d}{2}$ (c'est-à-dire si $\frac{d}{2}>n$ ou si $d$ est impair), alors $\mu^{-1}(0)/\!/G$ admet exactement deux résolutions symplectiques non-isomorphes (\cite[§2]{Fu3}) données par:
$$\xymatrix{
    T^*\Gr(N,E) \ar@{->>}[rd]_{\phi_1} & & T^*\Gr(d-N,E) \ar@{->>}[ld]^{\phi_2} \\
     & \mu^{-1}(0)/\!/G &
  }$$
où 
\begin{itemize} \renewcommand{\labelitemi}{$\bullet$}
\item $T^*\Gr(N,E)=\{(f,L_1) \in \mu^{-1}(0)/\!/G \times \Gr(N,E) \ |\ \Im(f) \subset L_1\}$ est le fibré cotangent de $\Gr(N,E)$,
\item $T^*\Gr(d-N,E) = \{(f,L_2) \in \mu^{-1}(0)/\!/G \times \Gr(d-N,E) \ |\ L_2 \subset \Ker(f)\}$ est son dual, 
\item les morphismes $\phi_i$ sont les projections sur le premier facteur.
\end{itemize}

Soient $\FF_{N,d-N}$ la variété de drapeaux partiels introduite dans la section \ref{MropRRED}, $\underline{E}$ le fibré trivial de fibre $E$ au dessus de $\FF_{N,d-N}$ et $T_1$, $T_2$ les fibrés vectoriels de la définition \ref{pulbacktauto}. On considère
\begin{align*}
R:&=\left\{(f,L_1,L_2) \in \mu^{-1}(0)/\!/G \times \Gr(N,E) \times \Gr(d-N,E)\ |\ \Im(f) \subset L_1 \subset L_2 \subset \Ker(f) \right\} \\
&\cong \Hom(\underline{E}/T_2,T_1) 
\end{align*} 
qui est une composante irréductible (lisse) du produit fibré de $T^*\Gr(N,E)$ et $T^*\Gr(d-N,E)$ au dessus de $\mu^{-1}(0)/\!/G$.
La première projection $\phi:\ R \rightarrow \mu^{-1}(0)/\!/G$ est une résolution de $\mu^{-1}(0)/\!/G$ et c'est la résolution minimale qui domine les deux résolutions symplectiques de $\mu^{-1}(0)/\!/G$. Le corollaire \ref{symplisssse} nous permet de déduire la

\begin{proposition}
Si $n=1$ et $d \geq 3$, alors $\gamma:\ \HHmp \rightarrow \mu^{-1}(0)/\!/G$ s'identifie à $\phi:\ R \rightarrow \mu^{-1}(0)/\!/G$. \\
Si $n=2$ et $d\geq 5$, alors $\gamma$ se factorise par $\phi$ mais $\HHmp$ n'est pas isomorphe à $R$. Plus précisément, on a le diagramme commutatif suivant:
\begin{equation}
\xymatrix{\HHmp \ar[rd]^{f} \ar@/_7pc/[rddd]_(0.4){\gamma} && \\
 &R \ar[dd]^{\phi} \ar[ld] \ar[rd] & \\
 T^*\Gr(N,E) \ar[rd]_{\phi_1} && T^*\Gr(d-N,E) \ar[ld]^{\phi_2} \\
 & \mu^{-1}(0)/\!/G &
}
\end{equation}
où $f$ est l'éclatement de la section nulle dans $R \cong \Hom(\underline{E}/T_2,T_1)$.
\end{proposition}

\subsection{Situation 3}

On reprend les notations de la section \ref{sektionsympOn}. On a vu que $\mu^{-1}(0)/\!/G=\overline{\OO_{[2^N,1^{2(d-N)}]}}$, et donc $\mu^{-1}(0)/\!/G$ admet une résolution symplectique si et seulement si $N=d$ (\cite[Proposition 3.19]{FuB}), c'est-à-dire si $d \leq n$. Dans ce cas, $\mu^{-1}(0)/\!/G$ admet une unique résolution symplectique; celle-ci est donnée par $T^*\IG(d,V'^*)$, le fibré cotangent de $\IG(d,V'^*)$ (c'est une conséquence de \cite[Main Theorem]{FuB} et de \cite[Proposition 3.5]{FuNa}). 

Pour pouvoir comparer cette résolution symplectique avec celles données par le corollaire \ref{symplisssseOn}, il faut exprimer le fibré $T^*\IG(d,V'^*)$ en fonction du fibré tautologique $T$ de $\IG(d,V'^*)$. La variété $\IG(d,V'^*)$ est $G'$-homogène, donc il existe un sous-groupe parabolique $P \subset G'$ tel que $\IG(d,V'^*) \cong G'/P$. On note $\pp$ l'algèbre de Lie de $P$, alors l'espace tangent à $G'/P$ au point base est isomorphe à $\gg'/ \pp$ comme $P$-module. On note $E_0 \subset V'^*$ le sous-espace de dimension $d$ stabilisé par $P$. Un calcul similaire à celui effectué dans la preuve de la proposition \ref{det_gorenstein} permet de montrer que $\gg'/ \pp \cong S^2(E_0^*)$ comme $P$-module. Il s'ensuit que 
$$T^* \IG(d,V'^*) \cong S^2(T).$$
Le corollaire \ref{symplisssseOn} nous permet de déduire la  

\begin{proposition}
Si $d \leq \frac{n+1}{2}$, alors $\gamma:\ \HHmgp \rightarrow \mu^{-1}(0)/\!/G$ est la résolution symplectique de $\mu^{-1}(0)/\!/G$.\\
En revanche, si $d=n=2$, alors $\HHmgp \cong Bl_0(S^2(T))$ et donc $\gamma$ n'est pas la résolution symplectique de $\mu^{-1}(0)/\!/G$ (mais $\gamma$ se factorise par celle-ci).
\end{proposition}

\subsection{Situation 5}

On reprend les notations de la section \ref{sektionsympSpn}. On a vu que 
$$
\mu^{-1}(0)/\!/G = \left\{
    \begin{array}{ll}
        \overline{\OO_{[2^n,1^{2(d-n)}]}} & \text{ si } d>n,\\
        \overline{\OO_{[2^{d-1},1^2]}} & \text{ si } d < n \text{ et } d \text{ est impair},\\
        \overline{\OO_{[2^d]}^{I}} \cup \overline{\OO_{[2^d]}^{II}} & \text{ si } d \leq n \text{ et } d \text{ est pair}.
    \end{array}
\right.
$$
On a le critère suivant, dû à Fu (\cite[Proposition 3.20]{FuB}), pour savoir si l'adhérence d'une orbite nilpotente dans $\gg'$ admet une résolution symplectique:
\begin{proposition}
La variété symplectique $\overline{\OO_{[k_1,\ldots, k_r]}}$ admet une résolution symplectique si et seulement si:
\begin{itemize}
\item ou bien il existe un entier pair $q \neq 2$ tel que les $q$ premières parts de $[k_1,\ldots, k_r]$ sont impaires et toutes les autres parts sont paires,
\item ou bien il existe exactement deux parts impaires dans $[k_1,\ldots, k_r]$ qui sont en position $2m-1$ et $2m$ pour un certain $m \geq 1$.
\end{itemize}
\end{proposition}

On en déduit que les composantes irréductibles de $\mu^{-1}(0)/\!/G$ admettent des résolutions symplectiques si et seulement si $d \leq n+1$. On distingue alors les deux cas suivants:

\noindent $\bullet$ Si $d \leq n$ et $d$ est pair, alors $Y_I=\overline{\OO_{[2^d]}^{I}}$ (resp. $Y_{II}=\overline{\OO_{[2^d]}^{II}}$) admet une unique résolution symplectique (à isomorphisme près) $\phi_I: T^*\OG^{I}(d,V'^*) \rightarrow Y_I$ (resp. $\phi_{II}: T^*\OG^{II}(d,V'^*) \rightarrow Y_{II}$) d'après \cite[Proposition 3.5]{FuNa} et \cite[Main Theorem]{FuB}. On déduit du corollaire \ref{resuss} la  

\begin{proposition}
Si $d \leq \frac{n}{2}+1$ et $d$ est pair, alors $\gamma:\ \HHmxp \rightarrow Y_I$ (resp. $\gamma:\ \HHmyp \rightarrow Y_{II}$) est l'unique résolution symplectique de $Y_I$ (resp. de $Y_{II}$).\\
En revanche, si $d=n=4$ alors $\HHmxp \cong Bl_0(\Lambda^2(T_I))$ (resp. $\HHmyp \cong Bl_0(\Lambda^2(T_{II}))$) et donc $\gamma:\ \HHmxp \rightarrow Y_I$ (resp. $\gamma:\ \HHmyp \rightarrow Y_{II}$) n'est pas la résolution symplectique de $Y_I$ (resp. de $Y_{II}$), mais $\gamma$ se factorise par celle-ci.
\end{proposition}

\noindent $\bullet$ Si $3 \leq d \leq n+1$ et $d$ impair, alors $\mu^{-1}(0)/\!/G=\overline{\OO_{[2^{d-1},1^2]}}$ admet exactement deux résolutions symplectiques non isomorphes (\cite[§2]{Fu3}) données par:
$$\xymatrix{
    T^*\OG^{I}(d,V'^*) \ar@{->>}[rd]_{\phi_1} & & T^*\OG^{II}(d,V'^*) \ar@{->>}[ld]^{\phi_2} \\
     & \mu^{-1}(0)/\!/G &
  }$$
où $\OG^{I}(d,V'^*)$ et $\OG^{II}(d,V'^*)$ sont les deux composantes irréductibles de la grassmannienne isotrope introduite dans la section \ref{description_quotientOn}, et $T^*\OG^{I}(d,V'^*)$ (resp. $T^*\OG^{II}(d,V'^*)$) est le fibré cotangent de $\OG^{I}(d,V'^*)$ (resp. de $\OG^{I}(d,V'^*)$). On souhaite comparer ces résolutions symplectiques avec les résolutions obtenues dans la section \ref{sssympSpn}. En raisonnant comme dans la situation 3, on montre que
\begin{equation} \label{isoSympn}
T^* \OG(d,V'^*) \cong \Lambda^2(T),
\end{equation}
où $T$ est le fibré tautologique de $\OG(d,V'^*)$. On note $T_I$ (resp. $T_{II}$) la restriction de $T$ à $\OG^{I}(d,V'^*)$ (resp. à $\OG^{II}(d,V'^*)$). On déduit de (\ref{isoSympn}) que 
$$
\left\{
    \begin{array}{l}
       T^* \OG^{I}(d,V'^*) \cong \Lambda^2(T_I), \text{ et } \\
       T^* \OG^{II}(d,V'^*) \cong \Lambda^2(T_{II}).
    \end{array}
\right.
$$
Le corollaire \ref{resuss} nous permet de déduire la 

\begin{proposition}
Si ($d=3$ et $n=2$) ou ($d=5$ et $n=4$), alors la résolution $\gamma:\ \HHmp \rightarrow \mu^{-1}(0)/\!/G$ n'est pas symplectique.
\end{proposition}

\noindent Malheureusement, j'ignore si $\gamma$ domine ou non les deux résolutions symplectiques de $\mu^{-1}(0)/\!/G$.

%% file: appendiceB.tex
\chapter{Les points fixes de \texorpdfstring{$B'$}{B'} comme sous-schémas de \texorpdfstring{$W$}{W}} \label{appendiceB}

On reprend les notations de la section \ref{lesdiffsituations}. Dans tout ce qui suit, on considère les points de $\HH(k)$ comme des sous-schémas fermés de $W$. On souhaite étudier les propriétés géométriques des déformations plates $G$-équivarientes de $G$ dans $W$. Par définition du schéma de Hilbert invariant, les points fermés de $\HHp$ sont précisément ces déformations dans les exemples que nous avons traités. Lorsque $G=SL_n(k)$ ou $\Gm$, on a déterminé la famille universelle $\pi:\ \XX \rightarrow \HH$ (propositions \ref{familleUnivSLn} et \ref{Hcasn11111}), on connaît donc explicitement les (deux) classes d'isomorphisme des déformations plates de $G$ dans $W$. Dans les autres situations, la famille universelle n'est pas connue. Néanmoins, nous allons voir que si l'on connaît les propriétés géométriques des points fixes de $B'$, alors on peut en déduire certaines propriétés géométriques de tous les points de $\HH(k)$.

\begin{proposition***}
On note $(P)$ l'une des propriétés suivantes:
\begin{itemize}
\item être de Cohen-Macaulay,
\item être normal,
\item être réduit,
\item être lisse.
\end{itemize}
Si tous les points fixes de $B'$ dans $\HH(k)$ vérifient la propriété $(P)$, alors tous les points de $\HH(k)$ vérifient la même propriété $(P)$. 
\end{proposition***}

\begin{proof}
D'après \cite[Théorème 12.1.1 et Théorème 12.1.6]{EGA4}, l'ensemble 
$$E:=\{z \in \XX(k) \ |\ \text{ la fibre schématique } \pi^{-1}(\pi(z)) \text{ ne vérifie pas $(P)$ au point $z$}\}$$ 
est un fermé de $\XX(k)$. Ce fermé est bien sûr $G' \times G$-stable par homogénéité. Le morphisme $\pi$ est $G$-invariant, affine et on a montré dans la preuve du lemme \ref{diagcom} que ${(\pi_* \OO_{\XX})}^G \cong \OO_{\HH}$. Donc, la famille universelle coïncide avec le morphisme de passage au quotient au sens de la théorie géométrique des invariants. Il s'ensuit que $\pi$ envoie les fermés $G$-stables de $\XX(k)$ dans des fermés de $\HH(k)$, et donc $\pi(E)$ est un fermé de $\HH(k)$. Le morphisme $\pi$ étant $G'$-équivariant, le fermé $\pi(E)$ est $G'$-stable. En particulier, d'après le lemme \ref{fixespoints}, ou bien il contient au moins un point fixe de $B'$ ou bien il est vide (auquel cas $E$ est aussi vide). On a donc deux alternatives:
\begin{itemize}
\item tous les points de $\HH(k)$ vérifient la propriété $(P)$, ou bien   
\item il existe au moins un point fixe de $B'$ dans $\HH$ qui ne vérifie pas la propriété $(P)$.
\end{itemize}
Et c'est précisément le résultat que nous avions annoncé.
\end{proof}

Motivés par cette proposition, nous allons étudier les propriétés géométriques des différents points fixes de $B'$ dans $\HH$. Tous les calculs qui suivent sont effectués à l'aide de \cite[Macaulay2]{Mac2}.

\section{Situation 2}  

\subsection{Cas \texorpdfstring{$\dim(V)=2$}{n=2}}
Avec les notations de la section \ref{dimTang}, on a $I=(f_1,\ldots,f_4,h_1,\ldots,h_4)$. Cet idéal n'est pas premier: il contient $x_{11} y_{11}$ mais ne contient ni $x_{11}$ ni $y_{11}$. On vérifie que l'idéal $I$ est égal à son radical, autrement dit, le schéma affine $Z_0=\Spec(k[W]/I)$ est réduit mais n'est pas irréductible. En fait, $Z_0$ est la réunion des quatre fermés irréductibles de dimension $4$ dans $W$ définis par les idéaux suivants de $k[W]$:
\begin{itemize} \renewcommand{\labelitemi}{$\bullet$}
\item $K_1:=(x_{11},x_{12},x_{21},x_{22}),$
\item $K_2:=(y_{11},y_{12},y_{21},y_{22}),$
\item $K_3:=(x_{11},x_{21},y_{22}y_{11}-y_{21}y_{12},y_{22}x_{12}+y_{12}x_{22},y_{21}x_{12}+y_{11}x_{22}),$
\item $K_4:=(y_{11},y_{21},x_{22}x_{11}-x_{21}x_{12},y_{22}x_{11}+y_{12}x_{21},y_{22}x_{12}+y_{12}x_{22}).$
\end{itemize}
Pour $i=1,\ldots,4$, on note $Z_0^{(i)}:=\Spec(k[W]/K_i)$ la composante irréductible de $Z_0$ définie par l'idéal $K_i$. On a $Z_0^{(1)}=\Hom(V,V_2)$ et $Z_0^{(2)}=\Hom(V_1,  V)$ qui sont des sous-espaces vectoriels de $W$. Et  
\begin{align*}
&Z_0^{(3)}=\left\{(\begin{bmatrix}
0 & v\end{bmatrix},u_2) \in W\ |\ \rg(u_2) \leq 1 \text{ et } u_2(v)=0\right\},\\
&Z_0^{(4)}=\left\{ \left(u_1,\begin{bmatrix}
0 \\
\phi
\end{bmatrix} \right) \in W\ |\ \rg(u_1) \leq 1 \text{ et } \phi \circ u_1 =0\right\},
\end{align*} 
qui sont des cônes affines singuliers uniquement en $0$. On vérifie que ces quatre composantes irréductibles sont normales, de Cohen-Macaulay (en fait $Z_0$ lui-même est de Cohen-Macaulay), qu'elles contiennent une orbite ouverte de $G$ mais qu'une orbite générale de $B'$ est de codimension 1 dans chacune d'elles. 
On s'intéresse maintenant aux intersections des composantes irréductibles de $Z_0$:
\begin{itemize} \renewcommand{\labelitemi}{$\bullet$}
\item $Z_0^{(1)} \cap Z_0^{(2)}=\{0\}$, 
\item $Z_0^{(2)} \cap Z_0^{(3)}=\left\langle  e_2^* \right\rangle \otimes V$, 
\item $Z_0^{(1)} \cap Z_0^{(4)}=V^* \otimes \left\langle  f_2 \right\rangle$, 
\item $Z_0^{(2)} \cap Z_0^{(4)}=\{ u_1 \in \Hom(V_1,V)\ |\ \rg(u_1) \leq 1 \}$ est un cône affine de dimension $3$ singulier uniquement en $0$,
\item $Z_0^{(1)} \cap Z_0^{(3i)}=\{ u_2 \in \Hom(V,V_2)\ |\ \rg(u_2) \leq 1 \}$ est un cône affine de dimension $3$ singulier uniquement en $0$,
\item $Z_0^{(3)} \cap Z_0^{(4)}=\{ (v,\phi) \in V \times V^* \ |\ \phi(v)=0\} $ est un cône affine de dimension $3$ singulier uniquement en $0$.
\end{itemize}
On en déduit que le lieu singulier de $Z_0$, noté $\Sing(Z_0)$, est la réunion de trois fermés irréductibles de codimension $1$ dans $Z_0$:
$$\Sing(Z_0)=(Z_0^{(1)} \cap Z_0^{(3)}) \cup  (Z_0^{(2)} \cap Z_0^{(4)}) \cup (Z_0^{(3)} \cap Z_0^{(4)}).$$

\subsection{Cas \texorpdfstring{$\dim(V)=3$}{n=3}}
Avec les notations de la section \ref{dimTang2}, on a $I=(f_1,\ldots,f_9,h_1,\ldots,h_9,s_1, \ldots, s_6, t_1, \ldots, t_6)$. Cet idéal n'est pas premier: il contient $x_{11} y_{11}$ mais ne contient ni $x_{11}$ ni $y_{11}$. On vérifie que $I$ est égal à son radical, autrement dit le schéma affine $Z_0=\Spec(k[W]/I)$ est réduit mais n'est pas irréductible. En fait, $Z_0$ est la réunion de huit fermés irréductibles de dimension $9$. Les composantes irréductibles de $Z_0$ sont: 

\begin{itemize} \renewcommand{\labelitemi}{$\bullet$}
\item $Z_0^{(1)}:=\Hom(V,V_2),$
\item $Z_0^{(2)}:=\Hom(V_1,V),$ 
\item $Z_0^{(3)}:=\{(\begin{bmatrix}
0 & 0 &v_3
\end{bmatrix},u_2) \in W \ |\ \rg(u_2) \leq 2 \text{ et } u_2(v_3)=0\},$
 \item $Z_0^{(4)}:=\left\{ \left(u_1,\begin{bmatrix}
0 \\
0 \\
\phi_3
\end{bmatrix} \right) \in W \ | \ \rg(u_1) \leq 2 \text{ et } \phi_3 \circ u_1=0\right\},$
\item $ Z_0^{(5)}:=\{(\begin{bmatrix}
0 & v_2 &v_3
\end{bmatrix},u_2) \in W \ |\ \rg(u_2) \leq 1 \text{ et } u_2(v_2)=u_2(v_3)=0\},$
\item $Z_0^{(6)}:=\left\{ \left(u_1,\begin{bmatrix}
0 \\
\phi_2 \\
\phi_3
\end{bmatrix} \right) \in W \ \mid \ \rg(u_1) \leq 1 \text{ et } \phi_2 \circ u_1=\phi_3 \circ u_1=0 \right\},$
\item $Z_0^{(7)}:=\left\{ \left(\begin{bmatrix}
0 & v_2 &v_3
\end{bmatrix},\begin{bmatrix}
\phi_1 \\
\phi_2 \\
\phi_3
\end{bmatrix} \right) \in W\ \middle| 
    \begin{array}{l}
      \text{ $v_2$ et $v_3$ sont colinéaires,}\\
      \text{ $\phi_2$ et $\phi_3$ sont colinéaires,}\\
       \forall i,j,\ \phi_i(v_j)=0, 
    \end{array}
 \right\},$ 
\item $Z_0^{(8)}:=\left\{\left( \begin{bmatrix}
v_1 & v_2 &v_3
\end{bmatrix},\begin{bmatrix}
0 \\
\phi_2 \\
\phi_3
\end{bmatrix} \right) \in W\ \middle| 
    \begin{array}{l}
    \text{ $v_2$ et $v_3$ sont colinéaires,}\\
      \text{ $\phi_2$ et $\phi_3$ sont colinéaires,}\\
       \forall i,j,\ \phi_i(v_j)=0, 
    \end{array}
\right\}.$ 
\end{itemize}
On vérifie que ces huit composantes irréductibles sont normales, de Cohen-Macaulay (mais pas $Z_0$ lui-même), qu'elles contiennent une orbite ouverte de $G$ mais qu'une orbite générale de $B'$ est de codimension 3. On vérifie ensuite que le lieu singulier de $Z_0$, noté $\Sing(Z_0)$, est de codimension $1$ dans $Z_0$ et que
$$\Sing(Z_0)=\bigcup_{i < j} (Z_0^{(i)} \cap Z_0^{(j)}).$$

\section{Situations 3 à 5}  

\subsection{Cas \texorpdfstring{$\dim(V)=2$}{n=2} dans la situation 3}

Avec les notations de la section \ref{espacTangO2}, on a $I=(f_1,f_2,f_3,h_1,h_2)$. Cet idéal n'est pas réduit: il contient $x_1^2$ mais ne contient pas $x_1$.
On vérifie que $I=(x_1,y_1,x_2^2) \cap (x_2,y_2,x_1^2)$, et donc $Z_0=\Spec(k[W]/I)$ est la réunion des deux "droites épaisses" définies par les idéaux $(x_1,y_1,x_2^2)$ et $(x_2,y_2,x_1^2)$.

\subsection{Cas \texorpdfstring{$\dim(V)=3$}{n=3} dans les situations 3 et 4}

On se place dans l'une des situations 3 ou 4 et pour $i=1,2$, on note $Z_i:=\Spec(k[W]/I_i)$. On vérifie que l'idéal $I_i$ n'est pas réduit mais que son radical est premier, ce qui signifie que $Z_i$ est un schéma irréductible mais non réduit. Dans la situation 4 (mais pas dans la situation 3), le schéma $Z_i$ est de Cohen-Macaulay. Enfin dans les situations 3 et 4, le schéma $Z_i$ muni de sa structure réduite est un cône affine singulier uniquement en l'origine, normal et de Cohen-Macaulay.

\subsection{Cas \texorpdfstring{$\dim(V)=4$}{n=4} dans la situation 5}

Avec les notations de la section \ref{tanggSp4}, on a $I=(f_1,\ldots,f_6,h_1,\ldots,h_5)$. Cet idéal est premier; il s'ensuit que $Z:=\Spec(k[W]/I)$ est une variété. On vérifie que $Z$ est normale, de Cohen-Macaulay et que son lieu singulier est de codimension $3$.